\documentclass[10pt]{article}
\usepackage[mathscr]{eucal}
\usepackage{graphicx}

\newcommand\DMO[2]{\DeclareMathOperator{#1}{#2}}

\usepackage{amsmath,amsthm,amsfonts,amscd,amssymb}
\usepackage[usenames,dvipsnames]{color}

\usepackage[all]{xy} 
\xyoption{color}

\usepackage{makeidx}

\usepackage{relsize}

\newcounter{enumitemp}
\newenvironment{enumeratecontinue}{
  \setcounter{enumitemp}{\value{enumi}}
  \begin{enumerate}
  \setcounter{enumi}{\value{enumitemp}}
}
{
  \end{enumerate}
}

\numberwithin{equation}{section}

\newcommand\pref[1]{(\ref{#1})}

\newcommand\ds\displaystyle

\theoremstyle{plain}

\newtheorem*{TheoremA}{Theorem A}
\newtheorem*{TheoremB}{Theorem B}

\newtheorem*{theorem*}{Theorem}

\newtheorem{theorem}{Theorem}[section]
\newtheorem{proposition}[theorem]{Proposition}
\newtheorem{lemma}[theorem]{Lemma}

\newtheorem{corollary}[theorem]{Corollary}
\newtheorem*{conjecture*}{Conjecture}
\newtheorem{conjecture}[theorem]{Conjecture}
\newtheorem{fact}[theorem]{Fact}
\newtheorem{LemmaAndDefinition}[theorem]{Lemma/Definition}
\newtheorem{reduction}[theorem]{Reduction}

\theoremstyle{definition}
\newtheorem{definition}[theorem]{Definition}

\DeclareMathOperator{\Out}{Out}
\DeclareMathOperator{\Aut}{Aut}
\DeclareMathOperator\Inn{Inn}
\DeclareMathOperator{\rank}{rank}
\DeclareMathOperator{\Length}{Length}

\DeclareMathOperator{\Stab}{Stab}
\DeclareMathOperator\diam{diam}
\DeclareMathOperator\image{image}

\DeclareMathOperator{\MCG}{\mathsf{MCG}}

\DeclareMathOperator\corank{corank}

\DeclareMathOperator\Krank{KR}
\DeclareMathOperator\KR{KR}

\newcommand\Id{\mathrm{Id}}

\newcommand\reals{{\mathbf R}}

\newcommand\Z{{\mathbf Z}}

\newcommand{\bdy}{\partial}
\newcommand\bdyinf{\bdy_\infty}

\newcommand{\from}{\colon}

\newcommand\suchthat{\bigm|}
\newcommand\inv{{-1}}
\newcommand\union{\cup}

\newcommand\abs[1]{\left| #1 \right|}

\newcommand\intersect{\cap}

\newcommand\meet{\wedge}

\newcommand\restrict{\bigm|}
\newcommand\subgroup{\leqslant}

\newcommand\A{\mathscr A}
\newcommand\B{\mathscr B}
\newcommand\F{\mathscr F}

\newcommand\FFS[1]{{\mathscr F}_{\!ell} #1}
\newcommand\FFSS{{\FFS  S}}
\newcommand\FST{{\FFS \, T}}
\newcommand\wtB{\wt\B}
\newcommand\wtBinf{\wt\B_\infty}

\newcommand\CFFS{\mathcal{F\!F}}

\newcommand\C{\mathcal C}

\renewcommand\H{{\mathcal H}}

\renewcommand\L{\mathcal L}
\newcommand\M{\mathcal M}

\newcommand\V{\mathcal V}

\newcommand\E{\mathcal E}

\newcommand\<\langle
\renewcommand\>\rangle
\newcommand\wh\widehat
\newcommand\disjunion\sqcup

\newcommand\act\curvearrowright

\newcommand\X{\mathcal{X}}
\newcommand\CV\X

\newcommand\LymanRTTTag{Lyman:RTT}

\newcommand\LymanCTTag{Lyman:CT}

\newcommand\BHTag{BestvinaHandel:tt}

\newcommand\BookOneTag{BFH:TitsOne}
\newcommand\BookOne{\cite{\BookOneTag}}

\newcommand\recognitionTag{FeighnHandel:recognition}

\newcommand\SubgroupsTag{HandelMosher:Subgroups}
\newcommand\Subgroups{\cite{\SubgroupsTag}}

\newcommand\FSHypTag{HandelMosher:FreeSplittingHyperbolic}
\newcommand\FSHyp{\cite{\FSHypTag}}
\newcommand\FSOneTag{HandelMosher:FreeSplittingHyperbolic}

\newcommand\RelFSHypTag{HandelMosher:RelComplexHyp}
\newcommand\FSRelHypTag\RelFSHypTag
\newcommand\RelFSHyp{\cite{\RelFSHypTag}}

\newcommand\RelFSHypTwoTag{HandelMosher:RelComplexHypII}
\newcommand\RelFSHypTwo{\cite{\RelFSHypTwoTag}}
\newcommand\RelFSTwoTag\RelFSHypTwoTag
\newcommand\RelFSTwo\RelFSHypTwo

\newcommand\RelFSHypThreeTag{HandelMosher:RelComplexHypIII}
\newcommand\RelFSHypThree{\cite{\RelFSHypThreeTag}}
\newcommand\RelFSThreeTag\RelFSHypThreeTag
\newcommand\RelFSThree\RelFSHypThree

\newcommand\ti {\tilde}

\DMO\Core{Core}
\DMO\ACore{{\hat{\mathcal{C}}}}
\DMO\truss{truss}
\newcommand\wt\widetilde
\renewcommand\O{{\mathscr O}}

\newcommand\FS{\mathcal{FS}}

\newcommand\collapsesto\succ
\newcommand\collapse\collapsesto
\newcommand\collapses\collapsesto
\newcommand\expandsto\prec
\newcommand\expand\expandsto
\newcommand\expands\expandsto

\newcommand\Vertices{{\mathcal V}}
\newcommand\Edges{{\mathcal E}}

\newcommand\TOAT{\emph{Two Over All Theorem}}
\newcommand\STOAT{\emph{Strong Two Over All Theorem}}

\newcommand\dt{{DT}}
\newcommand\DT\dt
\newcommand\SetTwo{{\{2\}}}

\newcommand\Over{{\textrm{o}}}
\newcommand\Fill{{\textrm{f}}}
\newcommand\CU{\Upsilon}

\newcommand\Fm{F_{\text{\!\tiny min}}}
\newcommand\Ana{\A_{\text{na}}}

\newcommand\relA{\emph{rel}~$\A$}

\title{Relative free splitting and free factor complexes III: \\
Stable translation lengths and filling paths
}

\author{Michael Handel and Lee Mosher}

\makeindex

\begin{document}

\maketitle

\begin{abstract}
Given $\phi$ in the relative outer automorphism group $\text{Out}(\Gamma;\mathscr{A})$, we describe quantitative relations between the stable translation length $\tau_\phi$ of $\phi$ acting on the relative free splitting complex $\mathcal{FS}(\Gamma,\mathscr{A})$ and the relative train track dynamics of~$\phi$. If $\phi$ has an orbit with a certain lower diameter bound $\Omega(\Gamma;\mathscr{A}) \ge 1$ then $\phi$ has a filling attracting lamination. Also, there is a positive lower bound $\tau_\phi \ge A(\Gamma;\mathscr{A}) > 0$ amongst all $\phi$ which have a filling attracting lamination. Both proofs rely on a study of \emph{filling paths} in a free splitting. These results are all new even for $\text{Out}(F_n)$. We also formulate conjectures and partial results concerning the action of $\phi$ on the complex of relative free factor systems $\CFFS(\Gamma;\mathscr{A})$.  
\end{abstract}

\section{Introduction to Part III}
\label{SectionIntro}
Given a group $\Gamma$ and a free factor system~$\A$, the relative outer automorphism group $\Out(\Gamma;\A)$ acts by isometries on the relative free splitting complex $\FS(\Gamma;\A)$. That complex is hyperbolic (see Part~I~\FSHyp), and so $\phi \in \Out(\Gamma;\A)$ has positive stable translation length $\tau_\phi$ if and only if each orbit map $n \mapsto \phi^n(T)$ is a quasigeodesic ($T \in \FS(\Gamma;\A)$), in which case $\phi$ is \emph{loxodromic}. One also says that $\phi$ is \emph{parabolic} if it has unbounded orbits but~$\tau_\phi=0$, and \emph{elliptic} if all orbits are bounded (hence~$\tau_\phi=0$). Our main results correlate the dynamics of $\phi$ acting on $\FS(\Gamma;\A)$ with the ``relative train track'' dynamics of $\phi$, particularly its attracting laminations:

\begin{TheoremA}
There exist \hbox{$A = A(\Gamma;\A)>0$} and $B=B(\Gamma;\A) >0$ such that for any $\phi \in \Out(\Gamma;\A)$, if $\phi$ has an attracting lamination which fills $\Gamma$ \relA, with associated expansion factor~$\lambda_\phi$, then $A \le \tau_\phi \le B \log(\lambda_\phi)$.
\end{TheoremA}

\begin{TheoremB}
There exists \hbox{$\Omega=\Omega(\Gamma;\A) > 0$} such that for any $\phi \in \Out(\Gamma;\A)$ the following are equivalent:
\begin{enumerate}
\item\label{ItemFillingLamExists}
$\phi$ has an attracting lamination that fills $\Gamma$ \relA;
\item\label{ItemActLox}
$\phi$ acts loxodromically on $\FS(\Gamma;\A)$;
\item\label{ItemActUnbdd}
$\phi$ acts with unbounded orbits on $\FS(\Gamma;\A)$;
\item\label{ItemActLargeOrbit}
Every orbit of the action of $\phi$ on $\FS(\Gamma;\A)$ has diameter $\ge \Omega$.
\end{enumerate}
\nobreak
As a consequence, no element of $\Out(\Gamma;\A)$ acts parabolically on $\FS(\Gamma;\A)$.
\end{TheoremB}
In these theorems and throughout this work, notations like $C=C(\Gamma;\A)$ mean that $C$ is a numerical constant depending only on two non-negative, integer valued numerical invariants of $\Gamma$ and $\A$, namely the cardinality $\abs{\A}$ and the $\corank(\A)=\corank(\Gamma;\A)$ (see Section~\ref{SectionSTOATTerms}). The constants $A$, $B$, $\Omega$ will themselves be defined by formulas involving constants drawn from \cite{\RelFSHypTag,\RelFSHypTwoTag}.

In the foundational case where $\Gamma=F_n$ and $\A=\emptyset$, the equivalences \pref{ItemFillingLamExists}$\iff$\pref{ItemActLox}$\iff$\pref{ItemActUnbdd} of Theorem~B were obtained in \cite{HandelMosher:FreeSplittingLox}; the hard implication \pref{ItemFillingLamExists}$\implies$\pref{ItemActLox} was proved using delicate relative train track methods. The methods we use here are new even in that foundational case, and the proofs here rely only the most basic relative train track methods. Also, there are powerful new conclusions of a quantitative nature, namely the additional equivalent item~\pref{ItemActLargeOrbit} in Theorem~B, and the lower bound in Theorem~A. 

Part~II of this work was devoted primarily to the proof and applications of the \emph{Two Over All Theorem}, including application to the upper bound of Theorem~A. Here in Part~III, our primary work is to produce the following: 
\begin{description}
\item[Section~\ref{SectionNoFillingLam}:] The constant $\Omega$ and a proof of the implication \pref{ItemActLargeOrbit}$\implies$\pref{ItemFillingLamExists} of Theorem~B.
\item[Section~\ref{SectionLowerBound}:] A proof of the lower bound of Theorem~A.
\end{description}
Applying these results, the proof of Theorem~A is complete, and we also have:

\begin{proof}[Proof outline for Theorem~B:] The implication \pref{ItemFillingLamExists}$\implies$\pref{ItemActLox} is an immediate consequence of the lower bound of Theorem~A. The implications \pref{ItemActLox}$\implies$\pref{ItemActUnbdd}$\implies$\pref{ItemActLargeOrbit} are immediate consequences of the definitions (using any value of $\Omega$). And in Section~\ref{SectionNoFillingLam} we construct an appropriate constant~$\Omega$ and use it to prove the the remaining implication \pref{ItemActLargeOrbit}$\implies$\pref{ItemFillingLamExists}.
\end{proof}

\paragraph{Section~\ref{SectionFFAandB}: Conjectures and partial results for $\CFFS(\Gamma;\A)$.} The dynamic classification of elements of $\Out(\Gamma;\A)$ acting on the complex of relative free factor systems $\CFFS(\Gamma;\A)$ was given by Guirardel and Horbez in \cite{GuirardelHorbez:Laminations}, generalizing the foundational case of $\Out(F_n)$ by Bestvina and Feighn \cite{BestvinaFeighn:FFCHyp} and the case of $\Out(F_n;\A)$ by Gupta \cite{Gupta:Loxodromic}. This general result says that $\phi \in \Out(\Gamma;\A)$ acts loxodromically on $\CFFS(\Gamma;\A)$ if and only if $\phi$ is fully irreducible \relA, and otherwise $\phi$ acts elliptically. For purposes of refining this classification, we formulate analogues of Theorems~A and~B in which the property ``$\phi$ has a filling lamination \relA'' is replaced with the property ``$\phi$ is fully irreducible \relA''. We prove several partial results regarding these analogues, some of which are easy to prove. Nore substantial partial results include the general $\CFFS(\Gamma;\A)$ analogue of the upper bound in Theorem~B; and the $\CFFS(F_n;\A)$ analogue of the implication \pref{ItemActLargeOrbit}$\implies$\pref{ItemFillingLamExists} of Theorem~B. The most substantial conjecture is the $\CFFS(\Gamma;\A)$ analogue of the lower bound of Theorem~A. We leave this conjecture untouched in \emph{all} cases, including the foundational case: 

\begin{conjecture*} There is a lower bound to the translation lengths of the actions on $\CFFS(F_n)$ of the fully irreducible elements of $\Out(F_n)$. 
\end{conjecture*}
See also the end of this introduction for a still stronger question underlying the lower bound of Theorem~A (for both $\FS(\Gamma;\A)$ and $\CFFS(\Gamma;\A)$), motivated by work of Bowditch \cite{Bowditch:tight}.

\subsection*{Methods of proof, based on filling paths in free splittings.} Concepts of ``filling'', and related concepts of ``free factor supports'', play an important role in the study of $\Out(F_n)$ and its generalizations $\Out(\Gamma;\A)$. 

For a simple example of filling consider any group element $\gamma \in \Gamma$. There is a unique smallest free factor $F$ of $\Gamma$ relative to~$\A$ such that $\gamma \in F$, and one can say that $F$ is the \emph{free factor support} of $\gamma$ \relA. To say that \emph{$\gamma$ fills $\Gamma$ \relA} means that $F=\Gamma$. There is another characterization of filling of a more ``concrete'' nature in that that it is expressed in terms of how $\gamma$ relates to each individual Grushko free splitting $T$ of $\Gamma$ \relA: $\gamma$ fills $\Gamma$ \relA\ if and only if the action of $\gamma$ on each $T$ has an invariant axis $L \subset T$ along which $\gamma$ acts as a translation, and the set of translates $\{\delta \cdot L \suchthat \delta \in \Gamma\}$ covers~$T$. More generally, given a lamination $\Lambda$ of $\Gamma$ relative to~$\A$ --- meaning a closed, $\Out(\Gamma;\A)$-invariant subset of the abstract line space $\B(\Gamma;\A)$ --- associated to $\Lambda$ is its free factor support $\F\Lambda$ of $\Lambda$ \relA, defined to be the minimal free factor system of $\Gamma$ \relA\ that ``supports'' $\Lambda$ \cite[Section 4.1.2]{\RelFSHypTwoTag}. One says that $\Lambda$ fills $\Gamma$ \relA\ if and only if $\F\Lambda = \{[\Gamma]\}$. This version of filling also has a more concrete version: $\Lambda$ fills $\Gamma$ \relA\ if and only if for every Grushko free splitting $R$ of $\Gamma$ \relA, the tree $R$ is covered by the set of lines in $R$ that realize leaves of $\Lambda$; see \cite[Section 4.1.2]{\RelFSHypTwoTag} for details, in particular \cite[Lemmas 4.2, 4.3]{\RelFSHypTwoTag}. Free factor supports of attracting laminations of elements of $\Out(F_n)$ play a central role in relative train track theory and its applications \cite{\BookOneTag,HandelMosher:FreeSplittingLox,HandelMosher:Subgroups,HandelMosher:BddCohomology}. For the general case of $\Out(\Gamma;\A)$, the concrete filling criterion for attracting laminations is applied here in the opening paragraphs of~Section~\ref{SectionEnablingProjection}.

The germ of this work was the following question that we formulated while first pondering the existence of a lower bound as in Theorem~A. Consider $\phi \in \Out(\Gamma;\A)$, and assume that $\phi$ has an attracting lamination $\Lambda$ that fills $\Gamma$ \relA; also, assume for simplicity that there exists an EG-aperiodic train track representative $F \from T \to T$ of $\phi$ defined on a Grushko free splitting $T$ of $\Gamma$ \relA\ (see \cite[Section 4.3.3]{\RelFSHypTag}). A~bi-infinite line $L \subset T$ represents a generic leaf of~$\Lambda$ if and only if~$L$ can be written as a nested union of ``natural iteration tiles'' of $F$, meaning paths of the form $F^k(E)$ for natural edges $E \subset T$  (see \cite[Section 4]{\LymanCTTag}, or see Section~\ref{SectionTTsAndLams} here for a quick summary). We asked ourselves: 
\begin{description}
\item[Vague Question:] What properties of the tiles $F^k(E)$ are correlated with the hypothesis that $\Lambda$ fills~$\Gamma$ \relA? Is there some kind of useful ``filling'' criterion for paths in $T$, such that $\Lambda$ fills if and only if the paths $F^k(E)$ fill $T$ for all sufficiently large $k$?
\end{description}
We found such a criterion (see Proposition~\ref{PropUniformFilling}, discussed below), expressed using a property roughly modeled on the ``concrete'' filling criteria discussed above. The idea is to ask how a path $\alpha \subset T$ relates to free splittings $S$ that are expansions of $T$: 
\begin{description}
\item[\qquad Filling paths (Definition~\ref{DefinitionFillingPaths})] In a free splitting $T$ of $\Gamma$ \relA, to say that a path $\alpha \subset T$ \emph{fills $T$} means that for any collapse map $S \mapsto T$ defined on a free splitting $S$ of $\Gamma$ \relA, letting $\tilde\alpha \subset S$ denote the natural lift of $\alpha$ to $S$ (Definition~\ref{DefinitionLiftingPaths}), the path $\tilde\alpha$ has an interior crossing of some edge in the orbit of every natural edge of $S$. More precisely: for every natural edge $E \subset S$ there exists $\gamma \in \Gamma$ such that $\gamma \cdot E$ is contained in the interior of $\alpha$.
\end{description}
The lifted path $\tilde\alpha$ itself may be characterized as the smallest path in $S$ (with respect to inclusion) whose image under the collapse map $S \mapsto T$ is $\alpha$. The earlier examples of filling concepts each had two equivalent versions: a ``concrete'' criterion and a ``free factor system'' criterion. Beside the ``concrete'' criterion for filling paths expressed just above (taken from Definition~\ref{DefinitionFillingPaths}), in Proposition~\ref{PropFillingPath} we state a criterion expressed in terms of a certain free factor system \relA\ denoted $\F[\alpha;T]$ that we call the \emph{filling support} of $\alpha$; see just below for a brief discussion.

Our two main technical results here regarding filling paths are: the \STOAT~\ref{TheoremStrongTwoOverAll}, which is applied in the proof of the implication \pref{ItemActLargeOrbit}$\implies$\pref{ItemFillingLamExists} of Theorem~B; and Proposition~\ref{PropUniformFilling} which answers the question above, and which is applied in finding the lower bound for Theorem~A. Although as said above our first progress on this project was Proposition~\ref{PropUniformFilling}, in retrospect the exposition is cleaner if one leads off with the \STOAT, which we do in Section~\ref{SectionSTOAT}.

\paragraph{The \STOAT\ (Section~\ref{SectionSTOAT}).} Recall the \TOAT\ of \RelFSHypTwo\ in its uniterated form, which says that for any foldable map $f \from S \to T$ between free splittings of $\Gamma$ \relA, if the distance $d_\FS(S,T)$ exceeds a certain constant then the free splitting $S$ has two natural edges in distinct orbits whose $f$-images in $T$ each cross a representative of every edge orbit of~$T$. The \STOAT\ says that if $d_\FS(S,T)$ exceeds a certain larger constant $\Theta=\Theta(\Gamma;\A)$ then $S$ has two natural edges in distinct orbits whose $f$-images in~$T$ each fill~$T$; there is also an iterated form of the theorem which gives conclusions when $d_\FS(S,T)$ exceeds a multiple $K\Theta$ of the constant~$\Theta$. 

In broad outline the two theorems have similar proofs: one factors $f$ as a Stallings fold path $S=S_0 \xrightarrow{f_1} \cdots\xrightarrow{f_M} S_M=T$, and for each natural edge $E \subset S$ one studies the evolution of the natural tiles $\alpha_l = f^0_l(E) \subset S_l$ as $l$ increases from~$0$ to~$M$. In the proof of the \TOAT, the game was over once the ``covering forest'' $\beta(\alpha_l;S_l) = \bigcup_{\gamma \in \Gamma} \gamma \cdot \alpha_l$ covered the whole of~$S_l$. But that is barely the start of the game in the \STOAT. We study certain decompositions of $\beta(\alpha_l;S_l)$ called ``protoforests'' (Definition~\ref{DefProtoforest}), which are $\Gamma$-invariant collections of non-overlapping subtrees called the ``protocomponents'' of the protoforest. The overlap protoforest $\beta^\Over(\alpha_l;S_l)$ (Definition~\ref{DefFineCoveringForest}) captures patterns of overlap amongst the translates of $\alpha_l$. The protocomponent stabilizers of $\beta^\Over(\alpha_l;S_l)$ need not be free factors of $\Gamma$ \relA; see examples and counterexamples in Section~\ref{SectionFillingCriterion}. The filling support $\F[\alpha_l;S_l]$ is indirectly defined as the smallest free factor system \relA\ that supports the protocomponent stabilizers of the overlap protoforest $\beta^\Over(\alpha_l;S_l)$ (Definition~\ref{DefinitionFineFFS}). Section~\ref{SectionCriterionProof} contains the proof of the filling criterion mentioned earlier. The key step occurs in Lemma~\ref{LemmaCheckDeltaConnected} which is a delicate construction of the \emph{filling protoforest} $\beta^\Fill(\alpha_l;S_l)$, whose protocomponent stabilizers \emph{do}~give the filling support~$\F[\alpha_l;S_l]$, as summarized at the end of Section~\ref{SectionFillingProtoforest}. After a further study of filling protoforests in Sections~\ref{SectionBlowingUpOnlyIf}--\ref{SectionBlowupProof}, focussing on their behavior under expansions of a free splitting, in Section~\ref{SectionProofStrongTOA} we prove the \STOAT\ by studying the evolution of the filling protoforests $\beta^\Fill(\alpha_l;S_l)$ as~$l$ increases.

\paragraph{The implication \pref{ItemActLargeOrbit}$\implies$\pref{ItemFillingLamExists} of Theorem~B (Section~\ref{SectionNoFillingLam}).} 
The proof will produce a certain constant $\Omega=\Omega(\Gamma;\A)$ and then use it to show that if all orbits of $\phi \in \Out(\Gamma;\A)$ have diameter $\ge\Omega$ then $\phi$ has a filling lamination. Carrying this out will require tools from relative train track theory (from \cite{\LymanCTTag}, from the prequels \cite{\RelFSHypTag,\RelFSHypTwoTag}, and from earlier sections of this paper), and a certain amount of the work in this section is to provide justifications for applying each of these tools. These justifications depend on assuming two constraints on the value of $\Omega$, in the form of two lower bounds $\Omega_1$, $\Omega_2$. The value of $\Omega$ is then set to be the maximum of those lower bounds. Here is a brief outline of the proof; for a fuller and more formal outline see Section~\ref{SectionOmegaAndProof}.

In Section~\ref{SectionFirstConstraint} the tool we use is \cite[Proposition 4.24]{\RelFSHypTwoTag}. That result, applied in conjunction with the first constraint $\Omega \ge \Omega_1 = 5$, let's us specify a free factor system~$\B$ \relA\ that is fixed by $\phi$, a train track representative $F \from T \to T$ of $\phi$ rel~$\B$ that is ``EG-aperiodic'' (short for ``exponentially growing aperiodic''), an attracting lamination $\Lambda$ of $\phi$ that is associated to~$F$, and a train track axis for $\phi$ in $\FS(\Gamma;\A)$ which is obtained by suspending~$F$. At each free splitting occupying a position along this axis, every generic leaf of $\Lambda$ is realized by a unique line in that free splitting, and each such line can be written as a nested union of tiles obtained by pushing forward edges from earlier free splittings along the axis (see~Conclusion (2c) of Section~\ref{SectionFoldAxisReview}).

In Section~\ref{SectionSecondConstraint}, one tool we use is the \STOAT. The hypothesis of that theorem is verified (uniformly along the axis) by applying the second constraint $\Omega \ge \Omega_2$. From the conclusion of that theorem we obtain (uniformly) filling tiles at every position along the axis. Using these tiles in conjunction with another tool --- namely, concepts of free splitting units from \cite[Section 4.5]{\RelFSHypTag} --- we then prove that the axis is a quasigeodesic line in $\FS(\Gamma;\A)$. Note that this proves item~\pref{ItemActLox} of Theorem~B as a step along the way to verifying item~\pref{ItemFillingLamExists}; nonetheless the remainder of that verification does not use \emph{only} item~\pref{ItemActLox} but continues to use the machinery built up by application of the lower bounds $\Omega \ge \max\{\Omega_1,\Omega_2\}$.

The remainder of the proof is found in Section~\ref{SectionEnablingProjection}. To show that $\Lambda$ fills it suffices to show that, in any Grushko free splitting $R$ of $\Gamma$ \relA\ and for any bi-infinite line in $R$ that realizes a generic leaf of~$\Lambda$, that line crosses a representative of every edge orbit of $R$ (see Corollary~\ref{CorollaryFillingLine}). We do this with the aid of a tool from \cite{\RelFSHypTag}, namely ``projection diagrams''. Given $R$, we carefully choose a finite subpath of the fold axis; the choice is informed by the \emph{Quasi-Closest Point Property} (a consequence of the Masur--Minsky axioms that is derived in \cite[Section 5.1]{\RelFSHypTag}) in conjunction with the fact that the axis is a quasigeodesic line. We then take a projection diagram from $R$ to the chosen axis subpath, and we use filling tiles in that diagram to prove the desired crossing property for generic leaves in $R$.

\smallskip

But there is a complication. When applying \cite[Proposition 4.24]{\RelFSHypTwoTag} in Section~\ref{SectionFirstConstraint} to obtain the train track representative $f \from T \to T$ of $\phi$, it may happen that the ``elliptic'' free factor system of the free splitting $T$ --- meaning the free factor system~$\F$ given by the vertex stabilizers of $T$ --- is strictly larger than the free factor system~$\A$. In order to even apply \cite[Proposition 4.24]{\RelFSHypTwoTag}, we must therefore choose $\F$ first, so as to be a maximal, non-filling, $\phi$-invariant free factor system \relA. It follows that $\Lambda$ is an attracting lamination of $\phi$ \emph{relative to~$\F$}. But, what we need is an attracting lamination of $\phi$ \emph{relative to~$\A$} --- and that's not what we have yet, in the case that $\F$ is strictly larger than~$\A$. So there is still some work to do in order to somehow identify $\Lambda$ with \emph{one of} the attracting laminations of~$\phi$ relative to~$\A$. This work is done in Sections~\ref{SectionDT} and~\ref{SectionRealizingGeneralLines}, applying work of Dowdall and Taylor \cite[Section 3]{DowdallTaylor:cosurface}; and then see Section~\ref{SectionTTsAndLams} for a summary of applications to the theory of train track maps and attracting laminations.

\paragraph{The lower bound of Theorem~A (Section~\ref{SectionLowerBound}).} 

In this section we find $A>0$ and prove if $\phi \in \Out(\Gamma;\A)$ acts loxodromically on~$\FS(\Gamma;\A)$ --- and so, by Theorem~B,~$\phi$ has an attracting lamination $\Lambda$ that fills $\Gamma$ \relA\ --- then $\tau_\phi \ge A$. Here is an outline.

Consider an EG-aperiodic train track representative $F \from T \to T$ (with respect to some maximal, non-filling, $\phi$-invariant free factor system~$\F$ \relA), and so each generic leaf of the lamination $\Lambda$ as a nested union of iteration tiles of $F$. In Lemma~\ref{LemmaQAPQuant}~\pref{ItemSTLLowerBound}, by applying concepts of component free splitting units, we obtain a positive lower bound for $\tau_\phi$ expressed by a formula depending only on $\Gamma$ and~$\A$, \emph{except} for the occurrence of two numerical quantities that depend putatively on the choice of $F$: a \emph{PF-exponent}~$\kappa$ (every entry of the transition matrix of $F^\kappa$ is $\ge 4$); and a \emph{filling exponent} $\omega$ (for every edge $E$ the tile $F^\omega(E)$ fills $T$). 

The desired lower bound for Theorem~A is therefore obtained if we can choose $F$ having a \emph{uniform} PF exponent and filling exponent, independent of all other choices. For this purpose we shall choose the train track representative $F \from T \to T$ to be defined on a natural free splitting~$T$. Such an $F$ was constructed in \cite{\BHTag} for the special case of $\Out(F_n)$, and in \cite{FrancavigliaMartino:TrainTracks} for the general case, using the Lipschitz semi-metric on the outer space of $\Gamma$ \relA. See Section~\ref{SectionNaturalTTRep} for details, including a further construction in the general case which mimics the special case construction of \cite{\BHTag}. An elementary construction of a uniform PF-exponent $\kappa_0(\Gamma;\A)$, for natural EG-aperiodic train track representatives of $\phi$, is found in Section~\ref{SectionUniformPF}.

The heart of the proof is thus Proposition~\ref{PropUniformFilling}, proved in Section~\ref{SectionUniformFilling}, which produces a uniform filling exponent $\omega_0=\omega_0(\Gamma;\A)$ as long as $F \from T \to T$ is defined on a natural free splitting~$T$. The scheme of the proof is similar to that of the \STOAT, which is a study of the evolution of filling support of paths under application of fold maps along a fold sequence. In Proposition~\ref{PropUniformFilling}, by contrast we study the evolution of filling support of paths under application of iterates of $F^{\kappa_0}$ --- namely, the filling supports $\F[\eta_m;T]$ of the sequence of paths $\eta_m = F^{m\kappa_0}(E)$, for any natural edge $E \subset T$. We first use the combinatorics of filling protoforests to prove that the sequence of free factor systems $\F[\eta_m;T]$ must increase strictly until, at some value $m=M$ having an upper bound $M \ge M_0$ depending only on $\Gamma$ and $\A$, the sequence stabilizes, that is, $\F[\eta_{M};T] = \F[\eta_{M+1};T] = \F[\eta_{M+2};T] = \cdots$. We then deduce that the stable value $\F[\eta_M;T]$ must be the full free factor system $[\Gamma]$, based on the fact that, up to translation, the tiles $\eta_m$ may be nested so that their union is a generic leaf of the filling attracting lamination~$\Lambda$. We thus obtain the uniform filing exponent $\omega_0 = M_0 \kappa_0$.

\subparagraph{Remarks and Questions.} The lower bound of Theorem~A was inspired by an analogous result of Bowditch \cite{Bowditch:tight} regarding the action of the mapping class group $\MCG(S)$ of a finite type surface $S$ on its curve complex $\C(S)$. It was known from \cite{MasurMinsky:complex1} that $\C(S)$ is Gromov hyperbolic and that the elements of $\MCG(S)$ acting loxodromically on $\C(S)$ are precisely the pseudo-Anosov mapping classes. Bowditch proved that the set of stable translation lengths $\tau_\phi$ of pseudo-Anosov elements $\phi \in \MCG(S)$ acting on $\C(S)$ has a positive lower bound. In fact he proved something much stronger, using an analysis of tight geodesics in $\C(S)$: there is an integer $M \ge 1$ depending only on the genus and number of punctures of $S$ such that for each pseudo-Anosov $\phi \in \MCG(S)$, its power $\phi^M$ has a bi-infinite geodesic axis in $\C(S)$ \cite[Theorem 1.4]{Bowditch:tight}; it follows that $\tau_\phi$ is a positive rational number with denominator~$M$. 

Our methods in this work are somewhat softer than the methods of  \cite{Bowditch:tight}, in that we use only quasigeodesics; we do not consider geodesics in $\FS(\Gamma;\A)$ nor in $\CFFS(\Gamma;\A)$. This motivates the following (Questions 2 and 3 in the \emph{Overview} \cite{HandelMosher:RelHypComplexIntro}):

\medskip\textbf{Questions:} Do there exist integers $m=m(\Gamma;\A) \ge 1$ and $n=n(\Gamma;\A) \ge 1$ such that for each $\phi \in \Out(\Gamma;\A)$ the following hold?
\begin{itemize}
\item If $\phi \in \Out(\Gamma;\A)$ has a filling lamination rel~$\A$ then the action of $\phi^m$ on $\FS(\Gamma;\A)$ has a bi-infinite geodesic axis.
\item If $\phi \in \Out(\Gamma;\A)$ is fully irreducible rel~$\A$ then the action of $\phi^n$ on $\CFFS(\Gamma;\A)$ has a bi-infinite geodesic axis.
\end{itemize}

\paragraph{}

\setcounter{tocdepth}{3}
\tableofcontents

\section{Filling paths and the \STOAT} 
\label{SectionSTOAT}

The \TOAT, in its ``uniterated'' form found in \cite[Section 5.1]{\RelFSHypTwoTag}, says the following. Consider a foldable map $f \from S \to T$ between free splittings in $\FS(\Gamma;\A)$. Assume, as a hypothesis, that the distance in $\FS(\Gamma;\A)$ between $S$ and $T$ satisfies a certain lower bound $d_\FS(S,T) \ge \Delta = \Delta(\Gamma;\A)$. Under this assumption, there exist two natural edges $E_i \subset S$ ($i=1,2$) in distinct $\Gamma$ orbits such for every natural edge $E' \subset T$ the path $f(E_i)$ crosses some translate $\gamma \cdot E'$ of $E'$ ($\gamma \in \Gamma$). For the precise definition of ``crossing'' see \cite[Definition 2.20]{\RelFSHypTwoTag}.

\paragraph{Vague Question:} The conclusions above control only how $\alpha$ crosses natural edges in~$T$. Can one, in addition,  also control how $\alpha$ crosses natural vertices of~$T$?

\paragraph{An example.} Suppose that $T$ is a free splitting of $\Gamma$ \relA\ with a vertex $v$ of valence~$4$ having trivial stabilizer. For example, with $\Gamma=F_2$ and $\A=\emptyset$ one can take $T$ to be the universal cover of a rank~$2$ rose. Let $e_1,e_2,e_3,e_4$ denote the four oriented edges with initial vertex $v$; choosing two out of those four gives a total of six different (nondegenerate) turns at $v$. For example one could label the loops of the rose as $a,b$, and choose the labels so that $e_1,e_2,e_3,e_4$ are initial directions of lifts of $a$, $b$, $\bar a$, and $\bar b$ respectively. Consider a path $\alpha \subset T$ and its collection of translates $\{g \cdot \alpha \suchthat g \in \Gamma\}$, and consider which of the six turns at $v$ are taken by that collection of paths. Suppose that the only turns at $v$ that are taken by these translates $g \cdot \alpha$ are the two turns $\{e_1,e_2\}$ and $\{e_3,e_4\}$; for example one could take $\alpha$ to be a lift of the path~$ab$. Intuitively this path $\alpha$ fails to ``fill'' $T$ in a certain visceral sense: there is an expansion $T \expands U$ in which the vertex $v \in \Vertices(T)$ expands to an edge $E \in \Edges(U)$ with endpoints $v',v''$, such that the directions $\{e_1,e_2\}$ are attached to $v'$ and $\{e_3,e_4\}$ are attached to $v''$; it follows that every translate in $U$ of the natural lift $\ti\alpha \subset U$ is disjoint from the interior of this expanded edge~$E$. In our example, $U$ would be the universal cover of a $\theta$-graph expansion of the rank~$2$ rose, and the projection of $\tilde \alpha \subset U$ to the $\theta$-graph would be a path that does not cross the expanded edge.

\medskip

What is only hinted at in the above example, and what we will supply in Definition~\ref{DefinitionFillingPaths} below, is a definition of what it means for a finite path to \emph{fill} a free splitting, which is then used to formulate the \STOAT. The proof of that theorem will take up the rest of Section~\ref{SectionSTOAT}.

\subsection{Statement of the \STOAT}
\label{SectionSTOATStatement}

\subsubsection{Terminology review.} 
\label{SectionSTOATTerms}
We assume that the reader is familiar with fundamental concepts regarding free factor systems, free splittings, relative free splitting complexes, and Stallings fold paths. The terminologies and notations that we use are laid out in \RelFSHypTwo, where one can also find more detailed references: 
\begin{description}
\item[\protect{\cite[Section 2.1]{\RelFSHypTwoTag}}] free factor systems, relative outer automorphism groups;
\item[\protect{\cite[Section 2.2]{\RelFSHypTwoTag}}] (2.2.1) graphs, the graph complement operator $G \setminus H$, maps of graphs (always assumed to be PL), paths in graphs (trivial and nontrivial); (2.2.2) trees and their paths, rays, and lines; 
%\marginparLee{Go back to \cite[Section 2.2.3]{\RelFSHypTwoTag} and check the usage of ``edgelets'' to make sure it is consistent with the new hierarchy here; always natural edges and (ordinary) edges; edgelets only for a further subdivision.}
(2.2.3) free splittings, natural vertices and edges, edgelets; (2.2.4) maps among free splittings, twisted equivariant maps, collapse maps; (2.2.5) relative free factor systems, relative free splittings, relative free factors, Kurosh rank; (2.2.6) relative free splitting complexes
\item[\protect{\cite[Section 2.3]{\RelFSHypTwoTag}}] foldable maps, foldable paths, folds, fold paths, Stallings fold theorem.
\end{description}
We quickly review a few highlights. 

\smallskip\textbf{Free factor systems.} A \emph{free factor system} of $\Gamma$ is a set of the form $\A = \{[A_1],\ldots,[A_I]\}$ (possibly empty, \hbox{i.e. $I=0$)} such that there exists a free factorization $\Gamma = A_1 * \ldots * A_I * B$ where $B$ is free of finite rank (possibly rank~$0$), each $A_i$ is nontrivial $(1 \le i \le I)$, and $[A]$ denotes the conjugacy class of $A$ in $\Gamma$. We denote $\abs{\A}=I$. We also denote $\corank(\A) = \rank(B)$, which is well-defined independent of $B$, equal to the rank of the free group which is the quotient of $\Gamma$ by the subgroup normally generated by $A_1 \union\cdots\union A_I$. The \emph{full} free factor system is $\A=\{[\Gamma]\}$; we use various synonyms for ``full'', such as ``filling'' and ``improper''.

Inclusion of free factors induces a partial order $\A \sqsubset \B$ on the set of free factor systems of~$\Gamma$. The natural action of $\Out(\Gamma)$ on free factor systems preserves this partial order, and the stabilizer of a free factor system~$\A$ under this action is a subgroup of $\Out(\Gamma)$ denoted $\Out(\Gamma;\A)$. The pre-image of $\Out(\Gamma;\A)$ under the natural homomorphism $\Aut(\Gamma) \mapsto \Out(\Gamma)$ is denoted $\Aut(\Gamma;\A)$, and it coincides with the subgroup of $\Aut(\Gamma)$ whose action on $\Gamma$ preserves the collection of subgroups $\{A \subgroup \Gamma \suchthat [A] \in \A\}$. 

Given one free factor system~$\A$ of~$\Gamma$, a \emph{free factor system of $\Gamma$ \relA} is a free factor system~$\F$ such that $\A \sqsubset \F$. A \emph{free factor} of $\Gamma$ \relA\ is a subgroup $F \subgroup \Gamma$ such that for some free factor system $\F$ \relA\ we have $[F] \in \F$; furthermore, we say that $F$ and $[F]$ are \emph{atomic} if $[F] \in \A$, otherwise $F$ and $[F]$ are \emph{nonatomic}. For any free factor $F$ of $\Gamma$ \relA, the \emph{restriction of~$\A$ to $F$} is a free factor system of the group $F$ denoted $\A \mid F$, consisting of the $F$-conjugacy classes of all free factors $A \subgroup F$ such that $[A] \in \A$; we note that the atomic/nonatomic dichotomy for $F$ (as a free factor rel~$\A$) is equivalent to the filling/nonfilling dichotomy of the free factor system $\A \mid F$ in the group $F$. The action of $\Out(\Gamma)$ on all free factor systems of $\Gamma$ restricts to an action of $\Out(\Gamma;\A)$ on free factor systems \relA.

%in Section~\ref{} under the heading \emph{Free splittings and their maps}

\smallskip\textbf{Free splittings, their maps, and their visible free factor systems.} A \emph{free splitting} of $\Gamma$ is a minimal simplicial action $\Gamma \act S$ on a simplicial tree with finitely many vertex and edge orbits and with trivial edge stabilizers. The given graph structure on $S$ is a subdivision of the unique \emph{natural} graph structure which has for its vertices those points of $S$ having either nontrivial stabilizer or valence~$\ge 3$. Sometimes the given graph structure on $S$ is further subdivided, in which case we use the term \emph{edgelets} to refer to the edges of the subdivided structure; proceeding from coarser to finer graph structures we thus have \emph{natural edges}, which decompose into \emph{edges}, and which further decompose into \emph{edgelets}. A \emph{nondegenerate subgraph} $\sigma \subset S$ is a proper, $\Gamma$-invariant subgraph with respect to some invariant subdivision of $S$, such that no component of $\sigma$ is a point, equivalently $\sigma$ is a union of edgelets of the subdivision; we allow $\sigma$ to be empty. The ``elliptic'' free factor system $\FFSS$ consists of the conjugacy classes of nontrivial vertex stabilizers for the action of $\Gamma$ on~$S$. To~say that $S$ is a \emph{Grushko} free splitting of $\Gamma$ \relA\ means that $\FFSS = \A$. 

To say that a map $f \from S \to T$ of free splittings of $\Gamma$ is \emph{$\Phi$-twisted equivariant} where $\Phi \in \Aut(\Gamma)$ means that $f(\gamma \cdot x) = \Phi(\gamma) \cdot f(x)$ for all $x \in S$, $\gamma \in \Gamma$; and $f$ is \emph{equivariant} if this holds with $\Phi = \Id$. All maps $f$ of free splittings are assumed to be equivariant, unless twisted equivariance is explicitly stated; and in either case $f$ is always assumed to be PL. \emph{Equivalence} of free splittings is defined by existence of an equivariant PL homeomorphism. In this section we will use a formal notation $[S]$ for the equivalence class of $S$; later on we usually abuse notation and let $S$ stand for its own equivalence class. The group $\Out(\Gamma)$ acts from the right on equivalence classes, with $[S] = [T] \cdot \phi$ if and only if there exists a $\Phi$-twisted equivariant PL homeomorphism $f \from S \to T$ for some $\Phi \in \Aut(\Gamma)$ representing~$\phi \in \Out(\Gamma)$.

The notation $f \from S \xrightarrow{\<\sigma\>} T$ represents a  \emph{collapse map} with \emph{collapse forest} $\sigma$, meaning that $\sigma \subset S$ is a subgraph of $S$, and $f$ is a quotient map that collapses each component of $\sigma$ to a point. The collapse map $f$ is \emph{trivial} if and only if $\sigma$ does not contain any natural edge of $S$, which occurs if and only if $S$ and $T$ are equivalent free splittings. In general the existence of a collapse map $S \mapsto T$ defines a well-defined partial order $[S] \collapse [T]$ on the set of equivalence classes of free splittings of~$\Gamma$, and two free splittings $S$ and $T$ are equivalent if and only if $[S] \collapse [T] \collapse [S]$. This \emph{collapse partial order} is respected by the action of $\Out(\Gamma)$ on equivalence classes of free splittings.

Given a free splitting $S$ of $\Gamma$ and a $\Gamma$-invariant subforest $\sigma \subset S$, as $s$ varies over the components of $\sigma \union \V(S)$ for which the stabilizer subgroup $\Stab(s)$ is nontrivial, the conjugacy classes of those subgroups form a free factor system of $\Gamma$ denoted $\F[\sigma]$; furthermore, under the associated collapse map $S \xrightarrow{\langle\sigma\rangle} T$ we have $\FFSS \sqsubset \F[\sigma] = \FST$ (see \cite[Lemma 2.3 and Definition 2.4]{\RelFSHypTwoTag}). A free factor system $\F$ of $\Gamma$ is said to be \emph{visible} in a free splitting $S$ if $\F=\F[\sigma]$ for some $\Gamma$-invariant subforest~$\sigma \subset S$ \cite[Definition 2.5]{\RelFSHypTwoTag}; if $\F$ is visible in $S$ then $\FFSS \sqsubset \F$. For any nested pair of free factor systems $\A \sqsubset \B$ there exists a Grushko free splitting $S$ of $\Gamma$ \relA\ in which $\B$ is visible \cite[Lemma 3.1]{\RelFSHypTag}.

\smallskip\textbf{The free splitting complex.}
Consider a free factor system~$\A$ of $\Gamma$. For any free splitting $T$ of $\Gamma$, to say that \emph{$T$ is a free splitting} \relA\ means that $\A \sqsubset \FST$, i.e. for every subgroup $A \subgroup \Gamma$ such that $[A] \in \A$, the restricted action of $A$ on $T$ fixes a vertex. The \emph{free splitting complex of $\Gamma$ \relA} is the ordered simplicial complex that is associated to the collapse partial ordering on equivalence classes of free splittings \relA. The action $\Out(\Gamma;\A) \act \FS(\Gamma;\A)$ is induced by the action of $\Out(\Gamma)$ on the collapse partial ordering.

\smallskip\textbf{Foldable maps and fold sequences.}
A \emph{foldable} map $f \from S \to T$ is an equivariant map of free splittings such that every point of $S$ is contained in the interior of some embedded arc $\alpha \subset S$ for which the restriction $f \restrict \alpha$ is injective. If this holds then, furthermore, $f$ is a \emph{fold} if it is defined by identifying initial segments of a pair of oriented natural edges having a common initial vertex. If $f$ is a fold map between free splittings of $\Gamma$ \relA\ then in $\FS(\Gamma;\A)$ we have a distance bound $d_\FS([S],[T]) \le 2$.

Consider a sequence of maps of the form \hbox{$\cdots \xrightarrow{f_i} T_i \xrightarrow{f_{i+1}} \cdots \xrightarrow{f_j} T_j \xrightarrow{f_{j+1}} \cdots$} with compositions denoted in ``index contraction'' notation $f^i_j = f_j \circ \cdots f_{i+1} \from T_i \to T_j$ whenever ~$i<j$. To say this is a \emph{foldable sequence} means that every $f^i_j$ is foldable; and if so then furthermore it is a \emph{fold sequence} if each $f_i$ is a fold map. We often regard a fold sequence as a path in $\FS(\Gamma;\A)$, where each step of that path consists of an edge path of length $\le 2$.

\smallskip\textbf{Notational convention:} In what follows we use the notation $[I,\ldots,J]$ for the \emph{integer interval} with endpoints $I,J \in \Z$: if $I \le J$ then $[I,\ldots,J] = \{i \in \Z \suchthat I \le i \le J\}$; and in general $[I,\ldots,J] = [J,\ldots,I]$. For example, $[I,\ldots,J]$ will be commonly used as an index set for a fold subpath $T_I \xrightarrow{f_{i+1}} \cdots \xrightarrow{f_J} T_J$ of a fold path as denoted in the previous paragraph.

\subsubsection{Filling paths}
\label{SectionFillingTiles}

Recall from \cite[Section 2.1]{\RelFSHypTwoTag} that \emph{paths} in a free splitting can be either nontrivial or trivial: a \emph{nontrivial path} is a concatenation of one or more edges without backtracking; and a \emph{trivial path} takes constant value at a single vertex. The unadorned terminology \emph{path} refers to both nontrivial and trivial paths, but beware that when the ``nontriviality'' qualifier is clear from the context we will sometimes drop it, as in the titles of the next couple of definitions.

\begin{definition}[Lifting paths through collapse maps]
\label{DefinitionLiftingPaths}
Consider two free splittings $U,T$ of $\Gamma$ relative to~$\A$ and a collapse map $\pi \from U \xrightarrow{\<\sigma\>} T$. After choosing subdivisions of $U$ and $T$ with respect to which $\sigma$ is a subcomplex and $\pi$ is simplicial, the projection map $\pi$ induces a bijection denoted $\tilde\alpha \leftrightarrow \alpha=\pi(\tilde\alpha)$ between the set of nontrivial edgelet paths $\tilde\alpha \subset U$ that begin and end with edgelets of $U \setminus \sigma$ and the set of nontrivial edgelet paths in~$T$. We refer to $\tilde\alpha$ as the \emph{lift} of $\alpha$ to~$U$. Note that $\tilde\alpha$ is well-defined independent of the choice of subdivision: as a subset of $U$, it is the unique compact arc in $U$ that projects onto $\alpha$ and that begins and ends with segments of $U \setminus \sigma$.
\end{definition}

\begin{definition}[Interior Crossings, and Filling Paths] 
\label{DefinitionFillingPaths}
Consider a free splitting $T$ and a path $\alpha \subset T$. Given another path $\eta \subset T$, to say that $\alpha$ \emph{has an interior crossing of the orbit of $\eta$} means that there exists $\gamma \in \Gamma$ such that the interior of $\alpha$ contains $\gamma \cdot \eta$, equivalently there exists $\delta \in \Gamma$ such that the interior of $\delta \cdot \alpha$ contains $\eta$. To say that $\alpha$ is a \emph{filling path in $T$}, or more briefly that \hbox{$\alpha$ \emph{fills $T$}}, means that for every free splitting $U$, every simplicial collapse map $\pi \from U \to T$, and every natural edge $E \subset U$, the lifted path $\tilde\alpha \subset U$ has an interior crossing of some natural edge in the orbit of~$E$. 
\end{definition}

\smallskip\noindent\textbf{An invariance principle for filling paths.} Filling paths are invariant under an equivariant homeomorphism of free splittings $f \from S \to T$, meaning that a given path $\alpha \subset S$ fills $S$ if and only if $f(\alpha)$ fills $T$: for every collapse map $U \to S$, its composition with $f$ is a collapse map and the respective lifts of $\alpha$ and $f(\alpha)$ are identical in $U$; and similarly for every collapse map $V \to S$.

\smallskip\noindent\textbf{An example.}  Notice that if $\alpha$ fills $T$ then $\alpha$ has an interior crossing of every natural edge~$E$ of~$T$, otherwise $U=T$ itself witnesses that $\alpha$ does not fill~$T$. But a simple counterexample shows that the converse fails. Let $U$ be the Cayley tree of the rank~$2$ free group $\<a,b\>$, let $\tilde\alpha \subset U$ be a path consisting of three consecutive edges labelled~$a$, and let~$T$ be obtained from~$U$ by collapsing each~$b$ edge. The path $\tilde\alpha \subset U$ is the lift of a path $\alpha \subset T$ that has an interior crossing of the unique natural edge orbit of $T$, but $U$ witnesses that $\alpha$ does not fill~$T$.

\medskip\noindent\textbf{The concept of tiles.} This concept was introduced into relative train track theory in \BookOne\ as a tool for constructing and applying attracting laminations (see also \cite[Section 4.3.2]{\RelFSHypTwoTag}). In that context, the tiles of a relative train track map $f \from G \to G$ are the (nonbacktracking) edge paths of the form $f^k_\#(E)$ where $E \subset G$ is an edge in a marked graph and $f^k_\#(E)$ is the path obtained by straightening the (possibly backtracking) edge path $f^k \restrict E$.

In Part II of this work \cite{\RelFSHypTwoTag}, in addition to tiles of relative train tracks, we used a more general concept of tiles of foldable maps in the \TOAT. The \emph{very} rough idea here is to think of a tile in a free splitting $T$ as a path which is determined by data in some far distant free splitting $S$, and which serves as a record in $T$ of some modicum of information about the relation between $S$ and $T$. In order to formulate, prove, and apply the \STOAT, we now refine and formalize the concept of tiles used in  \cite{\RelFSHypTwoTag}.

\begin{definition}[Tiles]
\label{DefGeneralTiles}
Consider a foldable map $f \from S \to T$. A \emph{natural $f$-tile} in $T$ is a path of the form $f(E)$ where $E \subset S$ is a natural edge; an \emph{$f$-tile} in $T$ is a path of the form $f(e)$ where $e \subset S$ is an edge. The ``natural'' qualifier may be dropped where it is understood. Also, the notation ``$f$'' may be replaced by other notations which, in context, have the effect of determining the map $f$; see for example ``iteration tiles'' in Section~\ref{SectionTTsAndLams}.
\end{definition}

Combining Definitions~\ref{DefinitionFillingPaths} and~\ref{DefGeneralTiles} one obtains the concept of a \emph{filling tile} which will play a prominent role in the statement, proof, and application of \STOAT.

\subsubsection{The statement: noniterated and iterated forms}
\label{SectionIteratedAndNon}
Just as with the \TOAT\ (see \cite[Section 5.1]{\RelFSHypTwoTag}), the \STOAT\ comes in iterated and noniterated forms, and the latter implies the former as we shall immediately show.

\begin{theorem}[\STOAT\ (noniterated form)]
\label{TheoremStrongTwoOverAll} 
For any group $\Gamma$ and any free factor system~$\A$ there exists a  constant $\Theta = \Theta(\Gamma;\A)$ such that for any foldable map $f \from S \to T$ of free splittings of $\Gamma$ \relA, if $d(S,T) \ge \Theta$ then there exist two natural edges $E_i \subset S$ $(i=1,2)$ in different orbits such that each of the natural $f$-tiles $f(E_i) \subset T$ is a filling path.
\end{theorem}

\begin{theorem}[\STOAT\ (iterated form)]
\label{TheoremSTOATIterated}
For any $\Gamma$ and $\A$ as above, using the same value of $\Theta$, and for any integer $N \ge 1$, if $d(S,T) \ge N \Theta$ then there exist two natural edges $E_i \subset S$ in different orbits such that each of the natural $f$ tiles $f(E_i) \subset T$ contains $2^{N-1}$ nonoverlapping filling paths.
\end{theorem}

\begin{proof}[Proof of the iterated form, assuming the noniterated form] The basis step $N=1$ follows immediately from the noniterated form of the theorem. 

Assume by induction that the theorem holds for $N$, and suppose that $d(S,T) \ge (N+1)\Theta$. Choose a Stallings fold factorization of $f$ in which each fold map has distance~$\le 1$ \cite[Theorem 2.17]{\RelFSHypTwoTag}. Along that fold path there exists a term $U$ such that $d(S,U)=\Theta$. It follows that $d(U,T) \ge N\Theta$, and we have a foldable factorization of $f$ of the form
$$S \xrightarrow{g} U \xrightarrow{h} T
$$
Applying the induction hypothesis to $h$ we obtain natural edges $E'_1,E'_2 \subset U$ in distinct orbits such that for each $j=1,2$ the path $h(E'_j)$ contains $2^{N-1}$ nonoverlapping filling paths. Applying the noniterated form to $g$ we obtain natural edges $E_1,E_2 \subset S$ in distinct orbits such that for each $i=1,2$ the path $g(E_i) \subset U$ is filling, and therefore for each $j=1,2$ the path $g(E_i)$ contains some translate $\gamma_{ij} \cdot E'_j$ of $E'_j$, $\gamma_{ij} \in \Gamma$. Noting that $\gamma_{i1} \cdot E'_1$ and $\gamma_{i2} \cdot E'_2$ are natural edges in distinct orbits, they do not overlap. Since the factorization above is foldable, the map $h$ is injective on $g(E_i)$. The path $f(E_i) = h \circ g(E_i)$ therefore contains two nonoverlapping paths $h(\gamma_{i1} \cdot E'_1)$ and $h(\gamma_{i2} \cdot E'_2)$, each of which contains $2^{N-1}$ nonoverlapping filling paths. It follows that $f(E_i)$ contains $2^N$ nonoverlapping filling paths.
\end{proof}

\subsection{A filling criterion for paths, expressed using protoforests}
\label{SectionFillingCriterion}
Given a free splitting $T$ and a path $\alpha \subset T$, in order to directly verify whether $\alpha$ fills $T$ according to Definition~\ref{DefinitionFillingPaths} one must search over all expansions of~$T$. In this section we give a more intrinsic description of what it means for $\alpha$ to fill~$T$, expressed without reference to any expansion of~$T$. To motivate that description, consider $\beta(\alpha;T) = \bigcup_{g \in \Gamma} g \cdot \alpha$, called the \emph{covering forest} in $T$ associated to~$\alpha$; see \cite[Definition 5.6]{\RelFSHypTwoTag}. In the proof of Step~1 of the original \TOAT, carried out in \cite[Section 5.4]{\RelFSHypTwoTag}, covering forests $\beta(\alpha;T)$ played the starring role, with guest star being the free factor system $\F[\beta(\alpha;T)]$ formed from the stabilizers of the components of $\beta(\alpha;T)$ (see \cite[Definition 2.4]{\RelFSHypTwoTag}). By studying how the translates $g \cdot \alpha$ of the path $\alpha$ overlap with each other in $T$, we shall describe the \emph{overlap protoforest} of $\alpha$ in $T$, a kind of refinement of the covering forest $\beta(\alpha;T)$ that we denote $\beta^\Over(\alpha;T)$. From the overlap protoforest we obtain (somewhat indirectly) a free factor system \relA\ called the \emph{filling support} $\F[\beta^\Over(\alpha;T)]$, which again is a refinement of $\F[\beta(\alpha;T)]$ in the sense of the nesting relation on free factor systems $\F[\beta^\Over(\alpha;T)] \sqsubset \F[\beta(\alpha;T)]$, and the nesting might be strict. Proposition~\ref{PropFillingPath} states our criterion for $\alpha$ to fill~$T$, one condition for which is that the free factor system $\F[\beta^\Over(\alpha;T)]$ must be \emph{full}, i.e.\ it must equal $\{[\Gamma]\}$. The proof of that proposition will follow in Section~\ref{SectionCriterionProof}, and its application to the \STOAT\ is found in Section~\ref{SectionProofStrongTOA}.

\begin{definition}[Protoforests in a free splitting]
\label{DefProtoforest}
Consider a tree $T$ and two subtrees $b,b' \subset T$ (with respect to the given simplicial structure on~$T$). To say that $b$ and $b'$ \emph{overlap} means that their intersection $b \, \intersect \, b'$ contains an edge of $T$; since $b,b'$ are subtrees, it follows that $b,b'$ do not overlap if and only if their intersection $b \intersect b'$ is a single point or empty. A \emph{protoforest} in $T$ is a set $\delta$ of subtrees of~$T$ called the \emph{protocomponents} of~$\delta$, such that no two distinct protocomponents overlap, and such that the collection of protocomponents is $\Gamma$-invariant in the sense that the image under any $g \in \Gamma$ of any protocomponent of~$\delta$ is also a protocomponent of~$\delta$. The partial order $\beta \prec \delta$ on protoforests in~$T$ is defined by requiring that every protocomponent of $\beta$ is a subset of some protocomponent of~$\delta$.
\end{definition}

\subparagraph{Remarks.} We have built $\Gamma$-invariance into the definition of protoforests. We did \emph{not} do that with ordinary subforests, of which there are many important non-invariant examples.

The concept of a \emph{protoforest} is very close to the concept of a \emph{transverse covering} which occurs in Guirardel's work \cite{Guirardel:GroupsActingFreely} in the broader setting of group actions on $\reals$-trees.

\medskip

An ordinary $\Gamma$-invariant forest is the same thing as a protoforest whose protocomponents are the same as its path components, for example $\beta(\alpha;T)$. We shall focus on two other special types of protoforests which are refinements of $\beta(\alpha;T)$: the ``overlap protoforest'' $\beta^\Over(\alpha;T)$; and the ``filling protoforest'' $\beta^\Fill(\alpha;T)$. See Definitions~\ref{DefFineCoveringForest} and~\ref{DefinitionFillingForest}, and see also the \emph{Protoforest Summary} just below Definition~\ref{DefinitionFillingForest}.

\begin{definition}[The overlap protoforest of a path in a free splitting] 
\label{DefFineCoveringForest}
Consider a free splitting $T$ of $\Gamma$ \relA, and a path $\alpha \subset T$ with corresponding covering forest $\beta(\alpha) = \beta(\alpha;T) = \bigcup_g g \cdot \alpha$. For each $g \in \Gamma$ we let $\beta_g(\alpha)$ denote the connected component of $\beta(\alpha)$ containing $g \cdot \alpha$. An \emph{$\alpha$-connection} is a sequence of paths of the form  
$$g_0 \cdot \alpha, \quad g_1 \cdot \alpha, \quad \ldots, \quad g_K \cdot \alpha \quad (g_0,g_1,\ldots,g_{K} \in \Gamma)
$$
such that any two consecutive terms $g_{k-1} \cdot \alpha$, $g_k \cdot \alpha$ overlap ($1\le k \le K$). If furthermore we are given edges $e,e' \subset T$ such that $e \subset g_0(\alpha)$ and $e' \subset g_K(\alpha)$ then we say this sequence is an \emph{$\alpha$-connection from $e$ to $e'$}. Also, given a subgraph $b \subset \beta(\alpha)$, we say that this sequence is an $\alpha$-connection \emph{in the subgraph $b$} if $g_k \cdot \alpha \subset b$ for $0 \le k \le K$. 

To say that a subgraph~$b \subset \beta(\alpha)$ is \emph{$\alpha$-connected} means that there exists an $\alpha$-connection in $b$ from any edge of $b$ to any other edge of $b$. Clearly every $\alpha$-connected subgraph of~$\beta(\alpha)$ is connected, i.e.\ it is a subtree. Also, if $b,b' \subset \beta(\alpha)$ are two $\alpha$-connected subgraphs, and if $b$ overlaps $b'$, then $b \union b'$ is clearly $\alpha$-connected. It follows that the invariant forest $\beta(\alpha)$ has a unique decomposition into maximal $\alpha$-connected subgraphs, each of which is a tree, no two of which overlap, and this decomposition is $\Gamma$-invariant. This collection of subgraphs therefore gives the covering forest $\beta(\alpha)$ the structure of a protoforest, which we call the \emph{overlap protoforest} of $\alpha$ in $T$, denoted $\beta^\Over(\alpha;T)$. For each $g \in \Gamma$, the protocomponent of $\beta^\Over(\alpha;T)$ that contains $g \cdot \alpha$ is denoted $\beta^\Over_g(\alpha;T)$, and clearly $\beta^\Over_g(\alpha;T) \subset \beta_g(\alpha;T)$. When $T$ and/or $\alpha$ are understood we often simplify the notation by writing $\beta^\Over(\alpha)$ or just $\beta^\Over$, and by writing $\beta^\Over_g(\alpha)$ or just $\beta^\Over_g$.
\end{definition}

Since $\Gamma$ acts transitively on the set $\{g \cdot \alpha\}$ we clearly obtain the following consequences of Definition~\ref{DefFineCoveringForest}:
\begin{description}
\item[Protocomponent Transitivity:] $\Gamma$ acts transitively on the protocomponents of $\beta^\Over(\alpha)$.
\item[Protocomponent Stabilizers:] The set of protocomponent stabilizers of $\Gamma$ acting on $\beta^\Over(\alpha)$ is the conjugacy class of the subgroup $\Stab(\beta^\Over_\Id(\alpha))$. 
\end{description}
Note that $\Stab(\beta^\Over_\Id(\alpha))$ need not be a free factor of $\Gamma$ \relA; see ``A new example'' below.

%Definition~\ref{DefinitionFineFFS} items~\pref{ItemFFFSTrivial}, \pref{ItemFFFSAtomic} and~\pref{ItemFFFSNonatomic}

\begin{definition}[The filling support of a path in a free splitting]
\label{DefinitionFineFFS}
Continuing with the notation of Definition~\ref{DefFineCoveringForest}, we define a free factor \relA\ denoted $F(\alpha)$, defined to be the unique minimal free factor of $\Gamma$ \relA\ containing the subgroup $\Stab(\beta^\Over_\Id(\alpha))$. The \emph{(Kurosh) filling rank} of $\alpha$ is defined to be the Kurosh rank of the free factor $\Fm(\alpha)$ relative to~$\A$, an integer denoted $\KR(\alpha)=\KR(\Fm(\alpha)) \ge 0$ \cite[Definition 2.8]{\RelFSHypTwoTag}. Let $\F[\alpha]$ denote the unique smallest free factor system of $\Gamma$ \relA\ which carries the free factor $\Fm(\alpha)$. It follows that one of three cases holds (see the final sentence of \cite[Definition 2.8]{\RelFSHypTwoTag}):
\begin{enumerate}
\item\label{ItemFFFSTrivial}
$\Fm(\alpha)$ is trivial, equivalently $\KR(\alpha)=0$; it follows that $\F[\alpha] = \A$.
\item\label{ItemFFFSAtomic}
$\Fm(\alpha)$ is atomic, equivalently $[\Fm(\alpha)] \in \A$; it follows that $\KR(\alpha)=1$ and $\F[\alpha]=\A$. 
\item\label{ItemFFFSNonatomic}
$\Fm(\alpha)$ is nontrivial and nonatomic, equivalently $\KR(\alpha)>0$ and $[\Fm(\alpha)] \not\in \A$. It follows that $\F[\alpha] \ne \A$; and furthermore there is a decomposition $\A = \A_{\Fm(\alpha)} \sqcup \overline\A_{\Fm(\alpha)}$ such that $\F[\alpha] = \{[\Fm(\alpha)]\} \union \overline\A_{\Fm(\alpha)}$. 
\end{enumerate}
In particular $\Fm(\alpha)$ determines $\F[\alpha]$, and furthermore $\F[\alpha]$ determines $\Fm(\alpha)$ up to conjugacy \emph{except} in the case that $\F[\alpha]=\A$: in that case, $\Fm(\alpha)$ can be either trivial or atomic. Despite this ambiguity we will refer to either the free factor $\Fm(\alpha)$ or the free factor system $\F[\alpha]$ as the \emph{filling support} of~$\alpha$; in context the exact referent should be clear.  

We sometimes use the notation $\F[\alpha;T]$ to highlight the dependence on the free splitting $T$ that contains $\alpha$. Also, we note that $\F[\alpha]$ is determined by the connected subgraph $\beta^\Over_\Id(\alpha)$, in that it is the unique smallest free factor system that ``carries'' the subgroup $\Stab(\beta^\Over_\Id(\alpha))$. For this reason we sometimes denote $\F[\alpha]=\F[\beta^\Over(\alpha;T)]$ and $\KR(\alpha)=\KR(\beta^\Over(\alpha;T))$.
\end{definition}

\noindent
\textbf{Remark.} In the statement of Definition~\ref{DefinitionFineFFS}, Case~\pref{ItemFFFSTrivial} holds if and only if $\beta^\Over_\Id(\alpha) = \alpha$ if and only if for all $g \ne h \in \Gamma$ the translated paths $g \cdot \alpha$ and $h \cdot \alpha$ do not overlap; in this case the protocomponents of $\beta^\Over(\alpha)$ are precisely the translates of $\alpha$. Also, Case~\pref{ItemFFFSAtomic} holds if and only if there exists a vertex $V \in \beta^\Over_\Id(\alpha)$ such that $\Stab(\beta^\Over_\Id(\alpha)) = \Stab(V)$ and $[\Stab(V)] \in \A$; it follows in this case that $V \in \alpha$, and that $\beta^\Over_\Id(\alpha) = \bigcup_{g \in \Stab(V)} g \cdot \alpha$.

\paragraph{A continued example.} In the example we gave preceding the statement of Theorem~\ref{TheoremStrongTwoOverAll}, the protocomponents of $\beta^\Over_\Id(\alpha)$ are the axes in $T$ of the infinite cyclic free factor $\<a\>$ of $\Gamma = \<a,b\>$, and so the filling support is $\F[\alpha;T] = \{[\<a\>]\}$ which is a nonfull free factor system, despite the fact that the path $\alpha \subset T$ does indeed have an interior crossing of a translate of every natural edge orbit of~$T$: the tree $T$ has only one natural edge orbit, and $\alpha$ crosses three different natural edges of that orbit, the middle one of those three being an interior crossing.

\paragraph{A new example.} This example shows that the subgroup $\Stab(\beta^\Over_\Id(\alpha))$ \emph{need not be} a free factor of $\Gamma$ \relA, in contrast to the fact that the stabilizer of every component of the $\Gamma$-invariant subforest $\beta(\alpha) \subset T$ is a free factor of $\Gamma$ \relA. Consider the rank 4 free group $\Gamma = \<a,b,c,d\>$. Let $G$ be a marked graph with two vertices $p,q$, with the $\<a,b\>$ rose attached to $p$, the $\<c,d\>$ rose attached to $q$, and an edge $e$ from $p$ to~$q$; we identify $\pi_1(G,p) \approx \<a,b\> * \<c,d\> \approx \<a,b,c,d\>$ by using $e$ as a maximal subtree of~$G$. Let $U = \wt G$ be the universal covering with edge labels lifted from~$G$. Choosing a lift $\tilde p \in U$ of $p$ determines the deck transformation action $\Gamma \act U$. Let $\tilde\alpha \subset U$ be the path with initial vertex $\tilde p$ that is labelled by the word $e \, c \, d \, \bar c \, \bar d \, \bar e \, a \, b \, \bar a \, \bar b \, e$. Note that $\tilde\alpha$ is contained in the $U$-axis of the infinite cyclic subgroup 
$$C = \< c \, d \, c^\inv \, d^\inv \, a \, b \, a^\inv \,  b^\inv \> < \<a,b,c,d\>
$$
Also, $\tilde\alpha$ consists of one entire fundamental domain for the action of $C$ on its $U$-axis, followed by the first $e$ edge of the next fundamental domain. Let $T$ be the free splitting obtained from $U$ by collapsing all edges labelled $a,b,c,d$, so $T$ has a single natural edge orbit, represented by the $T$-image of any $e$-edge of~$U$. Let $\alpha \subset T$ be the image of $\tilde\alpha$. Note that $\alpha$ crosses the unique natural edge orbit three times, the middle crossing being an interior crossing. Again $\alpha$ is contained in the axis of~$C$ in~$T$, and $\alpha$ consists of one entire fundamental domain of that axis followed by the first edge of the next fundamental domain. Note that no two distinct translates of the axis of $C$ have a common edge in~$T$. It follows that the protocomponent $\beta^\Over_\Id(\alpha)$ equals the axis of $C$ in~$T$, and that $C = \Stab(\beta^\Over_\Id(\alpha))$. But $C$ has trivial image under abelianization of $\<a,b,c,d\>$ and hence $C$ is \emph{not} a free factor of $\<a,b,c,d\>$. In fact we have the equation $\F[\alpha;T] = \{[\<a,b,c,d\>]\}$, in other words $C$ fills the group $\<a,b,c,d\>$, and hence $\alpha$ fills~$T$ by Proposition~\ref{PropFillingPath}. 

The fact that $C$ fills $\<a,b,c,d\>$ can be proved using a beautiful transversality argument that we learned from a paper of Stallings \cite{Stallings:transversality}. Here are some details of this argument. Consider $S$ a compact orientable surface of genus~2 with 1 boundary component $C$. There is an embedding $G \hookrightarrow S$ and a deformation retraction $\rho \from S \mapsto G$ such that $\rho(C)$ is the loop $aba^\inv b^\inv c d c^\inv d^\inv$. Arguing by contradiction, if $C$ did not fill  $\<a,b,c,d\>$ then there would exist a homotopy equivalence $h \from G \to H$ to a connected graph $H$ such that $h(C)$ is a circuit contained in a proper subgraph of $H$, thus missing the midpoint $m \in E$ of some edge $E \subset H$. After perturbing $h$ to be transverse to $m$, it follows that $h^\inv(m)$ is disjoint union of simple closed curves in the interior $S$. But $h$ is a homotopy equivalence and the image of each of these curves is the point $m$, hence each such curve bounds a disc in $S$. We can then homotope $h$ to remove these curves one-at-a-time, obtaining a homotopy equivalence $S \to H$ that misses $m$ entirely. But this implies that $\<a,b,c,d\> \approx \pi_1(S) \approx \pi_1(H)$ is contained in a proper free factor of itself, which is absurd. 

\smallskip

We now state our combinatorial criterion for a path to fill a free splitting:

%The \emph{Filling Criterion} of Proposition~\ref{PropFillingPath} 
%item~\pref{ItemAlphaInteriorCrossing}
%item~\pref{ItemAlphaFullRank}

\begin{proposition}
\label{PropFillingPath}
For each free splitting $T$ of $\Gamma$ \relA\ and each path $\alpha \subset T$, the path $\alpha$ fills $T$ if and only if the following holds:

\medskip
\textbf{\emph{The Filling Criterion:}}
\begin{enumerate}
\item\label{ItemAlphaInteriorCrossing}
$\alpha$ has an interior crossing of the orbit of every natural edge of $T$ 
\item\label{ItemAlphaFullRank} 
$\Fm(\alpha) = \Gamma$, which holds if and only if $\F[\alpha;T] = \{[\Gamma]\}$, which holds if and only if 
$$\KR(\alpha) = \KR(\Gamma;\A) \,\, (= \abs{\A} + \corank(\A))
$$ 
(see \cite[Definition 2.8 and Lemma 2.9]{\RelFSHypTwoTag}).
\end{enumerate}
\end{proposition}

\paragraph{A counterexample.} Condition~\pref{ItemAlphaInteriorCrossing} of Proposition~\ref{PropFillingPath} is necessary for $\alpha$ to fill~$T$, because if there is a natural edge $e \subset T$ such that $\alpha$ has no interior crossing of $e$ then the trivial expansion of~$T$, namely $U = T$, witnesses that $\alpha$ does not fill $T$. But one might wonder whether Condition~\pref{ItemAlphaInteriorCrossing} follows from Condition~\pref{ItemAlphaFullRank}, in which case the statement of Proposition~\ref{PropFillingPath} could be simplified by eliminating any mention of Condition~\pref{ItemAlphaInteriorCrossing}. Here is an example to dispel that wonder, in which $\alpha$ does \emph{not} have an interior crossing of the orbit of every natural edge of $T$, and yet $\Fm(\alpha) = \Gamma$. Let $T$ be the Cayley tree of $\<a,b\>$. Equivariantly subdivide each $a$ edge into three subedges $a=a_1 \, a_2 \, a_3$. Let $\alpha \subset T$ be a path labelled $a_2 \, a_3 \, b \, b \, a_1 \, a_2$, and note that $\alpha$ has no interior crossing of any natural edge of $T$ labelled~$a$. To prove $\Fm(\alpha)=\Gamma$ it suffices to show that the overlap protoforest $\beta^\Over(\alpha)$ has just one protocomponent, namely the entire tree $T$. Note first that each path in $T$ labelled $b \, b$ is contained in a single protocomponent of $\beta^\Over(\alpha)$, and hence each $b$-axis is contained in a single protocomponent. Note next that each path labelled $b \, a_1 \, a_2 \, a_3 \, b$ is contained in a single protocomponent of $\beta^\Over(\alpha)$, because its initial $b \, a_1 \, a_2$ subpath is a terminal subpath on one translate of~$\alpha$, its terminal $a_2 \, a_3 \, b$ subpath is an initial subpath of another translate of~$\alpha$, and those two subpaths share their $a_2$ edge. The unique protocomponent of $\beta^\Over(\alpha)$ is therefore~$T$. 

\subsection{Proof of the filling criterion}
\label{SectionCriterionProof}

In this section we prove the \emph{Filling Criterion}, Proposition~\ref{PropFillingPath}. The ``if'' direction is proved in Section~\ref{SectionFillingIf}. 

The ``only if'' direction requires considerably more work. Section~\ref{SectionFillingProtoforest} covers some preliminary work, defining the \emph{filling protoforest} $\beta^\Fill(\alpha;T)$ of a path $\alpha$ in a free splitting $T$ of $\Gamma$ \relA, whose protocomponent stabilizers realize the free factor system $\F[\alpha;T]$, and which sits between the overlap protoforest $\beta^\Over(\alpha)$ and the covering forest $\beta(\alpha)$ with respect to the partial order $\prec$. Section~\ref{SectionBlowingUpOnlyIf} reduces the ``only if'' direction to the \emph{Protoforest Blowup Lemma}~\ref{LemmaExpansionConstruction}, which gives criteria under which a protoforest in $T$ can be ``lifted'' to a forest in some expansion of~$T$; the reduction is carried out by verifying that the filling protoforest $\beta^\Fill(\alpha;T)$ satisfies those criteria. Section~\ref{SectionBlowupProof} contains a proof of \emph{Protoforest Blowup Lemma}.

\subsubsection{Proof of the ``if'' direction.} 
\label{SectionFillingIf}
Assuming that condition~\pref{ItemAlphaInteriorCrossing} of Proposition~\ref{PropFillingPath} holds, but that the path $\alpha$ fills $T$, we shall prove that condition~\pref{ItemAlphaFullRank} does not hold. We start by deriving a consequence of condition~\pref{ItemAlphaInteriorCrossing}.

\begin{lemma}\label{LemmaAllOrNone}
Consider a free splitting $T$ and a path $\alpha \subset T$ that has an interior crossing of a translate of every natural edge of $T$. For any expansion $T \expands U$ with lifted path $\tilde\alpha \subset U$ the following holds: for every natural edge $E$ of $U$, either $\tilde\alpha$ has an interior crossing of a translate of $E$, or $\tilde\alpha$ is disjoint from the orbit of the interior of~$E$. 
It follows that the covering forest $\beta(\tilde\alpha;U) = \bigcup_{g \in \Gamma} g \cdot \tilde\alpha$ is a natural (invariant) subforest of~$U$. 
\end{lemma}

\begin{proof} Denote the collapse map $q \from U \xrightarrow{\<\upsilon\>} T$. Let $e',e''$ denote the first and last edges of $\tilde\alpha$, both of which are in $U \setminus \upsilon$. We prove the lemma in two cases. 

\smallskip
\textbf{Case 1: $E \not\subset \upsilon$.} The image $q(E) \subset T$ is a natural edge of $T$ and so there exists $g \in \Gamma$ such that $g \cdot \alpha$ has an interior crossing of $q(E)$. We may decompose $E$ as a concatenation $E = \zeta \, \eta \, \theta$ so that $\zeta,\theta$ are the (possibly trivial) maximal prefix and suffix of $E$ that are contained in $\upsilon$, and $\eta$ is the (nontrivial) maximal subpath of $E$ that begins and ends with edges of $U \setminus \upsilon$. Using that $q(E)$ is contained in the interior of $g \cdot \alpha$, it follows that $\eta$ is contained in the interior of $g \cdot \tilde\alpha$, and that $g \cdot e' \not\subset \eta$ and $g \cdot e'' \not\subset \eta$. Since $g \cdot e'$ and $g \cdot e''$ are the first and last edges of $g \cdot \tilde\alpha$, and since $g \cdot e' \not\subset \zeta\union\theta$ and $g \cdot e''\not\subset \zeta\union\theta$, it follows that $E$ is contained in the interior of $g \cdot \tilde\alpha$.

\smallskip
\textbf{Case 2:} $E \subset \upsilon$. If the interior of $E$ is disjoint from the orbit of $\tilde\alpha$ then we are done. Otherwise, for some $g \in \Gamma$ the intersection $E \intersect g \cdot \tilde\alpha$ contains some edge of~$U$. Again $g \cdot e'$ and $g \cdot e''$ are the first and last edges in $g \cdot \tilde\alpha$, and neither is contained in $\upsilon$, hence neither is contained in $E$. It follows that $E$ is contained in the interior of $g \cdot \ti\alpha$.
\end{proof}

Continuing now with the proof of the ``if'' direction, assuming that $\alpha$ \emph{does} have an interior crossing of a translate of every natural edge orbit of $T$ but that $\alpha$ \emph{does not} fill $T$, we must prove that the free factor system $\F[\alpha;T]$ is not full. Using that $\alpha$ does not fill~$T$, choose a free splitting and collapse map $q \from U \xrightarrow{\<\upsilon\>} T$, and a natural edge $E \subset U$, such that the $\tilde\alpha \subset U$ does not have an interior crossing of the orbit of $E$. Applying Lemma~\ref{LemmaAllOrNone}, the interior of $E$ is disjoint from $\beta(\tilde\alpha;U)$.

We may assume that $q$ restricts to a simplicial isomorphism $\tilde\alpha \mapsto \alpha$, equivalently the invariant forests $\upsilon$~and~$\Gamma \cdot \tilde\alpha$ are edge disjoint. To justify making this assumption, we can factor $q$ as 
$$U \xrightarrow{\<\upsilon \intersect \Gamma \cdot \tilde\alpha\>} U' \xrightarrow{q'} T
$$
Letting $\tilde \alpha' \subset U'$ be the lift of $\alpha$, and so $\tilde\alpha$ is the lift of $\tilde\alpha'$, it follows that $\beta(\tilde\alpha';U')$ is disjoint from the interior of the natural edge $E' \subset U'$ that is the image of $E$, and that the collapse map $q' \from U' \mapsto T$ is injective on $\tilde\alpha'$, and hence $q'$ restricts to a simplicial isomorphism $\tilde\alpha' \to \alpha$. Replacing $U$ by $U'$ completes the justification.

Since $q$ restricts to a simplicial isomorphism $\tilde\alpha \mapsto \alpha$, also $q$ restricts to a simplicial isomorphism $g \cdot \tilde\alpha \mapsto g \cdot \alpha$ for each $g \in \Gamma$. Also, since the collapse map $q$ takes the edges of $U \setminus \upsilon$ bijectively to $T$, and since $\upsilon$ is edge disjoint from $\Gamma \cdot \tilde\alpha$, it follows that $q$ induces a bijection from the edges of $\beta(\tilde\alpha;U) = \Gamma \cdot \tilde\alpha$ to the edges of $\beta(\alpha;T) = \Gamma \cdot \alpha$. As a consequence, for any two edges $\tilde e, \tilde e' \subset \beta(\tilde\alpha;U)$ with images $e,e' \subset \beta(\alpha;T)$, each $\tilde\alpha$ connection from $\tilde e$ to $\tilde e'$ in $\beta(\tilde\alpha;U)$ maps to an $\alpha$ connection from $e$ to $e'$ in $\beta(\alpha;T)$, and conversely each $\alpha$ connection from $e$ to $e'$ in $\beta(\alpha;T)$ lifts to a $\tilde\alpha$ connection from $\tilde e$ to $\tilde e'$ in $\beta(\tilde\alpha;U)$. From this it follows that $q$ induces a $\Gamma$-equivariant bijective correspondence between protocomponents of $\beta^\Over(\tilde\alpha;U)$ and protocomponents of $\beta^\Over(\alpha;T)$, restricting to a simplicial isomorphism between corresponding protocomponents. The protocomponent stabilizers of $\beta^\Over(\tilde\alpha;U)$ and of $\beta^\Over(\alpha;T)$ are therefore identical, and hence $\F[\tilde\alpha;U] = \F[\alpha;T]$. But $\beta^\Over_\Id(\tilde\alpha;U) \subset \beta^\Over(\tilde\alpha;U) \subset U \setminus (\Gamma \cdot E)$, hence $\F[\tilde\alpha;U] \sqsubset \F[U \setminus (\Gamma \cdot E)]$ which is a nonfilling free factor system.

\subsubsection{The filling protoforest of a path.} 
\label{SectionFillingProtoforest}
To prepare for the proof of the ``only if'' direction of Proposition~\ref{PropFillingPath}, and for later purposes, given a path $\alpha \subset T$ in a free splitting of $\Gamma$ \relA\ we shall describe an invariant protoforest denoted $\beta^\Fill(\alpha)=\beta^\Fill(\alpha;T)$, and called the \emph{filling protoforest} of $\alpha$ in $T$; see Definition~\ref{DefinitionFillingForest} below. In the ``refinement'' partial ordering, this protoforest will fit between the overlap forest and the ordinary covering forest: $\beta^\Over(\alpha) \prec \beta^\Fill(\alpha) \prec \beta(\alpha)$ (see the ``Protoforest Summary'' at the end of Section~\ref{SectionFillingProtoforest}).

The protoforest $\beta^\Fill(\alpha)$ is a little tricky to define: it will be obtained by delicately conglomerating protocomponents of the overlap protoforest $\beta^\Over(\alpha;T)$. Like $\beta^\Over(\alpha;T)$, the filling protoforest $\beta^\Fill(\alpha;T)$ will have just one orbit of protocomponents. The protocomponent stabilizers of $\beta^\Fill(\alpha;T)$ will be the subgroups conjugate to~$\Fm(\alpha)$; recall that the protocomponent stabilizers of $\beta^\Over(\alpha;T)$ do not generally fit that bill, because $\Stab(\beta^\Over_\Id(\alpha;T))$ can be a proper subgroup of $\Fm(\alpha)$, as shown in the ``new example'' preceding Proposition~\ref{PropFillingPath}. 

The next lemma provides the basis for conglomerating components of the overlap protoforest $\beta^\Over(\alpha)$ to form the protocomponents of the filling protoforest $\beta^\Fill(\alpha)$. For any free splitting $\Gamma \act T$ and any nontrivial subgroup $H \subgroup \Gamma$, recall the notation $T_H$ for the minimal $H$-invariant subtree of~$T$.

%Lemma~\ref{LemmaCheckDeltaConnected}~\pref{ItemDeltaTranslates}
%Lemma~\ref{LemmaCheckDeltaConnected}~\pref{ItemDeltaStabilizer}
%Lemma~\ref{LemmaCheckDeltaConnected}~\pref{ItemDeltaIdChar}
%Lemma~\ref{LemmaCheckDeltaConnected}~\pref{ItemDeltaStabBiject}
%Lemma~\ref{LemmaCheckDeltaConnected}~\pref{ItemDeltaCompStab}

\begin{lemma} 
\label{LemmaCheckDeltaConnected}
Consider a free splitting $T$ of $\Gamma$ \relA\ and a path $\alpha \subset T$, associated to which we have the overlap protoforest $\beta^\Over(\alpha)$ with protocomponent $\beta^\Over_\Id(\alpha)$ containing $\alpha$, and with filling support $\Fm(\alpha)$. The following hold:
\begin{enumerate}
\item\label{ItemDeltaTree}
The subforest $\beta^\Fill_\Id(\alpha) = \Fm(\alpha) \cdot \beta^\Over_\Id(\alpha) \subset T$
is an $\Fm(\alpha)$-invariant subtree of $T$. If $\Fm(\alpha)$ is nontrivial with minimal subtree \hbox{$T_\alpha = T_{\Fm(\alpha)}$} then it follows that $T_\alpha \subset \beta^\Fill_\Id(\alpha)$.
\item\label{ItemDeltaTranslates}
The translates $\beta^\Fill_g(\alpha) = g \cdot \beta^\Fill_\Id(\alpha)$ form the protocomponents of a protoforest in $T$ denoted $\beta^\Fill(\alpha) = \beta^\Fill(\alpha;T)= \{\beta^\Fill_g(\alpha) \suchthat g \in \Gamma\}$.
\item\label{ItemDeltaStabilizer}
The protoforest $\beta^\Fill(\alpha)$ and the conjugacy class $[\Fm(\alpha)] = \{g \Fm(\alpha) g^\inv \suchthat g \in \Gamma\}$ are related as follows:
\begin{enumerate}
\item\label{ItemDeltaCompStab}
$\Stab(\beta^\Fill_g(\alpha)) = g \Fm(\alpha) g^\inv$;
\item\label{ItemDeltaStabBiject}
If $\Fm(\alpha)$ is not trivial then the function $\beta^\Fill_g(\alpha) \mapsto \Stab(\beta^\Fill_g(\alpha))$ is a well-defined bijection between the set $\beta^\Fill(\alpha)$ and the set of subgroups forming the conjugacy class $[\Fm(\alpha)]$. 
\end{enumerate}
\item\label{ItemDeltaIdChar}
The tree $\beta^\Fill_\Id(\alpha)$ is characterized as the smallest subtree of $T$ such that its stabilizer is a free factor of $\Gamma$ \relA, it contains $\alpha$, and it is a union of maximal $\alpha$-connected subtrees (i.e.\ protocomponents of $\beta^\Over(\alpha;T)$). 
\end{enumerate}
\end{lemma}

\begin{definition}[The filling protoforest of a path in a free splitting]
\label{DefinitionFillingForest}
Given a path $\alpha \subset T$ in a free splitting of $\Gamma$ \relA, 
the \emph{filling protoforest} of $\alpha$ is the invariant protoforest $\beta^\Fill(\alpha)=\beta^\Fill(\alpha;T)$ defined in Lemma~\ref{LemmaCheckDeltaConnected}~\pref{ItemDeltaTranslates}, with protocomponents $\beta^\Fill_g(\alpha) = g \cdot \beta^\Fill_\Id(\alpha)$ where $\beta^\Fill_\Id(\alpha)$ is the unique protocomponent containing $\alpha$. 
\end{definition}

\begin{proof}[Proof of Lemma \ref{LemmaCheckDeltaConnected}] To start the proof we note that $\Stab(\beta^\Over_\Id(\alpha))$ is trivial if and only if $\Fm(\alpha)$ is trivial. If this is so then $\beta^\Fill_\Id(\alpha) = \Fm(\alpha) \cdot \beta^\Over_\Id(\alpha) = \beta^\Over_\Id(\alpha) = \alpha$ is a tree, and the rest of the proof is straightforward.

We may therefore assume that $\Stab(\beta^\Over_\Id(\alpha))$ and $\Fm(\alpha)$ are both nontrivial, hence the minimal subtree $T_\alpha = T_{\Fm(\alpha)} \subset T$ is defined. Assuming for the moment that~\pref{ItemDeltaTree} holds, the remaining items are proved as follows. 

From~\pref{ItemDeltaTree} it follows that each translate $\beta^\Fill_g(\alpha) = g \cdot \beta^\Fill_\Id(\alpha)$ is a subtree of $T$, and this collection of subtrees is clearly $\Gamma$-invariant. The equation $\beta^\Fill_g(\alpha) = (g \Fm(\alpha)) \cdot \beta^\Over_\Id(\alpha)$, together with $\Gamma$-invariance of the protoforest $\beta^\Over(\alpha)$, shows that $\beta^\Fill_g(\alpha)$ is a union of protocomponents of $\beta^\Over(\alpha)$. In order to prove~\pref{ItemDeltaTranslates}, it suffices (by $\Gamma$-invariance of $\beta^\Fill(\alpha)$) to consider each $g \in \Gamma$, and to assume that the trees $\beta^\Fill_\Id(\alpha)$ and $\beta^\Fill_g(\alpha)$ overlap in an edge, and to prove that those trees are identical. From that assumption we obtain $h,h' \in \Fm(\alpha)$ such that $h \cdot \beta^\Over_\Id(\alpha)$ and $g h' \cdot \beta^\Over_\Id(\alpha)$ overlap in an edge. Since $\beta^\Over(\alpha)$ is a protoforest it follows that $h \cdot \beta^\Over_\Id(\alpha) = g h' \cdot \beta^\Over_\Id(\alpha)$ and therefore $h^\inv g h' \in \Stab(\beta^\Over_\Id(\alpha)) \subgroup \Fm(\alpha)$, implying that $g \in \Fm(\alpha)$ and therefore $\beta^\Fill_g(\alpha) = g \cdot \beta^\Fill_\Id(\alpha) = g \Fm(\alpha) \cdot \beta^\Over_\Id(\alpha) = \Fm(\alpha) \cdot \beta^\Over_\Id(\alpha) = \beta^\Fill_\Id(\alpha)$, completing the proof of~\pref{ItemDeltaTranslates}. 

Since $\Stab(\beta^\Over_\Id(\alpha)) \subgroup \Fm(\alpha)$ we have $\Stab(\beta^\Fill_\Id(\alpha)) = \Stab(\Fm(\alpha) \cdot \beta^\Over_\Id(\alpha)) = \Fm(\alpha)$, and item~\pref{ItemDeltaCompStab} immediately follows. To prove~\pref{ItemDeltaStabBiject}, for any $g,h \in \Gamma$ we have
\begin{align*}
\beta^\Fill_g(\alpha) = \beta^\Fill_h(\alpha) &\iff g^\inv h \in \Stab(\beta^\Fill_\Id(\alpha)) = \Fm(\alpha) \\
&\iff g \Fm(\alpha) g^\inv = h \Fm(\alpha) h^\inv 
\intertext{(using that the free factor $\Fm(\alpha)$ is its own normalizer)}
&\iff g \Stab(\beta^\Fill_\Id(\alpha)) g^\inv = h \Stab(\beta^\Fill_\Id(\alpha)) h^\inv \\
&\iff  \Stab(\beta^\Fill_g(\alpha)) = \Stab(\beta^\Fill_h(\alpha))
\end{align*}
To prove item~\pref{ItemDeltaIdChar}, we note first that $\beta^\Fill_\Id(\alpha)$ certainly does satisfy all of the properties required by~\pref{ItemDeltaIdChar}. Since the set of protocomponents of $\beta^\Over(\alpha)$ whose union comprises $\beta^\Fill_\Id(\alpha)$ is a single $\Fm(\alpha)$ orbit of protocomponents of $\beta^\Over(\alpha)$, it follows that for any proper subtree of $\beta^\Fill_\Id(\alpha)$ which is a union of protocomponents of $\beta^\Over(\alpha)$, its stabilizer is strictly smaller than $\Fm(\alpha)$; but since $\Fm(\alpha)$ is the smallest free factor of~$\Gamma$ \relA\ that contains $\Stab(\beta^\Over_\Id(\alpha))$, so the stabilizer of that subtree cannot be a free factor of $\Gamma$ \relA.

We turn to the proof of~\pref{ItemDeltaTree}, first proving the inclusion $T_\alpha \subset \beta^\Fill_\Id(\alpha)$. If $T_\alpha = \{V\}$ is a single vertex then it is the unique point fixed by each nontrivial element of $\Fm(\alpha)$, including each nontrivial element of $\Stab(\beta^\Over_\Id(\alpha))$. Therefore $T_\alpha$ is also the minimal invariant subtree for $\Stab(\beta^\Over_\Id(\alpha))$, and so $T_\alpha \subset \beta^\Over_\Id(\alpha) \subset \Fm(\alpha) \cdot \beta^\Over_\Id(\alpha) = \beta^\Fill_\Id(\alpha)$. 

We may therefore assume that $T_\alpha$ contains at least one edge; it follows that $\Fm(\alpha)$ is nontrivial and nonatomic (Definition~\ref{DefinitionFineFFS}~\pref{ItemFFFSNonatomic}). Recall the decomposition $\A = \A_{\Fm(\alpha)} \union \overline\A_{\Fm(\alpha)}$ from \cite[Lemma~2.7]{\RelFSHypTwoTag}. Recall also from that lemma that $\A_{\Fm(\alpha)}$ ``restricts to'' a free factor system of the group~$\Fm(\alpha)$, in the sense that there is a free factorization $\Fm(\alpha) = A_1 * \cdots * A_K * B$ such that $\A_{\Fm(\alpha)} = \{[A_1],\ldots,[A_K]\}$ (where $[A_k]$ is the $\Gamma$-conjugacy class of $A_k$). In this situation the action $\Fm(\alpha) \act T_\alpha$ is clearly a free splitting of $\Fm(\alpha)$ relative to~$\A_{\Fm(\alpha)}$. Let $\H$ denote the convex hull of $\beta^\Fill_\Id(\alpha)$, meaning the smallest subtree of $T$ containing $\beta^\Fill_\Id(\alpha)$. By $\Fm(\alpha)$-invariance of $\beta^\Fill_\Id(\alpha)$ it follows that the tree $\H$ is also $\Fm(\alpha)$-invariant. By minimality of $T_\alpha$ we have $T_\alpha \subset \H$.  Assuming by contradiction that $T_\alpha \not\subset \beta^\Fill_\Id(\alpha)$, there exists an edge $e \subset T_\alpha$ with interior disjoint from $\beta^\Fill_\Id(\alpha)$. The inclusion $T_\alpha \subset \H$ restricts to the following inclusion of $\Fm(\alpha)$-invariant subforests of~$T$:
$$T_\alpha \setminus (\Fm(\alpha) \cdot e) \,\subset\, \H  \setminus (\Fm(\alpha) \cdot e)
$$
This inclusion induces an $\Fm(\alpha)$-equivariant bijection between the components on either side of the inclusion, hence corresponding components have the same stabilizer subgroup with respect to the $\Fm(\alpha)$ action. From the left hand side of that inclusion, the stabilizer of each component is a proper free factor of $\Fm(\alpha)$ rel~$\A_{\Fm(\alpha)}$, and is therefore a free factor of $\Gamma$ \relA\ that is properly contained in $\Fm(\alpha)$. By construction $\beta^\Over_\Id(\alpha)$ is contained in some component of $\H \setminus (\Fm(\alpha) \cdot e)$, and therefore $\Stab(\beta^\Over_\Id)(\alpha)$ is contained in some free factor of $\Gamma$ \relA\ that is properly contained in~$\Fm(\alpha)$. But this contradicts the definition of $\Fm(\alpha)$ as the smallest such free factor (see Definition~\ref{DefinitionFineFFS}). From this contradiction, the desired inclusion $T_\alpha \subset \beta^\Fill_\Id(\alpha)$ follows.

Using the inclusion $T_\alpha \subset \beta^\Fill_\Id(\alpha)$ we now prove that the forest $\beta^\Fill_\Id(\alpha) = \Fm(\alpha) \cdot \beta^\Over_\Id(\alpha)$ is connected, i.e.\ it is a tree, which will complete the proof of~\pref{ItemDeltaTree}. Clearly $\Fm(\alpha)$ acts transitively on the components of $\beta^\Fill_\Id(\alpha)$. Knowing that $T_\alpha \subset \beta^\Fill_\Id(\alpha)$, it follows that some component of $\beta^\Fill_\Id(\alpha)$ intersects $T_\alpha$. By $\Fm(\alpha)$-invariance it follows that every component of $\beta^\Fill_\Id(\alpha)$ intersects~$T_\alpha$. But $T_\alpha$ is a connected subset of $\beta^\Fill_\Id(\alpha)$ hence $\beta^\Fill_\Id(\alpha)$ is connected.
\end{proof}

\subparagraph{Protoforest summary.} To summarize Definitions~\ref{DefFineCoveringForest}, \ref{DefinitionFineFFS} and~\ref{DefinitionFillingForest}, associated to any path $\alpha \subset T$ we have a nested sequence of three protoforests, namely the overlap protoforest, the filling protoforest, and the ordinary covering forest:
$$\beta^\Over(\alpha) \prec \beta^\Fill(\alpha) \prec \beta(\alpha)
$$
For each $g \in \Gamma$ the associated protocomponents containing $g \cdot \alpha$ are denoted 
$$\beta^\Over_g(\alpha) \subset \beta^\Fill_g(\alpha) \subset \beta_g(\alpha) \quad g \in \Gamma
$$
We have a corresponding nested sequence of stabilizer subgroups, which in the case $g=\Id$ takes the form 
$$\Stab(\beta^\Over_\Id(\alpha)) \subgroup \underbrace{\Stab(\beta^\Fill_\Id(\alpha))}_{\Fm(\alpha)} \subgroup \Stab(\beta_\Id(\alpha))
$$
We note that only $\Stab(\beta^\Fill_\Id(\alpha))$ and $\Stab(\beta_\Id(\alpha))$ are guaranteed to be free factors \relA. 

The filling support of $\alpha$, denoted $\F[\alpha;T]$, is the free factor system \relA\ whose unique nonatomic component is $[\Fm(\alpha)]$, \emph{if} the free factor $\Stab(\beta^\Fill_\Id(\alpha;T)) = \Fm(\alpha)$ is nontrivial and nonatomic \relA; otherwise $\F[\alpha;T]$ is just $\A$ itself. 

\medskip

\subsubsection{Blowing up the filling protoforest to prove the ``only if'' direction.}
\label{SectionBlowingUpOnlyIf}
Given a protoforest $\delta$ in a free splitting $T$ of $\Gamma$ \relA, its protocomponent stabilizers need not be free factors \relA\ (see the ``new example'' preceding Proposition~\ref{PropFillingPath}). But the filling protoforest $\beta^\Fill(\alpha)$ of a path $\alpha \subset T$ has been constructed so that its protocomponent stabilizers $\Stab(\beta^\Fill_g(\alpha)) = g \Fm(\alpha) g^\inv$ are indeed free factors \relA, and furthermore so that one of two alternatives holds: those stabilizers are all trivial or atomic, in which case the filling support is $\F[\alpha;T]=\A$; or they are all nontrivial, in which case they form a single conjugacy class $[\Fm(\alpha)]$, and that class is the unique non-atomic element of the filling support $\F[\alpha;T]$ (see Definition~\ref{DefinitionFineFFS} and Lemma~\ref{LemmaCheckDeltaConnected}~\pref{ItemDeltaStabilizer}). Furthermore, in the nontrivial case we get extra information about the individual protocomponent stabilizers, allowing us to conclude that $\beta^\Fill(\alpha)$ satisfies the following \emph{Tame Stabilizer Hypothesis}. We state this hypothesis in a general manner, allowing it to be tested on protoforests in any free splitting, not just filling protoforests of paths:

\paragraph{Tame Stabilizer Hypothesis (for a protoforest $\delta$ in a free splitting $T$ of $\Gamma$ \relA)} \quad
\begin{enumerate}
\item\label{ItemProtoDistinct}
Either all protocomponents of $\delta$ have trivial stabilizer, or any two distinct protocomponents have distinct stabilizers.
\item\label{ItemProtoStabSys}
There is a free factor system $\F[\delta]$ of $\Gamma$ \relA, and a decomposition \mbox{$\F[\delta] = \F_0[\delta] \sqcup \F_1[\delta]$} such that $\F_1[\delta] \subset \A$ and $\F_0[\delta]$ consists of those nontrivial, nonatomic conjugacy classes of the form $[\Stab(\delta_i)]$ as $\delta_i$ varies over the protocomponents of $\delta$.
\item
\label{ItemProtoStabInt}
For any two distinct protocomponents of $\delta$, the intersection of their stabilizers is trivial.
\end{enumerate}
In fact item~\pref{ItemProtoStabInt} follows from~\pref{ItemProtoDistinct} and~\pref{ItemProtoStabSys}: item~\pref{ItemProtoDistinct} reduces to the case that the two stabilizers are distinct nontrivial subgroups; and that case is covered by applying item~\pref{ItemProtoStabSys} together with \cite[Lemma 2.1]{\RelFSHypTag} which says that the collection of subgroups representing elements of a free factor system is mutually malnormal.

As alluded to above, by applying Definition~\ref{DefinitionFineFFS} and Lemma~\ref{LemmaCheckDeltaConnected}~\pref{ItemDeltaStabilizer}, it follows that the \emph{Tame Stabilizer Hypothesis} holds for the filling protoforest $\beta^\Fill(\alpha)$ associated to any path $\alpha \subset T$ in any free splitting $T$ of $\Gamma$ \relA, with the filling support $\F[\alpha;T]$ playing the role of $\F[\delta]$. Furthermore, in the nontrivial/nonatomic case there are corresponding decompositions using the notation of Definition~\ref{DefinitionFineFFS}~\pref{ItemFFFSNonatomic}:
$$\underbrace{\F[\alpha;T]}_{\F[\delta]} = \underbrace{\{[\Fm(\alpha)]\}}_{\F_0[\delta]} \union \underbrace{\overline\A_{\Fm(\alpha)}}_{\F_1[\delta]}
$$

Given any protoforest $\delta$ in $T$, a sufficient condition for the \emph{Tame Stabilizer Hypothesis} to hold is that there exists a simplicial collapse map $q : U \xrightarrow{\upsilon} T$ defined on a free splitting $U$ of $\Gamma$ \relA, and there exists an invariant subforest $\tilde\delta \subset U$ (with respect to the given simplicial structure on $U$), such that the forests $\upsilon$ and $\tilde\delta$ have no edge in common, the map $q$ induces a bijection between connected components of $\tilde\delta$ and protocomponents of $\delta$, and $q$ takes each connected component of $\tilde\delta$ to its associated protocomponent of $\delta$ by a simplicial isomorphism. The following lemma says that this condition is also necessary:

%Lemma~\ref{LemmaExpansionConstruction}~\pref{ItemBlowupFFEqual}

\begin{lemma}[The Protoforest Blowup Lemma]
\label{LemmaExpansionConstruction}
Consider a free splitting $T$ of $\Gamma$ \relA\ and a protoforest $\delta = \{\delta_i\}_{i \in I}$ in~$T$. If~$\delta$ satisfies the \emph{Tame Stabilizer Hypothesis}, then there exists a free splitting $U$ of $\Gamma$ \relA, a simplicial collapse map $q \from U \xrightarrow{\upsilon} T$, and an invariant subforest $\tilde\delta \subset U$ with connected components $\tilde\delta = \bigsqcup_{i \in I} \tilde\delta_i$ (with the same index set $I$), such that the following hold:
\begin{enumerate}
\item\label{ItemBlowupFFEqual}
For each $i \in I$ the projection $q$ restricts to a simplicial isomorphism $q \from \tilde\delta_i \to \delta_i$. 
\item $\F[\delta]=\F[\tilde\delta]$, hence the free factor system~$\F[\delta]$ is visible in $U$ as witnessed by $\tilde\delta \subset U$.
\end{enumerate}
\end{lemma}

Before proving this lemma in Section~\ref{SectionBlowupProof} to follow, we first apply it:

\begin{proof}[Proof of Proposition~\ref{PropFillingPath}: The ``only if'' direction]
Consider an arc $\alpha \subset T$ in a free splitting $T$ of $\Gamma$ \relA, with filling support $\F[\alpha;T]$ (Definition~\ref{DefinitionFineFFS}). We must prove that $\alpha$ does \emph{not} fill $T$ if either of the following two cases holds:
\begin{description}
\item[Case 1:] There exists a natural edge $E$ of $T$ such that $\alpha$ does not have an interior crossing of the orbit of $E$; or
\item[Case 2:] $\F[\alpha;T]$ is \emph{not} filling. 
\end{description}
In Case 1, the trivial expansion $U=T \xrightarrow{\Id} T$ witnesses that $\alpha$ does not fill $T$. 

In Case~2, consider the filling protoforest $\beta^\Fill(\alpha)$ in $T$, which as we have seen satisfies the \emph{Tame Stabilizer Hypothesis} and has associated free factor system $\F[\beta^\Fill(\alpha)] = \F[\alpha;T]$. Applying Lemma~\ref{LemmaExpansionConstruction} we obtain an expansion $U \mapsto T$ and for each $g \in \Gamma$ a simplicially isomorphic lift $\tilde\beta_g \subset U$ of $\beta^\Fill_g(\alpha)$, such that these lifts are the connected components of a proper, invariant subforest $\tilde\beta \subset U$ whose component stabilizers form the free factor system~$\F[\alpha;T]$. From the Case~2 hypothesis, $\tilde\beta$ is a proper subforest of~$U$. Since the lift $\tilde\alpha \subset U$ of $\alpha$ is contained in the component $\tilde\beta_\Id$, the translates of $\tilde\alpha$ are contained in the proper subforest $\tilde\beta$ and hence do not cover~$U$. The expansion $U \mapsto T$ thus witnesses that $\alpha$ does not fill~$T$.
\end{proof}

\subsubsection{Proof of the Protoforest Blowup Lemma}
\label{SectionBlowupProof}
The last remaining piece of the proof of Proposition~\ref{PropFillingPath}, the Filling Criterion, is the following:

\begin{proof}[Proof of the Blowup Lemma \ref{LemmaExpansionConstruction}] Let $T$ and $\delta = \{\delta_i\}_{i \in }$ be given. For all $i \ne j \in I$, the protocomponent intersection $\delta_i \intersect \delta_j$ is either empty or a single vertex of $T$. Let $W \subset \V T$ denote those vertices $w \in \V T$ for which there exist $i \ne j \in I$ such that $w = \delta_i \intersect \delta_j$; clearly $W$ is $\Gamma$-invariant. If $W$ is empty then distinct protocomponents of $\delta$ are disjoint, hence $\delta$ is a true subforest of $T$, and the conclusions of the lemma are true with the trivial expansion $U=T \xrightarrow{\Id} T$. Henceforth we assume that $W$ is not empty.

We have two actions $\Gamma \act I$ and $\Gamma \act W$, the former induced by the action $\Gamma \act \{\delta_i\}_{i \in I}$ on the protocomponents of $\delta$. Consider the diagonal action $\Gamma \act I \times W$. For each $(i,w) \in I \times W$ let $C_{iw} =  \Stab(\delta_i) \intersect \Stab(w)$. Since $T$ is a free splitting of~$\Gamma$, the following implication holds: if $C_{iw}$ is nontrivial then $w \in \delta_i$. 

We would like that the converse implication holds as well:
\begin{itemize}
\item[] $(*)$ \quad For all $(i,k) \in I \times K$, \, $C_{ik}$ is nontrivial if and only if $w_k \in \delta_i$.
\end{itemize}
While statement $(*)$ may fail to be true, we shall reduce the proof of Lemma~\ref{LemmaExpansionConstruction} to the case that $(*)$ \emph{is} true, repairing its failure by constructing an intermediate expansion $T' \mapsto T$ as follows. 

Consider the following $\Gamma$-set that witnesses the falsity of~$(*)$:
$$Z = \bigl\{(i,w) \in I \times W \suchthat w \in \delta_i \,\,\text{and $C_{iw}$ is trivial} \bigr\} 
$$
Let $W^f \subset W$ be the image of the restricted projection map $Z \subset I \times W \mapsto W$ --- that is, $W^f$~is the set of all $w \in W$ for which there exists $i \in I$ such that $w \in \delta_i$ and $C_{iw}$ is trivial. Note that $W^f$ is $\Gamma$-invariant. We construct $T'$ from $T$ in two stages, first detaching from each $w \in W^f$ those $\delta_i$ for which $(i,w) \in Z$, and then attaching new edges to~$T''$ to fill the gaps created by these detachments, thus constructing the free splitting~$T'$. 

In more detail, the forest $T''$ is constructed from $T$ as follows. First, remove each $w \in W^f$ from the vertex set of $T$ and replace it by a subset $\wh w = \{w'\} \union \{w''_{i} \suchthat (i,w) \in Z\}$ of vertices of~$T''$ (all vertices of $T-W^f$ remain as vertices of $T''$). For any edge $e \subset T$ that was attached to $w \in W^f$, we re-attach $e$ to an appropriate point of $\wh w$ as follows: if there exists $i \in I$ such that $(i,w) \in Z$ and $e \subset \delta_i$, re-attach $e$ to $w''_i$; otherwise, reattach $e$ to $w'$. There is a quotient map $T'' \mapsto T$ obtained by collapsing each $\wh w \subset T''$ to the corresponding point $w \in T$ (for $k \in K'$). That quotient map induces a bijection of edges, and each protocomponent $\delta_i \subset T$ of the protoforest $\delta$ lifts by a simplicial isomorphism to a subtree $\delta''_i \subset T''$. Intuitively, we think of $w''_i$ as a new, special copy of $w$ to which~$\delta''_i$ is attached if it so happens that $C_{iw}$ is trivial; and we think of $w'$ as the ``original'' copy of $w$ to which everything else stays attached.

For use below we record that for all $w \in W$ and for all $i \ne j \in I$ such that $w = \delta_i \intersect \delta_j$, the following dichotomy holds in~$T''$:
\begin{description}
\item[Case 1:] $\delta''_i \intersect \delta''_j = \emptyset$ $\iff$ at least one of $(i,w)$ or $(j,w)$ is in~$Z$ $\iff$ at least one of $C_{iw}$ or $C_{jw}$ is trivial.
\item[Case 2:] $\delta''_i \intersect \delta''_j = \{w'\} \ne \emptyset \iff$ neither $(i,w)$ nor $(j,w)$ is in $Z$ $\iff$ both $C_{iw}$ and $C_{jw}$ are nontrivial.
\end{description} 
We note that if $w \in W - W^f$ then Case~2 must hold; whereas if $w \in W^f$ then either of Case~1 or Case~2 could hold. 

We next attach to the forest $T''$ a new edge $e_{iw}$ with endpoints $w'$ and~$w''_{i}$, for all \hbox{$(i,w) \in Z$.} The overall result of this detachment--attachment operation is a new tree $T'$. The actions of $\Gamma$ on $T$, $I$, $W$ and $Z$ induce an action $\Gamma \act T'$. This action is a free splitting of $\Gamma$ \relA\ because $\Stab(e_{iw}) = C_{iw}$ is trivial for each $(i,w) \in Z$. There is a collapse map $T' \mapsto T$ under which the vertices of $T$ with nontrivial preimage are precisely the points $w \in W^f$, and the preimage of each such $w$ is a star graph with star vertex $w'$, consisting of the union of those new edges $e_{iw} \subset T'$ for which $(i,w) \in Z$. By construction we obtain a protoforest $\delta' = \delta''$ in $T' \supset T''$ with protocomponents $\delta'_i = \delta''_i \subset T'$, and the collapse map $T' \to T$ induces a $\Gamma$-invariant bijection of protocomponents, mapping corresponding protocomponents by simplicial isomorphism. The \emph{Tame Stabilizer Hypothesis} therefore still holds with respect to~$T'$ and $\delta'$. Also, as an immediate consequence of the dichotomy recorded above, condition $(*)$ now holds in $T'$, which completes the reduction argument. 

\smallskip

It remains to establish the conclusions of the lemma for given $T$ and $\delta$ under the additional assumption that condition $(*)$ holds, which we restate here:
\begin{description}
\item[$(*)$] For all $w \in W$ and all $\delta_i$ we have 
$w \in \delta_i \iff C_{iw} = \Stab(w) \intersect \Stab(\delta_i) \ne \{\Id\}$.
\end{description}
Denote $I_w = \{i \in I \suchthat w \in \delta_i\}$. By definition each $w \in W$ is contained in more than one protocomponent $\delta_i$, and so $\abs{I_w} \ge 2$. Combining this with $(*)$ it follows that
\begin{description}
\item[$(**)$] For all $w \in W$ we have $\Stab(w) \ne \{\Id\}$, and so $[\Stab(w)]$ is an element of the free factor system $\FST$.
\end{description}
Note that $[\Stab(w)] \in \FST$ depends only on the orbit of $w$ under the action $\Gamma \act W$, and distinct orbits produce distinct elements of $\FST$.

Since $\F[\delta]$ and $\FST$ are both free factor systems of $\Gamma$ \relA, their meet $\B = \F[\delta] \meet \FST$ is also a free factor system of~$\Gamma$ \relA, and clearly $\B \sqsubset \FST$ (see \cite[Section 2.2]{HandelMosher:RelComplexHyp} for discussion of the meet). For each $w \in W$ we have $[\Stab(w)] \in \FST$, and so $\Stab(w)$ is a free factor rel~$\B$. Applying \cite[Lemma/Definition 2.7]{\RelFSHypTwoTag} together with $(*)$ and the definition of the meet, in the group $\Stab(w)$ we have a free factor system $\B \restrict \Stab(w)$ of the following form:
$$\B \restrict \Stab(w) = \biggl\{[\Stab(w) \intersect \Stab(\delta_i)]^w \biggm| i \in I_w \biggr\}
$$
where we use the notation $[\cdot]^w$ to refer to conjugacy classes within the group $\Stab(w)$. 

Note that $\B \restrict \Stab(w)$ is a nonfull free factor system of the group $\Stab(w)$, for otherwise there would exist $i \in I_w$ such that $\Stab(w) = \Stab(w) \intersect \Stab(\delta_i)$, and by choosing $j \ne i \in I_w$ it would follow that
$$\Stab(w) \intersect \Stab(\delta_j) = \Stab(w) \intersect \Stab(\delta_i) \intersect \Stab(\delta_j)
$$
But $\Stab(\delta_i) \intersect \Stab(\delta_j)$ is trivial (by item~\pref{ItemProtoStabInt} of the \emph{Tame Stabilizer Hypothesis}) whereas $\Stab(w) \intersect \Stab(\delta_j)$ is nontrivial (by $(*)$), which is a contradiction.

For each $w \in W$ choose $\Stab(w) \act U_w$ to be a (nontrivial) Grushko free splitting of the group $\Stab(w)$ relative to its (nonfull) free factor system $\B \restrict \Stab(w)$. We may make these choices in a $\Gamma$-equivariant manner, meaning that the individual actions $\Stab(w) \act U_w$ extend to a $\Gamma$-action on the forest $\sqcup_{w\in W} U_w$, and that the stabilizer of $U_w$ is $\Stab(w)$: for each orbit of the action $\Gamma \act W$ one first chooses a single orbit representative $w$ and a Grushko free splitting $\Stab(w) \act U_w$ relative to~$\B \restrict \Stab(w)$; and then one extends that to a diagonal action $\Gamma \act (\Gamma / \Stab(w)) \times U_w$ using the natural left action of $\Gamma$ on the set of left cosets $\Gamma / \Stab(w)$. 

For each $w \in W$ and $i \in I_w$ consider the subgroup $C_{iw} = \Stab(\delta_i) \intersect \Stab(w)$. Since $[C_{iw}]^w \in \B \restrict \Stab(w)$, and since $C_{iw}$ is nontrivial, it follows that $C_{iw}$ fixes a unique point $u_{iw} \in U_w$ and that $C_{iw} = \Stab(u_{iw})$. Note that if $i \ne j \in I_w$ then $u_{iw} \ne u_{jw}$, because $C_{iw}, C_{jw}$ are nontrivial subgroups of $\Stab(w)$ whose intersection $C_{iw} \intersect C_{jw}$ is a subgroup of $\Stab(\delta_i) \intersect \Stab(\delta_j) = \{\Id\}$ and is therefore trivial.

\smallskip

For the remainder of the proof we shall construct the desired free splitting $\Gamma \act U$ from the actions $\Gamma \act T$ and $\Gamma \act \sqcup_w U_w$ using a detachment/re-attachment procedure carried out at each $w \in W$. Recall the direction set $D_w T$ consisting of all oriented edges $E$ with initial vertex~$w$. The set $D_w T$ subdivides as follows: one subset $D_w \delta_i$ for each $i \in I_w$; and a single \emph{leftover} direction for each $E \in D_w T - \bigsqcup_{i \in I_w} D_w \delta_i$ that we denote $D_w E$.

\smallskip
Starting with the free splitting action of $\Gamma$ on the tree $T$:

\smallskip\textbf{Protoforest detachments:} For each $w \in W$ and each $i \in I_w$, detach the entire set $D_w \delta_i$ from $w \in T$, and let $x_{iw} \in \delta_i$ denote the point that was formerly attached to $w$. We may thus identifying $D_{x_{iw}} \delta_i$ with $D_w \delta_i$. The points $x_{wi}$, one for each $w \in E$ and $i \in I_w$, are called \emph{protoforest detachment points}.

\smallskip\textbf{Leftover detachments:} For each $w \in W$ and each leftover oriented edge $E \subset T$ with initial vertex $w$, detach $E$ from $w$, and let $y_{Ew} \in E$ be the initial vertex that was formerly attached to $w$. The points $\{y_{Ew}\}$, one for each $w \in W$ and each leftover oriented edge $E$ with initial vertex $w$, are called \emph{leftover detachment points}.

\smallskip\textbf{The detachment forest $\Delta$:} The result of these detachments is a forest denoted $\Delta$ on which $\Gamma$ acts, together with a $\Gamma$-equivariant quotient map $\Delta \mapsto T$: the action of $\Gamma$ on detachment points is given by 
$$\gamma \cdot x_{iw} = x_{\gamma \cdot i, \gamma \cdot w} \quad\text{and}\quad \gamma \cdot y_{Ew} = y_{\gamma \cdot E, \gamma \cdot w}
$$
and the quotient map takes $x_{iw}$ and each $y_{Ew}$ to $w$. This quotient map is a bijection on edges, and so the $\Gamma$ stabilizer of every edge of $\Delta$ is trivial. Note that each protocomponent $\delta_i \subset T$ lifts to $\Delta$ by simplicial isomorphism, we identify $\delta_i$ with its lift, and we note that the protocomponents are \emph{pairwise disjoint} in $\Delta$.

\paragraph{Re-attachments:} Consider the disjoint union of the forest $\Delta$ and the forest $\sqcup_w U_w$, with the induced $\Gamma$ action on this disjoint union. 

\smallskip\textbf{Protoforest re-attachments:} Noting that both of the sets $\{x_{iw}\}$ and $\{u_{iw}\}$ are bijectively indexed by the set $\{(i,w) \suchthat w \in W, i \in I_w\}$, we may re-attach $x_{iw}$ to $u_{iw}$ in a well-defined manner, and these re-attachments are automatically $\Gamma$-equivariant. 

\smallskip\textbf{Leftover re-attachments:} Re-attach each leftover detachment point $y_{Ew} \in \Delta$ to \emph{some} point $z_{Ew} \in U_w$. The reattachment map $y_{Ew} \mapsto z_{Ew}$ is \emph{not} automatically $\Gamma$-equivariant. However, since each of the actions $\Gamma \act \Delta$ has trivial edge stabilizers, and since each leftover detachment point is incident to a unique edge of $\Delta$, the restricted action of $\Gamma$ on the set of leftover detachment points is a free action, and so we can force these reattachments to be $\Gamma$-equivariant in the usual manner: choose one representative $y_{Ew}$ of each $\Gamma$-orbit of leftover detachment points, choose $z_{Ew} \in U_w$ arbitrarily and reattach $y_{Ew}$ to $z_{Ew}$, and then for each $\gamma \in \Gamma - \{\Id\}$, the leftover detachment point $y_{\gamma \cdot E, \gamma \cdot w} = \gamma \cdot y_{Ew}$ is re-attached to the point $z_{\gamma \cdot E,\gamma \cdot w} := \gamma \cdot z_{Ew}$.

\medskip

Let $U$ be the quotient of the disjoint union of $\Delta$ and $\sqcup_w U_w$ under the reattachments just described. By construction the $\Gamma$ actions on $\Delta$ and on $\sqcup_w U_w$ induce an action $\Gamma \act U$. Since $\Gamma$ acts on each of $\Delta$ and $\sqcup_w U_w$ with trivial edge stabilizers, and since no edges are identified under the quotient map from their disjoint union to $U$, the $\Gamma$-stabilizer of each edge of $U$ is trivial. We have two $\Gamma$-equivariant quotient maps $\Delta \mapsto T$ and $\sqcup_w U_w \mapsto T$ (the latter defined by $U_w \mapsto w$), and these maps agree where points of $\Delta$ are identified with points of $\sqcup_w U_w$, hence we have an induced equivariant quotient map $U \mapsto T$. The only points of $T$ that have nontrivial pre-image under this quotient map are the points $w \in W$, and for each such point its pre-image is the tree $U_w$; furthermore, the collapse map $U \mapsto T$ takes the edges in $U$ that are incident to but not contained in the tree $U_w$ to the edges of $T$ that are incident to $w$. It follows that $U$ is a tree, and furthermore that $\Gamma \act U$ is a free splitting action and $U \mapsto T$ is a collapse map. Because the protocomponents are pairwise disjoint in $\Delta$, and because of the inequality $u_{iw} \ne u_{jw}$ for each $i \ne j \in I_w$, the protocomponents remain pairwise disjoint in $U$. This completes the proof of the Protoforest Blowup Lemma~\ref{LemmaExpansionConstruction}.
\end{proof}

\subsection{Proof of the \STOAT}
\label{SectionProofStrongTOA}
We have already proved, in Section~\ref{SectionIteratedAndNon}, that the uniterated form of the \STOAT~\ref{TheoremStrongTwoOverAll} quickly implies the iterated form. The hard work left is to prove the uniterated form, which we do here in Section~\ref{SectionTheProof}, by giving a formula for $\Theta = \Theta(\Gamma;\A) \ge 1$ and using it to prove that for any foldable map $f \from S \to T$, if $d(S,T) \ge \Theta$ then $S$ has two natural edges in different $\Gamma$ orbits each of whose $f$-images fills~$T$.

The proof will share certain strategies of the proof of the original \TOAT, with filling protoforests taking over the job of covering forests. Starting with a Stallings fold factorization $f \from S = S_0 \xrightarrow{f_1} S_1 \xrightarrow{f_2} \cdots \xrightarrow{f_M} S_M=T$, the rough idea is to consider a natural edge $E_0 \subset S_0$ and its corresponding sequence of natural tiles
$$\alpha_0 = E_0 \subset S_0 \qquad \alpha_1 = f_1(\alpha_0) \subset S_1 \qquad \cdots
\qquad \alpha_M = f_M(\alpha_{M-1}) \subset S_M
$$
and to study the evolution of the filling rank $\KR(\alpha_i)$: recall from Definition~\ref{DefinitionFineFFS} and Lemma~\ref{LemmaCheckDeltaConnected} that $\KR(\alpha_i)=\KR(F(\alpha_i))$ where $F(\alpha_i)$, a free factor of $\Gamma$ \relA, is the stabilizer of the protocomponent $\beta^\Fill_\Id(\alpha)$ of the filling protoforest $\beta^\Fill(\alpha)$ . We describe the evolution of the sequence $\KR(\alpha_i)$ in two lemmas. Section~\ref{ItemKuroshMotonocity} contains  Lemma~\ref{LemmaFoldFillingProtoforest}, a monotonicity property saying that the sequence $\KR(\alpha_i)$ is nondecreasing. Section~\ref{SectionConstantKRBound} contains Lemma~\ref{LemmaFillingDiameterBound} which gives a distance bound while Kurosh rank is constant: a universal upper bound (namely $6$) for the diameter of any fold subsequence along which $\KR(\alpha_i)$ is constant, \emph{assuming} that this constant is strictly less than the maximum value $\KR(\Gamma;\A) = \abs{\A} + \corank(\A)$ for Kurosh ranks.

\subsubsection{Monotonicity of filling ranks along fold paths.}
\label{ItemKuroshMotonocity}
To set up Lemma~\ref{LemmaFoldFillingProtoforest}, consider a single fold map $h \from S \to S'$ between free splittings of $\Gamma$ \relA, and paths $\alpha \subset S$, $\alpha' \subset S'$ such that $h \restrict \alpha$ is a homeomorphism onto $\alpha'$. Consider also the following associated objects:
\begin{itemize}
\item (Definition~\ref{DefFineCoveringForest}) The overlap protoforests $\beta^\Over(\alpha;S)$ and $\beta^\Over(\alpha';S')$ with protocomponents $\beta^\Over_g(\alpha;S)$ and $\beta^\Over_g(\alpha';S')$ containing $g \cdot \alpha$ and $g \cdot \alpha'$ respectively; 
\item (Definition~\ref{DefinitionFineFFS}) The filling supports $F = \Fm(\alpha)$ and $F' = \Fm(\alpha')$, which are the smallest free factors of $\Gamma$ \relA\ containing $\Stab(\beta^\Over_\Id(\alpha;S))$ and $\Stab(\beta^\Over_\Id(\alpha';S'))$ respectively; 
\item (Definition \ref{DefinitionFillingForest}) The filling protoforests $\beta^\Fill(\alpha;S)$ and $\beta^\Fill(\alpha';S')$, with protocomponents containing $g \cdot \alpha$, $g \cdot \alpha'$ respectively:
\begin{align*}
\beta^\Fill_g(\alpha;S) &= g F \cdot \beta^\Over_\Id(\alpha;S) = gFg^\inv \cdot \beta^\Over_g(\alpha;S) \\
\beta^\Fill_g(\alpha;'S') &= g F' \cdot \beta^\Over_\Id(\alpha';S') = gF'g^\inv \cdot \beta^\Over_g(\alpha';S')
\end{align*}
\end{itemize}

%Definition~\ref{DefinitionPreserveProtocomponents}

\begin{definition}
\label{DefinitionPreserveProtocomponents}
Using the above notation, to say that $h$ \emph{preserves filling protocomponents} (with respect to $\alpha$ and $\alpha'$) means:
\begin{enumerate}
\item\label{ItemProtoToProto}
$h(\beta^\Fill_g(\alpha;S)) = \beta^\Fill_g(\alpha';S')$ for each $g \in \Gamma$.
\item\label{ItemDistinctToDistinct}
For each $g_1,g_2 \in \Gamma$, we have $\beta^\Fill_{g_1}(\alpha;S) = \beta^\Fill_{g_2}(\alpha;S)$ if and only if $\beta^\Fill_{g_1}(\alpha;'S') = \beta^\Fill_{g_2}(\alpha';S')$.
\end{enumerate}
In applying this definition one should keep in mind that two protocomponents of the same protoforest are distinct if and only if they are edge disjoint.
\end{definition}

\begin{lemma}
\label{LemmaFoldFillingProtoforest}
For any fold map $h \from S \to S'$ and any paths $\alpha \subset S$ and $\alpha' \subset S'$, if $h \mid \alpha$ is a homeomorphism onto $\alpha'$ then (with notations as above including respective filling supports $F$ and $F'$) we have:
\begin{enumerate}
\item\label{ItemMono}
$F \subgroup F'$ and $h(\beta^\Fill_g(\alpha;S)) \subset \beta^\Fill_g(\alpha;'S')$ for all $g \in \Gamma$.
\item\label{ItemPlateau} 
$F = F'$ if and only if $h$ preserves filling protocomponents. 
\item\label{ItemKRFillingInequality}
$\KR(\alpha) \le \KR(\alpha')$, with equality if and only if $h$ preserves filling protocomponents.
\end{enumerate}
As a consequence, if $\alpha$ fills $S$ then $\alpha'$ fills $S'$.
\end{lemma}

\begin{proof} For the proof we abbreviate protoforest notations by dropping the argument $(\alpha;S)$, for example $\beta^\Over_\Id = \beta^\Over_\Id(\alpha;S)$. We also drop $(\alpha';S')$ but in its place we add a ``prime'' symbol, for example $\beta'{}^\Over_\Id = \beta^\Over_\Id(\alpha';S')$.

Item~\pref{ItemKRFillingInequality} follows immediately from~\pref{ItemMono} and~\pref{ItemPlateau} together with \cite[Lemma 2.9]{\RelFSHypTwoTag} and the equations $\KR(\alpha)=\KR(F)$ and $\KR(\alpha')=\KR(F')$.
 
Since $\beta^\Over_\Id$ is $\alpha$-connected and contains $\alpha$, and since $\alpha'=h(\alpha)$, it follows that $h(\beta^\Over_\Id)$ is $\alpha'$-connected and contains~$\alpha'$. But $\beta'{}^\Over_\Id$ is union of all $\alpha'$ connected subgraphs of $S'$ that contain~$\alpha'$, hence $h(\beta^\Over_\Id) \subset \beta'{}^\Over_\Id$. For each $g \in \Stab(\beta^\Over_\Id)$ it follows that $(g \cdot \beta'{}^\Over_\Id) \intersect \beta'{}^\Over_\Id$ contains $h(\beta^\Over_\Id)$ and so the two protocomponents $g \cdot \beta'{}^\Over_\Id$ and $\beta'{}^\Over_\Id$ of $\beta(\delta';S')$ are not edge disjoint, hence they are equal, and therefore $g \in \Stab(\beta'{}^\Over_\Id)$, proving that $\Stab(\beta^\Over_\Id) \subgroup \Stab(\beta'{}^\Over_\Id)$. Since $F$, $F'$ are the smallest free factors of $\Gamma$ \relA\ that contain the respective subgroups $\Stab(\beta^\Over_\Id)$ and $\Stab(\beta'{}^\Over_\Id)$, it follows that $F \subgroup F'$. It follows furthermore that 
$$h(\beta^\Fill_\Id) = h(F \cdot \beta^\Over_\Id) = F \cdot h(\beta^\Over_\Id)  
\subset F' \cdot \beta'{}^\Over_\Id = \beta'{}^\Fill_\Id
$$
which completes the proof of item~\pref{ItemMono} in the special case $g=\Id$; the general case follows by $\Gamma$-invariance of protoforest components and the identities $g \cdot \beta^\Fill_\Id = \beta^\Fill_g$ and $g \cdot \beta'{}^\Fill_\Id = \beta'{}^\Fill_g$.

If $h$ preserves filling protocomponents then it immediately follows that $F = \Stab(\beta^\Fill_\Id)=\Stab(\beta'{}^\Fill_\Id)=F'$. For the other direction, suppose that $F=F'$. Letting $g$ vary over a set of left coset representatives in $\Gamma$ of the subgroup $\Stab(\beta^\Fill_\Id)=F=F'=\Stab(\beta'{}^\Fill_\Id)$, the trees $g \cdot \beta^\Fill_\Id$ form an edge disjoint cover of $\beta^\Fill(\alpha;S)$ and the trees $g \cdot \beta'{}^\Fill_\Id$ form an edge disjoint cover of $\beta^\Fill(\alpha';S')$; but from \pref{ItemMono} we also have $h(\beta^\Fill_\Id) \subset \beta'{}^\Fill_\Id$, and by equivariance it follows that $h(g \cdot \beta^\Fill_\Id) \subset g \cdot \beta'{}^\Fill_\Id$. These facts taken together imply that $h(g \cdot \beta^\Fill_\Id)=g \cdot \beta'{}^\Fill_\Id$ for all $g \in \Gamma$, and the following chain of equivalences for each $g_1,g_2 \in \Gamma$ completes the proof that $h$ preserves protocomponents equivariantly: 
\begin{align*}
g_1 \cdot \beta^\Fill_\Id = g_2 \cdot \beta^\Fill_\Id &\iff g_2^\inv g_1(\beta^\Fill_\Id) = \beta^\Fill_\Id \iff g_2^\inv g_1 \in F=F' \\
&\iff g_2^\inv g_1(\beta'{}^\Fill_\Id)=\beta'{}^\Fill_\Id \\
&\iff h(g_1 \cdot \beta^\Fill_\Id)  = g_1(\beta'{}^\Fill_\Id)=g_2(\beta'{}^\Fill_\Id)= h(g_2 \cdot \beta^\Fill_\Id) 
\end{align*}
\end{proof}

\subsubsection{A distance bound while Kurosh rank is constant}
\label{SectionConstantKRBound}

In Lemma~\ref{LemmaFillingDiameterBound} to follow, we will be concerned with finding an upper bound on the diameter of a fold sequence along which there is a ``fold-invariant path sequence'' of constant filling complexity; but the lemma does have an additional hypothesis requiring that each relevant path has interior crossings of translates of every natural edge. 

The proof of the lemma will require the concept of \emph{combing rectangles} which we quickly review before stating the lemma (see from \cite[Section 4.3]{\RelFSHypTag}; we will also need this concept later, in Section~\ref{SectionEnablingProjection}, when reviewing \emph{projection diagrams}).

%Definition~\pref{DefCombingRectangle}~\pref{ItemCombRectPullback}

\smallskip

Recall the concept of a ``nondegenerate subgraph'' of a free splitting, in Section~\ref{SectionSTOATTerms} under the heading \emph{Free splittings and their maps}.

\begin{definition}[Pullbacks]
\label{DefNondegenSubgraph}
Consider a foldable map $g \from V' \to V$ of free splittings and a nondegenerate subgraph $\sigma \subset V$ (with respect to some subdivision of $V$). The \emph{pullback} of $\sigma$ (with respect to $g$, in $V'$) is the nondegenerate subgraph $g^*(\sigma) \subset V'$ that is obtained from $g^\inv(\sigma)$ by removing any single point component of $g^\inv(\sigma)$. When working with pullbacks we generally use ``edgelet'' terminology, because one cannot generally say that $g^*(\sigma) \subset V'$ is a subforest with respect to the given simplicial structure on $V'$; at best it is a subforest with respect to a subdivisions of $V'$ and $V$ such that $g$ is simplicial and $\sigma \subset V$ is a simplicial subforest. 

Along any foldable sequence $V_I \xrightarrow{g_{I+1}} \cdots \xrightarrow{g_J} V_J$, a \emph{pullback sequence} consists of nondegenerate subgraphs $\beta_i \subset V_i$, one for each $i \in [I,\ldots,J]$, such that $\beta_i$ is the pullback of $\beta_j$ under $f^i_j \from V_i \mapsto V_j$ for all $i \le j \in [I,\ldots,J]$. A pullback sequence determines and is determined by its final term: for any nonempty, nondegenerate subgraph $\beta_J \subset V_J$, by pulling back one map at a time we obtain the unique pullback sequence with final term~$\beta_J$.
\end{definition}

\begin{definition}
\label{DefCombingRectangle}
(\cite[Section 4.3]{\RelFSHypTag})
A \emph{combing rectangle} is a commutative diagram of free splittings and equivariant maps of the form
$$\xymatrix{
V_I \ar[r]^{g_{I+1}} \ar[d]_{\<\sigma_I\>}^{\pi_I} & \cdots \ar[r]^{g_{i-1}} 
& V_{i-1} \ar[r]^{g_i}\ar[d]_{\<\sigma_{i-1}\>} \ar[d]^{\pi_{i-1}}
& V_{i} \ar[r]^{g_{i+1}} \ar[d]_{\<\sigma_{i}\>} \ar[d]^{\pi_i}
& \cdots \ar[r]^{g_J} 
& V_J \ar[d]_{\<\sigma_J\>} \ar[d]^{\pi_J} \\
U_I \ar[r]_{f_{I+1}} & \cdots \ar[r]_{f_{i-1}} & U_{i-1} \ar[r]_{f_i} &U_{i} \ar[r]_{f_{i+1}} & \cdots \ar[r]_{f_J} & U_J
}$$
such that the following hold:
\begin{enumerate}
\item Each row is a foldable sequence.
\item Each vertical arrow is a collapse map $\pi_i \from V_i \to U_i$ with indicated collapse subforest $\sigma_i \subset V_i$.
\item\label{ItemCombRectPullback} 
The $\sigma_i$'s form a pullback sequence along the $V$-row.
\end{enumerate}
We recall the \emph{Combing By Expansion Lemma}, \cite[Lemma 4.9]{\RelFSHypTag} which, starting with any foldable sequence $U_I \mapsto\cdots\mapsto U_J$ and any collapse map $V_J \xrightarrow{\<\sigma_J\>} U_J$, produces a combing rectangle as denoted above. In addition, from the proof of that lemma we extract the following property which we add to the previous three properties to define a \emph{combing by expansion rectangle}:
\begin{enumeratecontinue}
\item\label{ItemFiberProduct}
For each $I< i \le J$ the $i^{\text{th}}$ square in the diagram is a fiber product diagram in the category of minimal $\Gamma$ trees, meaning that $V_{i-1}$ is the unique minimal $\Gamma$-invariant subtree of the $\Gamma$-forest 
$$\{(x,y) \in U_{i-1} \times V_i \suchthat f_i(x)=\pi_i(y)\}
$$
\end{enumeratecontinue}
If desired, after possibly subdividing each tree in the above combing diagram, we may assume that each map depicted is a simplicial map, so each horizontal arrow therefore takes each edgelet (of its domain) to an edgelet (of its range), and each vertical arrow takes each edgelet to either an edgelet or a vertex. Assuming that holds then, as with any simplicial collapse map, for each $I \le i \le J$ we have:
\begin{enumeratecontinue}
\item\label{ItemVUEdgeBijection}
The map $\pi_i \from V_i \to U_i$ induces an edgelet bijection $\E(V_i \setminus \sigma_i) \leftrightarrow \E(U_i)$. 
\end{enumeratecontinue}
This completes Definition~\ref{DefCombingRectangle}. 
\end{definition}

\medskip
 
Fix a fold sequence in $\FS(\Gamma;\A)$ of the form
$$U_I \xrightarrow{f_{I+1}} U_{I+1} \xrightarrow{f_{I+2}} \cdots \xrightarrow{f_J} U_J
$$ 
and a sequence of paths $\alpha_i \in U_i$, such that $f_i \restrict \alpha_{i-1}$ is a homeomorphism onto $\alpha_i$ (for $I <i \le J$). Let $F_i \subgroup \Gamma$ denote the filling support of~$\alpha_i$ (Definition~\ref{DefinitionFineFFS}), and $K_i=\KR(\alpha_i)=\KR(F_i)$ the filling rank. Applying Lemma~\ref{LemmaFoldFillingProtoforest} we have subgroup inclusions and numerical inequalities 
$$F_I \subgroup F_{I+1} \subgroup \cdots \subgroup F_J \quad\text{and}\quad K_I \le K_{I+1} \le \cdots \le K_J
$$

\begin{lemma}
\label{LemmaFillingDiameterBound}
If $\alpha_i$ has an interior crossing of a translate of each natural edge of $U_i$ (for $I \le i \le J$), and if the sequence of Kurosh ranks $K_i = K$ is a constant independent of $I \le i \le J$, and if the strict inequality $K < \KR(\Gamma;\A)=\abs{\A} + \corank{\A}$ holds, then the diameter in $\FS(\Gamma;\A)$ of the fold sequence \hbox{$U_I \mapsto\cdots\mapsto U_J$} is $\le 6$.
\end{lemma}

\begin{proof} For the proof we use abbreviated ``filling protoforest'' notation $\beta^\Fill_i = \beta^\Fill(\alpha_i,U_i)$ and $\beta^\Fill_{i,g} = g \cdot \beta^\Fill_{i,\Id} = \beta^\Fill_g(\alpha_i;U_i)$.

To set up the proof of Lemma~\ref{LemmaFillingDiameterBound}, since $K=\KR(F_i)$ is constant for $I \le i \le J$, it follows that $F=F_i$ is also constant (\cite[Lemma~2.9]{\RelFSHypTwoTag}). Applying Lemma~\ref{LemmaFoldFillingProtoforest} it follows that fold map $f_i \from U_{i-1} \to U_i$ preserves filling protocomponents (with respect to $\alpha_{i-1}$ and $\alpha_i$).

We note that $F$ is not the trivial subgroup. To see why, by hypothesis there is a subset $\hat\alpha$ of the interior of $\alpha_i$ consisting of a union of exactly one representative of each natural edge orbit of~$U_i$. It follows that $\alpha_i \setminus \hat\alpha$ contains some edgelet $e$ of~$U_i$, and it also follows that $g \cdot e \subset \hat\alpha$ for some $g \ne \Id \in \Gamma$. The paths $\alpha_i$ and $g^\inv \cdot \alpha_i$ therefore overlap, and so $g^\inv \in \Stab(\beta^\Over_{i,\Id}) \subgroup F$. 

For the rest of the proof we index the protocomponents of each protoforest $\beta^\Fill_i$ by choosing a set of left coset representatives of $\Gamma$ modulo $F$ that includes the identity element $\Id$. For notating these coset representatives and their associated protocomponents we use $\kappa$ instead of~$g$, writing $\beta^\Fill_{i,\kappa}=\kappa \cdot \beta^\Fill_{i,\Id}$. 

As seen in Section~\ref{SectionBlowingUpOnlyIf}, the protoforest $\beta^\Fill_J$ satisfies the Distinct Stabilizer Hypothesis. We may therefore apply the Proforest Blowup Lemma~\ref{LemmaExpansionConstruction}, obtaining an expansion $\pi_J \from V_J \xrightarrow{\<\sigma_J\>} U_J$ that satisfies the following property $(A_J)$ and a quick consequence $(B_J)$:
\begin{itemize}
\item[$(A_J)$] There is an invariant subforest $\tilde\beta_J \subset V_J$ with component decomposition $\tilde\beta_J = \bigsqcup_{\kappa} \tilde\beta_{J,\kappa}$, such that $\pi_J$ restricts to a simplicial isomorphism $\tilde\beta^{\vphantom{\Fill}}_{J,\kappa} \to \beta^\Fill_{J,\kappa}$ for each $\kappa$; in particular $\Stab(\tilde\beta_{J,\Id}) = \Stab(\beta^\Fill_{J,\Id}) = F_J=F$.
\item[$(B_J)$] $\tilde\beta_J \subset V_J$ is a proper subforest, equal to the uncollapsed subforest $V_J \setminus \sigma_J$.
\end{itemize}
Property $(B_J)$ is a consequence of $(A_J)$ which, together with the  fact that $\beta^\Fill_J$ covers $U_J$, implies the equation $\tilde\beta_J = V_J \setminus \sigma_J$; also properness of $\tilde\beta_J$ follows from the equation $\F[\tilde\beta_J] = \F[\beta^\Fill_J]$ (Lemma~\ref{LemmaExpansionConstruction}~\pref{ItemBlowupFFEqual}) and the strict inequality $\KR(F_J) = K < \KR(\Gamma;\A)$.

\smallskip

Applying the statement and the proof of the \emph{Combing By Expansion Lemma}, \cite[Lemma 4.9]{\RelFSHypTag} as indicated in Definition~\ref{DefCombingRectangle}, choose a combing-by-expansion rectangle with notation as in that definition. By downward induction we shall extend properties $(A_J)$, $(B_J)$ to a sequence of properties $(A_i)$, $(B_i)$, and add a new sequence of properties $(C_i)$, all defined for $I \le i \le J$, as follows:
\begin{itemize}
\item[$(A_i)$] There is an invariant subforest $\tilde\beta_i \subset V_i$ with component decomposition $\tilde\beta_i = \bigsqcup_{\kappa} \, \tilde\beta_{i,\kappa}$ such that $\pi_i$ restricts to a simplicial isomorphism $\tilde\beta_{i,\kappa} \to \beta^\Fill_{i,\kappa}$ for each $\kappa$; in particular $\Stab(\tilde\beta_{i,\Id})=\Stab(\beta^\Fill_{i,\Id}) = F_i = F$.
\item[$(B_i)$] $\tilde\beta_i \subset V_i$ is a proper subforest, equal to the uncollapsed subforest $V_i \setminus \sigma_i$.
\item[$(C_i)$] If $i \ge 1$, the foldable map $g_i$ restricts to a map $g_i \from \tilde\beta_{i-1} \to \tilde\beta_i$ that induces a component bijection $g_i(\tilde\beta_{i-1,\kappa}) = \tilde\beta_{i,\kappa}$.
\end{itemize}
To start the proof, we know already that $(A_J)$ and $(B_J)$ are true. For each $i \ge 1$ assuming by downward induction that $(A_i)$ and $(B_i)$ are true, we shall prove that $(C_i)$ and $(A_{i-1})$ are true; and then $(B_{i-1})$ is a consequence of $(A_{i-1})$ for the same reasons that $(B_J)$ is a consequence of~$(A_J)$.

To prove $(A_{i-1})$, from the inductive assumption that $(A_i)$ holds we get a restricted simplicial isomorphism $\pi_i \from \tilde\beta_{i,\kappa} \to \beta^\Fill_{i,\kappa}$ for each $\kappa$. Knowing that $f_i$ preserves protocomponents, we obtain a restricted simplicial map $f_i \from \beta^\Fill_{i-1,\kappa} \to \beta^\Fill_{i,\kappa}$. Let $\tilde\beta_{i-1,\kappa} \subset V_{i-1}$ be the subforest of all edges $e \subset V_{i-1}$ for which $\pi_{i-1}(e) \subset \beta^\Fill_{i-1,\kappa}$ and $g_i(e) \subset \tilde\beta_{i,\kappa}$. By applying Definition~\pref{DefCombingRectangle} item~\pref{ItemFiberProduct} and using that $\pi_i \from \tilde\beta_{i,\kappa} \to \beta^\Fill_{i,\kappa}$ is a simplicial isomorphism, the restricted map $\pi_{i-1} \from \tilde\beta_{i-1,\kappa} \to \beta^\Fill_{i-1,\kappa}$ is a simplicial isomorphism, and in particular $\tilde\beta_{i-1,\kappa}$ is connected. Since $g_i(\tilde\beta_{i-1,\kappa}) \subset \tilde\beta_{i,\kappa}$, it follows that as $\kappa$ varies the trees $\tilde\beta_{i-1,\kappa}$ are pairwise disjoint, and hence $\tilde\beta_{i-1} = \bigsqcup_\kappa \tilde\beta_{i-1,\kappa}$ is the component decomposition, which proves $(A_{i-1})$.

To prove $(C_i)$, by definition of $\tilde\beta_{i-1}$ we have $g(\tilde\beta_{i-1,\kappa}) \subset \tilde\beta_{i,\kappa}$. As $\kappa$ varies, we know that the components $\tilde\beta_{i,\kappa}$ of $\tilde\beta_i$ are pairwise disjoint. The trees $\tilde\beta_{i-1,\kappa}$ are therefore pairwise disjoint, hence they are indeed the components of the forest $\tilde\beta_{i-1}$, proving $(C_i)$.

\medskip

For each $I \le i < j \le J$ we must prove the distance bound $d(U_i,U_{j}) \le 6$, which we do using methods of \cite[Section 5]{\RelFSHypTwoTag} that go back to the proof of \cite[Lemma 5.2]{\FSOneTag}. Since each vertical arrow is a collapse map we have $d(V_i,U_i)$, $d(V_{j},U_{j}) \le 1$, and we shall prove $d(V_i,V_{j}) \le 4$. Consider the foldable map 
$$h = g^i_{j} \from V_i \to V_{j}
$$
Combining properties $(C_{i+1})$, \ldots, $(C_{j})$, the restricted map $h \from \tilde\beta_i \to \tilde\beta_{j}$ induces a component bijection $\tilde\beta_{i,\kappa} \to \tilde\beta_{j,\kappa}$. Applying \cite[Lemma/Definition~5.2]{\RelFSHypTwoTag} together with \cite[Propositions~5.3, 5.4]{\RelFSHypTwoTag}, by prioritizing folding the subforest $\tilde\beta_i \subset V_i$ we obtain a fold factorization of $h$ of the form
$$V_i = W_0 \xrightarrow{h_1} \cdots \xrightarrow{h_K} W_K \xrightarrow{h_{K+1}} \cdots \xrightarrow{h_L} W_L = V_{j}
$$
such that the following hold, with foldable maps denoted $h^k_l = h_l \circ\cdots\circ h_{k+1} \from W_k \to W_l$ and invariant subforests denoted $\delta_k = h^0_k(\tilde\beta_i) \subset W_k$:
\begin{enumerate}
\item\label{ItemTwoEdgesOfdelta}
Each fold map $h_k \from W_{k-1} \to W_k$ with $0 < k \le K$ is defined by folding two edges of $\delta_{k-1}$.
\item\label{ItemFirstHalfTwo} $d(W_0,W_K) \le 2$ (by combining \pref{ItemTwoEdgesOfdelta} with \cite[Proposition~5.4]{\RelFSHypTwoTag}).
\item\label{ItemRestrictInjectdeltaK}
The foldable map $h^K_L \from W_K \to W_L$ restricts to an injection on each component of $\delta_K$.
\end{enumerate}
Since the map $h = g^i_j$ induces a bijection from components of $\delta_0 = \tilde\beta_i$ to components of $\delta_L = \tilde\beta_j=\delta_L$, it follows for each $0 \le k < l \le L$ that the map $h^k_l$ induces a bijection from components of $\delta_k$ to components of $\delta_l$. Using this for the case $k=K$ and $l=L$, and combining it with item~\pref{ItemRestrictInjectdeltaK} above, it follows that $h^K_L$ restricts to an equivariant homeomorphism $\delta_K \mapsto \delta_L$. Finally, since $h$ does not identify any edgelet of $\tilde\beta_i$ with an edgelet of $V_i \setminus \tilde\beta_i$, it follows that $h^K_L$ does not identify any edgelet of $\delta_K$ with an edgelet of $W_K \setminus \delta_K$. In summary, the remainder of the fold sequence $W_K \mapsto\cdots\mapsto W_L$ prioritizes folding of the subforest $W_K \setminus \delta_K$. Again applying \cite[Proposition~5.4]{\RelFSHypTwoTag}, we obtain
\begin{enumeratecontinue}
\item\label{ItemSecondHalfTwo} $d(W_K,W_L) \le 2$. 
\end{enumeratecontinue}
Combining~\pref{ItemFirstHalfTwo} and~\pref{ItemSecondHalfTwo} we get $d(V_i,V_{j}) = d(W_0,W_L) \le d(W_0,W_K) + d(W_K,W_L) \le 2 + 2 = 4$ which completes the proof of Lemma~\ref{LemmaFillingDiameterBound}.
\end{proof}

\subsubsection{Proof of the \STOAT} 
\label{SectionTheProof}

Consider a foldable map $f \from S \to T$ between free splittings $S,T \in \FS(\Gamma;\A)$. Choose a Stallings fold factorization of $f$ with folds of length~$\le 1$ (see~\cite[Theorem 2.17~(4)]{\RelFSHypTwoTag}), denoted
$$f \from S = S_0 \xrightarrow{f_1} S_1 \xrightarrow{f_2} \cdots \xrightarrow{f_M} S_M=T
$$ 
Using the constant $\Delta = \Delta(\Gamma;\A) \ge 1$ from the \TOAT\ \cite{\RelFSHypTwoTag}, we assume a preliminary lower bound
$$d(S,T) \ge 3 \Delta
$$
Using the fact that each fold has length~$\le 1$, we can choose $L$ so that $0 < L \le M$ and so that $d(S,S_L) =  d(S_0,S_L) = 3 \Delta$. Applying the \TOAT\ to the foldable map~$f^0_L$ we obtain two natural edges $E_0,E'_0 \subset S_0$ in distinct orbits such that each of the natural tiles $f^0_L(E_0), f^0_L(E'_0) \subset S_L$ contains $2^{3-1}=4$ nonoverlapping subpaths each of which crosses a translate of every natural edge of~$S_L$. Choose $\alpha_L \subset S_L$ to be either of the paths $f^0_L(E_0)$ or $f^0_L(E'_0)$, and then for each $l$ such that $l \le l \le M$, let $\alpha_l = f^L_l(\alpha_L)$. The path $\alpha_M$ is thus either of $f^0_M(E_0)$ or $f^0_M(E'_0)$, and our goal has become: 
\begin{itemize}
\item To find some lower bound $d(S,T)=d(S_0,S_M) \ge \Theta = \Theta(\Gamma;\A)$ (with $\Theta \ge 3 \Delta$) which implies that $\alpha_M$ satisfies the \emph{Filling Criterion} of Proposition~\ref{PropFillingPath}:

\noindent
\textbf{(1)${}_M$:} \, $\alpha_M$ has an interior crossing of some translate of every natural edge of $S_M$;

\noindent
\textbf{(2):} $\KR(\alpha_M)=\KR(\Gamma;\A)$.
\end{itemize}
Once that goal is achieved, by applying Proposition~\ref{PropFillingPath} it will then follow that $\alpha_M$ fills $S_M=T$, and the proof is done.

For purposes of verifying \textbf{(1)${}_M$}, and for the broader purposes of applying Lemma~\ref{LemmaFillingDiameterBound} along the fold sequence from $S_L$ to $S_M$, for each $L \le l \le M$ we first verify:
\begin{description}
\item[(1)${}_l$:] $\alpha_l$ has an interior crossing of some translate of every natural edge of $S_l$
\end{description}
We start with the fact that the path $\alpha_L$ ($=f^0_L(E_0)$ or~$f^0_L(E'_0)$) contains $4$ nonoverlapping subpaths of $S_L$ each of which crosses a translate of every natural edge of~$S_L$. We next distill that fact down to the statement that $\alpha_L \subset S_L$ contains $3$ nonoverlapping subpaths each of which crosses a translate of every edge of $S_L$. That property is inherited by $\alpha_l \subset S_l$ for each $L \le l \le M$, and so we obtain a decomposition of the form 
$$\underbrace{\xymatrix{
\bullet \ar@{-}[r]^{\nu_0} & 
\bullet \ar@{-}[r]^{\mu_1} & 
\bullet \ar@{-}[r]^{\nu_1} & 
\bullet \ar@{-}[r]^{\mu_2} & 
\bullet \ar@{-}[r]^{\nu_2} & 
\bullet \ar@{-}[r]^{\mu_3}  &
\bullet \ar@{-}[r]^{\nu_3}  &
\bullet
}}_{\alpha_l}
$$
such that each of $\mu_1, \mu_2, \mu_3$ crosses a translate of every edge of $S_l$; the $\nu_i$ subpaths are allowed to be empty. Each $\mu_i$ must contain a natural vertex of $S_l$, for otherwise $\mu_i$ is a proper subpath of some natural edge $E \subset S_l$, contradicting that $\mu_i$ crosses a translate of every edge of $S_l$. Also, the subpath $\mu_1$ must contain a natural vertex $w_1$ of $S_l$ which is contained in the interior of $\alpha_l$: otherwise no point of $\mu_1$ except possibly its left endpoint is a natural vertex, implying that $\mu_1$ is a proper subpath of some natural edge of $S_l$, again a contradiction. The subpath $\mu_3$ must similarly contain a natural vertex $w_3$ of~$S_l$ which is in the interior of $\alpha_l$. It follows that $\overline{w_1 w_3}$ is a natural path in $S_l$, it is an interior subpath of $\alpha_l$, and it contains the path~$\mu_2$. Consider any natural edge $E \subset S_l$. Choosing any edge $e \subset E$ we obtain $\gamma \in \Gamma$ such that $\gamma \cdot e \subset \mu_2$. The natural edge $\gamma \cdot E$ therefore overlaps the natural path $\overline{w_1 w_3}$, implying that $\gamma \cdot E \subset \overline{w_1 w_3}$, and hence $\gamma \cdot E$ is contained in the interior of~$\alpha_l$, completing the proof of (1)${}_l$.

\medskip

We turn now to verifying~(2), starting with a formula for~$\Theta$. For $L \le l \le M$ consider the filling protoforest~$\beta^\Fill_l = \beta^\Fill(\alpha_l;S_l)$, with protocomponent $\beta^\Fill_{l,\Id} = \beta^\Fill_\Id(\alpha_l;S_l)$ containing $\alpha_l$, and with filling support $F_l = \Stab(\beta^\Fill_{l,\Id}) \subgroup \Gamma$ of Kurosh rank~$\KR(\alpha_l) = \KR(F_l)$. By applying Lemma~\ref{LemmaFoldFillingProtoforest} we have subgroup inclusions $F_L \subgroup \cdots \subgroup F_M$ and integer inequalities $\KR(\alpha_L) \le \cdots \le \KR(\alpha_M)$, and furthermore for each $L < l \le M$ we have $F_{l-1}=F_l$ if and only if $\KR(\alpha_{l-1}) = \KR(\alpha_l)$. We may therefore decompose the fold sequence $S_L \mapsto\cdots\mapsto S_M$ into maximal fold subsequences along which Lemma~\ref{LemmaFillingDiameterBound} may be applied to obtain constant Kurosh rank along each fold subsequence:
\begin{align*}
\underbrace{S_L = S_{M_0} \mapsto\cdots\mapsto S_{M_1-1}}_{\KR_1} \xrightarrow{f_{M_1}} \underbrace{S_{M_1} \mapsto\cdots\mapsto S_{M_2-1}}_{\KR_2} & \xrightarrow{f_{M_2}} \cdots \\
\cdots & \xrightarrow{f_{M_{N-1}}} \underbrace{S_{M_{L-1}} \mapsto\cdots\mapsto S_{M_N-1} = S_M}_{\KR_N} \\
\KR_n = \KR(\alpha_{M_{n-1}}) = \cdots = \KR(\alpha_{M_n - 1}) &\quad\text{for $1 \le n \le N$} \\
\KR_{n-1} < \KR_n &\quad\text{for $1 < n \le N$}
\end{align*}
Suppose now that:
\begin{description}
\item[$(\#)$] $\alpha_M$ does \emph{not} fill $S_M$, equivalently $\KR_N = \KR(\alpha_M) < \KR(\Gamma;\A)$
\end{description}
For each $1 \le n \le N$ it follows that $\KR_n \le \KR_N <  \KR(\Gamma;\A)$ and so we can apply Lemma~\ref{LemmaFillingDiameterBound} to each fold subsequence, with the conclusion that the diameter in $\FS(\Gamma;\A)$ of each fold subsequence is~$\le 6$. From the sequence of strict integer inequalities $0 \le \KR_1 < \cdots < \KR_N < \KR(\Gamma;\A)
$ we obtain $N \le \KR(\Gamma;\A)$. Taking into account the diameter bound of $1$ for each of the individual folds $f_{M_l} \from S_{M_l-1} \to S_{M_l}$, the fold subsequence $S_L \mapsto\cdots\mapsto S_M$ therefore has diameter bounded above by the constant $7 \cdot \KR(\Gamma;\A) - 1$.

Combining this with the previously obtained distance equation $d(S_0,S_L) = 3 \Delta$, under the supposition $(\#)$ we obtain the bound 
\begin{align*}
d(S,T) = d(S_0,S_M) &\le d(S_0,S_L) + d(S_L,S_M) \\
&\le 3 \Delta + 7 \KR(\Gamma;\A) - 1
\end{align*}
Adding $1$ to the above constant, if
$$d(S,T) \ge \Theta(\Gamma;\A) = 3 \Delta + 7 \KR(\Gamma;\A)
$$
then the filling rank of $\alpha_M$ must equal $\KR(\Gamma;\A)$. This holds for each choice of $\alpha_M = f(E_0)$ or $f(E'_0)$ as noted earlier. To summarize: having also proved that the natural tile $\alpha_M$ has an interior crossing of each natural edge orbit of $S_M$, by applying Proposition~\ref{PropFillingPath} it follows that $\alpha_M$ fills~$S_M$, completing the proof of the Strong Two-Over-All Theorem~\ref{TheoremStrongTwoOverAll}.

\subsection{Comparing filling supports of nested paths}
\label{SectionNestedFillingSupports}

The proof of the \STOAT, carried out in the previous section, was based on a study of the evolution of filling support as a path is pushed forward along a fold sequence (see Lemma~\ref{LemmaFoldFillingProtoforest} and its application in Lemma~\ref{LemmaFillingDiameterBound}).

For application in Section~\ref{SectionUniformFilling}, here we study a different evolution of filling support, namely evolution under nesting of paths within a single free splitting. To review protoforest concepts and notations see Definitions~\ref{DefFineCoveringForest}, \ref{DefinitionFineFFS} and~\ref{DefinitionFillingForest}, and the \emph{Protoforest Summary} that follows Definition~\ref{DefinitionFillingForest}.

%Lemma~\ref{LemmaNestingOfFillingSupport}~\pref{ItemStrictPrecSqsubset}

\begin{lemma}
\label{LemmaNestingOfFillingSupport}
For any free splitting $S$ of $\Gamma$ \relA\ and any paths $\eta,\eta' \subset S$ the following hold:
\begin{enumerate}
\item\label{ItemEtaEtapSubset}
 If $\eta \subseteq \eta'$ then $\eta \subseteq \beta^\Fill_\Id(\eta')$.
\item\label{ItemIfEtaInBlah}
If $\eta \subseteq \beta^\Fill_\Id(\eta')$ then
\begin{enumerate}
\item\label{ItemFourRelations} $\beta^\Fill_\Id(\eta) \subseteq \beta^\Fill_\Id(\eta')$ and $\beta^\Fill(\eta) \preceq \beta^\Fill(\eta')$ and $\Fm(\eta) \subseteq \Fm(\eta')$ and $\KR(\eta) \le \KR(\eta')$.
\item\label{ItemFourEquations}
If furthermore $\beta(\eta)=\beta(\eta')=S$ then the following are equivalent:
\begin{enumerate}
\item One of the four relations in \pref{ItemFourRelations} is an equality
\item All of the four relations in \pref{ItemFourRelations} are equalities.
\end{enumerate}
\end{enumerate}
\end{enumerate}
As a consequence of~\pref{ItemEtaEtapSubset} and~\pref{ItemFourRelations}, if $\eta \subseteq \eta'$ and if $\eta$ fills $S$ then $\eta'$ fills~$S$.
\end{lemma}

\subparagraph{Remark.} Regarding item~\pref{ItemFourEquations} of this lemma, without the hypothesis $\beta(\eta)=\beta(\eta')=S$ there are counterexamples. For a simple one, let $S$ be the universal cover of the rank~$2$ rose with edges labelled $a,b$, let $\eta' \subset S$ be an $ab$ path, let $\eta \subset \eta'$ be the $a$ subpath; then $\beta(\eta) \ne S = \beta(\eta')$ and so $\beta^\Fill(\eta) \ne \beta^\Fill(\eta')$ and yet $\Fm(\eta)=\Fm(\eta')$ is the trivial free factor.

\begin{proof} Item~\pref{ItemEtaEtapSubset} follows immediately from $\eta' \subset \beta^\Fill_\Id(\eta')$. Assuming that $\eta \subseteq \beta^\Fill_\Id(\eta')$, we prove item~\pref{ItemFourRelations} in a few steps. 

\medskip
\noindent
\textbf{Step 1:} For all $g,h \in \Gamma$, if $g \cdot \eta \subset \beta^\Fill_\Id(\eta')$, and if $h \cdot \eta$ overlaps $g \cdot \eta$, then $h \cdot \eta \subset \beta^\Fill_\Id(\eta')$. 

\medskip
For the proof, choose an edge $e' \subset h \cdot \eta \,\intersect\, g \cdot \eta$, and so $e' \subset \beta^\Fill_\Id(\eta')$. It also follows that $gh^\inv \cdot e' \subset gh^\inv h \cdot \eta = g \cdot \eta \subset \beta^\Fill_\Id(\eta')$. Since $\beta^\Fill(\eta')$ is a protoforest, and since its protocomponent $\beta^\Fill_\Id(\eta')$ contains both $e'$ and $gh^\inv \cdot e'$, it follows that $gh^\inv \in \Stab(\beta^\Fill_\Id(\eta'))$. Also, therefore, $h g^\inv \in \Stab(\beta^\Fill_\Id(\eta'))$. It follows that $h \cdot \eta = h g^\inv \cdot (g \cdot \eta) \subset \beta^\Fill_\Id(\eta')$.

\medskip
\noindent
\textbf{Step 2:} $\beta^\Over_\Id(\eta) \subset \beta^\Fill_\Id(\eta')$. 

\medskip
Choose an edge $e \subset \beta^\Over_\Id(\eta)$. The conclusion that $e \subset \beta^\Fill_\Id(\eta')$ follows by induction on the length of an $\eta$-connection between an edge of $\eta$ and the edge $e$: use the hypothesis $\eta \subset \beta^\Fill_\Id(\eta')$ as the basis step where that length equals~$0$; and use Step 1 as the induction step.

\medskip
\noindent
\textbf{Step 3:} $\Fm(\eta) \subgroup \Fm(\eta')$ and hence $\KR(\eta) \le \KR(\eta')$.

\medskip
Combining Step 2 with the fact that $\beta^\Fill_\Id(\eta')$ is a protoforest, it follows that $\Stab(\beta^\Over_\Id(\eta))$ stabilizes $\beta^\Fill_\Id(\eta')$, and so 
$$\Stab(\beta^\Over_\Id(\eta)) \subgroup \Stab(\beta^\Fill_\Id(\eta')) = \Fm(\eta')
$$
But $\Fm(\eta)$ is the smallest free factor containing  $\Stab(\beta^\Over_\Id(\eta))$, by application of Definition~\ref{DefinitionFineFFS}, and so \hbox{$\Fm(\eta) \subgroup \Fm(\eta')$.}

\medskip
\noindent
\textbf{Step 4:} $\beta^\Fill_\Id(\eta) \subseteq \beta^\Fill_\Id(\eta')$ and hence $\beta^\Fill(\eta) \preceq \beta^\Fill(\eta')$.
$$\beta^\Fill_\Id(\eta) = \Fm(\eta) \cdot \beta^\Over_\Id(\eta) \underbrace{\subseteq}_{\text{\tiny Step 2}} \Fm(\eta) \cdot \beta^\Fill_\Id(\eta') \underbrace{\subseteq}_{\text{\tiny Step 3}} \Fm(\eta') \cdot \beta^\Fill_\Id(\eta') = \beta^\Fill_\Id(\eta')
$$
This completes the proof of~\pref{ItemFourRelations}.

\medskip

We prove \pref{ItemFourEquations}. First, from the definition of protoforests it follows that $\beta^\Fill_\Id(\eta)  = \beta^\Fill_\Id(\eta')$ if and only if $\beta^\Fill(\eta)=\beta^\Fill(\eta')$. Next, a nested pair of free factors of $\Gamma$ \relA\ are equal to each other if and only if their Kurosh ranks are equal \cite[Lemma 2.9]{\RelFSHypTwoTag}. Next, since $\Fm(\eta)=\Stab(\beta^\Fill_\Id(\eta))$ and similarly for $\eta'$, the equation $\beta^\Fill_\Id(\eta) = \beta^\Fill_\Id(\eta')$ implies $\Fm(\eta)=\Fm(\eta')$.

Finally, assuming that $\beta^\Fill_\Id(\eta) \ne \beta^\Fill_\Id(\eta')$, we prove that $\Fm(\eta) \ne \Fm(\eta')$. This is where we use the hypothesis of \pref{ItemFourEquations} saying that \hbox{$\beta(\eta)=S$:} the protoforest $\beta^\Fill(\eta)$ being a refinement of $\beta(\eta)$, it follows that $\beta^\Fill(\eta)$ covers~$S$. Knowing also that $\beta^\Fill(\eta) \preceq \beta^\Fill(\eta')$ and that $\beta^\Fill_\Id(\eta) \ne \beta^\Fill_\Id(\eta')$, it follows that $\beta^\Fill(\eta)$ has a protocomponent $\beta^\Fill_g(\eta)$ for some $g \in \Gamma$, such that  $\beta^\Fill_g(\eta)$ does not overlap $\beta^\Fill_\Id(\eta)$, and such that $\beta^\Fill_g(\eta) \subseteq \beta^\Fill_\Id(\eta')$. It follows that $g \not\in \Fm(\eta) = \Stab(\beta^\Fill_\Id(\eta))$. It also follows that $\beta^\Fill_\Id(\eta')$ and $g \cdot \beta^\Fill_\Id(\eta')$ overlap along edges of $\beta^\Fill_g(\eta)$, and so from the definition of a protoforest we obtain $g(\beta^\Fill_\Id(\eta'))=\beta^\Fill_\Id(\eta')$ which implies that $g \in \Fm(\eta') = \Stab(\beta^\Fill_\Id(\eta'))$.
\end{proof}

\section{Large orbits in $\FS(\Gamma;\A)$: The implication \pref{ItemActLargeOrbit}$\implies$\pref{ItemFillingLamExists} of~Theorem~B}
\label{SectionNoFillingLam}

Given $\phi \in \Out(\Gamma;\A)$, with a value of $\Omega = \Omega(\Gamma;\A)$ as yet unspecified we are assuming: 
\begin{description}
\item[Large Orbit Hypothesis (Theorem~B~\pref{ItemActLargeOrbit}):] Every $\phi$-orbit in $\FS(\Gamma;\A)$ has diameter $\ge \Omega$.
\end{description}
and we must prove
\begin{description}
\item[Conclusion (Theorem~B~\pref{ItemFillingLamExists}):] $\phi$ has an attracting lamination that fills $\Gamma$ \relA.
\end{description}
In Section~\ref{SectionOmegaAndProof} we outline how to find $\Omega$ and how to use it prove the conclusion. That outline will motivate reviews and preliminary materials contained in Sections~\ref{SectionDT}--\ref{SectionTTsAndLams}; in particular for a very quick review of attracting laminations see Section~\ref{SectionTTsAndLams}. The steps of the outline are then carried out in Sections~\ref{SectionFirstConstraint}--\ref{SectionEnablingProjection}. 

\subsection{An outline: Finding $\Omega$ and proving the implication} 
\label{SectionOmegaAndProof}

After some significant preliminary work in Sections~\ref{SectionDT}--\ref{SectionTTsAndLams}, starting in Section \ref{SectionFirstConstraint} we will find~$\Omega$ by formulating two constraints on its value in the form of two lower bounds $\Omega_1$, $\Omega_2$, each chosen to enable application of different tools and methods (see also the parallel discussion in the Introduction). To get the proof off the ground, one starts by choosing a maximal, nonfull, $\phi$-invariant free factor system~$\F$ of $\Gamma$ \relA.

In Section \ref{SectionFirstConstraint}, the first constraint $\Omega \ge \Omega_1 = 5$ is chosen to enable application of \cite[Proposition 4.24]{\RelFSHypTwoTag}, the conclusion of which describes a dichotomy between two different types of train track representative. The first case of this dichotomy can be avoided by applying the constraint $\Omega \ge \Omega_1 = 5$. The other case of this dichotomy therefore applies, providing us with a unique attracting lamination $\Lambda$ of $\phi$ relative to~$\F$. Furthermore, there is an EG-aperiodic train track axis for $\phi$ in the outer space of $\Gamma$ rel~$\F$ which we may include into the free splitting complex $\FS(\Gamma;\A)$, each position $T_i$ along that axis is a free splitting rel~$\F$ supporting an EG-aperiodic train track representative $T_i \to T_i$ of $\phi$, and the generic leaves of $\Lambda$ are realized by lines in $T_i$ that are exhausted by iteration tiles of that train track map. The hard work that remains is to prove that $\Lambda$ fills $\Gamma$ \relA.

In Section \ref{SectionSecondConstraint}, the second constraint $\Omega \ge \Omega_2 = \Omega_2(\Gamma;\A)$ is chosen to enable application of the \STOAT\ to the train track axis. The hypothesis of that theorem can be verified \emph{uniformly} at every position along the train track axis for $\phi$, by applying the lower diameter bound $\Omega \ge \Omega_2$ to orbits along the axis. We can then apply the conclusion of the \STOAT\ to show that at each position $T_i$ along the fold axis, sufficiently high iteration tiles are filling tiles in the sense of Definition~\ref{DefGeneralTiles}. We then prove that the train track axis is a bi-infinite, quasigeodesic line in $\FS(\Gamma;\A)$, by applying those filling tiles in conjunction with concepts of free splitting units from \cite[Section 4.5]{\RelFSHypTag}.  

In Section~\ref{SectionEnablingProjection}, the proof is completed by applying concepts of projection diagrams that were developed in \cite[Section 5.2]{\RelFSHypTag} and were used there to define a Masur--Minsky projection function from $\FS(\Gamma;\A)$ to any finite fold path in $\FS(\Gamma;\A)$. To show that $\Lambda$ fills, by applying Corollary~\ref{CorollaryFillingLine} it suffices to find a ``witness path'' in any given Grushko free splitting $R$ of $\Gamma$ relative to $\A$, meaning a path in $R$ that crosses a representative of every edge orbit of $R$ and is contained in a bi-infinite line of $R$ that realizes a generic leaf of~$\Lambda$. We do this with the aid of a projection diagram from $R$ to a very carefully chosen finite subpath of the fold axis. This axis subpath is chosen with the aid of the \emph{Quasi-Closest Point Property}, a consequence of the Masur--Minsky axioms that is derived in \cite[Section 5.1]{\RelFSHypTag}, and which gives some control on the distance within the axis subpath between the projection of $R$ to that subpath and the point on that subpath which minimizes distance to~$R$. The choice of axis subpath enables us to identify a free splitting $T_i$ at a certain location along that subpath, and to take a line $\ell$ in $T_i$ which represents a generic leaf of $\Lambda$, and then to find a ``multi-filling'' tile in that line, meaning a tile that contains multiple nonoverlapping subpaths each of which fills $T_i$. That ``multi-filling'' property then lets us set up a bounded cancellation argument for pushing that tile to the desired witness path in~$R$.  

\medskip

In order to carry out the above outline one needs some preliminary work, carried out in Sections~\ref{SectionDT} --- \ref{SectionTTsAndLams}, regarding how to compare attracting laminations of $\phi$ relative to~$\F$ with attracting laminations of $\phi$ relative to~$\A$, whenever one is given a nested pair of $\phi$-invariant free factor systems~$\A \sqsubset \F$ of~$\Gamma$. The summary result that we will need is found in Section~\ref{SectionFirstConstraint} under the heading ``Conclusion~(2)''. Here is a simplified version: 
\begin{itemize}
\item $\phi$ has a unique attracting lamination relative to $\A$ that is not supported by~$\F$ if and only if $\phi$ has a unique attracting lamination relative to $\F$, and furthermore there is a natural equivariant bijection between the generic leaves of these attracting laminations. 
\end{itemize}
The ``natural equivariant bijection'' in this statement comes by combining some elements of relative train track theory (see Section~\ref{SectionTTsAndLams}) with a relation between the line spaces $\wtBinf(\Gamma;\A)$ and $\wtBinf(\Gamma;\F)$ (see Section~\ref{SectionRealizingGeneralLines}). That relation between line spaces is itself described using a relation between the canonical boundaries $\bdyinf(\Gamma;\A)$ and $\bdyinf(\Gamma;\F)$ which was first worked out carefully by Dowdall and Taylor \cite{DowdallTaylor:cosurface} (see Section~\ref{SectionDT}): there is a natural $\Gamma$-equivariant homeomorphism between $\bdyinf(\Gamma;\A)$ and a dense subset $\bdy_\F(\Gamma;\A) \subset \bdyinf(\Gamma;\A)$ called the \emph{$\F$ sub-boundary} of $\Gamma$ \relA.

\subsection{Boundaries under collapse: the Dowdall Taylor correspondence}
\label{SectionDT}

For recalling notations and concepts regarding ends of free splittings and boundaries of groups relative to free factor systems: see \cite[Section 4.1.1]{\RelFSHypTwoTag}, much of which follows \cite{GuirardelHorbez:Laminations}. In particular we recall: the end space $\bdyinf T$ of a free splitting $T$ of $\Gamma$, consisting of asymptotic classes of rays in~$T$; the action $\Gamma \act  \bdyinf T$ that is naturally induced by the free splitting action $\Gamma \act T$; and the unique action $\Aut(\Gamma;\A) \act \bdyinf T$ that extends the free splitting action~$\Gamma \act T$ with respect to the natural monomorphism $\Gamma \approx \Inn(\Gamma) \subgroup \Aut(\Gamma;\A)$. Recall also the canonical boundary $\bdyinf(\Gamma;\A)$ on which $\Gamma$ acts with unique extended action by $\Aut(\Gamma;\A)$, together with equivariant identifications $I_T \from \bdyinf(\Gamma;\A) \to \bdyinf T$ which commute with the equivariant map $\bdyinf S \mapsto \bdyinf T$ induced by any equivariant simplicial map $S \mapsto T$ (between two Grushko free splittings \relA). Also, for each nonatomic free factor $F \subgroup \Gamma$ \relA\ with restricted free factor system $\A \mid F$ there is an $F$-equivariant embedding $\bdyinf(F;\A \mid F) \hookrightarrow \bdyinf(\Gamma;\A)$ which, for any free splitting $T \in \FS(\Gamma;\A)$, is induced by the inclusion $T_F \hookrightarrow T$ of the minimal subtree for the restricted action $F \act T$. The action of $\Phi \in \Aut(\Gamma;\A)$ on $\bdyinf(\Gamma;\A)$ is denoted $\bdyinf\Phi \from \bdyinf(\Gamma;\A) \to \bdyinf(\Gamma;\A)$ and may regarded as the unique $\Phi$-twisted equivariant homeomorphism of $\bdyinf(\Gamma;\A)$ with respect to the action $\Gamma \act \bdyinf(\Gamma;\A)$ (see \cite[Lemma 4.1]{\RelFSHypTwoTag}). 

Consider a nested pair of free factor systems~$\A \sqsubset \F$ of $\Gamma$. Consider also two Grushko free splittings $S$ and $T$ of $\Gamma$ relative to $\A$ and $\F$ respectively, with end spaces $\bdyinf S$ and~$\bdyinf T$. Consider, finally, a choice of equivariant simplicial map $f \from S \to T$.\footnote{Some such $f$ always exists: for $v \in S$ such that $\Stab(v)$ is nontrivial we have $[\Stab(v)] \in \A$, and so $\Stab(v)$ is a subgroup of a unique $F$ such that $[F] \in \F$; let $f(v) \in T$ be the unique point stabilized by~$F$. One then chooses the value of $f$ arbitrarily on a single representative of each vertex orbit, extending over that orbit by equivariance, and then extending equivariantly over all edges.} As recalled above, if $\A=\F$ then $f$ naturally induces an equivariant homeomorphism $\bdyinf S \to \bdyinf T$. But when $\A \ne \F$, one would not expect $f$ to induce a natural equivariant \emph{homeomorphism}. Nonetheless Dowdall and Taylor discovered that $f$ induces a natural equivariant \emph{subboundary correspondence} in the form of an equivariant homeomorphism with codomain $\bdyinf T$, the domain of which is a certain $\Gamma$-invariant subspace of $\bdyinf S$ called the \emph{$T$-subboundary of $S$}, denoted  $\bdy_T S$ (see \cite[Section 3]{DowdallTaylor:cosurface}). In Lemma~\ref{LemmaDTBdy} we review this subboundary correspondence, emphasizing its connection with bounded cancellation concepts. 
In Lemma~\ref{LemmaRelBdyFunctor} we then translate the subboundary correspondence into the language of the canonical boundaries $\bdyinf(\Gamma;\A)$ and $\bdyinf(\Gamma;\F)$, producing an ``$\F$-subboundary'' $\bdy_\F(\Gamma;\A) \subset \bdyinf(\Gamma;\A)$ and a homeomorphism $\bdy_\F(\Gamma;\A) \mapsto \bdyinf(\Gamma;\F)$, all of which are well-defined independent of the choice of $S$, $T$ and~$f$; this homeomorphism is equivariant not just with respect to the actions of $\Gamma$ but also with respect to the extended actions of the subgroup
$$\Aut(\Gamma;\A,\F) = \Aut(\Gamma;\A) \intersect \Aut(\Gamma;\F) \subgroup \Aut(\Gamma)
$$
which contains the inner automorphism group $\Inn(\Gamma) \approx \Gamma$.

\paragraph{Examples of the subboundary correspondence.} Here are some simple examples that can be understood using collapse maps, without bounded cancellation. Consider the rank~3 free group $\Gamma = \<a,b,c\>$ and the free factor system $\F=\{\langle a \rangle\}$. The Cayley tree $S$ of the generating set $\{a,b,c\}$ is a free splitting of $\Gamma$ rel~$\A = \emptyset$, and by collapsing every $a$-edge of $S$ one obtains a Grushko free splitting $T$ of $\Gamma$ relative to~$\F$ with corresponding collapse map $S \xrightarrow{[\alpha]} T$ where $\alpha \subset S$ is the subforest of edges labelled~$a$; the components of $\alpha$ are precisely the axes of subgroups conjugate to $\<a\>$. There is no equivariant homeomorphism $\bdy_\infty S \to \bdy_\infty T$. Nonetheless, after removing from $\bdy_\infty S$ the ideal endpoints of components of $\alpha$ we obtain the subboundary $\bdy_T S$ and a natural induced bijection $\bdy_T S \to \bdy S$, and this bijection is an equivariant homeomorphism with respect to the subspace topology on $\bdy_T S$ that is induced by the end topology on $\bdy S$. 

One can continue onto another example using the free factor systems $\A'=\F$ and $\F' = \{\langle a,b \rangle\}$ of $\Gamma$: from a collapse map $S' = T \xrightarrow{[\beta]} T'$ where $\beta$ is the subforest of (the images in $T$ of) edges labelled~$b$, one obtains a Grushko free splitting $T'$ of $\Gamma$ relative to $\F'$, a subboundary $\bdy_{T'} S' \subset \bdy_\infty S'$ by removing the ideal endpoints of each component of $\beta$, and an equivariant subboundary homeomorphism $\bdy_{T'} S' \approx \bdy_\infty T'$.

\paragraph{Defining the subboundary correspondence.}
Consider now a nested pair of free factor systems $\A \sqsubset \F$ of $\Gamma$, a Grushko free splitting $S$ of $\Gamma$ \relA, a Grushko free splitting $T$ of $\Gamma$ rel~$\F$, and an equivariant map $f \from S \to T$. Recall that each $\xi \in \bdyinf S$ is represented by an asymptotic family of rays $[v,\xi) \subset S$, one for each  $v \in S$. One may consider the image $f[v,\xi) \subset T$ and ask whether $f[v,\xi)$ has finite or infinite diameter in~$T$. This property of $\xi$ is well-defined independent of $v$, because for any two $v_1,v_2 \in S$,
the symmetric difference closure $\overline{[v_1,\xi) \, \Delta \, [v_2;\xi)}$ is identical to the line segment $[v_1,v_2]$, the $f$-image of which is the finite diameter subset $f[v_1,v_2] \subset T$, hence $f[v_1,\xi)$ and $f[v_2,\xi)$ have finite Hausdorff distance in~$T$. The \emph{$T$-subboundary of~$S$} is defined as the subspace $\bdy_T S \subset \bdy_\infty S$ consisting of those $\xi \in \bdyinf S$ for which $f[v,\xi)$ has infinite diameter in~$T$. The subset $\bdy_T S$ is well-defined independent of the choice of~$f$, because any two equivariant maps $f,f' \from S \mapsto T$ are \emph{boundedly equivalent}, meaning that $\max_{x \in S} d_T(f(x),f'(x)) < \infty$ (the maximum is achieved on any finite subtree of $S$ whose $\Gamma$-orbit covers~$S$), and therefore $f[v,\xi)$, $f'[v,\xi)$ have finite Hausdorff distance in~$T$ for all $v \in T$ and all~$\xi \in \bdy T$. Note that the subset $\bdy_T S$ is invariant under the action of $\Gamma$ on $\bdy S$, for the following reasons. For any $\xi \in S - \bdy_T S$ and any $v \in S$ the set $f[v,\xi)$ has finite diameter in $T$, and it follows for each $\gamma \in \Gamma$ that the set $f[\gamma \cdot v, \gamma \cdot \xi) = f \left(\gamma \cdot [v,\xi)\right) = \gamma \cdot f[v,\xi)$ has finite diameter in $T$, hence $\gamma \cdot \xi \in S - \bdy_T S$.

In the following lemma we apply the bounded cancellation lemma for finite paths; see \cite[Section 4.1.4]{\RelFSHypTwoTag} for a discussion the origins of this lemma in \cite{Cooper:automorphisms} and in \cite[Section 3]{BFH:laminations}, and see \cite[Lemma 4.14]{\RelFSHypTwoTag} for the specific version we will be applying together with elements of its proof.

%Lemma~\ref{LemmaDTBdy}~\pref{ItemDTBdyHomeo}
%Lemma~\ref{LemmaDTBdy}~\pref{ItemDTBdyGromov}
%Lemma~\ref{LemmaDTBdy}~\pref{ItemDTEquiv}

\begin{lemma}
\label{LemmaDTBdy}
For any free factor systems $\A \sqsubset \F$ of $\Gamma$, any equivariant map $f \from S \to T$ from a Grushko free splitting $S$ of $\Gamma$ \relA\ to a Grushko free splitting $T$ of $\Gamma$ rel~$\F$ induces a map $\bdy f \from \bdy_T S \to \bdyinf T$ that satisfies the following properties:
\begin{enumerate}
\item\label{ItemDTBdyHomeo}
$\bdy f$ is a homeomorphism, and is characterized by the following property: for each $\xi \in \bdy_T S$ its image $\bdy f(\xi) \in \bdyinf T$ is the unique point satisfying the following (with constant $C$ depending only on~$f$ and on chosen equivariant geodesic metrics on $S$ and $T$):
\begin{description}
\item[Bounded Cancellation for Rays:] For each $p \in S$ we have 
$$\bigl[f(p),\bdy f(\xi)\bigr) \subset f[p,\xi) \subset N_C\bigl[f(p),\bdy f(\xi)\bigr)
$$
\end{description}
\item\label{ItemDTBdyGromov}
$f \union \bdy f \from S \union \bdy_T S \to T \union \bdyinf T$ is continuous with respect to end topologies. 
\item\label{ItemDTWellDef}
$\bdy f$ is well-defined independent of the choice of~$f$.
\item\label{ItemIfCollapse}
If $f$ is a collapse map then $\bigl[f(p),\bdy f(\xi)\bigr) = f[p,\xi)$ for all $p \in S$ and $\xi \in \bdy_T S$.
\item\label{ItemDTEquiv}
The subset $\bdy_T S \subset \bdyinf S$ is invariant under the action $\Aut(\Gamma;\A,\F) \act \bdyinf S$, and the map~$\bdy f$ is equivariant with respect to the actions of $\Aut(\Gamma;\A,\F)$ on $\bdy_T S$ and on~$\bdyinf T$.
\end{enumerate}
\end{lemma}

\begin{proof}[Proof of Lemma~\ref{LemmaDTBdy}]
We wish to apply \cite[Theorem 3.2]{DowdallTaylor:cosurface}, so we start by verifying the hypotheses of that theorem. One hypothesis is that $f$ is surjective, which follows by continuity of $f$ and minimality of the $\Gamma$ action on $T$. Another is that $f$ is Lipschitz, which follows from existence of a finite subtree of $S$ whose $\Gamma$-orbit covers~$S$. The remaining hypothesis says that $f$ is \emph{alignment preserving}, meaning that there exists a constant $D \ge 0$ such that for all $x,y,z \in S$ the following implication holds:
$$d_S(x,y) + d_S(y,z) = d_S(x,z) \implies d_T(f(x),f(y)) + d_T(f(y),f(x)) \le d_T(f(x),f(z)) + D
$$
We prove this by applying bounded cancellation to the path $[x,z] \subset S$, as follows. As in any tree equipped with a geodesic metric, from the equation $d_S(x,y) + d_S(y,z) = d_S(x,z)$ it follows that $y \in [x,z]$. Let $p$ be the unique point of $[f(x),f(z)]$ closest to $f(y)$. We may now the bounded cancellation lemma for finite paths in $S$, namely \cite[Lemma 4.14 (1)]{\RelFSHypTwoTag} (which applies even when $\A \ne \F$). From that lemma we obtain $f[x,z] \subset N_C[f(x),f(z)]$ with cancellation constant $C$ depending only on $f$, and hence $d_T(p,f(y)) \le C$. Noting that $p \in [f(x),f(y)]$ and $p \in [f(y),f(z)]$, we have
\begin{align*}
d_T(f(x),f(y)) + d_T(f(y),f(z)) &= \bigl( d_T(f(x),p) + d_T(p,f(y))\bigr) + \bigl(d_T(f(y),p) + d_T(p,f(z) \bigr) \\
   &= \bigl(d_T(f(x),p) + d_T(p,f(z))  \bigr) + 2 d_T(p,f(y)) \\
   & \le d_T(f(x),f(z)) + 2C
\end{align*}
With $D=2C$, this completes verification of the hypotheses of \cite[Theorem 3.2]{DowdallTaylor:cosurface}.

Applying \cite[Theorem 3.2]{DowdallTaylor:cosurface} and its proof, we obtain the homeomorphism $\bdy f \from \bdy_T S \to \bdy T$, and the extension $f \union \bdy f$ satisfies all the conclusions of~\pref{ItemDTBdyGromov}. 
To complete the proof of~\pref{ItemDTBdyHomeo}, consider any $p \in S$ and $\xi \in \bdy_T S$, and let $\eta = \bdy f(\xi) \in \bdyinf T$. The following statement may be found in the proof of  \cite[Theorem 3.2]{DowdallTaylor:cosurface}:
\begin{description} 
\item[$(*)$] There exists a sequence $(w_i)$ in $f[p,\xi)$ that converges to $\eta$ in $\overline T = T \union \bdyinf T$.
\end{description}
We now apply the proof of \cite[Lemma 4.14~(2)]{\RelFSHypTwoTag}, where it is shown that statement $(*)$ implies bounded cancellation for rays and the uniqueness of~$\eta = \bdy f(\xi)$.\footnote{In \cite[Lemma 4.14~(2)]{\RelFSHypTwoTag} where bounded cancellation for rays is proved under the stronger hypothesis $\A=\F$, that hypothesis was used to derive statement~$(*)$; the remainder of the proof did not depend on that hypothesis, it used only statement~$(*)$.}

Conclusion~\pref{ItemDTWellDef} follows, as in earlier arguments, using that any two equivariant maps $S \mapsto T$ are boundedly equivalent. To prove conclusion~\pref{ItemIfCollapse}, assume that $f[v,\xi) \not\subset \bigl[f(v),\bdy f(\xi)\bigr)$, choose $y \in [v,\xi)$ such that $f(y) \not\in \bigl[f(v),\bdy f(\xi)\bigr)$, and let $p \in \bigl[f(v),\bdy f(\xi)\bigr)$ be the point closest to $f(y)$. The segment $[f(v),y]$ and the ray $[f(y),\bdy f(\xi))$ both contain the point $p$; choose $x \in [v,y)$ and $z \in (y,\bdy\xi)$ such that $f(x)=f(z)=p$. Since $y \in [x,z]$ and $f(y) = p \not\in f[x,z]$, the set $f^\inv(p)$ is disconnected, contradicting the definition of a collapse map.

It remains to prove conclusion~\pref{ItemDTEquiv}. We have seen already that $\bdy_T S$ is invariant under the action $\Gamma \act \bdyinf S$, and we next show that the map $\bdy f \from \bdy_T S \to \bdyinf T$ is $\Gamma$-equivariant with respect to the action of $\Gamma$. Consider any $\xi \in \bdy_T S$ and $\gamma \in \Gamma$, let $\rho = [v,\xi) \subset S$ be a ray representing~$\xi$, and so $\gamma \cdot \rho = [\gamma \cdot v, \gamma \cdot \xi) \subset S$ is a ray representing $\gamma \cdot \xi \in \bdy_T S$. By bounded cancellation for rays it follows that $f(\gamma \cdot \rho) = f[\gamma \cdot v, \gamma \cdot \xi)$ has finite Hausdorff distance in $T$ from the ray $[f(\gamma \cdot v), \bdy f(\gamma \cdot \xi))$ representing $f(\gamma \cdot \xi) \in \bdyinf T$. Also by bounded cancellation for rays, it follows that $f[v,\xi)$ has finite Hausdorff distance in $T$ from the ray $[f(v),\bdy f(\xi))$ representing $\bdy f(\xi) \in \bdyinf T$. The point $\gamma \cdot \bdy f(\xi) \in \bdyinf T$, which is represented by the ray
\begin{align*}
[f(\gamma \cdot v), \gamma \cdot \bdy f(\xi)) &=  [\gamma \cdot f(v), \gamma \cdot \bdy f(\xi)) = \gamma \cdot [f(v),\bdy f(\xi))
\end{align*}
therefore has finite Hausdorff distance in $T$ from the ray $\gamma \cdot f[v,\xi) = f(\gamma \cdot [v,\xi)) = f(\gamma \cdot \rho)$ which, as we have seen, has finite Hausdorff distance in $T$ from a ray representing $f(\gamma \cdot \xi)$. It follows that $\gamma \cdot \bdy f(\xi) = f(\gamma \cdot \xi)$.

We next prove that $\bdy_T S$ is invariant under the action $\Aut(\Gamma;\A,\F) \act \bdyinf S$. The given action \hbox{$\Gamma \act \bdyinf S$} extends to an action \hbox{$\Aut(\Gamma;\A) \act \bdyinf S$} such that each element $\Phi \in \Aut(\Gamma;\A)$ acts by the unique $\Phi$-twisted equivariant homeomorphism of $\bdyinf S$, by applying \cite[Lemma 4.1]{\RelFSHypTwoTag}. Also applying the same lemma, for each $\Phi \in \Aut(\Gamma;\A,\F)$ there exist $\Phi^\inv$-twisted equivariant maps $h \from S \to S$ and $k \from T \to T$ inducing maps $\bdyinf h \from \bdyinf S \to \bdyinf S$ and $\bdyinf k \from \bdyinf T \to \bdyinf T$, each of which, on its respective domains, is the unique $\Phi^\inv$-twisted equivariant homeomorphisms. In particular we have the identity $\bdyinf h(\eta)=\Phi^\inv \cdot \eta$ for each $\eta \in \bdyinf S$.  The two maps $f \circ h$, $k \circ f \from S \to T$ are both $\Phi^\inv$-twisted equivariant, each being a composition (in one order or another) of a $\Phi^\inv$-twisted equivariant map and an \hbox{($\Id$-twisted)} equivariant map. It follows that $f \circ h$ and $k \circ f$ are boundedly equivalent: the distance between the values of two $\Phi^\inv$-twisted equivariant simplifical maps of $S$ is bounded by the maximum distance over a finite subtree of $S$ whose translates cover~$S$. Consider now $\xi \in \bdy_T S$ and $\Phi \cdot \xi \in \bdyinf S$, hence $\bdyinf h(\Phi \cdot \xi) = \Phi^\inv \cdot \Phi \cdot \xi = \xi$. Choose any $p \in S$ and consider the ray $[p,\Phi \cdot \xi) \subset S$. Its image $h[p,\Phi \cdot \xi)$ is Hausdorff close to $[h(p),\xi)$. Since $\xi \in \bdy_T S$ it follows that $f[h(p),\xi) \subset T$ has infinite diameter in~$T$; therefore $f(h[p,\Phi \cdot \xi)) = (f \circ h)[p,\Phi \cdot \xi)$ also has infinite diameter in $T$. By bounded equivalence of $f \circ h$ and $k \circ f$, it follows that $(k \circ f)[p,\Phi\cdot\xi) = k(f[p,\Phi\cdot\xi))$ has infinite diameter in $T$. Since $k$ is Lipschitz, it follows that $f[p,\Phi\cdot\xi)$ has infinite diameter in $T$, proving that $\Phi \cdot \xi \in \bdy_T S$. Since this holds for any $\Phi \in \Aut(\Gamma;\A,\F)$ and any $\xi \in \bdy_T S$, the proof of $\Aut(\Gamma;\A,\F)$-invariance of $\bdy_T S$ is complete.

Restricting the action $\Aut(\Gamma;\A,\F) \act \bdy S$ to $\bdy_T S$, each $\Phi \in \Aut(\Gamma;\A,\F)$ acts by a \hbox{$\Phi$-twisted} equivariant homeomorphism of $\bdy_T S$. Conjugating via the equivariant homeomorphism $\bdy f$ we obtain a $\Phi$-twisted equivariant homeomorphism of $\bdyinf T$. But \cite[Lemma 4.1]{\RelFSHypTwoTag} tells us that there is a unique $\Phi$-twisted equivariant homeomorphism of $\bdyinf T$, namely, the action of $\Phi \in \Aut(\Gamma;\A,\F) \subgroup \Aut(\Gamma;\F)$ as an element of the unique action $\Aut(\Gamma;\F) \act \bdyinf T$ that extends the given action $\Gamma \act \bdyinf T$. Putting this altogether, we have shown that $\bdy f$ is $\Aut(\Gamma;\A,\F)$ equivariant.
\end{proof}

\noindent
\emph{Remark.} Regarding conclusion~\pref{ItemIfCollapse} of Lemma~\ref{LemmaDTBdy}, one can observe that an equivariant simplicial map $f \from S \to T$ of free splittings is a collapse map if and only if $C=0$ is a cancellation constant for~$f$.

\paragraph{A canonical version of the subboundary correspondence.} Consider again a nested pair $\A \sqsubset \F$ of free factor systems of~$\Gamma$. Define the \emph{\hbox{$\F$-subboundary} of~$\Gamma$} \relA\ to be
$$\bdy_\F(\Gamma;\A) \quad = \quad \bdyinf(\Gamma;\A) \quad - \quad \bigcup_{F} \, \bdyinf(F;\A \mid F)
$$ 
where the union is taken over all nonatomic free factors $F$ of $\Gamma$ \relA\ such that $[F] \in \F$ (c.f.\ the earlier \emph{Examples of the subboundary correspondence}).

In Lemma~\ref{LemmaRelBdyFunctor} below we recast the subboundary correspondence in the language of the canonical boundaries $\bdyinf(\Gamma;\A)$: in brief, the lemma says that the inclusion map $\bdy_\F(\Gamma;\A) \hookrightarrow \bdyinf(\Gamma;\A)$ corresponds to the subboundary homeomorphism $\bdy_T S \mapsto \bdy T$ for any $\Gamma$-equivariant map $f \from T \to S$ from a Grushko free splitting \relA\ to a Grushko free splitting rel~$\F$. 

Our statement of Lemma~\ref{LemmaRelBdyFunctor} is expressed in terms of the \emph{bounded distance category of free splittings of $\Gamma$,} which has for its objects the class of free splittings of $\Gamma$, and for its morphisms the \emph{bounded equivalence classes} of the set of equivariant simplicial maps $S \mapsto T$ from one free splitting $S$ to another~$T$. But of course any two such maps \emph{are} boundedly equivalent, hence there is either no morphism at all from $S$ to~$T$ or there is a unique morphism; furthermore, a morphism exists if and only if $\FFSS \sqsubset \FST$ (see the opening passages of Section~\ref{SectionDT}). Together with the fact that every nonfull free factor system $\A$ of $\Gamma$ can indeed be realized as $\A = \FFSS$ for some free splitting $S$ of $\Gamma$ \cite[Lemma 3.1]{\RelFSHypTag}, we summarize as follows:

\begin{lemma}\label{LemmaBddEquiv}
The bounded distance category of free splittings of $\Gamma$ is equivalent, as a category, to the poset of nonfull free factor systems of~$\Gamma$ with respect to the partial ordering $\A \sqsubset \B$. \qed
\end{lemma}

\smallskip

Here is our ``canonical'' version of the subboundary correspondence:

\begin{lemma}
\label{LemmaRelBdyFunctor}
There is a contravariant functor from the poset of free factor systems of $\Gamma$ to the category of topological spaces equipped with a $\Gamma$-action, which assigns to each nonfilling free factor system $\A$ of $\Gamma$ the space $\bdyinf(\Gamma;\A)$, and to each nested pair of nonfilling free factor systems $\A \sqsubset \F$ a $\Gamma$-equivariant embedding $I^\F_\A \from \bdyinf(\Gamma;\F) \to \bdyinf(\Gamma;\A)$, such that the following hold:
\begin{enumerate}
\item\label{ItemRelBdyFunctorInclude}
$\text{image}(I^\F_\A) = \bdy_\F(\Gamma;\A)$, 
\item\label{ItemRelBdyFunctorEquiv}
$I^\F_\A$ is equivariant with respect to the actions on its domain and range of the group \\ $\Aut(\Gamma;\A,\F) = \Aut(\Gamma;\A) \intersect \Aut(\Gamma;\F)$
\item\label{ItemRelBdyFunctorCommutes}
For any equivariant map $f \from S \to T$ from a Grushko free splitting $S$~\relA\ to a Grushko free splitting $T$ rel~$\F$ we have the following $\Gamma$-equivariant commutative diagram of homeomorphims, injections, and the embedding $I^\F_\A$, and hence $\bdy_T S = I_S(\bdy_\F(\Gamma;\A))$:
\end{enumerate}
$$\xymatrix{
\bdyinf(\Gamma;\F) 
		\ar@[blue]@/^1.5pc/[rrrr]^*[blue]{I^\F_\A} 
		\ar@2{~>}[d]_{I_T} 
		\ar@[blue]@2{~>}[rr]
	&& \bdy_\F(\Gamma;\A) \,\, 
		\ar@{^{(}->}[rr] 
		\ar@[blue]@2{~>}[d]
		\ar@[blue]@2{~>}[dll]
	&& \bdyinf(\Gamma;\A) \ar@2{~>}[d]_{I_S} \\
\bdyinf T           
	&& \bdy_T S \,\, \ar@2{~>}[ll]_{\bdy f}  \ar@{^{(}->}[rr]^j     
	&& \bdyinf S 
}$$
\end{lemma}

\begin{proof} In the commutative diagram shown we have already defined all spaces and their $\Gamma$-actions, and we have defined all arrows \emph{except} for the blue ones, namely: the black arrows $\hookrightarrow$ denoting equivariant injections; and the black arrows $\xymatrix{\ar@2{~>}[r] &}$ denoting equivariant homeomorphisms. The blue embedding arrow $I^\F_\A$ is uniquely determined by the commutativity requirement: 
$$I^\F_\A = (I_S)^\inv \circ j \circ (\bdy f)^\inv \circ I_T
$$
As for equivariance properties: the $I_T$ and $I_S$ homeomorphism arrows are $\Aut(\Gamma;\F)$-equivariant (see the passages following \cite[Lemma 4.1]{\RelFSHypTwoTag}); the $\bdy f$ homeomorphism arrow and its inverse are $\Aut(\Gamma;\A,\F)$-equivariant (by Lemma~\ref{LemmaDTBdy}~\pref{ItemDTEquiv}); the subset $\bdy_T S$ of $\bdy_\infty S$ is invariant under the action of $\Aut(\Gamma;\A,\F)$ (by Lemma~\ref{LemmaDTBdy}~\pref{ItemDTEquiv}) and so the $j$ inclusion arrow is $\Aut(\Gamma;\A,\F)$-equivariant. It follows that $I^\F_\A$ is $\Aut(\Gamma;\A,\F)$ equivariant, and in particular is $\Gamma$-equivariant. 

The main steps that remain to be shown are:
\begin{description}
\item[Step 1:] $I^\F_\A$ is well-defined independent of the choice of $f \from S \to T$, 
\item[Step 2:] $\image(I^\F_\A) = \bdy_\F(\Gamma;\A)$, verifying conclusion~\pref{ItemRelBdyFunctorInclude}.
\item[Step 3:] Contravariance: for any triple of nonfilling free factor systems $\A \sqsubset \B \sqsubset \F$ of $\Gamma$ we have $I^\B_\A \circ I^\F_\B = I^\F_B \from \bdyinf(\Gamma;\F) \to \bdyinf(\Gamma;\A)$.
\end{description}
Once Steps 1 and 2 are complete, the remaining blue equivariant homeomorphism arrows~$\xymatrix{\ar@[blue]@2{~>}[r] &}$ are uniquely determined by the commutativity requirement, verifying conclusion~\pref{ItemRelBdyFunctorCommutes}.

\smallskip 

\textbf{Step 1:} Consider two equivariant maps $f \from S \to T$ and $f' \from S' \to T'$ from Grushko free splittings $S,S'$ \relA\ to Grushko free splittings $T,T'$ rel~$\F$. Choosing equivariant maps $h \from T \to T'$ and $g \from S \to S'$, the first diagram below coarsely commutes. By chasing around that first diagram we obtain (the black portion of) the second diagram, a commutative diagram of $\Gamma$-equivariant homeomorphisms and inclusions: 
$$
\xymatrix{
   &&&&*[red]{\bdyinf(\Gamma;\F)\,\,} 
	\ar@[red]@{-->}[rrrr]^*[red]{I^\F_\A}
   	\ar@[red]@2{~>}[dl]
	\ar@[red]@2{~>}[ddl]
   && 
   && *[red]{\,\, \bdyinf(\Gamma;\A)}
   	\ar@[red]@2{~>}[dl]
	\ar@[red]@2{~>}[ddl]
\\
T \ar[d]_h && S \ar[ll]_f \ar[d]_g & {\bdyinf T} \ar@2{~>}[d]_{\bdy h}
	&& \bdy_T S \,\, 
		\ar@2{~>}[ll]_{\bdy f} 
		\ar@2{~>}[d]_{\bdy g \restrict \bdy_T S} 
		\ar@{^{(}->}[rr]
	&& \bdyinf S \ar@2{~>}[d]_{\bdy g} 
\\
T' && S' \ar[ll]_{f'} & \bdyinf T' 
	&& \bdy_{T'} S' \,\,
		\ar@2{~>}[ll]_{\bdy f'} 
		\ar@{^{(}->}[rr]
	&& \bdyinf S'
}$$
In that second diagram, the two (partially red) triangles involving homeomorphisms from $\bdyinf(\Gamma;\F)$ and from $\bdyinf(\Gamma;\A)$ are already known to commute. It follows that $I^\F_\A$ is well-defined.

It remains to verify the diagram chase; here are a few details. Given $\xi \in \bdy_T S$ and any representative ray $[v,\xi) \subset S$, the $T$ image $f[v,\xi)$ has infinite diameter and has finite Hausdorff distance from the ray $\bigl[f(v),\bdy f(\xi)\bigr)$ (Lemma~\ref{LemmaDTBdy}). Since $h$ is a quasi-isometry, the $T'$ images $h\bigl[f(v),\bdy f(\xi)\bigr)$ and $h \circ f[v,\xi)$ have finite Hausdorff distance from each other and from the ray $\bigl[h(f(v)),\bdy h(\bdy f(\xi))\bigr)$. Since $g$ is a quasi-isometry, the $S'$ image $g[v,\xi)$ has finite Hausdorff distance from the ray $\bigl[g(v),\bdy g(\xi)\bigr)$, and so their $T'$ images $f' \circ g[v,\xi)$ and $f'\bigl[g(v),\bdy g(\xi)\bigr)$ have finite Hausdorff distance from each other; by coarse commutativity, they also have finite Hausdorff distance from $h \circ f[v,\xi)$ and hence also from the ray $\bigl[h(f(v)),\bdy h(\bdy f(\xi))\bigr)$. It follows that $f' \bigl[g(v),\bdy g(\xi)\bigr)$ has infinite diameter, proving that $\bdy g(\xi) \in \bdy_{T'}(S')$. This argument proves that the inclusion $\bdy g(\bdy_T S) \subset \bdy_{T'}(S')$ holds, as does the commutativity relation $\bdy h \circ \bdy f = \bdy f' \circ \left( \bdy g \restrict \bdy_T S \right)$. The reverse inclusion $\bdy_{T'}(S') \subset \bdy g(\bdy_T S)$ follows by a similar diagram chase using equivariant coarse inverses of $g$ and $h$, hence the restriction $\bdy g \restrict \bdy_T S$ is a homeomorphism $\xymatrix{\bdy_T S \ar@2{~>}[r] &\bdy_{T'} S'}$.

\medskip
\textbf{Step 2:} Applying Step~1, to prove the equation $I^{\vphantom{\F}}_S(\bdy_\F(\Gamma;\A)) = \bdy_T S$ for \emph{every} equivariant map $f \from S \to T$ from a Grushko free splitting $S$ \relA\ to a Grushko free splitting $T$ rel~$\F$, it suffices to prove that equation for just \emph{one choice} of such a map. We choose $S$ to be a free splitting of $\Gamma$ \relA\ in which the free factor system $\F$ is visible, and so there is a $\Gamma$-invariant subforest $\sigma \subset S$ such that $\F[\sigma] = \F$ (see Section~\ref{SectionSTOATStatement}; see also \emph{Examples of the subboundary correspondence} earlier here in Section~\ref{SectionDT}). We may assume that the components of $\sigma$ are precisely the minimal subtrees $S_F \subset S$ for the actions of those free factors $F \subgroup \Gamma$ such that $[F] \in \F$: we first replace $\sigma$ by its union with the set of vertices of $S$ having nontrivial stabilizer; we then discard components of $\sigma$ having trivial stabilizer; and finally we replace each component of $\sigma$ with the minimal subtree for the action of its stabilizer. We choose $T$ to be the quotient of $S$ obtained by collapsing each component of $\sigma$ to a point, with collapse map $f \from S \xrightarrow{\langle\sigma\rangle} T$. 

Given an end $\xi \in \bdyinf S$ with corresponding point $I_S^\inv(\xi) \in \bdyinf(\Gamma;\A)$, we must prove the following equivalence:
$$\xi \in \bdy_T S \iff I_S^\inv(\xi) \in \bdy_\F(\Gamma;\A)
$$
Pick a vertex $v \in S$. 

\smallskip
\emph{Case 1: The ray $[v,\xi)$ contains only finitely many edges of $S \setminus \sigma$.} In this case we prove both sides of the equivalence are false. First, clearly the $T$ image $f[v,\xi)$ consists of finitely many edges and hence has finite diameter in $T$, implying that $\xi \not\in \bdy_T S$. Also, choosing a vertex $w \in [v,\xi)$ beyond the finitely many edges of the intersection $[v,\xi) \intersect \sigma$, it follows that the $[w,\xi) \subset \sigma$. By connectivity, $[w,\xi)$ lies entirely in some component $S_F$ of $\sigma$. It follows that $\xi \in \bdyinf S_F$, hence $I_S^\inv(\xi)  \in \bdy(F;\Gamma \restrict \A)$, and therefore $I_S^\inv(\xi) \xi \not\in \bdy_\F(\Gamma;\A)$. 

\smallskip
\emph{Case 2: $[v,\xi)$ contains infinitely many edge of $S \setminus \sigma$.} In this case we prove both sides of the equivalence are true.  Enumerate the edges of $[v,\xi) \intersect (S \setminus \sigma)$, in order along $[v,\xi)$, as  $E_1 \, E_2 \, E_3 \, \cdots$. It follows that $f[v,\xi) = f(E_1) \, f(E_2) \,  f(E_3) \, \cdots$ is an infinite ray in $T$ and hence has infinite diameter, proving that $\xi \in \bdy_T S$. Also, consider any nonatomic free factor $F \subgroup \Gamma$ representing an element $[F] \in \F[\sigma]$, and consider the corresponding component $S_F$ of $\sigma$. The intersection $[v,\xi) \intersect S_F$ is connected, and if it were infinite then it would be a ray comprising all but finitely many edges of $[v,\xi)$, contradicting that $[v,\xi)$ has infinitely many edges not in $\sigma$. It follows that $[v,\xi) \intersect S_F$ is finite, hence $\xi \not\in \bdyinf S_F$ and $I_S^\inv(\xi) \not\in \bdy(F;\Gamma \restrict \A)$. Since this holds for any $F$, it follows that $I_S^\inv(\xi) \in \bdy_\F(\Gamma;\A)$.

\subparagraph{Step 3:} Applying Lemma~\ref{LemmaDTBdy} and the earlier steps, it suffices to consider $\Gamma$-equivariant maps $S \xrightarrow{f} T \xrightarrow{g} U$ where $S,T,U$ are Grushko free splittings with respect to $\A$, $\B$ and~$\F$ respectively, and to prove that for all $\xi \in \bdy_T S$, $\eta \in \bdy_U T$ and $\zeta \in \bdyinf U$, if $\bdy f(\xi)=\eta$ and $\bdy g(\eta)=\zeta$ then $\xi \in \bdy_U S$ and $\bdy(g \circ f)(\xi)=\zeta$. This is again a diagram chase argument. Picking $v \in S$, applying Lemma~\ref{LemmaDTBdy} it follows first that $f\bigl[v,\xi\bigr)$ has infinite diameter and is Hausdorff close to $[f(v),\eta)$. Applying that lemma again it follows that $g\bigl[f(v),\eta\bigr)$ has infinite diameter and is Hausdorff close to $[g(f(v)),\zeta)$. But $g \circ f\bigl[v,\xi\bigr)$ is Hausdorff close to $g\bigl[f(v),\eta\bigr)$ and hence $g \circ f\bigl[v,\xi\bigr)$ has infinite diameter, proving that $\xi \in \bdy_U S$; and furthermore $g \circ f\bigl[v,\xi\bigr)$ is Hausdorff close to $[g(f(v)),\zeta)$ proving that $\bdy(g \circ f)(\xi)=\zeta$. 
\end{proof}

\subsection{Line spaces under collapse.} 
\label{SectionRealizingGeneralLines}

We first recollect notations and concepts regarding line spaces of a group $\Gamma$ relative to a free factor system~$\A$ (see \cite[Section 4.1.2]{\RelFSHypTwoTag} which follows \cite{GuirardelHorbez:Laminations} and \cite{\LymanCTTag}). Then, in Lemma/Definition~\ref{LemmaDefBirecurrentRealization}, we describe a correspondence of bi-infinite lines spaces that is induced by the Dowdall--Taylor correspondence of end spaces.

By combining the double space construction \cite[Section 4.1.1]{\RelFSHypTwoTag}, the boundary identifications $I_T \from \bdy(\Gamma;\A) \approx \bdy T$, and the end subspace identifications $I_T \from \bdyinf(\Gamma;\A) \approx \bdyinf T$ (as $T$ varies over free splittings of $\Gamma$ \relA), one obtains line spaces and bi-infinite line subspaces with identifications and inclusions as follows: 
$$\xymatrix{
\wtB(\Gamma;\A) \ar@{=}[r] 
	& \bdy(\Gamma;\A)^\SetTwo \ar@2{~}[r] 
	& \bdy T^\SetTwo \ar@{=}[r] 
	&  \wtB T \\
\wtBinf(\Gamma;\A) 
		\ar@{=}[r] 
		\ar@{^{(}->}[u]
	& \bdyinf(\Gamma;\A)^\SetTwo 
		\ar@2{~}[r]
		\ar@{^{(}->}[u]
	& \bdyinf T^\SetTwo 
		\ar@{=}[r]
		\ar@{^{(}->}[u]
	& \wtBinf T
		\ar@{^{(}->}[u]
}$$
All arrows in this diagram are equivariant with respect to the action of $\Gamma$ and the extended action of $\Aut(\Gamma;\A)$. The group $\Out(\Gamma;\A)$ acts on $\Gamma$-invariant subsets of each space in this diagram, and we obtain an induced $\Out(\Gamma;\A)$-equivariant diagram of sets of $\Gamma$-invariant subsets. 

After some further notation review, for the remainder of the section we will mostly focus on the second line of the diagram above, for purposes of applying the Dowdall--Taylor correspondence on the level of bi-infinite lines.

Given $\ell \in \wtB(\Gamma;\A)$ and a free factor system~$\F$, to say that $\ell$ is \emph{carried by} or \emph{supported by}~$\F$ means that there exists $F \subgroup \Gamma$ such that $[F] \in \F$ and $\xi,\eta \in \bdy(F;\A \restrict F) \subset \bdy(\Gamma;\A)$. Given any subset $L \subset \wtB(\Gamma;\A)$, to say that a free factor system $\F$ \emph{supports} $L$ means that for each $\ell = \{\xi,\eta\} \in L$ there exists $F \subgroup \Gamma$ such that $[F] \in \F$ and $\xi,\eta \in \bdy(F;\A \restrict F) \subset \bdy(\Gamma;\A)$. The unique free factor system that supports $L$ and is minimal with respect to the partial order $\sqsubset$ is called the \emph{free factor support} of $L$; for the proof of existence and uniqueness see \cite[Lemma 4.2]{\RelFSHypTwoTag} following \cite[Corollary 4.12]{\LymanCTTag}. To say that $L$ \emph{fills} $\Gamma$ \relA\ means that its free factor support equals $\{[\Gamma]\}$, equivalently $L$ is not supported by any nonfull free factor system \relA. 

An individual element $\ell = \{\xi,\eta\} \in \wtBinf(\Gamma;\A)$ is called a \emph{bi-infinite line} of $\Gamma$ \relA\ with \emph{ideal endpoints}~$\xi,\eta \in \bdyinf(\Gamma;\A)$. For any Grushko free splitting $T$ of $\Gamma$ \relA\ the \emph{concrete realization} of $\ell$ in $T$, denoted either $\ell_T \subset T$ or $\ell(T)$, is the unique line in $T$ whose closure in $\overline T = T \union \bdyinf T$ is $T \union \{\xi,\eta\}$ (using $I_T$ to identify $\xi,\eta \in \bdyinf T$). Given a bi-infinite line $\ell \in \wtBinf(\Gamma;\A)$ and a Grushko free splitting $T$ of $\Gamma$ \relA, to say that the concrete realization $\ell_T \subset T$ is \emph{birecurrent in $T$} means that for any finite path $\alpha \subset \ell_T $ and any subray $\rho \subset \ell_T$ there exists $\gamma \in \Gamma$ such that $\gamma \cdot \alpha \subset \rho$. As shown in \cite[Lemma 4.1]{\LymanCTTag}, \emph{birecurrence of $\ell \in \wtBinf(\Gamma;\A)$} is well-defined, independent of the choice of Grushko free splitting $T$ \relA, by requiring that its concrete realization $\ell_T$ be birecurrent. 

For any nested pair of free factor systems $\A \sqsubset \F$, the boundary embedding $I^\F_\A \from \bdyinf(\Gamma;\F) \approx \bdy_\F(\Gamma;\A) \hookrightarrow \bdyinf(\Gamma;\A)$ given by Lemma~\ref{LemmaRelBdyFunctor} induces an embedding of line spaces
$$\underbrace{\wtBinf(\Gamma;\F)}_{\bdyinf(\Gamma;\F)^\SetTwo} \approx \underbrace{\wt\B_\F(\Gamma;\A)}_{\bdy_\F(\Gamma;\A)^\SetTwo} \hookrightarrow \underbrace{\wtBinf(\Gamma;\A)}_{\bdyinf(\Gamma;\A)^\SetTwo}
$$
and this embedding is equivariant with respect to the action of $\Aut(\Gamma;\A,\F)$.

\medskip

In our next result we take advantage of the embedding above in order to extend the concept of ``concrete realization'' of an abstract line in $\wtBinf(\Gamma;\A)$ beyond the case of Grushko free splittings \relA, making that concept apply to all free splittings \relA, but we do this at some cost: we require that the given line be birecurrent; and the ``concrete realization'' will not always be a concrete line, it could also be a point.

In the statement and later applications of the following lemma, for each nonfull free factor system $\F$ of $\Gamma$ \relA\ we identify $\wtB_\F(\Gamma;\A)$ with its embedded image in $\wtBinf(\Gamma;\A)$. Also, for each proper free factor $F \subgroup \Gamma$ \relA\ we identify $\bdyinf(F;\A \restrict F)$ with its embedded image in $\bdyinf(\Gamma;\A)$, and then in turn we identify $\wtBinf(F;\A \restrict F) = \bdyinf(F;\A \restrict F)^\SetTwo$ with its image under the induced embedding in~$\wtBinf(\Gamma;\A) = \bdyinf(\Gamma;A)^\SetTwo$. 

%Case 1 of Lemma~\ref{LemmaDefBirecurrentRealization}

\begin{LemmaAndDefinition}
\label{LemmaDefBirecurrentRealization} 
For any nonfull free factor system~$\F$ of $\Gamma$ \relA\ and any birecurrent abstract line $\ell = \{\xi,\eta\} \in \wtBinf(\Gamma;\A)$ exactly one of the following holds:
\begin{description}
\item[Case 1:] $\ell \in \wtB_\F(\Gamma;\A)$; or
\item[Case 2:] There exists a unique proper free factor $F_\ell \subset \Gamma$ \relA\ such that $[F_\ell] \in \F$ and such that $\ell \in \wtBinf(F_\ell;\A \restrict F_\ell)$.
\end{description}
For any Grushko free splitting $T$ of $\Gamma$ rel~$\F$ the \emph{concrete realization of $\ell$ in $T$}, a subset denoted $\ell_T \subset T$, is defined as follows:
\begin{description}
\item[Case 1 realization, a line:] $\ell_T \subset T$ is the concrete realization of the abstract line $\ell_\F \in  \wtBinf(\Gamma;\F) \approx \wtBinf T$ that is identified with $\ell$ under the homeomorphism $\wtBinf(\Gamma;\F) \approx \wtB_\F(\Gamma;\A)$.
\item[Case 2 realization, a point:] $\ell_T$ is a singleton, namely the unique point of $T$ fixed by $F_\ell$.
\end{description}
Furthermore, using this definition of $\ell_T$ we have:
\begin{description}
\item[Bounded Cancellation for lines:] For any Grushko free splitting $S$ of $\Gamma$ \relA\ and any map $f \from S \to T$ with cancellation constant $C$:
$$\ell_T \subset f(\ell_S) \subset N_C(\ell_T)
$$
If in addition $f$ is a collapse map then $f(\ell_S)=\ell_T$.
\end{description}
\end{LemmaAndDefinition}

\begin{proof} We first verify the case analysis. Because the subsets $\wtB_\F(\Gamma;\A)$ and $\wtBinf(F;\A \restrict F)$ form a pairwise disjoint collection in $\wtBinf(\Gamma;\A)$ as $F$ varies, clearly cases 1 and 2 are mutually exclusive, and the free factor $F$ in case 2 is unique, if it exists. To prove that one of cases 1 or 2 actually occurs, choose $S$ to be a free splitting \relA\ in which the free factor system~$\F$ is visible. For each free factor $F \subset \Gamma$ \relA, consider the minimal subtree $S_F$ for the restricted action $F \act S$, and so the inclusion $S_F \hookrightarrow S$ induces the embedding $\bdyinf(F;\A \restrict F) \approx \bdyinf S_F \hookrightarrow \bdyinf S \approx \bdyinf(\Gamma;\A)$. Because of the visibility assumption, as $F$ varies the subtrees $S_F$ are pairwise disjoint, and we have a $\Gamma$-invariant subforest with component decomposition $\bigsqcup_F S_F \subset S$. We break into two cases depending on the relation in $S$ between that subforest and the concrete realization $\ell_S \subset S$ of~$\ell$.

\smallskip
\textbf{Case 1 Hypothesis:} $\ell_S \not\subset \bigsqcup_F S_F$. 
We shall prove that 
$$\ell \in \wtB_\F(\Gamma;\A) = \wtBinf(\Gamma;\A) - \bigsqcup_F \bdy(F;\A \restrict F) \approx \bdyinf S - \bigsqcup_F \bdyinf S_F
$$
Given any $F$ and any ray in $S$, that ray represents a point of $\bdyinf S_F$ if and only if one of its subrays is contained in~$S_F$. It therefore suffices to choose any ray $\rho \subset \ell_S$ and prove that $\rho \not\subset \bigsqcup_F \bdyinf S_F$. Using the case hypothesis, there exists an edge $e \subset \ell_S$ such that $e \not\subset \bigsqcup_F S_F$. Since $\bigsqcup_F S_F$ is $\Gamma$ invariant, it follows for all $\gamma \in \Gamma$ that $\gamma \cdot e \not \subset \bigsqcup_F S_F$. Using birecurrence of $\ell$, there exists $\gamma \in \Gamma$ such that $\gamma \cdot e \subset \rho$, and therefore $\rho \not\subset \bigsqcup_F S_F$. 

\smallskip
\textbf{Case 2 Hypothesis:} $\ell_T \subset \bigsqcup_F S_F$. Using that $\ell_T$ is connected and that $\bigsqcup_F S_F$ is the decomposition into components, there exists $F$ such that $\ell_T \subset S_F$. It follows that both ideal endpoints of the concrete line $\ell_T$ are in $\bdyinf S_F$ and hence $\ell \in \wtBinf(S_F) \approx \wtBinf(F;\A \restrict F)$.

\medskip

We turn now to the proof of \emph{Bounded Cancellation for Lines}. Let $\ell = \{\xi,\eta\} \in \wtB_\F(\Gamma;\A)$, and so $\xi,\eta \in \bdy_\F(\Gamma;\A)$. Choose a map $f \from S \to T$ from a Grushko free splitting $S$ \relA\ to a Grushko free splitting $T$ rel~$\F$. Let $C$ be a cancellation constant for $f$ with respect to finite paths (for example, the one specified in \cite[Lemma~4.14]{\RelFSHypTwoTag}). 

\medskip

\textbf{Case 1:}  Let $\ell_T \subset T$ be the line as described under \emph{Case 1 realization}. We proceed as in the proof of Lemma~\ref{LemmaDTBdy}~\pref{ItemDTBdyHomeo}: by applying the proof of  \cite[Theorem 3.2]{DowdallTaylor:cosurface} to each of  two subrays of $\ell_T$, one with end $\xi$ and the other with end $\eta$, we obtain a bi-infinite sequence $(w_i)_{i \in \Z}$ of points in $f(\ell)$ such that in $T$ we have $w_i \to \xi$ as $i \to -\infty$ and $w_i \to \eta$ as $i \to +\infty$; using that sequence $(w_i)$ we may then apply the argument of \cite[Lemma~4.14~(3)]{\RelFSHypTwoTag}, leading to the conclusions of \emph{Bounded Cancellation for Lines}. The final sentence, for the case when $f$ is a collapse map, is a consequence of \hbox{Lemma~\ref{LemmaDTBdy}~\pref{ItemIfCollapse},} which tells us in this case that $[q,\xi(T))=f[p,\xi(S))$ and $[q,\eta(T))=f[p,\eta(S))$ hence $\ell_T=f(\ell_S)$. 

\smallskip
\textbf{Case 2:} Let $\ell_T \in T$ be as described under \emph{Case 2 realization}, namely, the point fixed by the action of $F_\ell$. Let $S_F \subset S$ be the minimal subtree for the action of $F$. We have $\ell_S \subset S_F$, hence $F \cdot \ell_S \subset S_F$, and so by minimality of $S_F$ this inclusion is an equation: $F \cdot \ell_S = S_F$. Since $f(S_F) \subset T$ is an $F$-invariant subtree of $T$ and $\ell_T$ is the minimal such subtree, we have $\ell_T \in f(S_F)=f(F \cdot \ell_S)$. Pick $p' \in S_F$ such that $f(p')=\ell_T$, and then pick $p \in \ell_S$ and $g \in F$ such that $p = g \cdot p'$, hence 
$$\ell_T = g \cdot \ell_T = g \cdot f(p') = f(g \cdot p')  = f(p) \in f(\ell_S)
$$
which is the first part of the \emph{Bounded Cancellation} conclusion. For the second part, since $f(\ell_S) \subset f(S_F)$, the inclusion $f(\ell_S) \subset N_C(\ell_T)$ will follow from the more general inclusion $f(S_F) \subset N_C(\ell_T)$ that we now verify. The union of the subsegments $[\gamma \cdot p, \delta \cdot p] \subset S_F$, taken over all $\gamma,\delta \in F$, is a $F$-invariant subtree of $S_F$ and hence is equal to~$S_F$. Applying \cite[Lemma~4.14~(1)]{\RelFSHypTwoTag}, namely \emph{Bounded Cancellation for Paths}, since $f(\gamma \cdot p) = f(\delta \cdot p)=\ell_T$ it follows that $f[\gamma \cdot p, \delta \cdot p] \subset N_C(\ell_T)$.  Since this holds for all~$\gamma,\delta \in F$, we obtain $f(S_F) \subset N_C(\ell_T)$. If in addition $f \from S \xrightarrow{\<\sigma\>} T$ is a collapse map then $S_F$ is contained in the component of $\sigma$ stabilized by~$F$, but that component collapses to the point~$\ell_T$, hence $f(S_F)=\ell_T$ which is the minimal subtree for $F$ acting on~$T$.
\end{proof}

\subparagraph{Remark.} There is a version of the Lemma~\ref{LemmaDefBirecurrentRealization} which drops the birecurrence hypothesis at the cost of more possibilities for the concrete realization; here is a brief description. As $F$ varies over all free factors such that $[F] \in \F$, the subsets $\bdyinf(F;\A \restrict F)$ and $\bdy_\F(\Gamma;\A)$ of $\bdyinf(\Gamma;\A)$ form a pairwise disjoint decomposition. And while the subsets $\wtBinf(F;\A \restrict F) = \wtBinf(F;\A \restrict F)^\SetTwo$ and $\bdy_\F(\Gamma;\A)^\SetTwo$ of $\bdyinf(\Gamma;\A)$ are pairwise disjoint, they do not form a decomposition, because the two ideal endpoints $\xi,\eta$ of an abstract line $\{\xi,\eta\} \in \wtBinf(\Gamma;\A)$ can lie in two different members of the decomposition of $\bdyinf(\Gamma;\A)$, leading to two other types of realizations: a ray, when one of $\xi,\eta$ lies in $\bdy_\F(\Gamma;\A)$ and the other lies in some set of type $\bdyinf(F;\A \restrict F)$; or a finite path, when $\xi,\eta$ lie im two different sets of type $\bdyinf(F;\A \restrict F)$. One could express this succinctly in the language of the Bowditch boundary $\bdy(\Gamma;\F)$ and its abstract line space $\bdy(\Gamma;\F)^\SetTwo$ by saying that there is a natural projection $\wtBinf(\Gamma;\A) \mapsto \bdy(\Gamma;\F)^\SetTwo \sqcup \bdy(\Gamma;\F)$. This is closely related to Lyman's treatment of attracting laminations \emph{including} their nongeneric leaves \cite[Section~4]{\LymanCTTag}.

\medskip
For stating the following corollary, consider a birecurrent abstract line $\ell \in \wtBinf(\Gamma;\A)$ and a free splitting $S \in \FS(\Gamma;\A)$, and let $\ell_S \subset S$ be the concrete realization of $\ell$ in $S$ as specified by Lemma/Definition~\ref{LemmaDefBirecurrentRealization}.  To say that \emph{$\ell$ fills~$S$} means that $\Gamma \cdot \ell_S = S$, where of course $\Gamma \cdot \ell_S = \bigcup_{\gamma \in \Gamma} \gamma \cdot \ell_S$. For example, if $\ell_S$ falls into Case~$1$ of Lemma/Definition~\ref{LemmaDefBirecurrentRealization} --- meaning that $\ell_S$ is supported by a nontrivial free factor that stabilizes a vertex of $S$ --- then $\ell_S$ is a point and so $\ell$ does not fill~$S$.

\begin{corollary}
\label{CorollaryFillingLine}
For each birecurrent abstract line $\ell \in \wtBinf(\Gamma;\A)$ the following are equivalent:
\begin{enumerate}
\item\label{ItemLineFillsRelA}
$\ell$ fills $\Gamma$ \relA.
\item\label{ItemLineFillsGrushko}
$\ell$ fills every Grushko free splitting $S \in \FS(\Gamma;\A)$.
\item\label{ItemLineFillsAllFS}
$\ell$ fills every free splitting $S \in \FS(\Gamma;\A)$.
\end{enumerate}
\end{corollary}

\begin{proof} To prove the implication \pref{ItemLineFillsRelA}$\implies$\pref{ItemLineFillsGrushko}, suppose that \pref{ItemLineFillsGrushko} fails as witnessed by~$S$. Since $S$ is Grushko, the realization $\ell_S$ is a concrete line in $S$, and so the invariant subforest $\sigma = \Gamma \cdot \ell_S \subset S$ is proper (by failure of~\pref{ItemLineFillsGrushko}). It follows that $\F\sigma$ is a nonfull free factor system \relA\ that supports $\ell$, hence \pref{ItemLineFillsRelA} fails. 

To prove the implication~\pref{ItemLineFillsGrushko}$\implies$\pref{ItemLineFillsAllFS}, consider any free splitting $S$ of $\Gamma$ \relA. We may choose a collapse map $f \from U \mapsto S$ defined on a Grushko free splitting $U$ of $\Gamma$ \relA\ \cite[Lemma 3.2~(3)]{\RelFSHypTag}. If $\ell$ is supported by $\FFS S$ then~\pref{ItemLineFillsRelA} fails. Otherwise the realizations $\ell_U \subset U$ and $\ell_S \subset S$ are both lines, and $f(\ell_U)=\ell_S$ (by Lemma/Definition~\ref{LemmaDefBirecurrentRealization}). Since~\pref{ItemLineFillsGrushko} holds we have $\Gamma \cdot \ell_U = U$, and hence $\Gamma \cdot \ell_S=S$.

To prove the implication \pref{ItemLineFillsAllFS}$\implies$\pref{ItemLineFillsRelA}, suppose that~\pref{ItemLineFillsRelA} fails, and hence $\ell$ is supported by some nonfull free factor system~$\F$ \relA. Letting $S$ be a free splitting \relA\ such that $\FFSS = \F$. It follows from Lemma/Definition~\ref{LemmaDefBirecurrentRealization} that $\ell_S$ is a point in $S$ hence hence~\pref{ItemLineFillsAllFS} fails.
\end{proof}

\subsection{EG-aperiodic train track maps and their attracting laminations}
\label{SectionTTsAndLams}

The theory of relative train track representatives and attracting laminations for elements of $\Out(\Gamma;\A)$ was developed by Lyman in \cite{\LymanRTTTag,\LymanCTTag}, adapting the methods of \cite{\BookOneTag} to make them work in the context of the Bowditch boundary that was introduced in \cite{GuirardelHorbez:Laminations}. For a full review see the following material in Part II:
\begin{itemize}
\item\label{ItemAttrLamReview} \cite[Section 4.1.2]{\RelFSHypTwoTag} Attracting laminations and their free factor supports; filling laminations.
\item\label{ItemRelTTReview} \cite[Section 4.3.1]{\RelFSHypTwoTag} Relative train track representatives;
\item \cite[Section 4.3.3]{\RelFSHypTwoTag} Train track representatives;
\item\label{ItemRelTTLamReview}
\cite[Theorem 4.17, in Section 4.3.2]{\RelFSHypTwoTag} A summary statement of the natural bijection between the finite set $\L(\phi)$  of attracting laminations of $\phi \in \Out(\Gamma;\A)$ and the set of EG-aperiodic substrata of a relative train track representative of~$\phi$.
\end{itemize} 
Everything we need about attracting laminations here in Part III is found in Theorem~\ref{TheoremLymanTTLams} below, which is the result of further whittling down the summary statement \cite[Theorem 4.17]{\RelFSHypTwoTag}, itself whittled down from \cite{\LymanCTTag}. Namely, Theorem~\ref{TheoremLymanTTLams} focusses on the special case of a unique attracting lamination and the expression of its generic leaves using tiles of an EG-aperiodic train track representative. 

\smallskip
\textbf{Train track representatives.} For the statement of Theorem~\ref{TheoremLymanTTLams} and its various supporting definitions, we switch to the notation $\F$ for our given free factor system of $\Gamma$, in order to prepare for later applications where $\F$ will be some $\phi$-invariant free factor system \relA.

Consider any free factor system~$\F$ of $\Gamma$ and any $\phi \in \Out(\Gamma;\F)$. As said in \RelFSHypTwo\ just preceding Section 4.1, we work entirely upstairs in trees and in $\wtB(\Gamma;\F)$ (rather than mostly downstairs in graphs of groups and in $\B(\Gamma;\F) = \wtB(\Gamma;\F) / \Gamma$ as is done in~\cite{\LymanCTTag}). 

Recall from \cite[Section 4.3.3]{\RelFSHypTag} that a \emph{train track representative} of $\phi$ rel~$\F$ is a map \hbox{$F \from T \to T$} defined on a Grushko free splitting $T$ rel~$\F$ such that the following hold: $F$~is $\Phi$-twisted equivariant with respect to some $\Phi \in \Aut(\Gamma;\F)$ representing $\phi$; the map $F$ takes vertices to vertices; and each point of~$T$ is contained in the interior of some edge path $\eta$ without backtracking such each restricted iterate $F^k \restrict \eta$ is injective ($k \ge 0$). It follows that the restrictions $f^k \mid e$ and $f^k \mid E$ are injective for each edge $e \subset T$ and each natural edge $E \subset T$; the image path $F^k(e)$ is referred to as a \emph{$k$-tile of $F$} or an \emph{iteration tile} when the value of $k$ is unimportant, and $F^k(E)$ is similarly referred to as a \emph{natural $k$-tile of $F$} or a \emph{natural iteration tile}. Fixing $e_1,\ldots,e_n \subset T_l$ to be an enumerated set of representatives of the $\Gamma$-orbits of edges of $T$, the \emph{transition matrix} of $F$ is the $n \times n$ matrix $\M$ whose entry $\M_{ij}$ counts how many times the $1$-tile $F(e_j) \subset T$ crosses translates of the edge~$e_i$. It follows that for any $k \ge 1$ the map $F^k$ is a train track representative of $\phi^k$ with transition matrix~$\M^k$, and so $\M^k_{ij}$ counts how many times the $k$-tile $F^k(e_j)$ crosses translates of~$e_i$. To say that $F$ is \emph{EG-aperiodic} means that there exists an exponent $k \ge 1$ such that each entry of $\M^k$ is positive (equivalently, each $k$-tile crosses at least one translate of each edge); and in this case $\M$ has Perron-Frobenius eigenvalue denoted $\lambda > 1$ and called the \emph{expansion factor} of $f$. 

To say that a concrete line $L \subset T$ is \emph{exhausted by iteration tiles (of $F$)} means that every finite subpath of $L$ is contained in some iteration tile.

\begin{theorem}[The special case of \protect{\cite[Theorem 4.17]{\RelFSHypTwoTag}}]
\label{TheoremLymanTTLams} For any free factor system~$\F$ of $\Gamma$ and any $\phi \in \Out(\Gamma;\F)$, if $\phi$ has an EG-aperiodic train track representative $F \from T \to T$ rel~$\F$ then $\phi$ has exactly one attracting lamination $\Lambda \subset \wtB(\Gamma;\F)$ relative to~$\F$. Furthermore, a concrete line $L \subset T$ is the realization of a generic leaf of $\Lambda$ if and only if $L$ is exhausted by iteration tiles~of~$F$. \qed
\end{theorem}

\subsection{The first constraint on $\Omega$:  The lamination and the fold axis}
\label{SectionFirstConstraint}
We are ready to start executing the outline laid out in Section~\ref{SectionOmegaAndProof}.

\subsubsection{Finding the attracting lamination.}
\label{SectionCountLams}
Consider any $\phi \in \Out(\Gamma;\A)$ satisfying the \emph{Large Orbit Hypothesis}, with a value of $\Omega$ yet to be specified: Every $\phi$-orbit in $\FS(\Gamma;\A)$ has diameter $\ge \Omega$. Among all $\phi$-invariant, non-filling free factor systems of $\Gamma$ \relA, choose $\F$ to be one which is maximal with respect to $\sqsubset$. Let $m \ge 0$ be the number of attracting laminations of $\phi$ \relA\ that are not supported by $\F$. 

We now apply \cite[Proposition 4.24]{\RelFSHypTwoTag}. That proposition has two possible conclusions depending on whether $m=1$. Conclusion~(2) uses ``fold axes'' which will be reviewed immediately in Section~\ref{SectionFoldAxisReview} immediately, where we will augment conclusion (2) with an additional subconclusion.

Here are the two conclusions of \cite[Proposition 4.24]{\RelFSHypTwoTag}:
\begin{description}
\item[Conclusion (1) (when $m \ne 1$):] There exists a Grushko free splitting $T$ of $\Gamma$ rel~$\F$ such that the diameter in $\FS(\Gamma;\A)$ of the orbit $\{T \cdot \phi^k\}_{k \in \Z}$ is $\le 4$. 
\end{description}
We immediately rule out Conclusion (1) by combining our \emph{Hypothesis} with the following:
\begin{description}
\item[Constraint \#1:] $\Omega \ge \Omega_1 = 5$. 
\end{description}
\emph{Conclusion (1)} and \emph{Constraint \#1} together imply that $m=1$. The other conclusion of \cite[Proposition 4.24]{\RelFSHypTwoTag} must therefore hold:
\begin{description}
\item[Conclusion (2) (when $m=1$):] Letting $\Lambda$ denote the unique attracting lamination of $\phi$ \relA\ that is not carried by~$\F$,  
there exists a fold axis of $\phi$ with respect to~$\F$, such that the following hold:
\begin{description}
\item[Conclusion (2a):] The first return map $F_0 \from T_0 \to T_0$ for that axis is an EG-aperiodic train track representative of~$\phi$~rel~$\F$. Let $\Lambda_\F$ be the unique attracting lamination of $\phi$ rel~$\F$ (see~Theorem~\ref{TheoremLymanTTLams}).
\item[Conclusion (2b):] There exists a collapse map $\mu \from U \to T_0$ defined on some Grushko free splitting $U$ of $\Gamma$ \relA, such that $\mu$ induces a bijection between concrete lines in $U$ realizing generic leaves of $\Lambda$ and concrete lines in $T_0$ realizing generic leaves of $\Lambda_\F$.
\end{description}
\end{description} 

%Conclusion (2c) of Section~\ref{SectionFoldAxisReview}

\subsubsection{Fold axes and tiles: a review.} 
\label{SectionFoldAxisReview}
We review here notation for fold axes taken from  \cite[Definition 4.8]{HandelMosher:RelComplexHypII}, for use in the setting of Conclusion~(2) above, and in settings to recur throughout this work.

An EG-aperiodic fold axis of $\phi$ with respect to $\F$ is a bi-infinite fold path of Grushko free splittings rel~$\F$, which fits into both the top and bottom rows of a bi-infinite commutative diagram shown just below, and which satisfies further properties to be described:
$$\xymatrix{
& \cdots \ar[r]  & T_{l-1} \ar[rr]^{f_{l}} \ar[dl]^<<<<{h^{l-1}_{l-p-1}} && T_{l} \ar[rr]^{f_{l+1}}  \ar[dl]^<<<<{h^l_{l-p}} && T_{l+1} \ar[r]  \ar[dl]^<<<<{h^{l+1}_{l-p+1}}& \cdots\\
\cdots \ar[r]  & T_{l-p-1} \ar[rr]_{f_{l-p}}  && T_{l-p} \ar[rr]_{f_{l-p+1}} && T_{l-p+1}  \ar[r] & \cdots
}$$
In this diagram, as said $T_i$ is a Grushko free splitting with respect to~$\F$, and is thus equipped with some simplicial structure. The~integer $p \ge 1$ is the \emph{period} of the axis. There exists a representative $\Phi \in \Aut(\Gamma;\F)$ of $\phi \in \Out(\Gamma;\F)$, such that each map \hbox{$h^l_{l-p} \from T_{l} \to T_{l-p}$} is a $\Phi$-twisted equivariant simplicial isomorphism ($l \in \mathbb Z$). Each of the foldable maps $f_l$ is, of course, $\Gamma$-equivariant. It follows that $T_l \cdot \phi = T_{l+p}$ for each $l \in \Z$, hence the $\phi$-orbit of $T_l$ can be expressed as $\{T_l \cdot \phi^k\}_{k \in \Z} = \{T_{l+kp}\}_{k \in \Z}$. For each $i < j$ we denote an equivariant map
$$f^i_j = f_j \circ \cdots \circ f_{i+1} \from T_i \to T_j
$$
For each $l \in \Z$ the \emph{first return map} $F_l \from T_l \to T_l$ is the unique $\Phi$-twisted equivariant topological representative of $\phi$ rel~$\F$ that makes the following diagram commute:
$$\xymatrix{
&& T_{l} \ar[rr]^{f^l_{l+p}}  \ar[dll]_{h^l_{l-p}} \ar[d]_{F_l}
&& T_{l+p} \ar[dll]^<<<<{h^{l+p}_{l}} \\
 T_{l-p} \ar[rr]_{f^{l-p}_l} && T_{l}
}
$$
Finally (as required in Conclusion~(2)), the topological representative $F_0 \from T_0 \to T_0$ is an EG-aperiodic train track representative of~$\phi$ rel~$\F$. 

For each $l \in \Z$, $k \ge 1$ and $i < j$ we denote a $\Phi^k$-twisted equivariant simplicial isomorphism
$$h^l_{l-kp} = h^{l-(k-1)p}_{l-kp} \circ \cdots \circ h^l_{l-p} \, \from \, T_l \to T_{l-kp}
$$
and its $\Phi^{-k}$-twisted equivariant inverse
$$h^{l-kp}_l = (h^l_{l-kp})^\inv \from T_{l-kp} \to T_l
$$
For each $l \in \Z$ and $k \ge 1$ we also have a commutative diagram for $(F_l)^k \from T_l \to T_l$, the \emph{$k^{\text{th}}$ return map}, which is a train track representative of $\phi^k$:
$$\xymatrix{
&& T_{l} \ar[rr]^{f^{l}_{l+kp}}  \ar[dll]_{h^l_{l-kp}} \ar[d]_{(F_l)^k}
&& T_{l+kp} \ar[dll]^<<<<{h^{l+kp}_{l}}\\
 T_{l-kp} \ar[rr]_{f^{l-kp}_l} && T_{l}
}
$$
From the previous diagram and the fact that $h^l_{l-kp}$ is a simplicial isomorphism one sees by inspection that in the free splitting $T_l$, an $(F_l)^k$ (natural) tile is the same thing as an $f^{l-kp}_l$ (natural) tile. These are referred to as \emph{(natural) $k$-tiles} (in $T_l$) or as \emph{(natural) iteration tiles} when the value of $k$ is not important.

Recall from \cite[Lemma~4.10]{\RelFSHypTag} that when the edge orbits of each $T_l$ are enumerated so that each simplicial isomorphism $h^l_{l-p}$ respects enumerations, the transition matrix $\M_l$ of the first return map $F_l \from T_l \to T_l$ is equal to the transition matrix $\M f^{l-p}_l$ of the foldable map $f^{l-p}_l$. More generally the matrix power $(\M_l)^k$ is equal to the transition matrix $\M f^{l-kp}_l$, the $i,j$ entry of~which counts how many times the path $f^{l-kp}_l(e)$ crosses edges in the orbit of $e'$, where $e \subset T_{l-kp}$ represents the $j^{th}$ edge orbit, and $e' \subset T_l$ represents the $i^{th}$ edge orbit.

This completes our review of fold axes. See also \cite[Lemma 4.22]{\RelFSHypTwoTag} which gives a general construction for a fold axis having a given EG-irreducible train track representative as its first return map.

\bigskip

Fixing notation for a fold axis of $\phi$ with respect to~$\F$ satisfying Conclusion~(2) above, we add another subconclusion regarding the behavior of generic leaves along the fold axis:

\begin{description}
\item[Conclusion (2c):] For every $i \in \Z$, for every collapse map $\mu \from V \to T_i$ defined on a Grushko free splitting $V$ \relA, and for~every birecurrent line $\ell \in \wtBinf(\Gamma;\A)$ with concrete realizations $\ell_V \subset V$ and $\ell_i \subset T_i$, we have $\mu(\ell_V)=\ell_i$. Furthermore, the following properties of $\ell$ are equivalent:
\begin{description}
\item[$(A):$] $\ell$ is a generic leaf of $\Lambda$, 
\item[$(B):$] $\ell_i$ is a line in $T_i$ that is exhausted by iteration tiles in $T_i$.
\item[$(C):$] $\ell_i$ is the concrete realization in $T_i$ of a generic leaf of $\Lambda_\F$.
\end{description}
\end{description}

\noindent
\textbf{Remark:} The main difference between Conclusions (2b) and (2c) is how $i \in \Z$ and \hbox{$\mu \from U \to T_i$} are quantified: (2b) is an existence statement; whereas (2c) is a universal statement.

\begin{proof} In the case that $\ell$ is supported by~$\F$, it follows that $\ell$ is not a generic leaf of~$\Lambda$, and that $\ell_l \subset T_l$ is a point instead of a line (by Lemma~\ref{LemmaDefBirecurrentRealization}~Case~1); each of (A), (B), (C) is therefore false in this case.

Consider now the case that $\ell$ is not supported by~$\F$. By Lemma~\ref{LemmaDefBirecurrentRealization}~Case~2 it follows that $\ell \in \wtB_\F(\Gamma;\A)$ and that $\ell_i$ is a line in~$T_i$. Note that the formulations of statements $(B)$ and $(C)$ depend on the choices of $i \in \Z$ and $\mu \from V \to T_i$, but on the other hand $(A)$ is independent of those choices. Applying Conclusion (2) it follows that statement~$(B)$, formulated with \emph{one particular} set of choices, is equivalent to statement $(A)$. But by applying \cite[\hbox{Proposition 4.25~(2b)}]{\RelFSHypTwoTag} it follows that statements $(B)$ and $(C)$ are all equivalent to each as one varies over all of the choices made in their formulations. Statement $(A)$ is therefore equivalent to all of those statements. 
\end{proof}

\subsubsection{PF-exponents and a uniform crossing property.}
\label{SectionPFExponentOnAxis}
The ``Two-over-all'' theorems --- weak and strong --- establish certain \emph{distances} along Stallings fold paths for which corresponding tiles satisfy strong properties. 

In this section (and in Section~\ref{SectionLowerBound}), one of our strategies  is to establish certain \emph{exponents} along fold axes for which corresponding tiles satisfy strong properties: see Definition~\ref{DefPFExponent} just below; and see also various properties to follow in Section~\ref{SectionNoFillingLam}, particularly Definition~\ref{DefFillingExponent} of Section~\ref{SectionUnifFillProps}.

%Definition~\ref{DefPFExponent} and Definition~\ref{DefFillingExponent}

\smallskip

We begin by observing that for every EG-aperiodic train track map $F \from T \to T$ with transition matrix $\M$, there exists an integer $\kappa \ge 1$ such that every entry of $\M^k$ is $\ge 4$: this follows because $\M$ has \emph{some} positive power, and it has Perron-Frobenius eigenvalue $\lambda>1$. In the language of our next definition one says that $\kappa$ is a \emph{PF-exponent} of $F$:

\begin{definition}[PF-exponents]
\label{DefPFExponent}
To say that an integer $\kappa \ge 1$ is a \emph{PF-exponent} of an EG-aperiodic train track map $F \from T \to T$ means that every $F^\kappa$ tile in $T$ crosses at least $4$ translates of every edge of $T$; equivalently, every entry of the transition matrix of $F$ is $\ge 4$. Also, given $\phi \in \Out(\Gamma;\A)$ and an EG-aperiodic fold axis of $\phi$ relative to a $\phi$-invariant free factor system~$\F$ \relA\ (denoted as in Section~\ref{SectionFoldAxisReview}), to say that $\kappa$ is a PF-exponent of the fold axis means that it is a PF-exponent of every train track map $F_l \from T_l \to T_l$ ($l \in \Z$); equivalently, every $\kappa p$ tile $f^{l-\kappa p}_l(e) \subset T_l$ crosses at least $4$ translates of every edge of $T_l$ (for edges $e \subset T_{l-\kappa p}$).
\end{definition}
\noindent
Note that every EG-aperiodic fold axis of $\phi$ does indeed have a PF exponent: letting $p$ be the period, simply take the maximum of the PF exponents of $F_i \from T_i \to T_i$ for $0 \le i < p$.

\medskip

\emph{Remark.} While the choice of ``$4$'' in Definition~\ref{DefPFExponent} may seem arbitrary, its utility should become apparent in the following lemma. 

\begin{lemma}[Uniform Crossing Property for Tiles] 
\label{LemmaUniformCrossing}
Given an EG-aperiodic train track map $F \from T \to T$ with PF-exponent $\kappa$, for every edge $e \subset T$ and for every \emph{natural} edge $E \subset T$, the tile $F^\kappa(e)$ crosses at least $2$ translates of~$E$. It follows that given any EG-aperiodic fold axis with period~$p$ and PF-exponent $\kappa$, and given any $T_i$ along that axis, every $\kappa p$ tile in $T_i$ crosses at least 2 translates of every natural edge of $T_i$.
\end{lemma}

\begin{proof} Since the tile \hbox{$F^\kappa(e) \subset T$} crosses $4$ translates of every edge of $T$, it is not contained in any natural edge of $T$. For each natural edge $\eta \subset T$ it follows that one of three alternatives holds: $F^\kappa(e)$ crosses $\eta$; or $F^\kappa(e)$ is disjoint from the interior of $\eta$; or the subpath $F^\kappa(e) \intersect \eta$ has one endpoint in common with an endpoint of $\eta$, and opposite endpoint in common with an endpoint of $F^\kappa(e)$. 

Since distinct natural edges of $T$ have disjoint interiors, and since $F^\kappa(e)$ has exactly two endpoints, it follows that the set of translates of the natural edge $E$ whose interior is not disjoint from $F^\kappa(e)$ may be listed in order (and with disjoint interiors) as $E_1,\ldots,E_m$, such that the subpath $F^\kappa(e) \intersect E_1$ contains one endpoint of $F^\kappa(e)$, and $E_i \subset F^\kappa(e)$ for $2 \le i \le m-1$, and the subpath $F^\kappa(e) \intersect E_m$ contains the opposite endpoint of $F^\kappa(e)$. The tile $F^\kappa(e)$ therefore crosses at least $m-2$ translates of~$E$. 

Choose an edge $e'$ of $T$ such $e' \subset E$. Using that $\kappa$ is a PF-exponent for $F$, the path $F^\kappa(e)$ crosses 4 or more distinct translates of~$e'$. Any two of these translates are contained in two distinct translates of $E$, because each natural edge has trivial stabilizer. It follows that $m \ge 4$, and so $m-2 \ge 2$.
\end{proof}

For the concept of ``pullback'' referred to in the following lemma, see Definition~\ref{DefNondegenSubgraph}. 

\begin{lemma}[Trivial Collapse Property] 
\label{LemmaTrivialCollapse}
Let $\kappa$ be a PF-exponent for the fold axis.  Consider any $i,\ell \in \Z$ such that $i \le \ell - \kappa p$. For any proper subforest $\mathscr T_\ell \subset T_\ell$, consider the pullback subforest $\mathscr T_i = (f^i_{\ell})^*(\tau_i)$ and its corresponding collapse map $q_i \from T_i  \xrightarrow{\<\mathscr T_i\>} S_i$. The following hold:
\begin{enumerate}
\item\label{ItemEdgeNotInTi}
For any edge $e \subset T_i$, we have $e \not\subset \mathscr T_i$. The collapse map $q_i$ is therefore injective on the set of vertices of $T$.
\item\label{ItemCollapseIsTrivial}
The collapse map $q_i \from T_i \to S_i$ is trivial. More precisely, the restrictions of $q_i$ to the edges of $T_i$ can be equivariantly tightened to produce an equivariant homeomorphism $\hat q_i \from T_i \to S_i$ (see \cite[Section 2.2.4]{\RelFSHypTwoTag}). As a consequence, the restrictions of $q_i$ and $\hat q_i$ to the vertex set of $T_i$ are identical, and for every path $\alpha \subset T_i$ we have $q_i(\alpha)=\hat q_i(\alpha)$.
\item\label{ItemCollapseSameFF}
$\FFS S_i = \FFS T_i = \F$.
\item\label{ItemFillsIFFFills}
For every path $\alpha \subset T_i$, $\alpha$ fills $T_i$ if and only if $q_i(\alpha)$ fills $S_i$.
\end{enumerate}
\end{lemma}

\textbf{Remark:} In this statement one should keep in mind the following somewhat tautological but nonetheless important statement: $\mathscr T_i$ is a subforest of $T_i$ with respect to the subdivision of $T_i$ for which the map $T_i \mapsto T_\ell$ is simplicial. An edge $e \subset T_i$ is typically a concatenation of many edgelets of this subdivision. The statement $e \not\subset \mathscr T_i$ is equivalent to saying that at least one edgelet of $e$ is not in the subforest $\mathscr T_i$, and to saying that $e$ does not collapse to a single point of~$S_i$.

\begin{proof} Conclusions~\pref{ItemCollapseIsTrivial} and~\pref{ItemCollapseSameFF} are both immediate consequences of Conclusion~\pref{ItemEdgeNotInTi}. Conclusion~\pref{ItemFillsIFFFills} then follows by combining the invariance principle stated after Definition~\ref{DefinitionFillingPaths} with the conclusions stated in~\pref{ItemCollapseIsTrivial} that $\hat q_i$ is an equivariant homeomrophism and that $q_i(\alpha)=\hat q_i(\alpha)$.

 To prove Conclusion~\pref{ItemEdgeNotInTi} we start with the special case $i = \ell - \kappa p$. Using that $\kappa$ is a PF-exponent, for any edge $e \subset T_{\ell - \kappa p}$ the image $\kappa$-tile $f^{\ell - \kappa p}_\ell(e) \subset T_{\ell}$ crosses a translate of every edge of $T_{\ell}$. Since the subforest $\mathscr T_\ell \subset T_\ell$ is proper, it follows that $f^{\ell - \kappa p}_\ell(e) \not\subset \mathscr T_{l}$, and therefore $e \not\subset \mathscr T_{\ell - \kappa p}$. The general case $i \le \ell-\kappa p$ reduces to the special case as follows. For each edge $e \subset T_i$, its image $f^i_{\ell-\kappa p}(e) \subset T_{\ell-\kappa p}$ crosses some edge $e' \subset T_{\ell-\kappa p}$. Knowing from the special case that $e' \not\subset \mathscr T_{\ell-\kappa p}$, it follows that $f^i_{\ell-\kappa p}(e) \not\subset \mathscr T_{\ell-\kappa p}$ and hence $e \not\subset \mathscr T_i = (f^i_{\ell-\kappa p})^*(T_{\ell-\kappa p})$. 
\end{proof}

\subsection{The second constraint on $\Omega$: The fold axis is quasigeodesic}
\label{SectionSecondConstraint}

%the Notation Overview at the beginning of Section~\ref{SectionSecondConstraint}.

\emph{Notation Overview:} In this and later subsections of Section~\ref{SectionNoFillingLam}, we continue with the objects and their notations and various properties that were established in Section~\ref{SectionFirstConstraint} as a consequence of Constraint~\#1 on the value of~$\Omega$, in particular: an element $\phi \in \Out(\Gamma;\A)$ satisfying the \emph{Large Orbit Hypothesis}; a choice of maximal, proper, $\phi$-invariant free factor system~$\F$ \relA; the unique attracting lamination $\Lambda$ for $\phi$ \relA\ that is not carried by~$\F$; and a choice of fold axis for $\phi$ with respect to~$\F$.

\smallskip

Our next constraint on $\Omega$ is expressed using the constant  $\Theta = \Theta(\Gamma;\A)$ from the \STOAT:
\begin{description}
\item[Constraint \#2:] $\Omega \ge \Omega_2 = \Theta$
\end{description}
Using this, the remainder of Section~\ref{SectionSecondConstraint} is dedicated to proving the following:
\begin{description}
\item[Quasigeodesic Axis Property (nonquantitative version):] The fold axis of $\phi$ is a bi-infinite quasigeodesic in $\FS(\Gamma;\A)$
\end{description}

\subparagraph{Remark} This property is equivalent to saying that $\phi$ acts loxodromically on $\FS(\Gamma;\A)$, which is item~\pref{ItemActLox} of Theorem~B. Perhaps one could think about breaking the proof of the implication~\pref{ItemActLargeOrbit}$\implies$\pref{ItemFillingLamExists} into two implications \pref{ItemActLargeOrbit}$\implies$\pref{ItemActLox}$\implies$\pref{ItemFillingLamExists}, but just setting up the proof of the latter implication would already require much of the rigamarole of axes and tiles that is already needed for the full implication.

\subsubsection{Filling exponents and uniform filling properties.} 
\label{SectionUnifFillProps}

The proof of the Quasigeodesic Axis Property will use two filling properties that are developed here in Section~\ref{SectionUnifFillProps} as consequences of Constraint \#2: the \emph{Uniform Filling Property} in Lemma~\ref{LemmaUniformFilling}; and the \emph{Uniform Crossing Filling Property} in Lemma~\ref{LemmaUniformCrossingFilling}.

%Definition~\ref{DefFillingExponent} in Section~\ref{SectionUnifFillProps}

\begin{definition}[Filling exponents]
\label{DefFillingExponent}
To say that an integer $\omega \ge 1$ is a \emph{filling exponent} of an EG-aperiodic train track map $F \from T \to T$ means that every $F^\omega$ tile in $T$ fills $T$. Also, given $\phi \in \Out(\Gamma;\A)$ and an EG-aperiodic fold axis of $\phi$ relative to a $\phi$-invariant free factor system~$\F$ \relA\ (denoted as in Section~\ref{SectionFoldAxisReview}), to say that $\omega$ is a filling exponent of the fold axis means that $\omega$ is a filling exponent of every train track map $F_l \from T_l \to T_l$ ($l \in \Z$); equivalently, every $\omega p$ tile $f^{l-\omega p}_l(e) \subset T_l$ fills $T_l$ (for edges $e \subset T_{l-\omega p}$).
\end{definition}

Along the $\phi$-orbit $\{T_0 \cdot \phi^k = T_{kp}\}_{k \in \Z}$, applying the \emph{Large Orbit Hypothesis} we obtain integers $I < J$ such that \hbox{$d_\FS(T_{Ip},T_{Jp}) \ge \Omega$}. Denoting $\delta=J-I$, and using that the~cyclic subgroup $\<\phi\> \subgroup \Out(\Gamma;\A)$ acts by isometries on the orbit of~$T_0$, we have:
$$d_\FS(T_{(m - \delta) p},T_{mp}) = d_\FS(T_{mp} \cdot \phi^{-\delta}, T_{mp}) \ge \Omega \quad\text{for each $m \in \Z$.}
$$
Any $\delta$ with the property is called a \emph{large orbit exponent} for the orbit $T_0 \cdot \phi^k$. 

Using Constraint \#2 we shall prove:

%the Uniform Fillling Property for Tiles (Lemma~\ref{LemmaUniformFilling})

\begin{lemma}[Uniform Filling Property for Tiles] 
\label{LemmaUniformFilling}
If $\kappa$ is a PF-exponent for the fold axis, and if $\delta$ is a large orbit exponent for the orbit $\{T_0 \cdot \phi^k\}$, then $\omega = \kappa + \delta + 1$ is a filling exponent for the fold axis.
\end{lemma}
 
\begin{proof} Let $m \in \Z$ be the least integer such that $\ell - \omega p \le mp$. It follows that $(m+\omega-1)p < \ell$, and so we have the following foldable decomposition of $f^{\ell - \omega p}_\ell \from T_{\ell-\omega p} \to T_\ell$:
$$T_{\ell-\omega p} \xrightarrow{f^{\ell-\omega p}_{mp}} T_{mp} \xrightarrow{f^{mp}_{(m+\kappa)p}} T_{(m+\kappa)p} \xrightarrow{f^{(m+\kappa)p}_{(m+\kappa+\delta)p}} T_{(m+\kappa+\delta)p} = T_{(m+\omega-1)p} \xrightarrow{f^{(m+\omega-1)p}_{\ell}} T_\ell
$$
Consider an arbitrary edge $e \subset T_{\ell-\omega p}$ with image tiles denoted $\tau_i = f^{\ell-\omega p}_i(e) \subset T_i$ ($i \ge \ell-\omega p$). We must prove that the tile $\tau_\ell$ fills $T_\ell$. The tile $\tau_{mp} \subset T_{mp}$ contains some edge of $T_{mp}$, so the tile $\tau_{(m+\kappa)p}$ contains some $\kappa$-tile of $T_{(m+\kappa)p}$. Applying the \emph{Uniform Crossing Property for Tiles} to that $\kappa$-tile, it follows that the tile $\tau_{(m+\kappa)p}$ crosses a translate of every natural edge $E \subset T_{(m+\kappa)p}$. Applying Constraint \#2 together with the \emph{Strong Two Over All Theorem}, we may choose the natural edge $E$ so that the natural tile $f^{(m+\kappa)p}_{(m+\kappa+\delta)p}(E)$ fills $T_{(m+\kappa+\delta)p} = T_{(m+\omega-1)p}$. Since that natural tile is contained in $\tau_{(m+\omega-1)p}$ it follows that $\tau_{(m+\omega-1)p}$ also fills $T_{(m+\omega-1)p}$. Applying Lemma~\ref{LemmaFoldFillingProtoforest} it follows that $\tau_l$ fills $T_l$. 
\end{proof}

By combining the concepts of PF-exponents and filling exponents we obtain:

\begin{lemma}[Uniform Crossing--Filling Property]
\label{LemmaUniformCrossingFilling}
Let $\kappa,\omega$ be a PF-exponent and a filling exponent for the fold axis, respectively. For all $\ell,m \in \Z$ such that $m \ge 0$, every $(m \kappa + \omega)p$ tile in $T_{\ell}$ crosses $4^m$ nonoverlapping natural $\omega p$-tiles each of which fills $T_{\ell}$. 
\end{lemma}

\begin{proof} We use the following foldable factorization of the map $f^{\ell-(m\kappa+\omega)p}_{\ell} \from T_{\ell - (m\kappa+\omega)p} \to T_\ell$, 
in which all arrows are maps of the form $f^i_j$:
$$T_{\ell - (m\kappa+\omega)p} \xrightarrow{\hphantom{f^{\ell-\omega p}_\ell}} T_{\ell - ((m-1)\kappa+\omega)p} \xrightarrow{\hphantom{f^{\ell-\omega p}_\ell}}\cdots\xrightarrow{\hphantom{f^{\ell-\omega p}_\ell}} T_{\ell - (\kappa+\omega)p} \xrightarrow{\hphantom{f^{\ell-\omega p}_\ell}} T_{\ell - \omega p} \xrightarrow{f^{\ell-\omega p}_\ell} T_\ell
$$ 
Consider any edge of $T_{\ell - (m\kappa+\omega)p}$ with image tiles denoted $\tau_i \subset T_i$ (for $i \ge \ell - (m\kappa+\omega)p$), and so in particular $\tau_\ell \subset T_\ell$ is an $(m\kappa + \omega)p$ tile. By $m$ successive applications of the definition of PF-exponent, the tile $\tau_{\ell - \omega p}$ crosses $4^m$ distinct edges of $T_{\ell - \omega p}$. The images of those edges in $\tau_\ell$ give $4^m$ nonoverlapping $\omega p$ tiles in $T_\ell$, and by definition of a filling exponent each of those tiles fills $T_\ell$.
\end{proof}

\paragraph{Remarks on uniformity of PF-exponents and of filling exponents.} In Section~\ref{SectionLowerBound} we shall describe a \emph{uniform} PF-exponent $\kappa=\kappa(\Gamma;\A)$ and a \emph{uniform} filling exponent $\omega=\omega(\Gamma;\A)$ for the fold axis currently under discussion, under the additional assumption that $T_0$ is a natural free splitting. While that kind of uniformity is not needed in our current setting here in Section~\ref{SectionNoFillingLam}, it could be applied to further uniformize other constructions found here in Section~\ref{SectionNoFillingLam} which rely on unspecified values of $\kappa$ and $\omega$. For example Lemma~\ref{LemmaQAPQuant}~\pref{ItemSTLLowerBound} will be applied in Section~\ref{SectionLowerBound}, in conjunction with uniform values of $\kappa$ and $\omega$ also described in Section~\ref{SectionLowerBound}, im order to obtain the lower bound on positive translation lengths needed for Theorem~A.

\subsubsection{Component free splitting units.}
\label{SectionCFSU}
The proof of the \emph{Quasigeodesic Axis Property} uses concepts of \emph{free splitting units} and associated distance estimates along fold paths in $\FS(\Gamma;\A)$, taken from \RelFSHyp, and applied here to our fold axis for~$\phi$. There are two different flavors of free splitting units: ``ordinary'' free splitting units; and component free splitting units \cite[Section 4.5]{\RelFSHypTag}. These two definitions are founded on two respective concepts of the complexity of a nondegenerate subgraph of a free splitting: ``ordinary'' complexity; and component complexity \cite[Section 4.4]{\RelFSHypTag}. Ordinary free splitting units have important internal applications in \cite{\RelFSHypTag}, but outside of that context component free splitting units seem generally easier to apply. Here we shall only review component complexity and component free splitting units (outside of brief remarks comparing ``component'' free splitting units to ``ordinary'' free splitting units).

Recall from Section~\ref{SectionSTOATTerms} the concept of a nondegenerate subgraph of a free splitting, namely a $\Gamma$-invariant subgraph no component of which is a point. 

For any free splitting $R$ and any nonempty, nondegenerate subgraph $\beta \subset R$, the \emph{component complexity} $C_1(\beta)$ is a positive integer equal to the number of $\Gamma$-orbits of components of $\beta$; equivalently, $C_1(\beta) = \abs{\pi_0(\beta/\Gamma)}$ where $\pi_0(\beta/\Gamma)$ denotes the set of components of the orbit subspace $\beta/\Gamma \subset R/\Gamma$.

\begin{lemma}[Monotonicity of component complexity]
\label{LemmaMonotonicityOfComplexity} \quad
For any foldable map \hbox{$f \from R' \to R$} of free splittings and any nonempty, nondegenerate subgraph $\beta_R \subset R$ with pullback $\beta_{R'} = f^*(\beta_R) \subset R'$, the map $f$ induces a surjection $\pi_0(\beta_{R'}/\Gamma) \mapsto \pi_0(\beta_R/\Gamma)$. It follows that $C_1(\beta_{R'}) \ge C_1(\beta_R)$, and equality holds if and only if that surjection is a bijection. 
\end{lemma}

\begin{proof} This follows immediately using that $f$ restricts to a continuous, equivariant surjection \hbox{$\beta_{R'} \mapsto \beta_R$}, and that the component sets $\pi_0(\beta_{R'}/\Gamma)$ and $\pi_0(\beta_R/\Gamma)$ are finite.
\end{proof}

Recall from Definition~\ref{DefNondegenSubgraph} the pullback operator on nondegenerate subgraphs with respect to foldable maps, and the concept of a pullback sequence along a foldable sequence.

\begin{definition}[Collapse expand diagrams]
\label{DefCollExp}
Consider any fold subpath $T_I \mapsto\cdots\mapsto T_J$ in our given axis. A \emph{collapse--expand diagram} over that subpath is a commutative diagram of two combing rectangles (see Definition~\ref{DefCombingRectangle}) having the following form, with collapse forests as indicated
$$\xymatrix{
R_I \ar[r] \ar[d]_{[\mathscr R_i]}  
                 & R_{I+1} \ar[r] \ar[d]_{[\mathscr R_{i+1}]}    
                 & \cdots \ar[r]
                 & R_{J-1} \ar[r] \ar[d]_{[\mathscr R_{j-1}]}   
                 & R_J  \ar[d]_{[\mathscr R_j]}
                                                                           \\
S_I \ar[r]                         
	& S_{I+1} \ar[r]                            
	& \cdots \ar[r]           & S_{J-1} \ar[r]   
                                                                            & S_J & \\
T_I \ar[u]^{[\mathscr T_i]} \ar[r]  
	& T_{I+1} \ar[u]^{[\mathscr T_{i+1}]} \ar[r] 
	& \cdots \ar[r]  
	& T_{J-1} \ar[u]^{[\mathscr T_{j-1}]} \ar[r]
	& T_J \ar[u]^{[\mathscr T_j]} \\
}$$
Combining Definition~\ref{DefCombingRectangle} with the discussion preceding Lemma~\ref{LemmaMonotonicityOfComplexity}, in this diagram we see two pullback sequences: the $\mathscr R_i$'s form a pullback sequence along the $R$-row; and the $\mathscr T_i$'s form a pullback sequence along the $T$-row. An \emph{edgelet subdivision} of this diagram is choice of subdivision of each free splitting in the diagram such that all arrows are simplicial maps, all collapse forests $\mathscr R_i$, $\mathscr T_i$ are subcomplexes, and $\Gamma$ acts by simplicial isomorphisms. For example, an edgelet subdivision of the diagram is determined by each subdivision of $S_J$ by pulling back vertices throughout the diagram. We use ``edgelet'' terminology to refer to $1$-simplices of an edgelet subdivision, and ``edge'' terminology to refer to $1$-simplices of the simplicial structures on each $T_i$ that were already given (in our review of fold axes in Section~\ref{SectionFoldAxisReview}).
\end{definition}

\begin{definition}[Component free splitting units]
\label{DefCompFSU}
For any fold subpath $T_I \mapsto\cdots\mapsto T_J$ in our given axis, to say that $T_I$ and $T_J$ \emph{differ by $<1$ component free splitting unit} means that there exists a collapse--expand diagram as denoted in Definition~\ref{DefCollExp}, such that in the $R$-row there there exists a pullback sequence $\beta_i \subset R_i$ ($i \in [I,\ldots,J]$) having constant component complexity, equivalently (by Lemma~\ref{LemmaMonotonicityOfComplexity}) each foldable map $R_i \mapsto R_j$ (for $i < j \in [I,\ldots,J]$) induces a bijection $\pi_0(\beta_i/\Gamma) \approx \pi_0(\beta_j/\Gamma)$. In this context, after further equivariant subdivisions we may assume that $\beta_i$ is an edgelet subcomplex of $R_i$ for each $i \in [I,\ldots,J]$; and since $\beta_i$ is nondegenerate, it is a union of~edgelets. 

More generally, the \emph{number of component free splitting units between $T_I$ and $T_J$} is an integer $\CU=\CU_{IJ} \ge 0$ equal to the maximum length of a sequence $i(0) < \cdots < i(\CU)$ in $[I,\ldots,J]$ such that if $1 \le u \le \CU$ then $T_{i(u-1)}$ and $T_{i(u)}$ do \emph{not} differ by $< 1$ component free splitting unit along the axis. 

Assuming that $T_I$ and $T_J$ differ by $<1$ free splitting unit, any collapse--expand diagram and pullback sequence which witnesses that fact can be restricted to the fold subsequence $T_i \mapsto\cdots\mapsto T_j$ for any $i \le j \in [I,\ldots,J]$, and those restrictions witness that $T_i,T_j$ also differ by $<1$ free splitting unit. It follows that the phrase ``$T_I$ and $T_J$ differ by $<1$ component free splitting unit'' is equivalent to the equation $\CU(T_I,T_J) = 0$. Also, the complementary phrase ``$T_I$ and $T_J$ do not differ by $<1$ component free splitting unit'' is equivalent to the inequality $\CU(T_I,T_J) \ge 1$. Combining this with Definition~\ref{DefCollExp}, for all $i \le j \in [I,\ldots,J]$ we have implications $\CU(T_I,T_J)=0 \implies \CU(T_i,T_j)=0$ and $\CU(T_i,T_j) \ge 1 \implies \CU(T_I,T_J) \ge 1$. 
\end{definition}

Here is the main feature of free splitting units that we shall be applying:

\begin{description}
\item[\protect{\cite[Corollary 5.5]{\RelFSHypTag}}:]
Along any fold subpath $T_I \mapsto \cdots \mapsto T_J$, component free splitting units are quasicomparable to distance: for any $i < j \in [I,\ldots,J]$ we have 
$$\frac{1}{\nu} \CU_{ij} - \xi \le d_\FS(T_i,T_j) \le \nu \CU_{ij} + \xi
$$ 
with quasicomparability constants $\nu=\nu(\Gamma;\A) \ge 1$ and $\xi=\xi(\Gamma;\A) \ge 0$. \qed
\end{description}

\emph{Remark.} The ``ordinary'' complexity $C(\beta)$ of \RelFSHyp\ is the sum of ``component'' complexity $C_1(\beta)$ plus three more uniformly bounded and non-negative terms $C_i(\beta)$ ($i=2,3,4$) measuring various other aspects of complexity of $\beta$. Corollary 5.5 of \RelFSHyp\ is derived from a corresponding result about ordinary complexity; uniform boundedness of the non-negative quantity $C(\beta)-C_1(\beta)$ is of key importance in that derivation.

\medskip

We next record a simple ``diagram chasing'' result in any collapse--expand diagram:

\begin{lemma}[Tiles throughout a collapse--expand diagram]
\label{LemmaTilesInMain}
Consider a collapse--expand diagram over fold subpath $T_I \mapsto\cdots\mapsto T_J$ as denoted in Diagram~\ref{DefCollExp}. For each edge $\theta \subset T_I$ such that $\theta \not\subset \mathscr T_I$ there is a unique function which associates to each free splitting $U$ in the diagram a nontrivial path $\theta(U) \subset U$ called the \emph{$\theta$-tile in $U$}, such that $\theta(T_I)=\theta$, and such that the following properties hold (for~$I \le i \le j \le J$):
\begin{enumerate}
\item\label{ItemTileInTRow} 
\emph{\textbf{In the $T$-row:}} 
The foldable map $T_i \mapsto T_j$ takes $\theta(T_i)$ homeomorphically to $\theta(T_j)$.
\item\label{ItemTileBetweenTandS} 
\emph{\textbf{Between the $T$ and $S$ rows:}}
The collapse map $T_i \mapsto S_i$ takes $\theta(T_i)$ onto $\theta(S_i)$.
\item\label{ItemTileInSRow} 
\emph{\textbf{In the $S$-row:}}
The foldable map $S_i \mapsto S_j$ takes $\theta(S_i)$ homemorphically to $\theta(S_j)$. Also, $\theta(S_{I})$ is contained in a natural edge of $S_{I}$. 
\item\label{ItemTileBetweenRAndS} 
\emph{\textbf{Between the $R$ and $S$ rows:}} The $\theta$-tile $\theta(R_i)$ is the lift of $\theta(S_i)$ with respect to the collapse map $R_i \mapsto S_i$ (thus $\theta(R_i)$ begins and ends with edgelets of $R_i \setminus \mathscr R_i$).
\item\label{ItemTileInRRow} 
\emph{\textbf{In the $R$ row:}}
The map $R_i \mapsto R_j$ takes $\theta(R_i)$ homeomorphically to $\theta(R_j)$. Also, $\theta(R_{I})$ is contained in a natural edge of $R_{I}$.
\end{enumerate}
\end{lemma}

\begin{proof} In the $T$-row we define $\theta(T_i)$ inductively: we already have $\theta(T_I)=\theta$; and $\theta(T_i)$ is the image of $\theta(T_{i-1})$ under the fold map $T_{i-1} \mapsto T_i$. Since all maps in the $T$-row are foldable, and since $\theta(T_I)$ is a subset of some natural edge $E \subset T_i$, the map $T_I \mapsto T_i$ takes $\theta(T_I)$ homeomorphically onto $\theta(T_i)$. Each path $\theta(T_i)$ is therefore nontrivial; item~\pref{ItemTileInTRow} also follows.

We turn to items~\pref{ItemTileBetweenTandS} and~\pref{ItemTileInSRow}. Since $\theta = \theta(T_I) \not \subset \mathscr T_I$, its image $\theta(S_I) \subset S_I$ is a nontrivial path contained in a natural edge of $S_I$, namely the image of $E$ under the collapse map $T_I \mapsto S_I$. Defining $\theta(S_i)$ inductively to be the image of $\theta(S_{i-1})$ under the foldable map $S_{i-1} \mapsto S_i$, the same argument as in the $T$ row proves item~\pref{ItemTileInSRow} and shows that each $\theta(S_i)$ is nontrivial. Using commutativity of the diagram, together with the fact that $\theta$-tiles within the $S$ and $T$ rows are mapped homeomorphically amongst themselves by the foldable maps in those rows, it follows that each collapse map $T_i \mapsto S_i$ takes $\theta(T_i)$ onto $\theta(S_i)$, proving item~\pref{ItemTileBetweenTandS}.

We may use the first sentence of item~\pref{ItemTileBetweenRAndS} (in conjunction with Definition~\ref{DefinitionLiftingPaths}) to define the path~$\theta(R_i)$. Nontriviality of $\theta(R_i)$ follows from nontriviality of $\theta(S_i)$, and the second sentence follows from Definition~\ref{DefinitionLiftingPaths}.

The second sentence of item~\pref{ItemTileInRRow} follows from the second sentence of item~\pref{ItemTileInSRow}, using the consequence of Definition~\ref{DefinitionLiftingPaths} that the collapse map $R_{i} \mapsto S_{i}$ takes the natural edges of $R_{i}$ not contained in $\mathscr R_{i}$ bijectively to the natural edges of $S_{i}$. Regarding the first sentence of~\pref{ItemTileInRRow}, for each $i \in [I,\ldots,J]$ we have the commutativity relation saying that the compositions $R_I \mapsto S_I \mapsto S_i$ and $R_I \mapsto R_i \mapsto S_i$ are the same, hence $\theta(R_i)$ maps onto $\theta(S_i)$. Furthermore, since $\mathscr R_I$ is the pullback of $\mathscr R_i$, and since $\theta(R_I)$ begins and ends with edgelets of $R_I \setminus \mathscr R_I$, it follows that $\theta(R_i)$ begins and ends with edgelets of $R_i \setminus \mathscr R_i$. But these properties characterize the lift of $\theta(S_i)$ through the collapse map $R_i \mapsto S_i$, hence $\theta(R_i)$ is the lift of $\theta(S_i)$.
\end{proof}

\subsubsection{A quantitative quasigeodesic axis property}
\label{SectionQuantBiinfinite}
In this section we complete the proof of the quasigeodesic axis property, deriving it quickly from Lemma~\ref{LemmaQAPQuant} just below, which expresses a quantitative version of the lower bound needed for the quasigeodesic axis property. The lemma also contains information about translation lengths that will be applied later in Section~\ref{SectionLowerBound}, to produce the uniform lower bound on positive translation lengths needed for Theorem~A. 

For purposes of application in Section~\ref{SectionLowerBound}, the statement of the lemma is formulated so that it does not require the Notation Overview established at the beginning of Section~\ref{SectionSecondConstraint}. 

%the \emph{Quasigeodesic Axis Property} from Section~\ref{SectionQuantBiinfinite}

%Lemma~\ref{LemmaQAPQuant}~\pref{ItemSTLLowerBound}

\begin{lemma}[Quasigeodesic Axis Property (quantitative version)] 
\label{LemmaQAPQuant} 
Consider any $\phi \in \Out(\F;\A)$, any $\phi$-invariant free factor system~$\F$ \relA, and any EG-aperiodic fold axis for $\phi$ with respect to~$\F$. If~$\kappa \ge 1$ is a PF-exponent and if $\omega\ge 1$ is a filling exponent for the fold axis then, denoting $\mu = 3\kappa+\omega$, and using the constants $\nu=\nu(\Gamma;\A)$ and $\xi = \xi(\Gamma;\A)$ from \cite[Corollary 5.5]{\RelFSHypTag}, the following hold:
\begin{enumerate}
\item
For any $I \in \Z$ and any integer $m \ge 1$, 
\begin{enumerate}
\item\label{ItemFSULowerBoundForPhi}
 Along the fold axis between $T_{I}$ and \hbox{$T_{I + m \mu p} = T_{I} \cdot \phi^{m \mu}$}, the number of component free splitting units is $\ge m$.
\item\label{ItemDistLowerBoundForPhi}
The distance in $\FS(\Gamma;\A)$ between $T_{I}$ and~$T_{I} \cdot \phi^{m\mu}$ is $\ge \frac{m}{\nu} \, - \, \xi$. 
\end{enumerate}
\item\label{ItemSTLLowerBound}
The stable translation length $\tau_\phi$ satisfies the lower bound $\ds\tau_\phi \ge  \frac{1}{\nu \mu} = \frac{1}{\nu(3\kappa + \omega)}$.
\end{enumerate}
\end{lemma}

Before embarking on the proof of this lemma, we first apply it:

\begin{proof}[Proof of the Quasigeodesic Axis Property (nonquantitative version):] For the fold axis established in the Notation Overview established at the beginning of Section~\ref{SectionSecondConstraint}, we obtain a PF exponent $\kappa$ from the paragraph of Section~\ref{SectionPFExponentOnAxis} surrounding Definition~\ref{DefPFExponent}. Also, we obtain a filling exponent $\omega$ from Lemma~\ref{LemmaUniformFilling} (which uses Constraint~\#2 on the value of $\Omega$). Using these exponents in combination with Lemma~\ref{LemmaQAPQuant}~\pref{ItemDistLowerBoundForPhi}, with the fact that $d(T_{i-1},T_i) \le 2$ for all $i$, and with the fact that $\phi$ acts isometrically on the $\phi$-orbit of~$T_0$, we obtain
$$\frac{1}{\nu} \abs{m-n} \, - \, \xi \, \le \, d(T_0 \phi^m, T_0 \phi^n) \, \le \, 2 \, p \abs{m-n} \quad\text{for all $m,n \in \Z$.}
$$
This proves that $\phi$ acts loxodromically on $\FS(\Gamma;\A)$, and so the fold axis is quasigeodesic. 
\end{proof}

\begin{proof}[Proof of the Quasigeodesic Axis Property (quantitative version)] \quad 
The implication \pref{ItemFSULowerBoundForPhi}$\implies$\pref{ItemDistLowerBoundForPhi} comes directly from \cite[Corollary 5.5]{\RelFSHypTag} cited earlier. Next, by applying~\pref{ItemDistLowerBoundForPhi} with $I=0$ it follows that
$$\frac{1}{m} d(T_0,T_0 \cdot (\phi^{\mu})^m) \ge \frac{1}{\nu} - \frac{\xi}{m}
$$
Letting $m \to \infty$ we have $\ds \tau_{\phi^\mu} \ge \frac{1}{\nu}$ and so $\ds\tau_{\phi} = \frac{1}{\mu} \, \tau_{\phi^\mu} \ge \frac{1}{\nu\mu}$, thus proving~\pref{ItemDistLowerBoundForPhi}$\implies$\pref{ItemSTLLowerBound}.

\smallskip

Next we reduce the general case of item~\pref{ItemFSULowerBoundForPhi} to its special case where $m=1$, by considering the following $m$-term foldable factorization of the map $f^I_{I+m\mu p} \from T_I \to T_{I + m \mu p}$:
$$T_I \xrightarrow{f^I_{I+\mu p}} T_{I + \mu p} \xrightarrow{}   \cdots \xrightarrow{} T_{I + (j-1) \mu p} \xrightarrow{f^{I + (j-1) \mu p}_{I + j \mu p}} T_{I + j \mu p}  \xrightarrow{} \cdots \xrightarrow{} T_{I + (m-1) \mu p} \xrightarrow{f^{I + (m-1) \mu p}_{I + m \mu p}}  T_{I + m \mu p}
$$
Applying the special case to each arrow in this factorization, there is $\ge 1$ component free splitting unit between $T_{I+(j-1)\mu p}$ and $T_{I+ j\mu p}$ for each $1 \le j \le m$. From Definition~\ref{DefCompFSU} it then follows that there are $\ge m$ component free splitting units between $T_I$ and $T_{I+m\mu p}$, which proves the general~case.

\smallskip

It remains to prove item~\pref{ItemFSULowerBoundForPhi} in the special case $m=1$. Arguing by contradiction, suppose that there exists $I \in \Z$ such that along the axis between $T_{I - \mu p}$ and $T_I$ there is $<1$ component free splitting unit (here we are shifting the value of $I$ by~one unit of $\mu p$). By applying Definition~\ref{DefCompFSU} we obtain a collapse expand diagram over the portion of the axis between $T_{I-\mu p}=T_{I - (3\kappa + \omega) p}$ and $T_I$, as depicted in Figure~\ref{FigureCollapseMuPee}, such that along the $R$-row of the diagram there is a pullback sequence $\beta_i \subset R_i$ of constant component complexity $\abs{\pi_0(\beta_i/\Gamma)}$. Each foldable map $R_i \mapsto R_j$ in the $R$-row therefore induces a bijection $\pi_0(\beta_i / \Gamma) \approx \pi_0(\beta_j / \Gamma)$ (for $i < j \in [I_2,\ldots,I_0]$). We may assume, after further subdivisions, that each $\beta_i$ is a union of edgelets of $R_i$.
\begin{figure}[h]
\label{FigureITwoOneZero}
$$\xymatrix{
R_{I_2} \ar[r] \ar[d]_{[\mathscr R_{I_2}]}   
        & \cdots \ar[r]  
        &R_{I_1} \ar[r] \ar[d]_{[\mathscr R_{I_1}]} 
        & \cdots \ar[r]     
        & R_{I_0}  \ar[d]_{[\mathscr R_{I_0}]}   \\
S_{I_2} \ar[r]                                     
	& \cdots \ar[r]              
	& S_{I_1} \ar[r]                           
	& \cdots \ar[r]     
        & S_{I_0} & \\
T_{I_2} \ar[u]^{[\mathscr T_{I_2}]}  \ar[r]
	& \cdots \ar[r] 
	& T_{I_1} \ar[u]^{[\mathscr T_{I_1}]}  \ar[r] 
	& \cdots \ar[r] 
	& T_{I_0} \ar[u]^{[\mathscr T_{I_0}]} 
}$$
\caption{Letting $\mu=3\kappa+\omega$, assuming $<1$ free splitting unit between $T_{I-\mu p}$ and $T_I$, and setting \, $I=I_0$, \, $I_1=I-\kappa p$, \, and $I_2 = I - \mu p = I_1 - (2\kappa+\omega)p$, \, there is a collapse expand diagram as depicted with a pullback sequence $\beta_i \subset R_i$ of constant component complexity.}
\label{FigureCollapseMuPee}
\end{figure}

Since $I_1 = I_0 - \kappa p$, we can apply the \emph{Trivial Collapse Property}, Lemma~\ref{LemmaTrivialCollapse}, to conclude that for each $i \in [I_2,I_1] = [I_1 - (2\kappa + \omega)p,I_1]$ the collapse map $q_i \from T_i \to S_i$ can be equivariantly tightened, relative to the vertex set of $T_i$, to obtain an equivariant homeomorphism $\hat q_i \from T_i \to S_i$, and the following hold:

\begin{description}
\item[Persistence of Filling and Crossing under Collapse:] For every $i \in [I_2,I_1]$ and every path $\alpha \subset T_i$ the following hold:
\begin{enumerate}
\item\label{ItemCollapsedPathNice} $\alpha$ fills $T_i$ if and only if $q_i(\alpha)=\hat q_i(\alpha)$ fills $S_i$.
\item\label{ItemCollapsedEdgeNice}
For every natural edge \hbox{$E \subset T_i$,} the path $\alpha$ has an interior crossing of $E$ if and only if $q_i(\alpha)$ has an interior crossing of $q_i(E)=\hat q_i(E)$.
\end{enumerate}
\end{description}

Consider now an arbitrary edge $\theta$ of $T_{I_2}$ and its associated $\theta$-tiles throughout the diagram as defined for all $i \in [I_2,\ldots,I_0]$ in Lemma~\ref{LemmaTilesInMain}, denoted $\theta(T_i) \subset T_i$, $\theta(S_i) \subset S_i$, $\theta(R_i) \subset R_i$, and so $\theta(R_{I_2})$ is contained in some natural edge $E \subset R_{I_2}$. Since $I_2 = I_1 - (2\kappa+\omega) p$, we can apply Lemma~\ref{LemmaUniformCrossingFilling}, the \emph{Uniform Crossing--Filling Property} of Section~\ref{SectionUnifFillProps}, using the values $\ell = I_1$ and $m=2$; from the conclusion of that lemma it follows that the tile $\theta(T_{I_1})$ has $4^2=16$ non-overlapping paths each of which fills~$T_{I_1}$. Applying \emph{Persistence of Filling and Crossing under Collapse}, the tile $\theta(S_{I_1})$ also has 16 non-overlapping paths each of which fills $S_{I_1}$. Lifting to $R_{I_1}$, we conclude that the tile $\theta(R_{I_1})$ has 16 non-overlapping subpaths each of which has an interior crossing of a natural edge in every orbit of natural edges of $R_{I_1}$. What we need from the latter conclusion is only 5 of those 16 subpaths, and we need only edgelet crossings: there is a subdivision into edgelet subpaths
$$\nu_0 \, \mu_1 \, \nu_1 \, \mu_2 \, \nu_2 \, \mu_3 \, \nu_3 \, \mu_4 \, \nu_4 \, \mu_5 \, \nu_5 = \theta(R_{I_1})
$$
such that for each $i=1,2,3,4,5$, the path $\mu_i$ crosses some edgelet in every orbit of edgelets of~$R_{I_1}$; the subpaths $\nu_i$ are allowed to be trivial. By pulling back to the tile $\theta(R_{I_2})$ we obtain a subdivision into edgelet paths
$$\nu'_0 \, \mu'_1 \, \nu'_1 \, \mu'_2 \, \nu'_2 \, \mu'_3 \, \nu'_3 \, \mu'_4 \, \nu'_4 \, \mu'_5 \, \nu'_5 = \theta(R_{I_2})  \subset E
$$
Choose any edgelet $\eta \subset \beta_{I_1}$. Choose edgelets $\eta_2 \subset \mu_2$ and $\eta_4 \subset \mu_4$ in the orbit of $\eta$. Pull back to get edgelets $\eta'_2 \subset \mu'_2 \subset \beta_{I_2}$ and $\eta'_4 \subset \mu'_4 \subset \beta_{I_2}$ that map to $\eta_2$ and $\eta_4$ respectively. Let $b_2,b_4$ denote the components of $\beta_{I_2}$ containing $\eta'_2,\eta'_4$ respectively. Since each of $\mu_1,\mu_3,\mu_5$ contains an edgelet of $R_{I_I} \setminus \beta_{I_I}$, it follows that that each of $\mu'_1,\mu'_3,\mu'_5$ contains an edgelet of $R_{I_2} \setminus \beta_{I_2}$. From this it follows that $b_2,b_4$ are both contained in the interior of the natural edge $E$ and that they are \emph{distinct} components of $\beta_{I_2}$. Since the stabilizer of $E$ with respect to the action of $\Gamma$ on $R_{I_2}$ is trivial, it follows that $b_2$ and $b_4$ are in distinct orbits of components of $R_{I_2}$. However, the images of $b_2,b_4$ under the foldable map $R_{I_2} \mapsto R_{I_1}$ both contain edgelets in the orbit of $\eta$, and so $b_2,b_4$ map to components of $\beta_{I_1}$ that are contained in the same orbit of components of $\beta_{I_1}$. This contradicts the fact that the induced map $\pi_0(\beta_{I_2}/\Gamma) \mapsto \pi_0(\beta_{I_1} / \Gamma)$ is a bijection, thus completing the proof of the quantitative version of~\pref{ItemFSULowerBoundForPhi}.
\end{proof}

We extract from the above proof the following statement that will be useful in Section~\ref{SectionEnablingProjection}:

\begin{lemma}[Uniform Lifting--Crossing Property] 
\label{LemmaUniformLiftingCrossing}
Fix a PF-exponent $\kappa$ and a filling exponent $\omega$ for the fold axis, and let $\mu = 3 \kappa + \omega$. Consider any $I \in \Z$ and any collapse expand diagram over the fold subpath going from $T_{I_2}$ through $T_{I_1}$ to $T_{I_0}$, where $I_2 = I - \mu p$ and $I_1 = I - \kappa p$ and $I_0=I$ (see Figure~\ref{FigureCollapseMuPee}). Letting $\theta \subset T_{I_2}$ be any edge, with tiles throughout the diagram denoted as in Lemma~\ref{LemmaTilesInMain}, the tile $\theta(R_{I_1})$ contains sixteen non-overlapping subpaths each of which has an interior crossing of a natural edge in every natural edge orbit of $R_{I_1}$.
\qed\end{lemma}

\subsection{Finishing the proof: Projection diagrams and bounded cancellation}
\label{SectionEnablingProjection}

In this section we complete the proof of the implication \pref{ItemActLargeOrbit}$\implies$\pref{ItemFillingLamExists}  of Theorem~B. Besides applying many tools from previous sections, additional tools that we bring to bear in this section include: projection diagrams from \cite{\RelFSHypTag}; and bounded cancellation.

We continue with the Notation Overview established at the beginning of Section~\ref{SectionSecondConstraint}. 

Let $\Omega = \max\{\Omega_1,\Omega_2\}$ using the constants $\Omega_1 = 5$ as chosen in Section~\ref{SectionFirstConstraint}, and $\Omega_2 = \Theta$ from the \STOAT\ as chosen in Section~\ref{SectionSecondConstraint}. Given $\phi \in \Out(\Gamma;\A)$, we assume that all orbits of the action of $\phi$ on $\FS(\Gamma;\A)$ have diameter $\ge \Omega$. To review the proof so far, in Section~\ref{SectionFirstConstraint} we chose a maximal proper $\phi$-invariant free factor system~$\F$ \relA, and we used the constraint $\Omega \ge \Omega_1$ to produce a unique attracting lamination $\Lambda$ for $\phi$ \relA\ that is not supported by~$\F$, together with a fold axis for $\phi$ relative to~$\F$ with an EG-aperiodic train track map for its first return. In each free splitting $T_i$ along that axis, the lines of $T_i$ that realize generic leaves of $\Lambda$ are precisely those lines that are exhausted by tiles. In Section~\ref{SectionSecondConstraint} we used the constraint $\Omega \ge \Omega_2$ to prove that the fold axis is quasigeodesic.

Let $\kappa$ be any PF-exponent and $\omega$ any filling exponent of the fold axis; for example $\kappa$ could be chosen as in the paragraph surrounding Definition~\ref{DefPFExponent}; and $\omega$ as in Lemma~\ref{LemmaUniformFilling}. We also let $\mu = 3\kappa + \omega$, for purposes of applying Lemma~\ref{LemmaUniformLiftingCrossing}, the \emph{Uniform Lifting--Crossing Property}.

Our goal is to prove that $\Lambda$ fills $\Gamma$ \relA, by showing that a generic leaf $\ell \in \Lambda$ fills $\Gamma$ \relA\ (\cite[Lemma 4.11]{\LymanCTTag}, and see \cite[Lemma 4.6]{HandelMosher:RelComplexHypII}). We do this by applying Corollary~\ref{CorollaryFillingLine}: for an arbitrary Grushko free splitting $R$ of $\Gamma$ \relA, letting $\ell_R \subset R$ be the concrete line realizing~$\ell$, we must prove that the orbit of $\ell$ covers $R$, that is, $\Gamma \cdot \ell_R = R$. To put this another way, we must prove that $\ell_R$ crosses some translate of every edge of~$R$. 

In order to relate a general Grushko free splitting $R$ of $\Gamma$ \relA\ to the free splittings $T_i$ in the given axis for $\phi$, we shall apply the methods of projection diagrams that were introduced in \cite{\RelFSHypTag} and were used there to define projection maps to fold paths for purposes of verifying the Masur--Minsky axioms.

Let $\FS^{(0)}(\Gamma;\A)$ denote the $0$-skeleton of the simplicial complex $\FS(\Gamma;\A)$. 

\begin{definition}[Projection Diagrams and Projection Functions] 
\label{DefProjections}
(\cite[Sections 4.3, 5.2]{\RelFSHypTag}) \hfill\break Consider a (finite) fold path $T_I \mapsto\cdots\mapsto T_J$ in $\FS(\Gamma;\A)$. For any free splitting $R$ of $\Gamma$ \relA\ and any integer~$\Delta \in [I,\ldots,J]$, a \emph{projection diagram from $R$ to $T_I \mapsto\cdots\mapsto T_J$ of depth $\Delta$} is a diagram of the form
$$\xymatrix{
R_{I} \ar[r] \ar[d] & \cdots \ar[r] & 
	R_i \ar[r] \ar[d] & \cdots \ar[r] &
	R_{\Delta} \ar[rrr] \ar[d] 
	&&&  R \\
S_{I} \ar[r]          & \cdots \ar[r] & 
	S_i \ar[r] & \cdots \ar[r] &
	S_{\Delta}  \\
T_{I} \ar[r] \ar[u] & \cdots \ar[r] &
	T_i \ar[r] \ar[u] & \cdots \ar[r] 
	& T_{\Delta} \ar[r] \ar[u]  
	& \cdots \ar[r] & T_J \\
}$$
such that the bottom row is the given fold path, the $R$ and $S$ rows are both foldable sequences (see Section~\ref{SectionSTOATTerms}), and the two rectangles depicted are combing rectangles (see Definition~\ref{DefCombingRectangle}) and hence those two rectangles taken together form a collapse expand rectangle (see Section~\ref{SectionSecondConstraint}).

Associated to the fold path $T_I \mapsto\cdots\mapsto T_J$ is its \emph{projection function} 
$$\Delta \from \FS^{(0)}(\Gamma;\A) \to [I,\ldots,J]
$$ 
which assigns to each free splitting $R \in \FS^{(0)}(\Gamma;\A)$ the maximum value $\Delta(R)$ of the depths $\Delta$ of all projection diagrams from $R$ to $T_I \mapsto\cdots\mapsto T_J$. A \emph{maximal depth} projection diagram from $R$ to $T_I \mapsto\cdots\mapsto T_J$ is one of depth $\Delta = \Delta(R)$.
\end{definition}

In the context of Definition~\ref{DefProjections}, we shall silently use the fact that for $\Delta \le \Delta' \in [I,\ldots,J]$, if there exists a projection diagram of depth $\Delta'$ then there exists one of depth $\Delta$: remove $S_{\Delta+1},\ldots,S_{\Delta'}$ and all incident arrows from the diagram of depth $\Delta$; and compose the sequence of maps 
$$R_{\Delta} \mapsto \ldots \mapsto R_{\Delta'} \mapsto R
$$ 
from the depth $\Delta'$ diagram to form the single map $R_{\Delta} \mapsto R$ needed for the depth $\Delta$ diagram.

\smallskip

In \cite{\RelFSHypTag}, the way that hyperbolicity of $\FS(\Gamma;\A)$ was proved was to show that the collection of fold paths and their projection functions satisfy the Masur--Minsky axioms with respect to constants $A=A(\Gamma;\A)$, $a=a(\Gamma;\A)$, $b=b(\Gamma;\A)$, and $c=c(\Gamma;\A)$. We shall need a further property that follows from the Masur--Minsky axioms, as shown in \cite{\RelFSHypTag}. Here is the statement of that property in the context of Definition~\ref{DefProjections}:

\begin{description}
\item[Quasi-Closest Point Property \protect{\cite[Proposition 5.6]{\RelFSHypTag}}] For any finite fold path $T_I \mapsto\cdots\mapsto T_J$ and any $R \in \FS(\Gamma;\A)$ with projection $\Delta(R) \in [I,\ldots,J]$, choosing $M \in [I,\ldots,J]$ so that \hbox{$D = d(R,T_{M})$} minimizes the distances $d(R,T_i)$ over all $i \in [I,\ldots,J]$, it follows that 
$$\diam\bigl\{ T_i \suchthat i \in [\Delta(R),\ldots,M] \bigr\} \le K \log D + C
$$ 
with constants $K=K(\Gamma;\A)$, $C=C(\Gamma;\A) \ge 0$ depending only $A,a,b,c$.
\end{description}

For the rest of this section we fix $R \in \FS(\Gamma;\A)$. Let $M_R \in \Z$ be chosen so that $D_R=d(R,T_{M_R})$ minimizes the distances $d(R,T_i)$ over all $i \in \Z$, hence also $D_R$ minimizes distances over all $i$ in any integer interval containing $M_R$. 

\smallskip

Projection diagrams are only defined from $R$ to \emph{finite} fold paths; in particular they are not defined to an entire fold axis. In order to effectively apply projection diagrams along our fold axis, depending on $R$ we need to carefully choose indices $I \le J$ for the endpoints of a fold subpath $T_I \mapsto \cdots \mapsto T_J$ to which we shall project $R$. Our first strategy is to choose $I,J$ so that we obtain some kind of ``global control'' over the projection of $R$ to the subpath. First we constrain $I$ and $J$ so that $I \le M_R \le J$; it follows $D_R$ minimizes $d(R,T_i)$ over all $i \in [I,\ldots,J]$. Our next constraint is a certain upper bound $I_0$ for $I$ and a lower bound $J_0$ for~$J$ which we may choose as follows, using that the fold axis is a bi-infinite quasigeodesic: 
\begin{enumerate}
\item\label{ItemConstrainedMin}
$I_0 \le M_R \le J_0$
\item\label{ItemWideSubpath}
If $i \in \Z - [I_0,\ldots,J_0]$ then $d(T_i,T_{M_R}) > K \log D_R + C$. 
\end{enumerate}
By combining \pref{ItemConstrainedMin} and \pref{ItemWideSubpath} with the \emph{Quasi-Closest Point Property} we obtain the desired global control on the projection of $R$:
\begin{description}
\item[Restricted Projection Property:] If $I \le I_0 \le J_0 \le J$ then, letting $\Delta_{IJ}(R) \in [I,\ldots,J]$ denote the projection of $R$ to $T_I \mapsto\cdots\mapsto T_J$, we have $I_0 \le \Delta_{IJ}(R) \le J_0$.
\end{description}

\paragraph{Further remarks: On ``Quasi-Closest'' versus ``Almost Closest'':} For most of the long lifetime of the composition of this work, our proof of the implication \pref{ItemActLargeOrbit}$\implies$\pref{ItemFillingLamExists} of Theorem~B had applied the ``Almost Closest Point Property'', a stronger version of the \emph{Quasi-Closest Point Property} in which the upper bound $K \log d + C$ is replaced by a constant~$C$. The way that stronger property had been applied was to choose $I_0 \le M_R \le J_0$ so as to have a sufficiently large but fixed radius $M_R - I_0 = J_0 - M_R$, so large that we could deduce the \emph{Restricted Projection Property} from the ``Almost Closest Point Property'' and from a certain constraint on $\Omega$. Only in a very late draft did we realize that this choice was impossible, because the ``Almost Closest Point Property'' is generally false (see remarks found at the end of \cite[Section 5.1]{\RelFSHypTag}). But two further realizations saved the proof: that the ``Almost Closest Point'' property could be replaced by the \emph{Quasi-Closest Point Property}; and the value $\Omega_2$ of the second constraint on $\Omega$ could be chosen so as to force the \emph{Quasigeodesic Axis Property} to hold. These, taken together, allowed us to recover the \emph{Restricted Projection Property} by being \emph{very much more careful} about the choice of the axis subpath $T_{I_2} \mapsto \cdots\mapsto T_\Delta$ to which $R$ is projected, in particular by making a more realistic allowance for the positions of $I_0$ and $J_0$ that allows the distance $d(T_{I_0},T_{J_0})$ to enlarge as the free splitting $R$ retreats further from the fold axis.

\bigskip

The next step of our strategy is to choose an integer interval that contains $[I_0,\ldots,J_0]$ and that extends far enough to allow application of Lemma~\ref{LemmaUniformLiftingCrossing}, the \emph{Uniform Lifting--Crossing Property} of Section~\ref{SectionQuantBiinfinite}, (there is no need to extend any further to the right). We use again the notation from Section~\ref{SectionQuantBiinfinite}, namely $\mu = 3\kappa+\omega$, and $I_1 = I_0 - \kappa p$, and $I_2 = I_0 - \mu p = I_1 - (2 \kappa + \omega) p$. Working with the fold subpath \hbox{$T_{I_2} \mapsto\cdots\mapsto T_{J_0}$}, let $\Delta = \Delta_{I_2 J_0}(R)$ denote the projection of $R$ to that subpath. By applying the \emph{Restricted Projection Property} we have $I_0 \le \Delta \le J_0$ and so we can depict a maximal depth projection diagram as shown in Figure~\ref{FigureMainProjDia}, which we refer to as the \emph{Main Projection Diagram}.

\begin{figure}
$$\xymatrix{
R_{I_2} \ar[r] \ar[d]_{[\mathscr R_{I_2}]} & \cdots \ar[r] &
	R_{I_1} \ar[r] \ar[d]_{[\mathscr R_{I_1}]} & \cdots \ar[r] & 
	R_{I_0} \ar[r] \ar[d]_{[\mathscr R_{I_0}]} \ar@/^2pc/[rrrrr]^{g_{I_0}}  & \cdots \ar[r] &
	R_{\Delta} \ar[rrr] \ar[d]_{[\mathscr R_\Delta]} 
	& & & R \\
S_{I_2} \ar[r] & \cdots \ar[r] &
	S_{I_1} \ar[r] & \cdots \ar[r] & 
	S_{I_0} \ar[r] & \cdots \ar[r] &
	S_{\Delta}  \\
T_{I_2} \ar[r] \ar[u]^{[\mathscr T_{I_2}]} & \cdots \ar[r] &
	T_{I_1} \ar[r] \ar[u]^{[\mathscr T_{I_1}]} & \cdots \ar[r] &
	T_{I_0} \ar[r] \ar[u]^{[\mathscr T_{I_0}]} & \cdots \ar[r] 
	& T_{\Delta} \ar[r] \ar[u]^{[\mathscr T_\Delta]}  
	& \cdots \ar[r] & T_{J_0} \\
}$$
\caption{The Main Projection Diagram. This is a projection diagram from $R$ to $T_{I_2} \mapsto\cdots\mapsto T_{J_0}$ of maximal depth $\Delta$ satisfying $I_0 \le \Delta \le J_0$ (by the \emph{Restricted Projection Property}). The portion of this diagram between column $I_2$ and column $I_0$ is a collapse expand diagram to which Lemma~\ref{LemmaUniformLiftingCrossing} may be applied. The diagram also depicts the map $g_{I_0}$, which is one of the foldable maps $g_i \from R_i \to R$ obtained by composing arrows in the $R$-row.}
\label{FigureMainProjDia}
\end{figure}

Next we state facts about realizations of generic leaves of $\Lambda$ in the \emph{Main Projection Diagram} that parallel facts about $\theta$-tiles described in Lemma~\ref{LemmaTilesInMain}.

%Lemma~\ref{LemmaLeavesInMain}~\pref{ItemWithinT}
%Lemma~\ref{LemmaLeavesInMain}~\pref{ItemLineTtoS}
%Lemma~\ref{LemmaLeavesInMain}~\pref{ItemLineRFromS}

%\begin{realizingleaf}
\begin{lemma}[Generic leaves throughout the \emph{Main Projection Diagram}]
\label{LemmaLeavesInMain}
Any generic leaf $\ell \in \Lambda$ is realized as a concrete line in every free splitting of the \emph{Main Projection Diagram}. Furthermore, the following hold for all $I_2 \le i \le j \le I_0$: 
\begin{enumerate}
\item\label{ItemWithinT}
\emph{\textbf{In the $T$ row:}} 
The foldable map $T_i \mapsto T_j$ takes $\ell(T_i)$ homeomorphically to $\ell(T_j)$. 
\item\label{ItemLineTtoS}
\emph{\textbf{Between the $T$ and $S$ rows:}}
The collapse map $T_i \to S_i$ takes $\ell(T_i)$ onto $\ell(S_i)$. 
\item\label{ItemWithinS} 
\emph{\textbf{In the $S$-row:}} The foldable map $S_i \mapsto S_j$ takes $\ell(S_i)$ homeomorphically to $\ell(S_j)$.
\item\label{ItemLineRFromS}
\emph{\textbf{Between the $S$ and $R$ rows:}}
The collapse map $R_i \to S_i$ takes $\ell(R_i)$ onto $\ell(S_i)$.
\item\label{ItemToEndOfR}
\emph{\textbf{In the $R$-row:}} For each $I_2 \le i \le I_0$, letting $C$ be a bounded cancellation constant for $g_i \from R_i \to R$, we have $\ell(R) \subset g_i(\ell(R_i)) \subset N_C(\ell(R))$.
\end{enumerate}
\end{lemma}
%\end{realizingleaf}

\begin{proof} We start by proving the opening sentence of the lemma, and item~\pref{ItemLineTtoS}, and item~\pref{ItemLineRFromS}. In the $T$~row which is a fold subpath of the axis, we may apply Conclusion~(2c) of Section~\ref{SectionFoldAxisReview}: for any choice of collapse map $V \mapsto T_i$ defined on any Grushko free splitting $V$ of $\Gamma$ \relA, the line $\ell(V)$ is mapped onto the concrete realization $\ell(T_i)$, and the latter is a bi-infinite line that is exhausted by iteration tiles. Since $\ell(T_i)$ contains a $\kappa$-tile of $T_i$, and since each $\kappa$-tile crosses a translate of every edge (by Definition~\ref{DefPFExponent}), it follows that $\ell(T_i) \not\subset \mathscr T_i$. In the $S$~row, the image of $\ell(T_i)$ under the collapse map $T_i \mapsto S_i$ is therefore \emph{not} a point. That image is equal to the image of $\ell(V)$ under the composed collapse map $V \mapsto T_i \mapsto S_i$. Since $\ell(V)$ is birecurrent, by application of Lemma~\ref{LemmaDefBirecurrentRealization} that image is a concrete line in $S_i$ equal to the realization $\ell(S_i)$; this proves~\pref{ItemLineTtoS}. 

In the $R$ row, since the free splitting $R$ itself is a Grushko free splitting \relA, the realization $\ell(R)$ is a concrete line. Also, by equivariance of each map $g_i \from R_i \to R$ it follows that each $R_i$ is also a Grushko free splitting \relA. The line $\ell$ is therefore realized as a concrete line $\ell(R_i) \subset R_i$, and~\pref{ItemLineRFromS} holds again~by Lemma~\ref{LemmaDefBirecurrentRealization}. 

We next prove~\pref{ItemWithinT}. By foldability along the axis, it follows that the map $T_i \mapsto T_j$ restricts to an injection on each edgelet tile in $\ell(T_i)$.  Using that $\ell(T_i)$ is exhausted by edgelet iteration tiles of~$T_i$, each the image of an edgelet somewhere earlier in the fold axis, it follows that the map $T_i \mapsto T_j$ restricts to an injection on all of $\ell(T_i)$, and the image line in $T_j$ is clearly $\ell(T_j)$. 

To prove item~\pref{ItemWithinS}, in the \emph{Main Projection Diagram} consider the commuting square with $S_i,S_j,T_i,T_j$ at the corners. We use the edgelet subdivisions under which the subforests $\mathscr T_i \subset T_i$ and $\mathscr T_j \subset T_j$ are subcomplexes and all four maps in that square are simplicial. Let $h^i_j \from S_i \to S_j$ and $f^i_j \from T_i \to T_j$ denote the foldable maps in the $S$-row and $T$ rows. It suffices to show for that any two distinct $1$-simplices $e \ne e' \subset \ell(S_i)$ we have $h^i_j(e) \ne h^i_j(e')$. Lifting $e,e'$ we get distinct $1$-simplices $\tilde e \ne \tilde e' \subset \ell(T_i) \setminus \mathscr T_i$, and by~\pref{ItemWithinT} we have distinct $1$-simplices $f^i_j(\ti e) \ne f^i_j(\ti e') \subset \ell(T_j) \setminus \mathscr T_j$, which therefore project to distinct $1$-simplices in $S_j$; the latter, by commutativity, are identified with $h^i_j(e)$, $h^i_j(e')$ respectively.

Item~\pref{ItemToEndOfR} follows from Lemma~\ref{LemmaDefBirecurrentRealization}.
\end{proof}

We turn now to the proof, using bounded cancellation, that $\ell(R)$ crosses a translate of every edgelet of $R$. Fix any edge $\theta \in T_{I_2}$, and consider the $\theta$-tiles throughout the portion of the \emph{Main Projection Diagram} between columns $I_2$ and $I_0$. By applying Lemma~\ref{LemmaUniformLiftingCrossing}, the \emph{Uniform Lifting--Crossing Property}, we conclude that $\theta(R_{I_1})$ contains sixteen nonoverlapping subpaths each of which has an interior crossing of a representative of every natural edge orbit of $R_{I_1}$. We need only three of those sixteen subpaths, using which we obtain a decomposition of the form
$$\theta(R_{I_1}) = \alpha_0 \, \beta_1 \, \alpha_1 \, \beta_2 \, \alpha_2 \, \beta_3 \, \alpha_3
$$
in which the $\alpha$'s are possibly trivial subpaths, and each subpath $\beta_j$ ($1 \le j \le 3$) has an interior crossing of a representative of every natural edge orbit of $R_{I_1}$. We will need only the weaker property that each $\beta_j$ crosses a representative of every edgelet orbit of $R_j$.
 
By Conclusion~\pref{ItemTileInRRow} of Lemma~\ref{LemmaTilesInMain}, $\theta(R_{I_2})$ is contained in a natural edge of $R_{I_2}$, and so the foldable map $g_{I_2} \from R_{I_2} \to R$ restricts to an injection on $\theta(R_{I_2})$. Also by that same conclusion, the map $R_{I_2} \mapsto R_{I_1}$ takes $\theta(R_{I_2})$ homeomorphically to $\theta(R_{I_1})$. Since $g_{I_2}$ factors as $R_{I_2} \mapsto R_{I_1} \xrightarrow{g_{I_1}} R$, it follows that the map $R_{I_1} \mapsto R$ restricts to an injection on $\theta(R_{I_1})$. Denoting the image of this injection as $\theta(R)$, we obtain a decomposition
$$\theta(R) = A_0 \, B_1 \, A_1 \, B_2 \, A_2 \, B_3 \, A_3 
$$
where $A_i$, $B_i$ are the images of the $\alpha_i$, $\beta_i$ respectively. Since each $\beta_i$ crosses a representative of every edgelet orbit of $R_{I_1}$, it follows that each $B_i$ crosses a representative of every edgelet orbit of~$R$.

Consider a generic leaf $\ell \in \Lambda$. Since its realization $\ell(T_{I_2})$ crosses a translate of every edge of $T_{I_2}$, we may rechoose $\ell$ in its orbit so that in $T_{I_2}$, the line $\ell(T_{I_2})$ crosses the edge $\theta=\theta(T_{I_2})$. Combining Lemma~\ref{LemmaLeavesInMain}~\pref{ItemWithinT} with Lemma~\ref{LemmaTilesInMain}~\pref{ItemTileInTRow} it follows that in $T_{I_1}$, the line $\ell(T_{I_1})$ crosses the tile $\theta(T_{I_1})$. Combining Lemma~\ref{LemmaLeavesInMain}~\pref{ItemLineTtoS} with Lemma~\ref{LemmaTilesInMain}~\pref{ItemTileBetweenTandS} it follows in $S_{I_1}$, the line $\ell(S_{I_1})$ crosses the tile $\theta(S_{I_1})$. Combining Lemma~\ref{LemmaLeavesInMain}~\pref{ItemLineRFromS} with Lemma~\ref{LemmaTilesInMain}~\pref{ItemTileBetweenRAndS} it follows that in $R_{I_1}$, the line $\ell(R_{I_1})$ crosses the tile $\theta(R_{I_1})$.

We shall finish the proof by applying bounded cancellation of the foldable map $g_{I_1} \from R_{I_1} \mapsto R$, using the concrete lines $\ell(R_{I_1})$ in $R_{I_1}$ and $\ell(R)$ in $R$. Since $R_{I_1}$ and $R$ are both Grushko free splittings \relA, one can use the form of bounded cancellation stated in \cite[Lemma 4.14~(3)]{\RelFSHypTwoTag}: using the edgelet subdivisions of $R_{I_1}$ and $R$ with respect to which the map $g_{I_1}$ is simplicial, and assigning length~$1$ to all edgelets so that $g_{I_1}$ restricts to an isometry from each edgelet of $R_{I_1}$ to an edgelet of $R$, it follows that the map $g_{I_1}$ has bounded cancellation constant 
$$C = \Length(R_{I_1}) - \Length(R) = \#\text{edges}(R_{I_1}) - \#\text{edges}(R)
$$
and that
$$g_{I_1}(\ell(R_{I_1})) \subset N_C(\ell(R))
$$
Restricting to the subpath $\theta(R_{I_1}) \subset \ell(R_{I_1})$ we therefore have
$$(\#) \qquad \underbrace{g_{I_1}(\theta(R_{I_1}))}_{\theta(R)} \subset N_C(\ell(R)) \qquad\hphantom{(\#)}
$$

Consider the decomposition $\theta(R) = \nu_- \, \nu_0 \, \nu_+$ where $\nu_0 = \theta(R) \intersect \ell(R)$ is the maximal common subpath of $\theta(R)$ and $\ell(R)$. The initial path $\nu_-$ meets the line $\ell(R)$ only in the terminal endpoint of $\nu_-$, and so the initial endpoint of $\nu_-$ achieves the maximum distance to $\ell(R)$ amongst all points on $\nu_-$, that maximal distance being equal to $\Length(\nu_-)$; applying $(\#)$ it follows that $\Length(\nu_-) \le C$. By a similar argument we have $\Length(\nu_+) \le C$. Since $\beta_1$ crosses a translate of every edge of $R_{I_1}$, it follows that 
$$\Length(A_0 \, B_1) = \Length(\alpha_0 \, \beta_1) \ge C
$$
and similarly
$$\Length(B_3 \, A_3) = \Length(\beta_3 \, \alpha_3) \ge C
$$
thus $\nu_-$ is a subpath of $A_0 B_1$, and $\nu_+$ is a subpath of $B_3 A_3$. It then follows that $A_1  B_2 A_2$ is a subpath of $\nu_0$ and hence of $\ell(R)$. Since $B_2$ crosses a translate of every edgelet of $R$, so does $\ell(R)$. This completes the proof that $\Lambda$ fills $\Gamma$ \relA, and hence the proof of the implication \pref{ItemActLargeOrbit}$\implies$\pref{ItemFillingLamExists} of Theorem~B.

\section{The lower bound in Theorem A.}
\label{SectionLowerBound}

Consider an outer automorphism $\phi \in \Out(\Gamma;\A)$ which has a filling attracting lamination \relA. Having proved Theorem~B (in Part II \cite{\RelFSHypTwoTag} and in previous sections here in Part III), this is equivalent to saying that $\phi$ acts loxodromically on $\FS(\Gamma;\A)$, with stable translation length~$\tau_\phi > 0$. In this section we find a constant $A = A(\Gamma;\A) > 0$ independent of $\phi$ such that $\tau_\phi \ge A$.

The proof was outlined in the introduction, and a more detailed outline is found in Section~\ref{SectionFindingBound}. The proof applies tools from Section~\ref{SectionNoFillingLam}, namely PF-exponents (Definition~\ref{DefPFExponent}) and filling exponents (Definition~\ref{DefFillingExponent}), and Lemma~\ref{LemmaQAPQuant}~\pref{ItemSTLLowerBound} which gives a positive lower bound for $\tau_\phi$ expressed as a function of a PF-exponent $\kappa$ and a filling exponent~$\omega$ of an EG-aperiodic train track axis for $\phi$, the form of that function depending only on $\corank(\Gamma;\A)$ and $\abs{\A}$. This reduces the problem to finding uniform values of $\kappa$ and $\omega$, depending only on $\corank(\Gamma;\A)$ and~$\abs{\A}$. We describe such values for the special class of EG-aperiodic train track axes obtained by suspending a train track representative defined on a natural free splitting. For such axes, a uniform PF-exponent is described in Proposition~\ref{PropUniformPF} proved in Section~\ref{PropUniformPF}, and a uniform filling exponent is described in Proposition~\ref{PropUniformFilling} proved in Section~\ref{PropUniformFilling}. Proposition~\ref{PropNaturalAxis} produces the axes that we need for these applications.

\subsection{Finding the lower bound.} 
\label{SectionFindingBound}
Here is the description of the strong class of EG-aperiodic fold axes that will be used in the proof:
\begin{proposition} 
\label{PropNaturalAxis}
For any $\phi \in \Out(\Gamma;\A)$ acting loxodromically on $\FS(\Gamma;\A)$, and for any maximal, proper, $\phi$-invariant free factor system~$\F$ rel~$\A$, in the relative free splitting complex $\FS(\Gamma;\A)$ there exists an EG-aperiodic fold axis for $\phi$ such that $T_0$ is a natural, Grushko free splitting with respect to~$\F$.
\end{proposition}

\begin{proof} First we apply the following:

\begin{theorem}[\protect{\cite[Theorem 8.24]{FrancavigliaMartino:TrainTracks}}]
\label{TheoremNatIrrTT}
There exists an irreducible train track representative $f \from T \to T$ with respect to~$\F$ where $T$ is a natural Grushko free splitting relative to~$\F$.
\end{theorem}
\noindent
See Section~\ref{SectionNaturalTTRep} for a further discussion of this theorem, including: some details of translation between the language of \cite{FrancavigliaMartino:TrainTracks} and our language here, in order to aid the citation; and a different proof of the theorem, adapting the original construction for $\Out(F_n)$ found in \cite{\BHTag}.

To finish the proof, by applying Conclusion (2c) of Section~\ref{SectionFoldAxisReview} it follows that $\Lambda$ --- the unique filling lamination for $\phi$ \relA\ --- is identified with the unique filling lamination for $\phi$ rel~$\F$, under any collapse map to $T$ defined on a Grushko free splitting \relA. Applying relative train track methods as summarized in \cite[Theorem 4.17]{\RelFSHypTwoTag}, it follows that the train track representative $f \from T \to T$ given in Theorem~\ref{TheoremNatIrrTT} is EG-aperiodic. Applying \cite[Proposition 4.22]{\RelFSHypTwoTag} we obtain an EG-aperiodic fold axis for $\phi$ with respect to~$\F$ such that $T=T_0$ and such that its first return map is $f \from T \to T$.
\end{proof}

We next state propositions regarding a uniform PF-exponent (Definition~\ref{DefPFExponent}) and a uniform filling exponent (Definition~\ref{DefFillingExponent}):

\begin{proposition}[Uniform PF exponents for natural iteration tiles]
\label{PropUniformPF}
There exist an integer $\kappa_0 = \kappa_0(\Gamma;\A) \ge 1$ such that for any $\phi \in \Out(\Gamma;\A)$ that acts loxodromically on $\FS(\Gamma;\A)$, for any maximal, proper, $\phi$-invariant free factor system~$\F$ rel~$\A$, and for any EG-aperiodic fold axis of $\phi$ relative to~$\F$ such that $T_0$ is a natural free splitting, $\kappa_0$ is a PF exponent.
\end{proposition}

\begin{proposition}[Uniform filling exponents for natural iteration tiles] 
\label{PropUniformFilling}
There exist an integer $\omega_0 = \omega_0(\Gamma;\A) \ge 1$ such that for any $\phi \in \Out(\Gamma;\A)$ that acts loxodromically on $\FS(\Gamma;\A)$, for any maximal, proper, $\phi$-invariant free factor system~$\F$ rel~$\A$, and for any EG-aperiodic fold axis of $\phi$ relative to~$\F$ such that $T_0$ is a natural free splitting, $\omega_0$ is a filling exponent.
\end{proposition}

The proofs of these two propositions are found in Sections~\ref{SectionUniformPF} and~\ref{SectionUniformFilling} respectively. Here we quickly apply them:

\begin{proof}[Proof of the lower bound for Theorem A]
Using the constants $\kappa_0(\Gamma;\A)\ge 1$ and $\omega_0(\Gamma;\A) \ge 1$ from Propositions~\ref{PropUniformPF} and~\ref{PropUniformFilling}, and the constant $\nu = \nu(\Gamma;\A) \ge 1$ from Proposition~\ref{LemmaQAPQuant} (which ultimately derives from \cite[Corollary 5.5]{\RelFSHypTag}), by applying that proposition we obtain the inequality $\ds \tau_\phi \ge \frac{1}{\nu (3 \kappa_0 + \omega_0)}$. 
\end{proof}

%Definition~\ref{DefPFExponent} and Definition~\ref{DefFillingExponent}

\subsection{Finding a uniform PF exponent: Proof of Proposition \ref{PropUniformPF}}  
\label{SectionUniformPF}

%Section~\ref{SectionUniformPF} proof of Proposition~\ref{PropUniformPF}

First we show that if $\kappa$ is a PF exponent for for $F_0 \from T_0 \to T_0$ then $\kappa+1$ is a PF exponent for each $F_j \from T_j \to T_j$. By periodicity, $\kappa$ is also a PF-exponent for every $F_{ip} \from T_{ip} \to T_{ip}$ ($i \in \Z$). Using the notation of Section~\ref{SectionFoldAxisReview}, we have a factorization $F^\kappa_{ip} = f^{(i-\kappa)p}_{ip} \circ h^{ip}_{(i-\kappa)p}$. Recalling that the $h$-maps are all simplicial isomorphisms, the transition matrix of the map $f^{(i-\kappa)p}_{ip}$ has all entries $\ge 4$. Next we show that $\kappa+1$ is a PF exponent for every $F_j$ ($j \in \Z$). By periodicity it suffices to show this for $0 < j < p$. Noting that $j-(\kappa+1)p \le -\kappa p$, consider the map $F_j^{\kappa+1} \from T_j \to T_j$ which factors as 
$$F_j^{\kappa+1} \from T_j \xrightarrow{h^j_{j-(\kappa+1)p}} T_{j-(\kappa+1)p} \xrightarrow {f^{j-(\kappa+1)p}_{-\kappa p}} T_{-\kappa p} \xrightarrow{f^{-\kappa p}_0} T_0 \xrightarrow{f^0_{j}} T_j
$$
Since the transition matrix for the map $f^{-\kappa p}_0$ has all entries $\ge 4$, it follows that the transition matrix for $F_j^{\kappa+1}$ has all entries $\ge 4$.

Using that $T_0$ is a natural free splitting, what remains is to prove that there is a uniform PF-exponent $\kappa_1 = \kappa_1(\Gamma;\A)$ for all EG-aperiodic train track maps $F \from T \to T$ defined on natural free splittings $T$ of $\Gamma$ \relA, for then we can take $\kappa_0 = \kappa_1 + 1$. The number of natural edge orbits of $T$ is bounded above by $m_1 = 3 \corank(\A) + 2 \abs{\A} - 3$, and so transition matrix $M$ for~$F$ is an $m \times m$ matrix of non-negative integers with $m \le m_1$. Consider the $m \times m$ Boolean matrix $\wh M$ that is obtained from $M$ by replacing every positive entry with a~$1$. For each $i \ge 1$, note that the Boolean power $\wh M^i$ is obtained from the ordinary power $M^i$ in a similar fashion, by replacing every positive entry with a~$1$. There are only finitely many $m \times m$ Boolean matrices such that $m \le m_1$, and so amongst those that have a Boolean power with all positive entries, there is a uniform exponent $\kappa_2=\kappa_2(\Gamma;\A)$ achieving that power. The matrix $M^{\kappa_2}$ therefore has all positive entries. It follows that every entry of $M^{3\kappa_2}$ is $\ge 4$, hence $\kappa_1 = 3 \kappa_2$ is a PF-exponent for $F$.

\subsection{Finding a uniform filling exponent: Proof of Proposition \ref{PropUniformFilling}}  
\label{SectionUniformFilling}

As in the proof of Proposition~\ref{PropUniformPF}, first we show that if $\omega$ is a filling exponent for $F_0$ then $\omega+1$ is a filling exponent for each $F_j \from T_j \to T_j$. And again, we may assume $0<j<p$, hence $j-(\omega+1)p \le - \omega p$, and we consider the factorization of $F_j^{\omega+1} \from T_j \to T_j$ depicted in the top line of the following diagram:
$$\xymatrix{
F^{\omega+1}_j : T_j  \ar[rr]^{h^j_{j-(\omega+1)p}} &&
	T_{j-(\omega+1)p} \ar[rr]^{f^{j-(\omega+1)p}_{-\omega p}} &&
	T_{-\omega p} \ar[rr]^{f^{-\omega p}_{0}} &&
	T_{0} \ar[r]^{f^{0}_{j}} &
	T_{j} \\
	&& 
	{\vphantom{E^{2^x}}} E=\eta \ar[rr]^{h^j_{j-(\omega+1)p}} \ar@{^{(}->}[u] && 
	{\vphantom{E^{2^x}}} \eta' \ar[rr]^{f^{-\omega p}_{0}} \ar@{^{(}->}[u] && 
	{\vphantom{E^{2^x}}} \eta'' \ar[r]^{f^{0}_{j}} \ar@{^{(}->}[u] & 
	{\vphantom{E^{2^x}}} \eta''' \ar@{^{(}->}[u]
}$$
In this diagram we choose an arbitrary edge of $T_j$, whose image under the simplicial isomoprhism $h^j_{j-(\omega+1)p}$ is an edge of $T_{j-(\omega+1)p}$ denoted $E = \eta$, with images $\eta',\eta'',\eta'''$ (along the foldable path from $T_{j-(\omega+1)p}$ to $T_j$) as depicted in the bottom line of the diagram. What we have to prove is that $\eta'''$ fills $T_j$. Each map takes vertices to vertices, and so $\eta'$ is a nontrivial path in $T_{-\omega p}$ which crosses at least one whole edge $E'$ of $T_{-\omega p}$. Since $\omega$ is a filling exponent for $F$, in $T_0$ it follows that the path $f^{-\omega p}_0(E')$ fills $T_0$, but that is a subpath of $\eta''$ and so $\eta''$ itself fills $T_0$. Then, applying Lemma~\ref{LemmaFoldFillingProtoforest}, it follows that the path $\eta''' = f^0_j(\eta'')$ fills $T_j$.

\medskip

What remains is to find a uniform constant $\omega=\omega(\Gamma;\A)$ which is a filling exponent for the first return map $F_0 \from T_0 \to T_0$ of the axis. Henceforth we drop the subscript $0$ and write this map as $F \from T \to T$. Since $T$ is a natural free splitting, applying Proposition~\ref{PropUniformPF} we obtain a uniform PF exponent $\kappa = \kappa(\Gamma;\A)$ for $F$. The map $G=F^\kappa \from T \to T$ is an EG-aperiodic train track representative of $\phi^\kappa$, it has filling exponent~$1$, and the outer automorphisms $\phi$ and $\phi^\kappa \in \Out(\Gamma;\A)$ share the same filling lamination $\Lambda$. It suffices to show that the uniform constant $\KR(\Gamma;\A) = \abs{\A} + \corank(\A)$ is a filling exponent for $G$, because then the uniform constant $\omega = \kappa \cdot \KR(\Gamma;\A)$ is a filling exponent for~$F$. 

Fix an arbitrary edge $E \subset T$. Since $G$ has filling exponent $1$, the tile $G(E)$ has four crossings of translates of $E$, and hence $G(E)$ has an interior crossing of some translate of $E$. After replacing~$G$ with $\gamma \cdot G$ for an appropriately chosen $\gamma \in \Gamma$, we may assume that $G(E)$ has an interior crossing of $E$~itself. Iterating $G$, we obtain the following monotonically nested sequence of iteration tiles 
$$\underbrace{\vphantom{T(E)}E}_{\eta_0} \subset \underbrace{G(E)}_{\eta_1} \subset \underbrace{G^{2}(E)}_{\eta_2} \subset \underbrace{G^{3}(E)}_{\eta_3} \subset \cdots \subset \underbrace{G^{m}(E)}_{\eta_m} \subset \cdots
$$
such that each tile in this sequence has an interior crossing of every previous tile. For later use we choose $\Psi \in \Aut(\Gamma;\A)$ representing $\phi^\kappa$ so that $G=F^\kappa$ is $\Psi$-twisted equivariant.

For the remainder of the proof we shall show, using $\mu = \KR(\Gamma;\A)$, that the tile $\eta_\mu$ fills $T$. Since $\eta_\mu = G^\mu(E)=F^{\kappa\mu}(E)$, this shows that that the uniform constant $\omega=\kappa\mu=\kappa \cdot \KR(\Gamma;\A)$ is a filling exponent for $F$.

\medskip

Since each tile $\eta_m$ has an interior crossing of a translate of every edge of $T$, the covering forest of $\eta_m$ is the full tree $\beta(\eta_m)=T$ (\cite[Definition 5.6]{\RelFSHypTwoTag}). Recall from Definition~\ref{DefinitionFineFFS}, from Lemma~\ref{LemmaCheckDeltaConnected} and from Definition~\ref{DefinitionFillingForest} the following objects associated to $\eta_m$: its filling protoforest $\beta^\Fill(\eta_m)$; the unique protocomponent $\beta^\Fill_\Id(\eta_m)$ of $\beta^\Fill(\eta_m)$ that contains~$\eta_m$; the associated filling support $\Fm(\eta_m) = \Stab(\beta^\Fill_\Id(\eta_m))$ which is a free factor of $\Gamma$ \relA, and which has Kurosh rank \relA\ denoted $\Krank(\eta_m) = \Krank(\Fm(\eta_m))$; and the associated free factor system denoted $\F[\eta_m]=\F[\eta_m;T]$. By applying Lemma~\ref{LemmaNestingOfFillingSupport} to the monotonic sequence of paths above, these objects may all be arranged in the matrix with monotonic rows depicted in Figure~\ref{FigureMatrix}.

\begin{figure}
\fbox{
\centerline{
$\begin{matrix}
\hphantom{0 \,\, \le \,\, } \beta^\Fill_\Id(\eta_0) & \subseteq         & \beta^\Fill_\Id(\eta_1) 
	& \subseteq        & \beta^\Fill_\Id(\eta_2) & 
	\subseteq  & \cdots  & \subseteq & 
	\beta^\Fill_\Id(\eta_m) & \subseteq  \,\, \cdots \hphantom{\,\, \le \abs{\A} + \corank(\A)}
\\ \\
\hphantom{0 \,\, \le \,\, } \beta^\Fill(\eta_0) & \preceq         & \beta^\Fill(\eta_1) 
	& \preceq        & \beta^\Fill(\eta_2) & 
	\preceq  & \cdots  & \preceq & 
	\beta^\Fill(\eta_m) & \preceq  \,\, \cdots \hphantom{\,\, \le \abs{\A} + \corank(\A)}
\\ \\
\hphantom{0 \,\, \le \,\, } \Fm(\eta_0)              & \sqsubseteq & \Fm(\eta_1) 
	& \sqsubseteq &     \Fm(\eta_2)        & 
	\sqsubseteq  & \cdots  & \sqsubseteq & 
	\Fm(\eta_m) &\sqsubseteq  \,\, \cdots \hphantom{\,\, \le \abs{\A} + \corank(\A)}
\\ \\
0 \,\, \le \,\, \KR(\eta_0)         & \le             & \KR(\eta_1) 
	&      \le       & \KR(\eta_2)        & 
	\le             & \cdots &      \le       & 
	\KR(\eta_m) & \le \,\, \cdots \,\, \le \underbrace{\abs{\A} + \corank(\A)}_{\KR(\Gamma;\A)}
\end{matrix}
$
}
}
\caption{Associated to the monotonically nested tile sequence $\eta_0 \subset \ldots \subset \eta_m \subset \ldots$ is this matrix of monotonic relation sequences, such that in each column of relations \emph{one} equation holds if and only if \emph{all} equations hold (see Lemma~\ref{LemmaNestingOfFillingSupport}). In the final row, the upper bound $\KR(\Gamma;\A) = \abs{\A} + \corank(\A)$ on relative Kurosh ranks $\KR(\eta_i)=\KR(\Fm(\eta_i))$ is found in \cite[Lemma 2.9]{\RelFSHypTwoTag}.
}
\label{FigureMatrix}
\end{figure}

For each column of relations in Figure~\ref{FigureMatrix}, (i.e. for each $m \ge 0$), since $\beta(\eta_{m}) = \beta(\eta_{m+1}) = T$ we may apply Lemma~\ref{LemmaNestingOfFillingSupport} to conclude that the following two statements are equivalent:
\begin{description}
\item[(i)] One of these four equations holds:
\end{description}
$$\beta^\Fill_\Id(\eta_{m}) = \beta^\Fill_\Id(\eta_{m+1}) , \quad \beta^\Fill(\eta_{m}) = \beta^\Fill(\eta_{m+1}),\quad \Fm(\eta_{m}) = \Fm(\eta_{m+1}),\quad\KR(\eta_{m}) = \KR(\eta_{m+1})
$$

\begin{description}
\item[(ii)] \emph{All four} of those equations hold.
\end{description}

Applying monotonicity and boundedness of the Kurosh ranks in the fourth row of matrix in Figure~\ref{FigureMatrix}, it follows that there exists an integer $M$ with $1 \le M \le \KR(\Gamma;\A)$ such that $\KR(\eta_M)=\KR(\eta_{M+1})$. Applying equivalence of $(i)$ and $(ii)$ in the case $m=M$ we get the following statement:
\begin{description}
\item[(iii)] All relations in the matrix between columns $M$ and $M+1$ are equations:
\end{description}
$$\beta^\Fill_\Id(\eta_M) = \beta^\Fill_\Id(\eta_{M+1}) , \quad \beta^\Fill(\eta_{M}) = \beta^\Fill(\eta_{M+1}) , \quad  \Fm(\eta_{M}) = \Fm(\eta_{M+1}) , \quad  \KR(\eta_{M}) = \KR(\eta_{M+1}) 
$$
Using this statement we now prove:
\begin{description}
\item[(iv)] $G(\beta^\Fill_\Id(\eta_M)) = \beta^\Fill_\Id(\eta_M)$
\end{description}
To start the proof choose $\Psi \in \Aut(\Gamma;\A)$ representing $\phi^{\kappa}$ so that $G=F^{\kappa}$ is $\Psi$-twisted equivariant: 
$$G(\gamma \cdot x) = \Psi(\gamma) \cdot G(x) \, \text{for all} \, x \in S
$$
Knowing that $G$ restricts to a homeomorphism $\eta_M \mapsto \eta_{M+1}$, it follows for each $\gamma \in \Gamma$ that $G$ restricts to a homeomorphism $\gamma \cdot \eta_M \mapsto \Psi(\gamma) \cdot \eta_{M+1}$. And from this it follows that $G$ maps each $\eta_M$-connection of the form
$$(*)\qquad\qquad \eta_M, \quad g_1 \cdot \eta_M, \quad \ldots, \quad g_I \cdot \eta_M
$$
to an $\eta_{M+1}$ connection of the form
$$(**) \qquad\qquad \underbrace{\eta_{M+1}}_{G(\eta_M)}, \quad \underbrace{\Psi(g_1) \cdot \eta_{M+1}}_{G(g_1 \cdot \eta_M)}, \quad\ldots, \quad \underbrace{\Psi(g_I) \cdot \eta_{M+1}}_{G(g_I \cdot \eta_M)}
$$
We use this to prove that for any edge $e \subset G(\beta^\Over_\Id(\eta_M))$ we have $e \subset \beta^\Over_\Id(\eta_{M+1})$. Choosing any $e' \subset \eta_M \subset \eta_{M+1} = G(\eta_M)$, it suffices to construct an $\eta_{M+1}$ connection from $e'$ to $e$.  
Since $e \subset G(\beta^\Over_\Id(\eta_M))$, we may choose an edge $e'' \subset \beta^\Over_\Id(\eta_M)$ such that $e \subset G(e'')$. By definition of~$\beta^\Over_\Id(\eta_M)$, there is an $\eta_M$-connection from $e'$ to $e''$ as shown in $(*)$ above; in particular $e'' \subset g_I \cdot \eta_M$. Its $G$ image, shown in $(**)$, is the desired $\eta_{M+1}$-connection from $e' \subset \eta_{M+1}$ to $e \subset G(e'') \subset G(g_I \cdot \eta_M)$. We have proved that
$$G(\beta^\Over_\Id(\eta_M)) \subset \beta^\Over_\Id(\eta_{M+1}) 
$$

We now consider the stabilizer subgroup $\Stab(\beta^\Over_\Id(\eta_M))$ and its $\Psi$ image $\Psi(\Stab(\beta^\Over_\Id(\eta_M)))$. For all $g \in \Stab(\beta^\Over_\Id(\eta_M))$ we have
$$\Psi(g) \cdot G(\beta^\Over_\Id(\eta_M)) = G(g \cdot \beta^\Over_\Id(\eta_M)) = G(\beta^\Over_\Id(\eta_M))
$$
The subgroup $\Psi(\Stab(\beta^\Over_\Id(\eta_M)))$ therefore stabilizes the tree $G(\beta^\Over_\Id(\eta_M))$. Since $G(\beta^\Over_\Id(\eta_M))$ is a subtree of $\beta^\Over_\Id(\eta_{M+1})$, and since $\beta^\Over_\Id(\eta_{M+1})$ is a protocomponent of the protoforest $\beta^\Over(\eta_{M+1})$, it follows that the subgroup $\Psi(\Stab(\beta^\Over_\Id(\eta_M)))$ stabilizes $\beta^\Over_\Id(\eta_{M+1})$:
$$\Psi(\Stab(\beta^\Over_\Id(\eta_M))) \subgroup \Stab(\beta^\Over_\Id(\eta_{M+1})) \subgroup \Fm(\eta_{M+1})=\Fm(\eta_M)
$$
where the last equation $\Fm(\eta_{M+1})=\Fm(\eta_M)$ comes from item (iii) above.

By definition $\Fm(\eta_M)$ is the minimal free factor of $\Gamma$ \relA\ containing $\Stab(\beta^\Over_\Id(\eta_M))$. Since $\Psi \in \Aut(\Gamma;\A)$, it follows that the minimum free factor of $\Gamma$ \relA\ containing $\Psi(\Stab(\beta^\Over_\Id(\eta_M)))$ is $\Psi(\Fm(\eta_M))$. So we have two free factors of $\Gamma$ \relA\ of equal Kurosh ranks, namely $\Fm(\eta_M)$ and $\Psi(\Fm(\eta_M))$, each containing the subgroup $\Psi(\Stab(\beta^\Over_\Id(\eta_M)))$; and it is known that one of these two, namely $\Psi(\Fm(\eta_M))$, is the unique minimal such free factor of $\Gamma$ \relA. These two free factors are therefore equal:
$$\Psi(\Fm(\eta_M))=\Fm(\eta_M)
$$
It follows that
\begin{align*}
G(\beta^\Fill_\Id(\eta_M)) &= G(\Fm(\eta_M) \cdot \eta_M) = \Psi(\Fm(\eta_M)) \cdot G(\eta_M) \\
&= \Fm(\eta_M) \cdot \eta_{M+1} = \Fm(\eta_{M+1}) \cdot \eta_{M+1} = \beta^\Fill_\Id(\eta_{M+1}) \\
&= \beta^\Fill_\Id(\eta_M)
\end{align*}
where, again, the last equation $\beta^\Fill_\Id(\eta_{M+1}) = \beta^\Fill_\Id(\eta_M)$ comes from item~(iii). This completes the proof of~(iv).

Next we prove:
\begin{description}
\item[(v)] $\Fm(\eta_M) = \Gamma$, hence the tile $\eta_M$ fills $S$.
\end{description}
Since $G(\beta^\Fill_\Id(\eta_M)) = \beta^\Fill_\Id(\eta_M)$, starting from $\eta_M \subset \beta^\Fill_\Id(\eta_M)$ and using that $G(\eta_m)=\eta_{m+1}$ for all $m$ it follows by induction that $\eta_{m} \subset \beta^\Fill_\Id(\eta_M)$ for all $m \ge M$. The generic leaf $\ell = \bigcup_{m \ge 0} \eta_{m}$ is therefore contained in $\beta^\Fill_\Id(\eta_M)$. Since $\Fm(\eta_M) = \Stab(\beta^\Fill_\Id(\eta_M))$, and since $\F[\eta_M]$ is the smallest free factor system having $[\Fm(\eta_M)]$ as a component, it follows that $\ell$ is supported by $\F[\eta_M]$. But $\ell$ is a generic leaf of the attracting lamination $\Lambda$, and hence $\Lambda$ is supported by $\F[\eta_M]$. Since $\Lambda$ fills $\Gamma$ \relA, it follows that $\F[\eta_M]=\{[\Gamma]\}$, equivalently $\Fm(\eta_M) = \Gamma$, equivalently $\eta_M$ fills~$S$.

Since $\eta_M \subset \eta_\mu$, and since $\eta_M$ fills~$T$, it follows from Proposition~\ref{LemmaNestingOfFillingSupport} that $\eta_\mu$ also fills~$T$.

\subsection{Finding a natural train track representative}
\label{SectionNaturalTTRep}
In this section we discuss the following result:
\begin{theorem*} \cite[Theorem 8.24]{FrancavigliaMartino:TrainTracks} For any $\phi \in \Out(\Gamma;\A)$ and a maximal, $\phi$-invariant, proper free factor system~$\F$ of $\phi$ \relA, there exists an irreducible train track representative $f \from T \to T$ of $\phi$ with respect to~$\F$ defined on a natural free splitting~$T$. 
\end{theorem*}

\paragraph{Translating the theorem.} The statement above is not an exact statement of \cite[Theorem 8.24]{FrancavigliaMartino:TrainTracks}; it is instead a significant translation of that theorem into our language. Some remarks are therefore in order to aid the translation. In the notation of \cite{FrancavigliaMartino:TrainTracks}, consider a group equipped with a free factorization $G = G_1 * \cdots * G_p * F_k$; we use the notations $\Gamma = G$ and $\F = \{[G_i]\}_{i=1}^p$. Our group $\Aut(\Gamma;\F)$ is identified with the group denoted $\Aut(G,\mathcal O)$ in \cite[Definition 3.2]{FrancavigliaMartino:TrainTracks} (the ``$\Out$'' notation is rarely used in \cite{FrancavigliaMartino:TrainTracks}). Note that $\mathcal O$ represents the unprojectivized outer space of $\Gamma$ rel~$\F$; projectivized outer space $P\mathcal O$ may be identified with the open subset of the relative free splitting complex $\FS(\Gamma;\F)$ represented by Grushko free splittings of $\Gamma$~rel~$\F$. Every free splitting in \cite{FrancavigliaMartino:TrainTracks} is a natural free splitting in our language: see properties (C0)--(C3) in \cite[Section 3]{FrancavigliaMartino:TrainTracks}, particularly (C0) which is equivalent to naturality. Our concept of a ``$\Phi$-twisted equivariant'' map $f \from T \to T$ is generally referred to in \cite{FrancavigliaMartino:TrainTracks} as a map ``representing~$\Phi$'' (see \cite[Definition 8.1]{FrancavigliaMartino:TrainTracks} and the preceding paragraph).  Given $\phi \in \Out(\Gamma;\F)$, for any representative $\Phi \in \Aut(\Gamma;\F)$ of $\phi$, the statement that $\F$ is a maximal $\phi$-invariant free factor system is equivalent to the the statement in \cite{FrancavigliaMartino:TrainTracks} that $\Phi$ is irreducible. The exact statement of Theorem 8.24 is this (using our notation for $\Aut(\Gamma;\F) = \Aut(G,\O)$):
\begin{theorem*}\cite[Theorem 8.24]{FrancavigliaMartino:TrainTracks}
Let $\Phi$ be an irreducible element of $\Aut(\Gamma;\F)$. Then there exists a simplicial optimal train track map representing $\Phi$.
\end{theorem*}
\noindent
Here are a few more remarks on translation. Regarding equivalence of two definitions of train track maps, see the discussion in \cite[Section 4.3.3]{\RelFSHypTwoTag}. Let $f \from T \to T$ be the ``simplicial, optimal train track map'' of the conclusion. The meaning of ``simplicial'' in this statement is not the one we use --- a map in the category of simplicial complexes --- but instead is simply the statement that vertices go to vertices (see the paragraph preceding Theorem 8.24). ``Optimality'' is a property that we do not need. But we do need for $f \from T \to T$ to be irreducible. If not then there is a proper subforest $\tau \subset T$ invariant under $\Gamma$ and $f$, with a corresponding collapse map $T \xrightarrow{\<\tau\>} T'$; by maximality of~$\F$, each component of $\tau$ has trivial stabilizer, and so $T'$ is still a Grushko free splitting of $\Gamma$ rel~$\F$. Under the collapse map $T \to T'$, the map $f$ on $T$ induces a map on $T'$, and by tightening we obtain a train track representative $f' \from T' \to T'$ of $\phi$. Proceeding by induction on the number of edge orbits, we obtain an irreducible train track representative of~$\phi$.

\paragraph{Another proof of the Theorem.}
Consider any $\phi \in \Out(\Gamma;\A)$ and any maximal, $\phi$-invariant, non-filling free factor system~$\F$ of $\phi$ \relA. We may think of $\phi$ as an irreducible element of the relative outer automorphism group $\Out(\Gamma;\F)$. For the foundational special case of $\Out(F_n)$ --- where $\Gamma=F_n$ and $\F=\emptyset$ --- an algorithmic construction of an irreducible natural train track representative can be found in \cite[Section 1]{\BHTag}. We give a straightforward adaptation of that construction to our present circumstances, aided by applying tools of \cite{\LymanRTTTag}. 

We review concepts of topological representatives needed for this proof (see \cite[Section 4.3.1]{\RelFSHypTwoTag}). Consider the class $\C^{\text{irr}}_\phi$ of irreducible topological representatives $F \from T \to T$ of $\phi$ defined on a Grushko free splitting $T$ of $\Gamma$ rel~$\F$. To see that this class is nonempty, start from an arbitrary topological representative $G \from U \to U$ defined on a Grushko free splitting rel~$\F$. From irreducibility of $\phi$ it follows that for every $\Gamma$-invariant subforest of $U$ that is either $G$-pretrivial or $G$-invariant, each component of that subforest has trivial stabilizer. So by repeatedly collapsing such forests one eventually arrives at a Grushko free splitting rel~$\F$ on which $G$ induces a topological representative having no pretrivial or $F$-invariant subforests, and any such representative is irreducible. 

Each $F \in \C^{\text{irr}}_\phi$ is $\Phi$-twisted equivariant with respect to some $\Phi \in \Aut(\Gamma;\A)$ representing~$\phi$. Also, $F$ takes each vertex to a vertex and each edge to a nontrivial edge path without backtracking. There is an $n \times n$ transition matrix $M$ for~$F$, meaning that for some enumerated set of edgelet orbit representatives $e_1,\ldots,e_n \subset S$, each matrix entry $M_{ij}$ counts the number of times that $F(e_j)$ crosses translates of $e_i$ (for $1 \le i,j \le n$). Irreducibility of $F$ means that for each $i,j$ there exists $k \ge 1$ such that $M^k_{ij} \ne 0$. We let $\lambda \ge 1$ denote the Perron-Frobenius eigenvalue, and we let $\vec w$ denote a Perron-Frobenius column eigenvector of $M$, and so $\sum_j M_{ij} w_j = \lambda w_i$. 

If there exists $F \in \C^{\text{irr}}_\phi$ which is NEG --- meaning that $\lambda=1$ --- then $F$ is automatically a train track representative of $\phi$. In this case $F$ is also a PL homeomorphism, and so we can replace the given simplicial structure on $S$ by a natural simplicial structure, completing the proof. 

We may therefore assume that $\lambda > 1$ for every~$F$ --- that is, every $F \in \C^{\text{irr}}$ is~EG. Amongst all members of $\C^{\text{irr}}_\phi$ there exists one for which $\lambda$ is minimal, and furthermore any member with minimal $\lambda$ is a train track representative: for $\Out(F_n)$ see \cite[Theorem 1.7]{\BHTag} and its proof; and for $\Out(\Gamma;\A)$ see \cite[Theorem 3.2]{\LymanRTTTag} and its proof.

Consider the subclass $\C^{\text{min}}_\phi \subset \C^{\text{irr}}_\phi$ of representatives $F \from S \to S$ for which $\lambda$ is minimal, so every member of $\C^{\text{min}}_\phi$ is a train track representative of~$\phi$. Let $\abs{S}$ denote the number of orbits of valence~$2$ vertices of $S$ with trivial stabilizer, and note that if $\abs{S}=0$ then $S$ is a natural free splitting. It therefore suffices to choose any $F \from S \to S$ in the class $\C^{\text{min}}_\phi$ and to prove, using concepts of ``valence~$2$ homotopy'', that if $\abs{S} > 0$ then there is another member $F' \from S' \to S'$ of $\C^{\text{min}}_\phi$ such that $\abs{S'} < \abs{S}$. 

Using the assumption $\abs{S} > 0$, choose $v \in S$ to be a vertex of valence~$2$ with trivial stabilizer. The two edgelets of $S$ incident to $v$ lie in the same natural edge of $S$ and they lie in different edgelet orbits. We may assume those two edgelets are the chosen representatives of their orbits, denoted $e_i$ and $e_j$. By permutating the notation we may assume $w_i \ge w_j$. Let $u$ be the vertex of $e_j$ opposite $v$. Since edge stabilizers of $S$ are trivial, the vertex $v$ is ``nonproblematic'' in the sense of \cite[Remark 3.6]{\LymanRTTTag} (strictly speaking, the projection of $v$ to the quotient graph of groups $S/\Gamma$ is nonproblematic). We may therefore apply \cite[Lemma 3.5]{\LymanRTTTag} to obtain a new irreducible topological representative $F_1 \from S_1 \to S_1$ by the operation of ``valence-two homotopy of $v$ across $e_j$''. From the conclusion of that lemma, the Perron-Frobenius eigenvalue of $F_1$ satisfies $\lambda_1 \le \lambda$, and by minimality of $\lambda$ it then follows that $\lambda_1=\lambda$. Also, the number $\abs{S_1}$ of vertex orbits of $S_1$ is strictly less than the number $\abs{S}$ of vertex orbits of $S$.

%Conjectures and partial results for $\FF(\Gamma;\A)$ (Section~\ref{SectionFFAandB})

\section{Conjectures and partial results for $\CFFS(\Gamma;\A)$}
\label{SectionFFAandB}

As stated in the \emph{Overview}, Theorems~A and~B, regarding elements $\phi \in \Out(\Gamma;\A)$ acting on the relative free splitting complex $\FS(\Gamma;\A)$, may be viewed as quantitative refinements of a qualitative classification: $\phi$ acts loxodromically on $\FS(\Gamma;\A)$ if and only if $\phi$ has a filling laminations; otherwise $\phi$ acts elliptically. 

Regarding the action of $\phi \in \Out(\Gamma;\A)$ on the complex of relative free factor systems $\CFFS(\Gamma;\A)$ --- equivalently, on~$\mathcal F(\Gamma;\A)$ (Proposition~6.3 of Part I) --- the qualitative classification is given in the following theorem of Guirardel and Horbez (for earlier special cases see the \emph{Overview}):

%\marginparLee{Still need to squeeze in a discussion of ``sporadic''. But somewhere we should adopt the ``exceptional'' definitions of $\CFFS(\Gamma;\A)$ for the exceptional nonsporadic cases.}
\begin{theorem}[\protect{\cite[Theorem 4.1]{GuirardelHorbez:SubgroupClassification}}]
\label{TheoremGH}
Assuming $(\Gamma;\A)$ is nonsporadic, $\phi \in \Out(\Gamma;\A)$ acts loxodromically on $\mathcal F(\Gamma;\A)$ if and only if $\phi$ is fully irreducible \relA; otherwise $\phi$ is elliptic. \qed
\end{theorem}
\noindent
The ``sporadic'' cases that are ruled out in this theorem are a subset of the ``exceptional'' cases defined in Section 2.5 of Part~I: to say that $(\Gamma;\A)$ is sporadic means that $\abs{\A}=2$ and $\corank(\Gamma;\A)=0$, or $\abs{\A}=1$ and $\corank(\Gamma;\A)=1$. For a brief comments regarding the exceptional, nonsporadic cases, see the paragraphs just preceding Section~\ref{SectionFFUpperBound}. 

%\begin{description}
%\item[Question:] Assuming $(\Gamma;\A)$ is nonsporadic, does there exist $\phi \in \Out(\Gamma;\A)$ acting loxodromically on $\CFFS(\Gamma;\A)$? Equivalenly, does there exists an element of $\Out(\Gamma;\A)$ that is fully irreducible rel~$\A$?
%\end{description}

\medskip

Here in Section~\ref{SectionFFAandB} we consider the problem of adding quantitative details to Theorem~\ref{TheoremGH}, in a manner analogous to Theorems~A and~B. We do not know proofs of all of the resulting statements, but we offer some partial results, and certain conjectures regarding the remainder. The proofs of Theorems~A and~B break naturally into six pieces: two bounds in Theorem~A; and a cycle of four implications in Theorem~B. We formulate $\CFFS(\Gamma;\A)$ analogues of these six pieces, starting with two bounds, with positive constants $A^\CFFS=A^\CFFS(\Gamma;\A)$ and $B^\CFFS=B^\CFFS(\Gamma;\A)$ that are to be determined:
\begin{description}
\item[Section~\ref{SectionFFUpperBound}: Upper bound] (\emph{proved}) \,\, If $\phi$ is fully irreducible \relA\ then $\tau^\CFFS_\phi \le B^\CFFS\log(\lambda_\phi)$, where \hbox{$\lambda_\phi > 0$} is the expansion factor of the unique attracting lamination of $\phi$.
\item[Section~\ref{SectionFFLowerBound}: Lower bound] (\emph{conjectured}) If $\phi \in \Out(\Gamma;\A)$ is fully irreducible \relA\ then \hbox{$\tau^\CFFS_\phi \ge A^\CFFS$\!.}
\end{description}
The remaining pieces form a cycle of four implications (1)$\implies$(2)$\implies$(3)$\implies$(4)$\implies$(1) amongst four statements, with a positive constant $\Omega^\CFFS=\Omega^\CFFS(\Gamma;\A)$ to be determined for purposes of the final implication (4)$\implies$(1):
\begin{enumerate}
\item $\phi$ is fully irreducible \relA.
\item $\phi$ acts loxodromically on $\CFFS(\Gamma;\A)$.
\item Every $\phi$-orbit on $\CFFS(\Gamma;\A)$ has infinite diameter.
\item Every $\phi$-orbit on $\CFFS(\Gamma;\A)$ has diameter $\ge \Omega^\CFFS$.
\end{enumerate}
The equivalence (1)$\iff$(2) is already known from Theorem~\ref{TheoremGH} above; if the full cycle of implications was known then equivalence of all four statements would follow. Note that, just as occurred for Theorems~A and~B, the implication (1)$\implies$(2) is also a logical consequence of the conjectured \emph{Lower Bound}, and the implications (2)$\implies$(3)$\implies$(4) are trivial. All that remains of the cycle of implications is the following:
\begin{description}
\item[Section~\ref{SectionFFLargeOrbits}: The implication (4)$\implies$(1)] (\emph{conjectured in general; proved if $\Gamma=F_n$}) \hfill\break There exists a constant $\Omega^\CFFS(\Gamma;\A) > 0$ such that if every orbit of $\phi$ acting on $\CFFS(\Gamma;\A)$ has diameter $\ge \Omega^\CFFS(\Gamma;\A)$ then $\phi$ is fully irreduclble.
\end{description}
As will be seen in Section~\ref{SectionFFLargeOrbits}, the proof for the case $\Gamma=F_n$ uses details of relative train track theory that are known for $\Out(F_n)$ (and applicable to $\Out(F_n;\A)$) but not known in general.

%One can formulate analogues of each of these pieces --- regarding the dynamics of $\phi \in \Out(\Gamma;\A)$ acting on the complex of relative free factor systems $\CFFS(\Gamma;\A)$ --- by replacing the property \emph{``$\phi$ has a filling lamination \relA''} with the property \emph{``$\phi$ is fully irreducible \relA''}. Recall that to say $\phi$ is \emph{irreducible \relA} means that for every $\phi$-invariant free factor system~$\A$ \relA\ either $\A = \B$ or $B = \{[\Gamma]\}$; and to say that $\phi$ is \emph{fully irreducible \relA} means that $\phi^k$ is irreducible for all $k \ge 1$. For the action of $\Out(F_n)$ on the complex of free factor systems $\CFFS(\F_n)$, the analogue of the implication (1)$\implies$(2) of Theorem~B was proved by Bestvina and Feighn in \cite[Theorem 9.3]{BestvinaFeighn:FFCHyp}: if $\phi \in \Out(F_n)$ is fully irreducible then $\phi$ acts loxodromically on $\CFFS(F_n)$. 
%\marginparLee{Go through and use the terminology ``complex of free factor systems''.}

\smallskip
The only nonsporadic, exceptional case is $\abs{\A}=3$ and $\corank(\Gamma;\A)=0$, which is discussed separately near the end of Section 6.1 of Part I, with the conclusion that in this case $\CFFS(\Gamma;\A)$ is equivariantly quasi-isometric to $\FS(\Gamma;\A)$. In this case there is also an easy relative train track argument to show that $\phi$ is fully irreducible if and only if it has a filling lamination (just as occurs when $\A=\emptyset$, $\corank(\A)=2$, i.e.\ the case $\Out(F_2)$). It follows that the two bounds and the four implications above are all known for $\CFFS(\Gamma;\A)$ in this one nonsporadic, exceptional case. 

Thus, for the remainder of Section~\ref{SectionFFAandB}, we may silently assume that $(\Gamma;\A)$ is non-exceptional.

\subsection{Translation lengths in $\CFFS(\Gamma;\A)$: Proof of an upper bound.} 
\label{SectionFFUpperBound}

\begin{proposition}
\label{PropUpperBoundFF}
There exists $B^\CFFS = B^\CFFS(\Gamma;\A) > 0$ such that for any fully irreducible $\phi \in \Out(\Gamma;\A)$ we have 
$$\tau^\CFFS_\phi \le B^\CFFS \log(\lambda_\phi)
$$
\end{proposition}

\begin{proof} Recall from \cite[Section~6.2]{\RelFSHypTag} the equivariant projection set map $\Pi \from \FS(\Gamma;\A) \mapsto \CFFS(\Gamma;\A)$ and associated concepts, motivated by methods of Kapovich and Rafi \cite{KapovichRafi:HypImpliesHyp}. This map $\Pi$ associates, to each $0$-simplex $[T] \in \FS(\Gamma;\A)$ represented by a free splitting $T$ of $\Gamma$ \relA, the set of $0$-simplices $\Pi[T] \subset \CFFS(\Gamma;\A)$ representing free factor systems of $\Gamma$ \relA\ that are visible in~$T$. The map~$\Pi$ satisfies the ``inverted equivariance property'' $\Pi([T] \cdot \phi) = \phi^\inv \cdot \Pi[T]$ for each $\phi \in \Out(\Gamma;\A)$, and it maps the $0$-skeleton of $\FS(\Gamma;\A)$ surjectively to the $0$-skeleton of $\CFFS(\Gamma;\A)$ in the sense that the subsets $\Pi[T]$ cover the $0$-skeleton of $\CFFS(\Gamma;\A)$. A \emph{projection map} $\pi$ is a function from the $0$-skeleton of $\FS(\Gamma;\A)$ to the $0$-skeleton of $\CFFS(\Gamma;\A)$ satisfying $\pi[T] \in \Pi[T]$; any projection map is $4$-Lipschitz. 

Choose any free factor system~$\B$ of $\Gamma$ \relA\ representing a $0$-simplex of $\CFFS(\Gamma;\A)$, and consider its orbit $\B_i = \phi^i(\B)$. Applying surjectivity of the map~$\Pi$, we can choose a free splitting $T$ of $\Gamma$ \relA\ so that $\B \in \Pi(T)$. Consider the orbit $T_i = T_i \cdot \phi^i$; from inverted equivariance it follows $\B_i \in \Pi[T_{-i}]$. We may define a projection map $\pi$ by choosing $\pi[T] \in \Pi[T]$ arbitrarily for $[T] \not\in \{[T_i]\}_{i \in \Z}$ and defining $\pi[T_i]=\B_{-i} = \phi^{-i} \cdot \B$. It follows that
$$d_\CFFS(\B_i,\B_j) \le 4 \, d_\FS\bigl([T_{-i}],[T_{-j}]\bigr)
$$
Theorem~A gives us a constant $B=B(\Gamma;\A)$ and an upper bound $\tau_\phi \le B \log(\lambda_\phi)$, and it follows that $\tau^\CFFS_\phi \le 4 B \log(\lambda_\phi)$. Setting $B^\CFFS = 4B$, the proof is done.
\end{proof}

\subsection{Orbits of large diameter in $\CFFS(\Gamma;\A)$: A conjecture, proved for \hbox{$\Gamma=F_n$.}} 
\label{SectionFFLargeOrbits}

The conjectural analogue for $\CFFS(\Gamma;\A)$ of the implication (4)$\implies$(1) of Theorem~B is the following statement, which we shall prove in the ``classical'' case $\Gamma=F_n$:

\begin{conjecture}
\label{ConjectureFourImpliesOne}
For any group $\Gamma$ and any free factor system~$\A$ of $\Gamma$ there exists a constant $\Omega^\CFFS = \Omega^\CFFS(\Gamma;\A) > 0$ such that for all $\phi \in \Out(\Gamma;\A)$, if every orbit of $\phi$ has diameter $\ge \Omega^\CFFS$ then $\phi$ is fully irreducible \relA.
\end{conjecture}

\begin{theorem} 
\label{TheoremClassical}
In the ``classical'' case $\Gamma = F_n$, Conjecture~\ref{ConjectureFourImpliesOne} is true.
\end{theorem}

The proof of this theorem takes up the rest of Section~\ref{SectionFFLargeOrbits}. For a while we express the proof in the general language of $\Out(\Gamma;\A)$, switching to $\Out(F_n;\A)$ when we need concepts of \emph{geometric strata}\footnote{Regarding geometric strata in the general setting of $\Out(\Gamma;\A)$, see comments in the introduction of \cite{\LymanCTTag}.} of relative train track maps that are not known in the general case.

The proof is structured with several reduction steps, in which a certain special case is settled by exhibiting one of three alternatives: a proof that $\phi$ is fully irreducible \relA; or an orbit the diameter of which has a specific upper bound; or a contradiction to an earlier reduction step. The complementary hypothesis of that special case may then be adopted as a further assumption which holds for the rest of the proof, 

\medskip\noindent
\textbf{Special Case \ref{ReductionIrreducible}:} If $\phi \in \Out(\Gamma;\A)$ is reducible rel~$\A$ then there is a non-filling $\phi$-invariant free factor system of $\Gamma$ \relA\ that is distinct from~$\A$, which gives us a diameter~$0$ orbit of the action of $\phi$ on $\CFFS(\Gamma;\A)$. 

\medskip
\noindent
We may therefore assume:

\begin{reduction}[Irreducibility] 
\label{ReductionIrreducible}
$\phi$ is irreducible \relA. \qed
\end{reduction}

We now fix a \emph{(forward) rotationless} power $\phi^k$ ($k \ge 1$); it follows that for every free factor system $\B$ of $\Gamma$ \relA, $\B$ is $\phi$-periodic if and only if $\B$ is fixed by $\phi^k$ (see \cite[Definition 3.13, Lemma 3.30, and Lemma 4.43]{\recognitionTag} for $\Gamma=F_n$; and for the general case see \cite[Section 5]{\LymanCTTag}). 

Choose $\B$ to be a maximal, nonfull, $\phi^k$-invariant free factor system \relA. There is a function $\iota \from \A \mapsto \B$ which is well-defined by the requirement that for each $[A] \in \A$ and $[B] \in \B$, the equation $\iota[A]=[B]$ holds if and only if $A$ is conjugate to a subgroup of~$B$; well-definedness of this function uses mutual malnormality of~$\B$ (see \Subgroups(I) Lemma 2.1). Clearly $\A \sqsubset \iota(\A) \sqsubset \B$, and so $\iota(\A)$ is a free factor system \relA. Since $\A$ is $\phi$-invariant, it follows that $\iota(\A)$ is $\phi$-invariant. If $\iota(\A) \ne \A$ then $\iota(\A)$ witnesses that $\phi$ is reducible \relA, contradicting Reduction~\ref{ReductionIrreducible}. It follows that $\iota(\A)=\A$. To put this another way, $\A$ is a~\emph{subset} of~$\B$. 

\medskip\noindent
\textbf{Special Case \ref{ReductionStrictInc}:} If $\A=\B$ then, by maximality of~$\B$, $\phi$ is fully irreducible \relA. 

\medskip
\noindent
We may therefore assume: 

\begin{reduction}[Strict Inclusion] 
\label{ReductionStrictInc}
$\A$ is a proper subset of~$\B$. \qed
\end{reduction}

Using $\phi$-invariance of $\A$ we may conclude:

\medskip\noindent
\textbf{Special Case \ref{ReductionNoninvariance}:} If there exists a nonempty, $\phi$-invariant subset $\C \subset \B-\A$ then $\A \union \C$ is a nonfilling, $\phi$-invariant free factor system \relA\ that is not equal to $\A$, contradicting the \emph{Irreducibility Assumption}. 

\medskip
\noindent
We may therefore assume: 

\begin{reduction}[Non-Invariance] 
\label{ReductionNoninvariance}
No nonempty subset of $\B-\A$ is invariant under $\phi$. To put it another way, $\A$ is the maximal $\phi$-invariant subset of $\B$. 
\qed
\end{reduction}
Applying Reduction~\ref{ReductionNoninvariance}, the action of $\phi$ on the set of conjugacy classes of nonatomic free factors induces a partition of the set $\B-\A$ into \emph{maximal partial $\phi$-orbits} each of which is a sequence of the form $[B_1],\ldots,[B_J]$ ($J \ge 1$) such that $\phi^\inv[B_1] \not\in \B$, and $\phi[B_j]=[B_{j+1}]$ for $1 \le j < J$, and $\phi[B_J] \not\in\B$. It follows that ``$\phi$-saturation of $\B-\A$'' --- meaning the set $\bigcup_{i \in \Z} \phi^i(\B-\A)$ --- must contain at least one additional element per maximal partial orbit in $\B-\A$. Bringing the atomic free factors $\A$ into the picture, the set \hbox{$\{[B]  \union \A\suchthat [B] \in \B-\A\}$} is a subset of the \hbox{$0$-skeleton} of $\CFFS(\Gamma;\A)$, this subset has a corresponding partition into maximal partial $\phi$-orbits, and its $\phi$-saturation contains at least one additional element per maximal partial orbit.

Our strategy in the remainder of the proof is to carry out further reductions, eventually leading to a case where the $\phi$-saturation of $\B-\A$ is a single finite $\phi$-orbit obtained by adding just \emph{one more element} to the set $\B-\A$; see below, particularly item~\pref{FactOrbitKPlusOne} under the heading \emph{Further facts about geometric strata}. The description of this one orbit will be sufficiently explicit that we obtain a concrete bound on the diameter of the corresponding orbit in the $0$-skeleton of $\CFFS(\Gamma;\A)$.

\medskip\noindent
\textbf{Special Case \ref{ReductionFillingLam}:} If no attracting lamination of $\phi$ fills $\Gamma$ rel~$\A$ then we may apply the implication (4)$\implies$(1) of Theorem B, and so with the constant $\Omega$ from that theorem it folllows that the action of $\phi$ on the relative free splitting complex $\FS(\Gamma;\A)$ has an orbit $\{[T] \cdot \phi^i\}$ of diameter $\le \Omega$. By the same argument as in the proof of Proposition~\ref{PropUpperBoundFF} we obtain a $4$-Lipschitz map $\pi$ taking the $\FS(\Gamma;\A)$ orbit $\{[T]\phi^i\}_{i \in \Z}$ to the set $\{\phi^{-i} \cdot \pi[T]\}_{i \in \Z} = \{\phi^i \cdot \pi[T]\}_{i \in \Z}$, and the latter set is an orbit of $\phi$ in $\CFFS(\Gamma;\A)$ having diameter bounded above by $4 \, \Omega$.

\smallskip
\noindent
We may therefore assume:

\begin{reduction}[A filling lamination] 
\label{ReductionFillingLam}
$\phi$ has an attracting lamination $\Lambda$ that fills~$\Gamma$ \relA. Furthermore, $\phi(\Lambda)=\Lambda$. \qed
\end{reduction}
\noindent
To prove the ``Furthermore'' statement, the function that associates to each attracting lamination of $\phi$ its free factor support is a $\phi$-equivariant injection to the set of nonatomic free factor systems of $\Gamma$ rel~$\A$. Since the free factor support \relA\ of $\Lambda$ is $\Gamma$, which is $\phi$-invariant, it follows that $\Lambda$ is $\phi$-invariant. 

\medskip
 
We now specialize to the case $\Gamma=F_n$. This allows us to work with relative train track representatives defined on \emph{marked graphs} --- as opposed to \emph{marked graphs-of-groups} as in \cite{\LymanRTTTag,\LymanCTTag}, and to \emph{Bass-Serre trees} as elsewhere in this work --- and to use concepts of geometric strata. 

We assume familiarity with the special class of relative train track representatives known as CT's, short for ``completely split relative train track representatives''. For full details see \cite{\recognitionTag}, or see \Subgroups\ Part I Section~1 for a thorough but terse review. Recall in particular that every element of $\Out(F_n)$ has a positive power that is rotationless, and every rotationless element of $\Out(F_n)$ is represented by a CT \cite[Section 4.5]{\recognitionTag}. 

Let $\psi = \phi^k \in \Out(F_n;\A) \subgroup \Out(F_n)$ be a rotationless power of $\phi$ ($k \ge 2$), and so $\Lambda$ is also a filling attracting lamination of $\psi$. Let $f \from G \to G$ be a CT representative of $\psi$ defined on a marked graph $G$ equipped with filtration $\emptyset = G_0 \subset G_1 \subset \cdots \subset G_R=G$, such that the free factor system $\B$ is represented by a filtration element that (by maximality of $\B$) we may take to be $G_{R-1}$, and such that the free factor system $\A$ is represented by some $G_a$ with $0 \le a < R-1$. From Reduction~\ref{ReductionStrictInc} we know that $\A = [G_a]$ is a subset of $\B = [G_{R-1}]$ and that $\B-\A$ is nonempty. It follows that $G_a$ is a union of components of $G_{R-1}$ and that $G_{R-1}-G_a$ is a nonempty union of components of $G_{R-1}$, and we get a corresponding pair of decompositions
$$(*) \qquad\qquad \underbrace{G_{R-1}}_{\B} = \bigl(\underbrace{G_{R-1} - G_a}_{\B-\A}\bigr) \,  \sqcup \underbrace{G_a}_{\A} \,\,   \qquad\qquad \hphantom{(*)}
$$
Since $\Lambda$ fills $F_n$ relative to~$\A=[G_a]$, it follows that the top stratum $H_R = G \setminus G_{R-1}$ is the EG-aperiodic stratum that is associated to~$\Lambda$. 

\subparagraph{Geometric models.} (\Subgroups\ Part I Section 2) To say that the EG stratum $H_R$ is \emph{geometric} means that it has a \emph{geometric model} (relative to~$f$), as defined in \Subgroups(I) Section 2.1.3. Because we are in the situation where $H_R$ is the top stratum, a geometric model is the same thing as a \emph{weak geometric model} as defined in \Subgroups(I) Section 2.1.2. At risk of confusion, we shall follow the latter definition but drop the ``weak'' adjective. 

The description of a geometric model comes in two parts: \emph{static data}; and \emph{dynamic data}. 

The static data starts with a compact, connected surface $S$ with nonempty boundary and one special boundary component denoted $\bdy_0 S$; the full component decomposition of $\bdy S$ is denoted 
$$\bdy S = \bdy_0 S \union\cdots\union \bdy_M S \quad M \ge 0
$$ 
Next, for each $1 \le m \le M$ one is given a homotopically nontrivial map $\alpha_m \from \bdy_m S \to G_{R-1}$, and using these as attaching maps one forms a quotient map
$$q \from S \sqcup G_{R-1} \to Y
$$
and so we may therefore regard $Y$ as a $2$-complex. Next, one is given an extension of the restricted embedding $q \restrict G_{R-1} \hookrightarrow Y$ to an embedding $G \hookrightarrow Y$ such that $G \intersect \bdy_0 S$ is a single point~$p$. Finally, one is given a closed, locally injective path $\rho \from [0,1] \to G \hookrightarrow Y$ based at $p$ that is homotopic rel~$p$ to a locally injective path that goes exactly once around the circle $\bdy_0 S$. It follows that there exists a deformation retraction $d \from Y \to G$ such that $d \restrict \bdy_0 S$ is a parameterization of $\rho$. We may~regard $Y$ as a ``marked $2$-complex'' by using the deformation retraction $d \from Y \to G$ to identify $\pi_1(Y,p) \approx \pi_1(G,p) \approx F_n$. The restricted map $d \from \bdy_0 S \to G$ induces an injection $\pi_1 (\bdy_0 S,p) \mapsto \pi_1(G,p) \approx F_n$ well defined up to conjugacy in $F_n$; let $[\bdy_0 S]$ denote the corresponding subgroup conjugacy class, and so $[\bdy_0 S] = [\<\rho\>]$ where $\<\rho\>$ denotes the infinite cyclic subgroup of $\pi_1(G,p) \approx F_n$ generated by path homotopy class of the closed path~$\rho$.

The dynamic data of the geometric model consists of a pseudo-Anosov homeomorphism $F \from S \to S$ fixing the point $p$ and preserving orientation on $\bdy_0 S$, together with a homotopy equivalence $h \from Y \to Y$, such that the following diagram is homotopy commutative rel~$p$:
$$\xymatrix{
S \ar[r]^q \ar[d]_{F}   &  Y \ar[r]^d \ar[d]^h    & G \ar[d]^{f}         \\
S \ar[r]^q                    & Y \ar[r]^d    & G
} \qquad \qquad\hphantom{(*)}
$$
The path $\rho$ is therefore a \emph{Nielsen path} of~$f$, meaning that $f \circ \rho$ is homotopic rel endpoints to $\rho$. 

%By \Subgroups(I) Fact 2.3, a geometric model exists if and only if there is a closed Nielsen path which intersects the interior of $H_R$.

The existence of the pseudo-Anosov homeomorphism $F$ puts a useful constraint on $S$ which we note here for later use:

\begin{fact}
\label{FactSurfaceConstraint}
In any geometric model for $H_R$, the surface $S$ is not a sphere with $\le 3$ holes.
\end{fact}
\noindent
We also know that $S$ is not a projective plane with $\le 2$ holes, but we will not need that fact.

\subparagraph{The nonattracting subgroup system.} (\Subgroups\ Part III Section 1) Continuing with the rotationless power $\psi = \phi^k \in \Out(\Gamma;\A)$ and its filling attracting lamination $\Lambda$ rel~$\A$, we next describe the nonattracting subgroup system $\Ana(\Lambda,\psi)$, expressing it in terms of the choice of CT representative $f \from G \to G$ of $\psi$ that was made earlier;
%whose top stratum is the EG-stratum associated to $\Lambda$; 
as reviewed a bit later, $\Ana(\Lambda,\psi)$ is well-defined independent of that choice. 

Like a free factor system, $\Ana(\Lambda,\psi)$ is a finite set each of whose elements is a conjugacy class of subgroups of~$\Gamma$, as follows:

\begin{fact}[\Subgroups\ Part III, Definition 1.2, ``The case of a top stratum''] 
\label{FactAnaDescribed} \quad\\
The nonattracting subgroup system $\Ana(\Lambda,\psi)$ is described in two cases: 
\begin{description}
\item[The nongeometric case.] $H_R$ is \emph{not} geometric $\iff$ there does \emph{not} exist a closed indivisible Nielsen path based at an interior point of $H_R$ $\iff$ $\Ana(\Lambda,\psi)$~is a free factor system of~$F_n$. In this case we have
$$\Ana(\Lambda,\psi) = [G_{R-1}] = \B 
$$
\item[The geometric case.] $H_R$ is geometric $\iff$ there exists a closed indivisible Nielsen path $\rho$ based at a point in the interior of $H_R$ $\iff$ $\Ana(\Lambda,\psi)$ is \emph{not} a free factor system of $F_n$. In this case, using a geometric model for $H_R$ as described above, we have
$$\Ana(\Lambda,\psi) = \{[\bdy_0 S]\} \union [G_{R-1}]  =  \{[\bdy_0 S]\} \union \B  
$$
\end{description}
\end{fact}

\medskip\noindent
\textbf{Special Case \ref{ReductionNielsen}:} If $H_R$ is nongeometric then $\B = \Ana$ and so, by combining $\phi(\A)=\A$ and $\phi(\Ana)=\Ana$, it follows that $\phi(\B-\A) = \phi(\Ana-\A) = \Ana-A = \B-\A$, which contradicts Reductions~\ref{ReductionStrictInc} and~\ref{ReductionNoninvariance}.

\medskip
We may therefore assume: 
\begin{reduction}[Geometricity]
\label{ReductionNielsen}
$H_R$ is geometric.
\end{reduction}

\paragraph{The complementary subgraph $L$.} By \Subgroups\ Part III, Corollary 1.9 (2), the nonattracting subgroup system $\Ana(\Lambda;\psi)$ is well-defined independently of the choice of CT representative $f \from G \to G$. From here onwards we abbreviate $\Ana=\Ana(\Lambda;\psi)$. 

In $Y$ the \emph{complementary subgraph} $L \subset Y$ is defined as 
$$L = Y - \text{interior}(S) = \bdy_0 S \sqcup G_{R-1}
$$
It follows that $\bdy_0 S$ is a component of $L$, and so $G_{R-1}$ is the union of its remaining components. It also follows that $q$ induces a map of pairs 
$$q \from (S,\bdy S) \to (Y,L)
$$ 
having the property that $q(\bdy_m S) \subset L_{k(m)}$. 
\begin{enumerate}
\item\label{ItemGRMinusOne}
 (\Subgroups\ Part I, Fact 2.3) \, The filtration element $G_{R-1}$ has no contractible components, and so $L$ has no contractible components.
\end{enumerate}
Combining this with the open decomposition $G_{R-1} = (G_{R-1} - G_a) \sqcup G_a$ given earlier in $(*)$, the component decomposition of $L$ can be denoted as follows:
\begin{align*}
(**) \qquad\qquad L \, &= \underbrace{L_0}_{\bdy_0 S} \, \sqcup \, \underbrace{(L_1 \sqcup \cdots \sqcup L_I)}_{G_{R-1}-G_a} \, \sqcup \, \underbrace{(L_{I+1} \sqcup\cdots\sqcup L_{I+K})}_{G_a} 
\end{align*}
\begin{enumeratecontinue}
\item (\Subgroups\ Part I, Lemma 2.7)
\begin{enumerate}
\item\label{ItemLComponents}
For each $k=0,\ldots,I+K$ the inclusion $L_k \hookrightarrow Y$ induces an injection $\pi_1 L_k \hookrightarrow \pi_1(Y) \approx F_n$. Let $[\pi_1 L_k]$ denote the conjugacy class of the image of this injection. 
\item\label{ItemLDistinct}
For any $k \ne l \in \{0,\ldots,I+K\}$ we have $[\pi_1 L_k] \ne [\pi_1 L_l]$. 
\item The restricted quotient map $S \xrightarrow{q} Y$ induces an injection $\pi_1 S \hookrightarrow \pi_1(Y) \approx F_n$. Let $[S]$ denote the conjugacy class of the image of this injection.
\end{enumerate}
\end{enumeratecontinue}
Putting this together with $(*)$ and $(**)$ we obtain a bijection between the set of connected components of $L$ and the nonattracting subgroup system $\Ana$, and we may denote this bijection as follows, using notations $B_i = \pi_1 L_i \subgroup F_n$ for $0 \le i \le I$ and $A_i = \pi_1 L_{I+i} \subgroup F_n$ for $0 \le i \le K$:
$$\Ana \, = \, \{\underbrace{[\pi_1 L_0]}_{[B_0]}\} \, \union \, \underbrace{\bigl\{\underbrace{[\pi_1 L_1]}_{[B_1]},\ldots,\underbrace{[\pi_1 L_I]}_{[B_I]}\bigr\}}_{\B-\A} \, \union \, \underbrace{\{\underbrace{[\pi_1 L_{I+1}]}_{[A_1]}, \ldots, \underbrace{[\pi_1 L_{I+K}]}_{[A_K]}\}}_{\A}
$$
\begin{enumeratecontinue}
\item\label{ItemLamNatural}
 (\Subgroups\ Part III, Corollary 1.9 (3)) For any $\theta \in \Out(\Gamma;\A)$, \, $\theta(\Lambda)$ is an attracting lamination of $\theta\psi\theta^\inv$, and 
$$\theta(\Ana) = \theta(\Ana(\Lambda,\psi)) = \Ana(\theta(\Lambda),\theta \psi \theta^\inv)
$$
\end{enumeratecontinue}
We next draw a conclusion regarding $\phi$ itself, as opposed to its rotationless power $\psi=\phi^k$:
\begin{enumeratecontinue}
\item\label{ItemPhiAnaLambdaPsi}
$\phi(\Ana)=\Ana$.
\end{enumeratecontinue}
This follows from item~\pref{ItemLamNatural} together with Reduction~\ref{ReductionFillingLam} and a short computation:
$$\phi(\Ana(\Lambda,\psi)) = \Ana(\phi(\Lambda),\phi \phi^k \phi^\inv)) = \Ana(\Lambda,\phi^k) = \Ana(\Lambda,\psi)
$$

\bigskip

Knowing that both $\A$ and $\Ana$ are $\phi$-invariant, it follows that the set
$$\Ana - \A = \{[B_0]\} \union \underbrace{\{[B_1],\ldots,[B_I]\}}_{\B-\A}  \qquad (I \ge 1)
$$ 
is also $\phi$-invariant. It follows that $\B-\A$ is a single partial orbit of cardinality $I \ge 1$, and that its $\phi$-saturation is obtained by adding a single element $[B_0]$ to form a single $\phi$-orbit $\Ana-A$ of cardinality $I+1 \ge 2$. After permuting the index set $\{1,\ldots,I\}$, this orbit may be written a  
$$\Ana-\A=\{[B_0],\underbrace{[B_1],\ldots,[B_I]}_{\B-\A}\} \quad\text{with}\quad \phi[B_j]=[B_{j+1}] \quad \text{for all $j$ modulo $I+1$}
$$
%\marginparMichael{3 lines before (5)}
Knowing that $B_0$ has rank~$1$, and knowing also that $\{[B_1]\} \union \A$ is a $0$-simplex of $\CFFS(F_n;\A)$ --- i.e.\ a non-full, non-atomic free factor system of $F_n$ rel~$\A$ --- it follows that these properties are inherited around the whole $\phi$-orbit $\{[B_0],\ldots,[B_I]\}$.
\begin{enumeratecontinue}
\item \label{FactOrbitKPlusOne}
For each $0 \le j \le I$ the subgroup $B_j$ has rank~$1$, and $ \{[B_j]\} \union \A$ is a $0$-simplex of $\CFFS(F_n;\A)$; it follows that $B_j$ is a rank~$1$ free factor of a cofactor of~$A$. These vertices thus form a $\phi$-orbit of the $0$-skeleton of $\CFFS(F_n;\A)$ of cardinality $I+1 \ge 2$:
$$\{[B_0]\}\union \A   \,\, , \,\,  \{[B_1]\}\union\A  \,\, , \,\ldots\, , \,\, \{[B_I]\}\union \A 
$$
\end{enumeratecontinue}
\noindent

\medskip\noindent
\textbf{Special Case \ref{ReductionOrbitWithTwo}:} 
%\marginparMichael{k is overloaded. \\ line 2 dangling ``We''}
Suppose that the cardinality of the orbit $\{[B_0],\ldots,[B_I]\}$ satisfies $I+1 \ge 3$. Consider any two indices $i \ne j \in \{0,\ldots,I\}$ and any other index $m \in \{0,\ldots,I\} - \{i,j\}$. We~have $[B_{i-m}]$, $[B_{j-m}] \in \B$ (with subscripts modulo $I+1$), and so $\{[B_{i-m}],[B_{j-m}]\}  \union \A \sqsubset \B$ and hence $\{[B_{i-m}],[B_{j-m}]\} \union \A$ is a vertex of $\CFFS(F_n;\A)$. The set
$$\{[B_i],[B_j]\}  \union \A = \phi^m(\{[B_{i-m}],[B_{j-m}]\} \union \A)
$$
is therefore a vertex of $\CFFS(F_n;\A)$, and we obtain a length $2$ edge path in $\CFFS(F_n;\A)$:
$$\{[B_i]\}  \union \A \sqsubset \{[B_i],[B_j]\} \union \A  \sqsupset  \{B_j]\}  \union \A 
$$
The distance between $\{[\B_i]\} \union \A$ and $\{[B_j]\} \union \A$ is therefore $\le 2$. Since $i,j$ are arbitrary, this proves that the whole orbit under consideration has diameter~$\le 2$.

\medskip
%\marginparLee{Notation like $V[B_i]$ for the vertex $\{[B_i]\} \union \A$?}
We may therefore assume:
\begin{reduction}[Orbit of Two]
\label{ReductionOrbitWithTwo} The orbit $\Ana-\A$ has cardinality $I+1=2$, and so
$$\Ana \, = \, \{\underbrace{[\pi_1 L_0]}_{[B_0]}\} \, \union \, \underbrace{\bigl\{\underbrace{[\pi_1 L_1]}_{[B_1]}\bigr\}}_{\B-\A} \, \union \, \underbrace{\{\underbrace{[\pi_1 L_{2}]}_{[A_1]}, \ldots, \underbrace{[\pi_1 L_{K+1}]}_{[A_K]}\}}_{\A}
$$
The vertices $\{[B_0]\} \union \A$ and $\{[B_1]\} \union \A$ in $\CFFS(\Gamma;\A)$ form a $\phi$-orbit of cardinality~$2$.
\end{reduction}

\paragraph{Remaining goal:} We shall describe an edge path in $\CFFS(F_n;\A)$ of length~$\le 4$ between the two vertices in Reduction~\ref{ReductionOrbitWithTwo}, thus giving us a $\phi$-orbit of diameter~$\le 4$ in $\CFFS(F_n;\A)$. For this we need a few more pieces from the theory of geometric strata, in particular ``free boundary circles''; each of $L_0$ and $L_1$ will turn out to be free boundary circles, and we will use this fact together with a topological construction on $S$ to produce the desired edge path.

\paragraph{Dynamic data for $\phi$.} Recall that the dynamic data of the geometric model for $\psi = \phi^k$ is expressed primarily as a certain homotopy commutative diagram. Knowing from item~\pref{ItemPhiAnaLambdaPsi} that $\phi[\Ana]=\Ana$, we can apply the results cited in item~\pref{ItemHEPairs} just below, with conclusions regarding a homotopy commutative diagram for $\phi$ itself: 
\begin{enumeratecontinue}
\item\label{ItemHEPairs} (\Subgroups\ Part I, Prop. 2.20 and Lemma 2.21) 
There exists a homotopy equivalence of pairs $\Phi_Y \from (Y,L) \to (Y,L)$ and a homeomorphism $\Phi_S \from S \to S$ such that the following diagram is homotopy commutative in the category of pairs: 
$$\xymatrix{
(S,\bdy S) \ar[r]^q \ar[d]_{\Phi_S}   &  (Y,L) \ar[d]^{\Phi_Y}         \\
(S,\bdy S) \ar[r]^q                    & (Y,L)
}
$$
\end{enumeratecontinue}
Using~\pref{ItemHEPairs} we next derive the following:
\begin{enumeratecontinue}
\item\label{ItemIndexPermutation}
After re-indexing the components of $\bdy S$ using a permutation of the index set $\{0,1,\ldots,M\}$ that fixes the index $0$, for each $i \in \{0,1\}$ the boundary circle $\bdy_i S$ is the unique component of $\bdy S$ such that $q(\bdy_i S) \subset L_i$, and furthermore
$$\Phi_S(\bdy_0 S) = \bdy_1 S \qquad\text{and}\qquad \Phi_S(\bdy_1 S) = \bdy_0 S
$$
\end{enumeratecontinue}
To derive~\pref{ItemIndexPermutation}, consider the component sets 
$$\pi_0 (\bdy S) = \{\bdy_0 S,\ldots,\bdy_M S\} \quad\text{and}\quad \pi_0 L = \{L_0,L_1, \ldots, L_{K+1}\}
$$
Applying~\pref{ItemHEPairs}, the restriction of $\Phi_S$ to $\bdy S$ induces a permutation $\Phi_{S*} \from \pi_0 (\bdy S) \to \pi_0 (\bdy S)$, and the restriction of $\Phi_Y$ to $L$ induces a permutation $\Phi_{Y*} \from \pi_0 L \to \pi_0 L$. Also, the map \hbox{$q \from \bdy S \to L$} induces a component map $q_* \from \pi_0 (\partial S) \to \pi_0 L$ given by $q_*(\bdy_m S) = L_{k(m)}$, and $q_*$ is equivariant with respect to those two permutations, meaning that $q_* \circ \Phi_{S*}(\bdy_m S) = \Phi_{Y*} \circ q_*(\bdy_m S)$ for each $m=0,\ldots,M$. Also, we know from Reduction~\ref{ReductionOrbitWithTwo} that $\phi[B_0]=[B_1]$ and that $\phi[B_1]=[B_0]$, and it follows that $\Phi_{Y*}(L_0) = L_1$ and that $\Phi_{Y*}(L_1)=L_0$. We also know already, from the definition of the complementary subgraph $L$, that $\bdy_0 S$ is the \emph{unique} component of $\bdy S$ such that \hbox{$q_*(\bdy_0 S) = L_0$.} It follows that $\Phi_S(\bdy_0 S)$ is the unique component of $\bdy S$ whose $q_*$-image is $\Phi_{Y*}(L_0)=L_1$, and after an index permutation as stated in \pref{ItemIndexPermutation} we may denote $\bdy_1 S = \Phi_S(\bdy_0 S)$, and so $q_*(\bdy_1 S) = L_1$ (we do \emph{not} know at this stage that $q(\bdy_1 S) = L_1$). The two equations of \pref{ItemIndexPermutation} then follow immediately.

\subparagraph{Free boundary circles.} (\Subgroups\ Part 1 Section 2.6) To say that $\bdy_m S$ is a \emph{free boundary circle} means that the attaching map $\alpha_m \from \bdy_m S \to L_{k(m)}$ is a homotopy equivalence, and that $L_{k(m')} \ne L_{k(m)}$ for all $m' \ne m \in \{0,\ldots,M\}$. It follows, from the definition of the complementary subgraph $L$, that $\bdy_0 S$ is a free boundary circle. Applying~\pref{ItemHEPairs}, the condition that $\bdy_m S$ be a free boundary circle is invariant under the permutation $\Phi_{S*} \from \pi_0(\bdy S) \to \pi_0(\bdy S)$, and therefore $\bdy_1 S$ is also a free boundary circle.

A free boundary circle $\bdy_m S$ is said to be \emph{strongly free} if $\alpha_m \from \bdy_m S \to L_{k(m)}$ is a homeomorphism. If this holds then we can use $\alpha_m$ to identify $\bdy_m S$ with $L_{k(m)}$, and it follows that there is an open neighborhood $U_m \subset Y$ of $\bdy_m S$ such that $U_m$ is a surface-with-boundary and $\bdy U_m = \bdy_m S$.

Clearly $\bdy_0 S$ is strongly free in $Y$. In the general theory of geometric models presented in \Subgroups\ Part I, a free boundary circle need not be strongly free. Nevertheless we may assume that $\bdy_1 S$ is also strongly free in $Y$, by making the following small adjustment to~$Y$.\footnote{There is a relative train track argument which uses that $H_R$ is the top stratum to prove that $\bdy_1 S$ is strongly free in $Y$, but the reduction argument we present is simpler.} Since the finite graph $L_1$ is homotopy equivalent to a circle, there exists a unique subgraph $L'_1 \subset L_1$ which is homeomorphic to a circle, and $L_1$ deformation retracts to $L'_1$. We may therefore homotope the attaching map $\alpha_1$ to a homeomorphism $\alpha'_1 \from \bdy_1 S \to L'_1$, and we may then delete the subset $L_1-L'_1$, resulting in a new 2-complex $Y'$. Once this is done, there exist arbitrarily small open neighborhoods $U \subset Y$ of $L_1$ and $U' \subset Y'$ of $L'_1$, and a homotopy equivalence $Y \to Y'$, and a homotopy inverse $Y' \to Y$, such that the homotopy from the composition $Y \to Y' \to Y$ to the identity on $Y$ is stationary on $Y-U$, and the homotopy from the composition $Y' \to Y \to Y'$ to the identity on $Y'$ is stationary on $Y'-U'$. Clearly both of $\bdy_0 S$ and $\bdy_1 S$ are strongly free in~$Y'$. Replacing $Y$ with $Y'$, for the rest of the argument we may therefore assume that $\bdy_0 S$ and $\bdy_1 S$ are both strongly free in~$Y$.

\paragraph{Conclusion of the proof.} In $S$ consider the subset $\hat\bdy S = \bdy S - (\bdy_0 S \union \bdy_1 S) = \bigcup_{m=2}^M \bdy_m S$; and in $L \subset Y$ consider the subset $\hat L = L - (L_0 \union L_1) = \bigcup_{j=2}^{I+K} L_j$. Because $\bdy_0 S$ and $\bdy_1 S$ are both strongly free in $Y$, we may regard $Y$ as the quotient of the disjoint union $S \sqcup \hat L$ obtained by attaching $\hat\bdy S$ to $\hat L$ using the attaching map 
$$\hat\alpha = \alpha_2 \sqcup \cdots \sqcup \alpha_M \from \hat\bdy S \to \hat L
$$
Let $\hat q \from S \sqcup \hat L \to Y$ denote the corresponding quotient map.

A \emph{subgraph of $S$ rel $\hat \bdy S$} is a finite embedded graph $\Gamma \subset S$ such that $\hat\bdy S \subset \Gamma$, and such that for each $i=0,1$, either $\Gamma$ contains $\bdy_i S$ or is disjoint from $\bdy_i S$. Consider any subgraph $\Gamma$ of $S$ rel $\hat\bdy S$. Associated to $\Gamma$ is a \emph{subgraph of $Y$ rel $\hat L$}, namely $\hat\Gamma = q(\Gamma \union \hat L)$. To say that $\Gamma$ is a \emph{spine of $S$ rel~$\hat\bdy S$} means that $\Gamma$ is a strong deformation retract of $S$; this is equivalent to saying $\hat\Gamma$ being a strong deformation retract of $Y$ which we express by saying that $\hat\Gamma$ is a \emph{spine of $Y$ rel~$\hat L$}; this is also equivalent to the complement $\Gamma^c = S-\Gamma = Y - \hat\Gamma$ satisfying one of the following: 
\begin{itemize}
\item $\Gamma^c$ consists of \emph{two} half-open annuli, one containing $\bdy_0 S$ and the other containing $\bdy_1 S$;~\emph{or}
\item There exist $i \ne j \in \{0,1\}$ such that $\Gamma^c$ is a \emph{single} half-open annulus containing $\bdy_i S$ but~not~$\bdy_{j} S$.
\end{itemize}
More generally, given a subgraph $\Gamma$ of $S$ rel~$\hat\bdy S$, to say that $\Gamma$ is a \emph{subspine of $S$ rel~$\hat\bdy S$} means that $\Gamma$ is a subgraph of some spine of $S$ rel~$\hat\bdy S$; this holds if and only if $\hat\Gamma$ is a subgraph of some spine of $Y$ rel~$\hat L$. In general if $\Gamma$ is a subspine rel~$\hat\bdy S$ then $\hat\Gamma$ represents a free factor system of $F_n$ rel~$\A$ that we denote $[\hat\Gamma]$. Here are some examples:
\begin{itemize}
\item $[\hat\Gamma]$ is filling if and only if $\Gamma$ is a spine rel~$\hat \bdy S$; 
\item $[\hat\Gamma]=\A$ if and only if $\Gamma$ deformation retracts to $\hat\bdy S$.
\item For $i \ne j \in \{0,1\}$, $\Gamma_i = \bdy S - \bdy_j S = \bdy_i S \union \hat\bdy S$ is a subspine rel~$\hat\bdy S$, and the following hold:
\begin{itemize}
\item $[\hat\Gamma_i] = \{[B_i]\} \union \A$.
\item $\Gamma_i$ is not a spine of $S$ rel~$\hat\bdy S$, and so $[\hat\Gamma_i]$ is non-filling.
\end{itemize}
\end{itemize}
Thus $[\hat\Gamma_0]$ and $[\hat\Gamma_1]$ form the $\phi$-orbit of cardinality~$2$ in the $0$-skeleton of $\CFFS(\Gamma;\A)$, that we have been considering since Reduction~\ref{ReductionOrbitWithTwo}. 

We shall interpolate between $\Gamma_0$ and $\Gamma_1$ using three more subspines of $S$ rel~$\hat\bdy S$, none of which is a spine rel~$\hat\bdy S$ nor deformation retracts to $\hat\bdy S$, and with inclusion relations as follows:
$$\Gamma_0 \subset \Delta_0 \supset \Theta \subset \Delta_1 \supset \Gamma_1
$$
From this interpolation we obtain the following edge path of length $4$ in $\CFFS(F_n;\A)$, and so we will be done:
$$\{[B_0]\}  \union \A= [\hat\Gamma_0] \, \sqsubset \, [\hat\Delta_0] \, \sqsupset \, [\hat\Theta] \, \sqsubset \, [\hat\Delta_1] \, \sqsupset \,  [\hat\Gamma_1] =  \{[B_1]\} \union \A 
$$
To define $\Theta$, choose $\theta \subset S$ to be either a nonseparating simple closed curve in the interior of $S$ or a properly embedded nonseparating arc in $S$ with $\bdy\theta \subset \hat\bdy S$; and then let $\Theta = \hat\bdy S \union \theta$. In more detail:
\begin{itemize}
\item If $S$ is nonorientable, choose $\theta$ to be an orientation reversing simple closed curve.
\item If $S$ is orientable of genus $\ge 1$ choose $\theta$ to be a nonseparating simple close curve.
\item If $S$ is orientable of genus $0$ then we know from Fact~\ref{FactSurfaceConstraint} that $\bdy S$ has $\ge 4$ components and so we can choose $\alpha$ to be a properly embedded arc whose two endpoints lie on two distinct components of $\hat\bdy S = \bdy S - (\bdy_0 S \union \bdy_1 S)$. 
\end{itemize}
It is clear from this construction that $\Theta$ does not deformation retract to $\hat\bdy S$, and also that the connected complement $S - \theta$ contains each of $\bdy_0 S$ and $\bdy_1 S$ and is not a half-open annulus, and so $\Theta$ is not a spine rel~$\hat\bdy S$. Thus $[\hat\Theta]$ is indeed a vertex of $\CFFS(\F_n;\A)$.

Next let $\Delta_i = \Theta \union \bdy_i S$ ($i=0,1$). Again clearly $\Delta_i$ does not deformation retract to $\hat\bdy S$. Also, the surface $S-\Delta_i$ is connected and it is not homeomorphic to a half-open annulus, for the following reasons: the surface $S-\Delta_i$ does have exactly one boundary component, namely $\bdy_j S$ with $j \ne i \in \{0,1\}$; but the surface $S - (\Delta_i \union \bdy_j S)$ is not homeomorphic to an open annulus, because it has $\ge 3$ topological ends: one end compactified by $\bdy_0 S$; another end compactified by $\bdy_1 S$; and a third end compactified by the component of $\hat\bdy S \union \theta$ that contains $\theta$. It is clear from this construction that $\Delta_i \subset \Theta$ and $\Delta_i \subset \Gamma_i$, and we are done.

\subsection{Translation lengths in $\CFFS(\Gamma;\A)$: A conjectured lower bound.}
\label{SectionFFLowerBound}

The following conjecture is an analogue for $\CFFS(\Gamma;\A)$ of the lower bound in Theorem~$A$ for $\FS(\Gamma;\A)$:

\begin{conjecture}
\label{ConjectureFFTransLowerBound}
There exists a constant $A^\CFFS = A^\CFFS(\Gamma;\A) > 0$ such that for all $\phi \in \Out(\Gamma;\A)$, if $\phi$ is fully irreducible rel~$\A$ then the stable translation length of $\phi$ acting on $\CFFS(\Gamma;\A)$ is $\ge A^\CFFS$.
\end{conjecture}

Here is an idea for attacking this conjecture, in a manner that is analogous to the proof of the lower bound of Theorem~A. Our proof of the latter applied ``free splitting units'' along fold axes in $\FS(\Gamma;\A)$. One way to think of this situation is that when a fold path is parameterized by counting folds, free splitting units give that fold path the structure of a reparameterized quasigeodesic with uniform constants. We already know, from Theorem 6.8 at the very end of Part~I, that the image of a fold path under the projection map $\pi \from \FS(\Gamma;\A) \to \CFFS(\Gamma;\A)$ is a reparameterized quasigeodesic in $\CFFS(\Gamma;\A)$ with uniform constants. But is there a nice formula for such a reparameterization, a combinatorial formula expressed in terms of some kind of ``free factor units''? More to the point, is there such a formula that could be applied to fold axes to find the lower bound in Conjecture~\ref{ConjectureFFTransLowerBound}?

For this purpose, here is a proposed combinatorial formulation of ``free factor units'', expressed using the projection set map $\Pi \from \FS(\Gamma;\A) \to \CFFS(\Gamma;\A)$. 
%\marginparLee{Gotta think about this a bit more\ldots and perhaps call Mark. Maybe replace this. Forget about paths. Just use realized free factor systems, so $<1$ free factor unit means that there exists an invariant subforest $\upsilon \subset U_0$ such that $\F[\upsilon]$ is nontrivial and nonfull, and such that the nesting relation $\F[\upsilon]$ is \emph{properly} nested in $\F[f(\upsilon)]$. When free factor units are defined in this fashion, they are equivalent to ``coarser'' free factor units which say that for every such $\upsilon$, $f(\upsilon)=T_K$. Even better, the ``coarser'' units are equivalent to the ones where defined by restricting to the case that $\upsilon$ is a circuit.}
Given a pair of Grushko free splittings $S,T$ rel~$\A$, let us say that there is $\ge 1$ free factor unit between $S$ and $T$ if the sets $\Pi(S), \Pi(T) \subset \CFFS(\Gamma;\A)$ are disjoint; equivalently, there does not exist a nonfull, nonatomic free factor system $\F$~\relA\ which is visible in both $S$ and $T$. Consider next a pair of nonatomic, nonfull free factor systems $\B,\C$ rel~$\A$. Let us say that there are $\ge K$ free factor units between $\B$ and $\C$ if for every foldable map $f \from S \to T$ between a pair of Grushko free splittings $S$ and $T$ rel~$\A$ and for every Stallings fold factorization of $f$, if $\B$ is visible in $S$ and $\C$ is visible in $T$ then the given fold path has a subdivision into a concatenation of $\ge K$ fold subpaths such that between the beginning and the end of each such subpath there is $\ge 1$ free factor unit. 
%It is not hard to show that the method just described for counting free factor units is quasicomparable to the alternate ``coarser'' method, in which one redefines ``\,$\ge 1$ free factor unit between $S$ and $T$\,'' by first restricting attention to those $\F$ that are visible in~$S$ such that the extension $\A \sqsubset \F$ is elementary, and then requiring that the only free factor system \emph{relative to~$\F$} that is visible in $T$ is the full free factor system $\{[\Gamma]\}$. 
%In the case of~$F_n$, the alternate method of defining $\ge 1$ free factor unit is equivalent to saying that in the quotient marked graphs $S / F_n \mapsto T / F_n$, for each circuit in $S / F_n$ its straightened image in $T / F_n$ covers the whole of $T / F_n$.

\begin{description}
\item[Question:] Using the formulation proposed above, is the number of free factor units between $\B$ and~$\C$ quasicomparable to the distance between $\B$ and $\C$ in $\CFFS(\Gamma;\A)$?
\end{description}

Even if the answer to this question is \emph{yes}, there will still be some work needed in order to determine how it might be applied to Conjecture~\ref{ConjectureFFTransLowerBound}, perhaps by some analogue of the work done in Section~\ref{SectionLowerBound} where the lower bound of Theorem~A was proved by applying free splitting units.

On the other hand, the answer to this question could be \emph{no} in various scenarios, for instance if there existed a sequence of examples where the number of free factor units between $\B$ and $\C$ is unbounded but the distance between $\B$ and $\C$ in $\CFFS(\Gamma;\A)$ is bounded. In such a situation one might still hope that while the formulation given above for the property \emph{$\ge 1$ free factor unit} is evidently too weak, perhaps there is a stronger but still useful combinatorial formulation of that property.

%%%%%%%%%%%%%%%%%%%%%%%%%%%%%%%%
\bibliographystyle{amsalpha} 

%Michael: You will probably need your own bibliography statement if you ever want to run bibtex

%\bibliography statement for Lee's system
\bibliography{/Users/Lee/Dropbox/Handel_Lyman_Mosher_shared/Lee_bibtex_file/mosher.bib} 
%%%%%%%%%%%%%%%%%%%%%%%%%%%%%%%%

\end{document}